\documentclass[10pt]{amsart}
\usepackage{enumitem}

\usepackage{graphicx}
\usepackage{geometry}
\usepackage{appendix}
\usepackage{color}
\usepackage{tikz-cd}

\definecolor{refblue}{RGB}{0, 0, 153}
\definecolor{citegreen}{RGB}{0, 115, 0}
\definecolor{linkred}{RGB}{191, 26, 61}

\usepackage{leftidx}
\usepackage{graphicx}

\usepackage[all]{xy}
\usepackage{amssymb}
\usepackage{setspace}
\usepackage{stmaryrd}
\usepackage{amsthm}
\usepackage{colonequals}
\usepackage[colorlinks,linkcolor=refblue,citecolor=citegreen,urlcolor=linkred,bookmarks=false,hypertexnames=true]{hyperref} 
\usepackage[nameinlink]{cleveref}

\usepackage[backend=biber, doi=false, url=false, isbn=false, maxbibnames=50, style=alphabetic, giveninits=true]{biblatex}
\renewbibmacro{in:}{}

\newtheorem{theorem}{Theorem}[section]
\newtheorem{lemma}[theorem]{Lemma}
\newtheorem{conjecture}[theorem]{Conjecture}
\newtheorem{proposition}[theorem]{Proposition}
\newtheorem{corollary}[theorem]{Corollary}
\newtheorem*{corollary*}{Corollary}
\newtheorem{construction}[theorem]{Construction}

\newtheorem*{conjecture*}{Conjecture}

\newtheorem{maintheorem}{Theorem}

\newcounter{mainconj}
\newtheorem{mainconj}[mainconj]{Conjecture}

\theoremstyle{definition}
\newtheorem{definition}[theorem]{Definition}
\newtheorem*{definition*}{Definition}
\newtheorem{example}[theorem]{Example}

\theoremstyle{remark}
\newtheorem{remark}[theorem]{Remark}

\numberwithin{equation}{section}

\newcommand{\stackcite}[1]{\cite[\href{https://stacks.math.columbia.edu/tag/#1}{#1}]{stk}}

\usepackage{tikz}
\usetikzlibrary{shapes.geometric, arrows}

\newcommand{\bbA}{\mathbb A}

\newcommand{\C}{\mathbb C}

\newcommand{\G}{\mathbb G}

\newcommand{\bbN}{\mathbb N}

\newcommand{\Q}{\mathbb Q}
\newcommand{\R}{\mathbb R}

\newcommand{\bw}{\mathbf{w}}

\newcommand{\Z}{\mathbb Z}

\newcommand{\scrL}{\mathscr L}

\newcommand{\cA}{\mathcal{A}}

\newcommand{\cC}{\mathcal{C}}

\newcommand{\cE}{\mathcal{E}}
\newcommand{\cF}{\mathcal{F}}
\newcommand{\cG}{\mathcal{G}}
\newcommand{\cH}{\mathcal{H}}

\newcommand{\cM}{\mathcal{M}}

\newcommand{\cO}{\mathcal{O}}

\newcommand{\cQ}{\mathcal{Q}}
\newcommand{\cR}{\mathcal{R}}
\newcommand{\cS}{\mathcal{S}}
\newcommand{\cT}{\mathcal{T}}

\newcommand{\cV}{\mathcal{V}}
\newcommand{\cW}{\mathcal{W}}
\newcommand{\cX}{\mathcal{X}}

\newcommand{\cZ}{\mathcal{Z}}

\newcommand{\frakS}{\mathfrak S}

\newcommand{\frakb}{\mathfrak b}

\newcommand{\frakg}{\mathfrak g}

\newcommand{\frakm}{\mathfrak m}

\newcommand{\frakt}{\mathfrak t}

\newcommand{\bB}{\mathbf{B}}

\newcommand{\bG}{\mathbf{G}}
\newcommand{\bT}{\mathbf{T}}

\newcommand{\GIT}{/\hspace{-1mm}/}

\newcommand{\Fil}{\mathrm{Fil}}

\newcommand{\HT}{\mathrm{HT}}

\newcommand{\GL}{\operatorname{GL}} 
\newcommand{\SL}{\operatorname{SL}} 
\newcommand{\SO}{\operatorname{SO}} 
\newcommand{\Sp}{\operatorname{Sp}} 
\newcommand{\GSp}{\operatorname{GSp}} 
\newcommand{\Gm}{\operatorname{\G}_{\mathrm m}} 

\newcommand{\co}{\colon} 

\newcommand{\Aff}{\operatorname{Aff}} 
\newcommand{\CRing}{\operatorname{CRing}} 
\newcommand{\Gpd}{\operatorname{Gpd}} 
\newcommand{\Rep}{\operatorname{Rep}} 
\newcommand{\Set}{\operatorname{\mathrm{Set}}} 

\newcommand{\Spec}{\operatorname{Spec}} 
\newcommand{\Spf}{\operatorname{Spf}} 
\newcommand{\Spm}{\operatorname{Spm}} 

\newcommand{\gr}{\operatorname{gr}} 
\newcommand{\Gr}{\operatorname{Gr}} 

\newcommand{\ab}{\operatorname{ab}} 
\newcommand{\ad}{\operatorname{ad}} 
\newcommand{\Ad}{\operatorname{Ad}} 
\newcommand{\Art}{\mathfrak A} 
\newcommand{\Aut}{\operatorname{Aut}} 
\newcommand{\coker}{\operatorname{coker}} 
\newcommand{\colim}{\mathop{\mathstrut\rm colim}\limits} 
\newcommand{\cont}{\operatorname{cont}} 
\newcommand{\End}{\operatorname{End}} 
\newcommand{\Gal}{\operatorname{Gal}} 
\newcommand{\Hom}{\operatorname{Hom}} 
\newcommand{\id}{\operatorname{id}} 
\newcommand{\im}{\operatorname{im}} 
\newcommand{\Lie}{\operatorname{Lie}} 
\newcommand{\Pic}{\operatorname{Pic}} 
\newcommand{\pr}{\operatorname{pr}} 
\newcommand{\rank}{\operatorname{rank}} 
\newcommand{\Res}{\mathrm{Res}}
\newcommand{\Tot}{\operatorname{Tot}} 
\newcommand{\tr}{\operatorname{tr}} 
\newcommand{\red}{\operatorname{red}} 
\newcommand{\et}{\operatorname{\acute{e}t}} 
\newcommand{\eqto}{\xrightarrow{\sim}} 
\newcommand{\rhobar}{\overline\rho}
\newcommand{\Qp}{{\Q}_p}
\newcommand{\Zp}{{\Z}_p}
\newcommand{\Qpbar}{\overline\Q_p}

\newcommand{\rig}{\mathrm{rig}}
\newcommand{\tri}{\mathrm{tri}}
\newcommand{\vreg}{\mathrm{vreg}}
\newcommand{\cris}{\mathrm{cris}}

\newcommand{\cyc}{\mathrm{cyc}}
\newcommand{\dR}{\mathrm{dR}}
\newcommand{\pdR}{\mathrm{pdR}}
\newcommand{\Ga}{\mathbb G_{\mathrm a}}
\newcommand{\triv}{\mathrm{triv}}
\newcommand{\ver}{{\mathrm{ver}}}
\newcommand{\Vect}{\mathrm{Vect}}
\newcommand{\br}[1]{\llbracket #1\rrbracket}
\renewcommand{\hyphen}{\text{-}}
\newcommand{\Tors}{\mathrm{Tors}}
\newcommand{\Ru}{\mathrm R_{\mathrm u}}
\newcommand{\Mtri}{\mathcal M^{\triangle}}
\newcommand{\Wtri}{W^{\triangle}}

\newcommand{\GK}{{\cG_K}}

\newcommand{\BdR}{B_{\dR}}

\newcommand{\unr}{\mathrm{unr}}

\newcommand{\PG}{{(\varphi,\Gamma_K)}}
\newcommand{\GPG}{{\gG\text{-}(\varphi,\Gamma_K)}}

\newcommand{\PGcat}{{\Phi\Gamma}}

\newcommand{\uHom}{\underline{\Hom}}
\newcommand{\uAut}{\underline{\Aut}}
\newcommand{\uk}{\underline{k}}

\newcommand{\vareps}{{\varepsilon}}

\newcommand{\wh}{\widehat}
\newcommand{\ovl}{\overline}
\newcommand{\wtl}{\widetilde}

\newcommand{\sq}{\square}

\newcommand{\xto}{\xrightarrow}
\newcommand{\into}{\hookrightarrow}
\newcommand{\onto}{\twoheadrightarrow}
\newcommand{\Cont}{\mathrm{Cont}}
\newcommand{\Map}{\mathrm{Map}}
\newcommand{\wt}{\mathrm{wt}}
\newcommand{\fil}{\mathrm{fil}}

\newcommand{\Sen}{\mathrm{Sen}}

\newcommand{\Lift}{\mathrm{Lift}}
\newcommand{\rec}{\mathrm{rec}}
\newcommand{\reg}{\mathrm{reg}}
\newcommand{\irreg}{\mathrm{irreg}}
\newcommand{\alg}{\mathrm{alg}}
\newcommand{\der}{\mathrm{der}}

\newcommand{\gtilde}{\widetilde{\mathfrak g}}
\newcommand{\gth}{\widehat{\widetilde{\mathfrak g}}}
\newcommand{\that}{\widehat{\mathfrak t}}

\newcommand{\gG}{\mathsf G}
\newcommand{\gH}{\mathsf H}
\newcommand{\gB}{\mathsf B}
\newcommand{\gM}{\mathsf M}

\newcommand{\gT}{\mathsf T}

\newcommand{\gU}{\mathsf U}
\newcommand{\gP}{\mathsf P}

\newcommand{\wtimes}{\widehat\otimes}
\newcommand{\sic}{\mathrm{sc}}
\newcommand{\gW}{\mathsf W}
\newcommand{\bW}{\mathbf W}

\newcommand{\Dtri}{D^{\triangle}}
\newcommand{\sat}{\mathrm{sat}}
\newcommand{\sss}{\mathrm{ss}}
\newcommand{\wdelta}{{w_\delta}}
\newcommand{\wsat}{{w_{\mathrm{sat}}}}
\newcommand{\adm}{\mathrm{adm}}
\newcommand{\Frac}{\mathrm{Frac}}
\newcommand{\Lbar}{{\overline{L}}}
\newcommand{\Rig}{\mathrm{Rig}}
\newenvironment{enum}{
\begin{enumerate}
  \setlength\itemsep{0pt}
  \setlength{\parskip}{0pt}
  \setlength{\parsep}{0pt}
}{\end{enumerate}}

\usepackage{microtype}

\usepackage{mathrsfs}

\begin{document}

\title[The trianguline variety for reductive groups]{The trianguline variety for reductive groups}

\author[A.~Conti]{Andrea Conti}
\address{Andrea Conti, IWR, Heidelberg University, Heidelberg, Germany}
\email{contiand@gmail.com, andrea.conti@iwr.uni-heidelberg.de}
\thanks{The first author is funded by the Deutsche Forschungsgemeinschaft (DFG, German Research Foundation) - Project-ID 444845124 – TRR 326.}

\author[M.~Moakher]{Mohamed Moakher} 
 \address{Mohamed Moakher, Department of Mathematics, University of Pittsburgh, Pittsburgh, USA}
 \email{mom224@pitt.edu}

\author{Julian Quast}
 \address{Julian Quast, Faculty of Mathematics, University of Duisburg-Essen, Essen, Germany}
\email{julian.quast@uni-due.de, me@julianquast.de}
\thanks{The third author is funded by the Deutsche Forschungsgemeinschaft (DFG, German Research Foundation) – Project-ID 444845124 – TRR 326 and Project-ID 517234220}


\date{\today}



\begin{abstract} 
    We study the trianguline variety for split connected reductive groups. 
    We generalize a theorem of Breuil, Hellmann, and Schraen about its local structure, establishing smoothness over the loci determined by various regularity conditions on the triangulation parameter, and normality at certain points outside of these smooth loci.
    Along the way, we prove a crystallinity criterion for $\PG$-modules with $\gG$-structure. 
\end{abstract}

\maketitle




\setcounter{tocdepth}{1}

\tableofcontents

\section*{Introduction}
\label{sec_intro}

Let $p$ be a prime. 
Let $K/\Qp$ be a finite extension, with fixed algebraic closure $\overline{K}$, and absolute Galois group $\GK := \Gal(\overline{K}/K)$. 
Emerton's proof of the overconvergent Fontaine--Mazur conjecture \cite[Proposition 1.1.2]{Emertonlocglob} shows that, up to a cyclotomic twist, a global Galois representation $\rho:\Gal(\ovl\Q/\Q)\to\GL_2(\Qpbar)$ is attached to a $p$-adic overconvergent eigenform if and only if it is almost everywhere unramified and \emph{trianguline} at $p$ in the sense of Colmez \cite{ColmezTri}, i.e., the associated cyclotomic $(\varphi,\Gamma)$-module is an extension of two rank-one $(\varphi,\Gamma)$-modules. Those $\rho$ that are crystalline at $p$ correspond to classical eigenforms.  Thus, after restriction to a decomposition group at $p$, crystalline representations appear naturally inside families of trianguline representations. Moreover, one can deduce geometric properties of the eigencurve from properties of the (local) trianguline deformation spaces.

When $\GL_2$ is replaced with $\GL_n$, there is an obvious notion of a trianguline representation $\rho\colon \GK \to\GL_n(\Qpbar)$, namely that the associated $(\varphi,\Gamma_K)$-module is obtained by successive extensions of rank-one objects. The eigenvariety of a unitary group split at $p$ contains a Zariski-dense set of classical points, whose associated Galois representation is crystalline (hence trianguline) at $p$, so that the results of Kedlaya--Pottharst--Xiao \cite{KPX} allow one to show that the Galois representation is trianguline at $p$ everywhere on the eigenvariety. In the series of works \cite{BHSmodulaire,bhssmooth,BHS19} Breuil, Hellmann and Schraen exploit this fact to describe some properties of such eigenvarieties via the \emph{trianguline variety}. This is a  rigid analytic variety that parametrizes pairs consisting of a trianguline Galois representation and the parameter of a triangulation, and which originates in the work of Kisin \cite{Kis03}. 
The local geometry of such a variety is described by an explicit scheme constructed in the spirit of the Grothendieck--Springer resolutions of singularities.

When $\GL_n$ is replaced by an arbitrary split connected reductive group $\gG$, de Daruvar \cite{dedar2020} proposed a Tannakian definition of trianguline representations, and used it to define an analogue of the trianguline variety of Breuil--Hellmann--Schraen, and to establish several of its basic properties. However, his results rely on a number of technical assumptions, and do not provide a description of a local model of the trianguline variety.

In this work, we extend the study of the $\GL_n$-trianguline variety carried out in \cite{BHS19} to the case of an arbitrary split connected reductive group. Our approach builds in part on \cite{dedar2020}, but we are able to remove the technical assumptions imposed there. We also generalize the local results in \cite{BHS19}, describing in particular the local geometry of the trianguline variety in terms of Grothendieck--Springer resolutions. 



\subsection*{Extensions of $\PG$-modules}
Let $L/\Qp$ be a finite extension with ring of integers $\cO_L$ and residue field $k$.
We assume that $L$ contains the Galois closure of $K$ in $\overline{K}$.
Let $\gG$ be a split connected reductive group over $\cO_L$ with Borel subgroup $\gB$ and fiberwise maximal torus $\gT$ over $\cO_L$. We let $\gU$ be the unipotent radical of $\gB$.

 We let $\Gamma_K=\Gal(K(\mu_{p^\infty})/K)$, for $\mu_{p^\infty}$ the group of $p$-power roots of unity in $\overline{K}$, and work with the category $\PGcat_{K,L}^+$ of cyclotomic $(\varphi,\Gamma_K)$-modules over the Robba ring $\cR_{K,L}$. 
We say that a Galois representation $\rho\colon \GK\to \gG(L)$ is \emph{trianguline} if the associated $\PG$-module,  which we can see as a fiber functor $\Rep_{L}(\gG)\to \PGcat_{K,L}^+$, factors through $\Rep_L(\gB)$ after possibly extending $L$. The parameter $\delta\colon \gT^\vee(K) \to L^\times$ of the triangulation, where $\gT^\vee$ is the dual torus, is obtained by restricting the fiber functor from $\Rep_L(\gB)$ to $\Rep_L(\gT)$.

Consider a continuous representation $\rhobar : \GK\to \gG(k)$, and
let $R^{\square}_{\rhobar}$ be the universal framed deformation ring of $\rhobar$ with coefficients in $\cO_L$ as in \cite{defT, defG}. 
Let $\cX_{\rhobar}$ be the rigid generic fiber of $\Spf(R^{\square}_{\rhobar})$. Let $\cT^{\gT}$ be the rigid analytic space parametrizing continuous characters of $\gT^{\vee}(K)$ and let $\cT_{0,\gB}^{\gT}$ be the subspace of those which are \emph{$\gB$-regular} in the sense of \Cref{def_Breg}.
We denote by $U_{\tri}(\rhobar) \subseteq \cX_{\rhobar} \times \cT^{\gT}$ the subset of pairs $(\rho, \delta)$ such that $\delta\in \cT_{0,\gB}^{\gT}$, and $\rho$ admits a triangulation with parameter $\delta$.

The \emph{trianguline variety} $X_{\tri}(\rhobar)$ is defined as the Zariski closure of $U_{\tri}(\rhobar)$, endowed with the reduced rigid analytic structure. In order to investigate its geometric properties, the standard approach is to introduce a larger space $\cS_{\gB}$ of \emph{$\gB$-regular rigidified $\gB$-$\PG$-modules}, where in addition to the data corresponding to a triangulation, one also specifies a ``rigidification'' of the corresponding $\gT$-$\PG$-module (see \Cref{sec_SB}). Our first main result establishes the geometric structure of this space.

\begin{maintheorem}[\Cref{SB_representable}]\label{mainA}
The space $\cS_\gB$ is representable by a rigid analytic space over $L$. Moreover, it is a vector bundle, in particular smooth, of relative dimension $\dim \gU \cdot [K:\Qp]$ over $\cT_{0,\gB}^{\gT}$.
\end{maintheorem}

\Cref{mainA} was proved by de Daruvar under the technical hypothesis that $p > 2$ and $\gG$ has no simple factors of type $G_2$, $F_4$ or $E_8$ (see \cite[Theorem 5.10]{dedar2020}). The main difficulty is handling non-abelian extension problems of $\PG$-modules. 
We give a new proof  inspired by the construction of the Selmer varieties in the Chabauty--Kim method \cite{Kim}. It involves an induction on the central series of the unipotent radical $\gU$, and allows us to remove all of these hypotheses.


For a summary of the proof of \Cref{mainA}, we refer to the end of the introduction.

\subsection*{Regularity conditions} We are interested in regularity conditions on the parameter $\delta$ that guarantee the uniqueness of a triangulation of this given parameter (see \Cref{sec:regularity}). Unfortunately, while we expect $\gB$-regularity to be sufficient for this purpose, we do not currently know how to prove it. 

We therefore consider conditions of two kinds. The first involves choosing a finite set $\Lambda\subseteq X^*_+(\gT^{\der})$ of $\Q$-generators of $X^*_+(\gT^{\der})_\Q$ and requiring that $\delta\circ(\mu-\lambda)^{\vee}$ be a regular character for each $\lambda\in \Lambda$ and each weight $\mu\neq \lambda$ appearing in the representation $V_{\lambda}$ of highest weight $\lambda$ (see \Cref{def_reg}); we call this $\Lambda$-regularity. Under such condition, uniqueness of a triangulation is proven in \Cref{uniqueness_Lambda_1} using the Plücker datum introduced below. When $\Lambda$ is taken to be the set of quasi-fundamental weights, one recovers the notion of very regularity used in \cite{dedar2020}. Another useful choice is the \emph{Hilbert basis} $\frakS$ of $X^*_+(\gT^{\der})$, i.e. its minimal generating set as a monoid. 

The second possibility is to consider a carefully chosen faithful representation $r:\gG\hookrightarrow\GL_m$ and transport to $\gG$ a condition on the parameter that is sufficient for uniqueness of triangulations over $\GL_m$ (see \Cref{def:rregular} and \Cref{uniqueness_r_1}); we call this $r$-regularity. We make such a condition explicit for several classical groups in \Cref{ex:rregular}. In general, we expect a form of ``$\gB$-regularity'' to be sufficient for the uniqueness of a triangulation (see \Cref{conj_Ger}), and we verify this for the classical groups $\Sp_{2n}$ and $\GSp_{2n}$. 

As an immediate consequence of our study of $\cS_\gB$, we deduce that the trianguline variety $X_{\tri}(\rhobar)$ can equivalently be defined as the closure of the points satisfying a stronger regularity condition on the parameter, if it is given by a Zariski dense Zariski open subset of $\cT_{0,\gB}^{\gT}$ (\Cref{independence_of_regularity_condition}). This applies in particular to the two kinds of regularity conditions introduced above. To the best of our knowledge, this argument is new even for $\gG = \GL_n$.
In particular, $X_{\tri}(\rhobar)$ coincides with the trianguline variety of \cite{BHSmodulaire} for $\gG = \GL_n$ and \cite{dedar2020} for general $\gG$.

\subsection*{The smooth locus of the trianguline variety}
Inside of the space $\cS_{\gB}$, we can identify the locus $\cS_{\gB}(\rhobar)$ of deformations of our fixed $\rhobar$. It comes equipped with a surjection $\pi_{\rhobar} : \cS_{\gB}(\rhobar) \to U_{\tri}(\rhobar)$, that allows us to study the geometry of $X_{\tri}(\rhobar)$ via that of $\cS_{\gB}(\rhobar)$. 
We write $U^*_{\tri}(\rhobar) \subset U_{\tri}(\rhobar)$ for the subset of points for which $\delta$ satisfies one of the regularity conditions mentioned above. We write $\cS^{*}_{\gB}(\rhobar) = \pi_{\rhobar}^{-1}(U^*_{\tri}(\rhobar))$, and denote by $\omega' \colon X_{\tri}(\rhobar) \to \cT^{\gT}$ the natural projection.

\begin{maintheorem}[{\Cref{thm_equidimensional}}]\label{mainB} Assume that $\cS^{*}_{\gB}(\rhobar) \neq \emptyset$. 
\begin{enum}
    \item $X_{\tri}(\rhobar)$ is equidimensional of dimension $\dim \gG_L + [K:\Qp] \cdot \dim \gB_L$.
    \item $\pi_{\rhobar}$ is smooth of relative dimension $\dim \gT_L$.
    \item $U^*_{\tri}(\rhobar)$ is smooth and Zariski open in $X_{\tri}(\rhobar)$ and $\omega'|_{U_{\tri}^{*}(\rhobar)}$ is smooth.
\end{enum}
\end{maintheorem}
The theorem does not imply the existence of a regular  trianguline lift of $\rhobar$, rather it depends on it. \Cref{mainB} was proved by de Daruvar under the same technical assumptions as \Cref{mainA}, that we remove. 
In particular, this allows us to remove the same assumptions on a result of Fakhruddin--Khare--Patrikis \cite{Fakhruddin_2022} on the existence of lifts of global mod $p$ representations trianguline at $p$. We present this global application in \Cref{sec:application}.

\subsection*{Triangulinity and parameters of points of $X_\tri(\rhobar)$}
For a point $x = (\rho, \delta) \in X_{\tri}(\rhobar) \setminus U_{\tri}(\rhobar)$ it is not a priori clear whether $\rho$ is trianguline and how $\delta$ is related to $\rho$. The following theorem ensures that the trianguline variety deserves its name for any $\gG$.



\begin{maintheorem}[{\Cref{prop:boundary}}]\label{main:boundary}
If $x=(\rho,\delta)\in X_{\tri}(\rhobar)$, then $\rho$ is trianguline of a parameter $\wtl\delta$ such that $\delta^{-1}\wtl\delta$ is algebraic. 
\end{maintheorem}


The fact that $\rho$ is trianguline follows from the results of \cite[\S 6]{KPX}. To apply these, it is essential to describe a triangulation $\Dtri$, seen as a fiber functor $\eta_{\Dtri}:\Rep_L(\gB)\to \PGcat^+_{K,L}$, in terms of its \emph{Pl\"ucker datum}. This consists of the data of a saturated rank-one $\PG$-module over $\cR_{K,L}$ inside of $\eta_{\Dtri}(V_\lambda)$ for each irreducible $\gG$-representation $V_\lambda$ of highest weight $\lambda\in X_+^*(\gT)$ (see \Cref{pluckerdatum}). The key point is that this combinatorial datum extends to the boundary of the trianguline variety thanks to \cite[Theorem 6.3.9]{KPX}.
Pl\"ucker data were originally introduced in unpublished work of Drinfeld to give a Tannakian description of the flag variety of $\gG$ (see for instance \cite{Haines}), and they were first applied to the setting of triangulable $\PG$-modules in de Daruvar’s work.

For the description of the parameter $\wtl\delta$, it is necessary to view $X_\tri(\rhobar)$ as the closure of the set of ``$\frakS$-regular'' points; this is justified by our earlier result on the independence of $X_\tri(\rhobar)$ from the chosen regularity condition. 
The advantage of working with $\frakS$-regularity is that the algebraicity of $\delta^{-1}\wtl\delta$ can be recovered from the corresponding result for $\GL_n$, by applying it to finitely many representations of $\gG$ whose highest weights generate the monoid $X^*_+(\gT^{\der})$.


\subsection*{Local geometry outside the smooth locus}
Our goal in the following is to describe the geometry of $X_{\tri}(\rhobar)$ outside of the smooth loci $U_{\tri}^{*}(\rhobar)$. As in \cite{BHS19}, we define in \Cref{sec:Xnew} the reductive group
\begin{align*}
    \bG \colonequals &\Spec(L) \times_{\Spec(\Qp)} \Res_{K/\Qp}\gG_K \cong \underbrace{\gG_L \times \dots \times \gG_L}_{[K:\Qp] \text{ times}}. 
\end{align*}
We define a Borel subgroup $\mathbf B$ and a torus $\mathbf T$ similarly, and the Lie algebras $\frakg,\frakb,\frakt$. 
The group $\bG$ appears naturally when considering weights of characters $\gT^\vee(K)\to L^\times$, e.g. of the parameter of a triangulation. 
We attach to $\bG$ 
an $L$-scheme
$$ X = \{(g_1\bB,g_2\bB,\psi) \in \bG/\bB \times \bG/\bB \times \frakg \mid \Ad(g_1^{-1})\psi, \Ad(g_2^{-1})\psi \in \frakb\}$$
which parametrizes pairs of flags compatible with an element of the Lie algebra of $\bG$. Its irreducible components $X^w$ are indexed by elements $w$ of the Weyl group $\bW$ of $(\bG,\bT)$: each $X^w$ is the Zariski closure of the locus $V^w$ where the two flags are in relative position $w$, that is $(g_1\bB,g_2\bB)\in V^{w}$ if and only if $g_1^{-1}g_2\in \bB w\bB$.
By a result of Bezrukavnikov--Riche \cite[Theorem 2.2.1]{zbMATH06155581}, $X^w$ is Cohen--Macaulay. Using Serre's criterion, Breuil--Hellmann--Schraen showed that $X^w$ is normal \cite[Theorem 2.3.6]{BHS19}.

Let $x=(\rho,\delta)\in X_{\tri}(\rhobar)$.
The local geometry of $X_{\tri}(\rhobar)$ at $x$ can be related to that of an irreducible component of $X$, as we will now explain. By \Cref{main:boundary}, the $\gG$-$\PG$-module $D$ associated to $\rho$ admits a triangulation $\Dtri$ of some parameter $\wtl\delta$. In this situation, we have the following simple characterization of almost de Rham representations.

\begin{proposition}[\Cref{lemmaKPX6313}]\label{propIntro}
The following are equivalent:
\begin{enumerate}[label=(\roman*)]
    \item $\rho$ is Hodge--Tate;
    \item $\rho$ is almost de Rham;
    \item $\delta$ is locally algebraic;
    \item $\wtl\delta$ is locally algebraic.
\end{enumerate}
If any of the above holds, then both $\delta$ and $\wtl\delta$ are the product of a smooth character with an algebraic character whose dual is conjugate to the Hodge--Tate cocharacter $\varpi_\HT$ of $\rho$.
\end{proposition}

Assume from now on that $\delta$ is locally algebraic and is either $\Lambda$-regular or $r$-regular.  By \Cref{propIntro}, $\delta$ is the product of a smooth character with the dual of a conjugate of $\varpi_\HT$ by an element $\wdelta\in\bW$. Let $\Mtri \colonequals \Dtri[\tfrac{1}{t}]$ be the induced triangulation of the $\gG$-$\PG$-module $\cM \colonequals D[\tfrac{1}{t}]$ over $\cR_{K,L}[\tfrac{1}{t}]$. We consider the almost de Rham $\gG$-$\BdR$-representation $W$ associated to $\cM$, the almost de Rham $\gB$-$\BdR$-representation $\Wtri$ associated to $\Mtri$, and the almost de Rham $\gG$-$\BdR^+$-representation $W^+$ associated to $D$ (see \Cref{secBdRBdRplusnew}). We further assume that the Hodge--Tate cocharacter $\varpi_\HT$ attached to $\rho$ is regular.

Let $X_{\rho,\Mtri}$ be the functor of deformations to Artinian $L$-algebras of $\rho$ and $\Mtri$. 
Given an Artinian $L$-algebra $A$, $X_{\rho,\Mtri}^\sq(A)$ consists of tuples $x_A=(\rho_A,\Mtri_A,\Wtri_A,W_A^+)$ deforming $(\rho,\Mtri,\Wtri,W^+)$ and satisfying the natural compatibility conditions. By choosing a trivialization of the potentially de Rham fiber functor $D_{\pdR}\circ\eta_W\colon \Rep_L(\gG)\to \Vect_L$,
we obtain a framed deformation functor $X_{\rho,\Mtri}^\sq$.
To each $x_A$ we attach:
\begin{itemize}
\item a Hodge--Tate flag $x_{W^+_A,\HT}\in (\bG/\bB)(A)$ as in \Cref{firstflag} by constructing a Plücker datum from the Hodge--Tate filtration and $\varpi_{\HT}$,
\item a flag $x_{\Wtri_A,\tri}\in (\bG/\bB)(A)$  as in \Cref{secondflag}, coming from the triangulation $\Mtri_A$,
\item a nilpotent element $N_{W_A}\in \frakg(A)$ as in \Cref{pdR_charanew}, coming from the $\Ga$-action on $D_{\pdR}\circ \eta_{W_A}$.
\end{itemize}
By \Cref{pointgtilde1} and \Cref{pointgtilde2}, these data determine an $A$-point of $X$:
$$ x_{\pdR,A}=(x_{\Wtri_A,\tri}, x_{W_A^+,\HT}, N_{W_A})\in X(A). $$  
For $\gG=\GL_n$, this recovers the element $w$ of \cite{BHS19}.

The resulting map $X_{\rho,\Mtri}^\sq \to X$ allows us to define the $w_\delta$-component $X_{\rho,\Mtri}^{\sq,w_\delta}$ of $X_{\rho,\Mtri}^\sq$, and we define $X_{\rho,\Mtri}^{w_\delta}$ as its image under the forgetful map $X_{\rho,\Mtri}^\sq\to X_{\rho,\Mtri}$. We prove that there exists a chain of formally smooth morphisms
\begin{align}\label{correspondence}
    \wh{X_{\tri}(\rhobar)}_x \xrightarrow{\cong} X_{\rho, \Mtri}^{w_\delta} \leftarrow X_{\rho, \Mtri}^{\square, w_\delta} \rightarrow X_{D, \Mtri}^{\square, w_\delta} \rightarrow X^{\square, w_\delta}_{W^+, \Wtri} \xrightarrow{\cong} \widehat X^{w_\delta}_{x_\pdR}.
\end{align}
Here $X_{D, \Mtri}^{\square, w_\delta}$ is defined in terms of deforming the $\gG$-$\PG$-module $D$ instead of $\rho$, and $X^{\square, w_\delta}_{W^+, \Wtri}$ is given by deforming $W^+$ together with the induced triangulation $\Wtri$ of $W$. 

For the purposes of this introduction, we say that $x=(\rho,\delta)$ is ``sufficiently regular'' if it satisfies the assumptions listed above. As a consequence of \eqref{correspondence} we have the following.
\begin{maintheorem}[\Cref{three_five_eleven}, \Cref{three_seven_eight}, Local model of the trianguline variety]\label{mainthm:local}
For a sufficiently regular point $x$, the completion of $X_{\tri}(\rhobar)$ at $x$ is pro-represented by a noetherian complete local domain of residue field $L$ and dimension $[K:\Q_p]\cdot (\dim\gG+\dim\gB)$, whose associated formal scheme is  isomorphic to $\wh X_{x_{\pdR}}^{w_\delta}$ up to formally smooth morphisms.
\end{maintheorem}

We defer an explanation of the proof of \Cref{mainthm:local} to the end of the introduction, and turn now to a concrete consequence for the geometry of the trianguline variety. 

\begin{corollary}[\Cref{three_seven_ten}]
At a sufficiently regular point $x$, the variety $X_{\tri}(\rhobar)$ is normal, hence irreducible, and Cohen--Macaulay.
\end{corollary}

\subsection*{Accumulation of crystalline points}

We wish to show that sufficiently nice crystalline points of $X_{\tri}(\rhobar)$ accumulate at every sufficiently regular point. As an intermediate step towards this result, we prove a criterion under which a triangulable $\gG$-$\PG$-module is crystalline. 

\begin{maintheorem}[\Cref{lincrys}, \Cref{cryscriterion}, Crystallinity criterion]
Let $D$ be a $\gG$-$\PG$-module over $\cR_{K,L}$ equipped with a triangulation $\Dtri$ of parameter $\delta\colon\gT^\vee(K)\to L^\times$. Assume that $\delta\vert_{\gT^\vee(\cO_K)}$ is algebraic and that the weights of $\delta$ have ``enough gaps'' with respect to $\gB$. 
Then $D$ is crystalline.
\end{maintheorem}
We outline the proof. Let $D_0$ denote the restriction of $\Dtri$ to a $\gT$-$\PG$-module. Under the hypotheses of the theorem, we wish to show that every lift of $D_0$ to a $\gB$-$\PG$-module is crystalline. By the proof of \Cref{mainA}, we know that the set of all lifts is in bijection with the $\PG$-cohomology group $H^1_{\varphi,\Gamma_K}(\eta_{D_0}(\Lie \gU))$. We adapt the criterion of Emerton--Gee \cite[Lemma 6.3.1]{emertongee} to the setting of $\PG$-modules and show that every class in $H^1_{\varphi,\Gamma_K}(\eta_{D_0}(\Lie \gU))$ is crystalline. However, this alone is insufficient to conclude since it is not immediately clear, unlike the $\GL_n$ case, that these crystalline classes correspond to crystalline lifts. To bridge this gap, we use the following observation: the crystallinity of a $\PG$-module $M$ over $\cR_{K,L}$ is equivalent to the triviality of the $\Gamma_K$-action on $M[\tfrac{1}{t}]$. Consequently, within the set $H^{1}_{\varphi,\Gamma_K}(\gB(\cR_{K,L}))$ of isomorphism classes of $\gB$-$\PG$-modules, we can show that the crystalline ones are precisely the preimage of the trivial element under the natural map $H^{1}_{\varphi,\Gamma_K}(\gB(\cR_{K,L}))\to H^{1}(\Gamma_K,\gB(\cR_{K,L}[\tfrac{1}{t}]))$, where the target denotes the non-abelian continuous cohomology set. The result then follows by exploiting the compatibility of this map with the exponential map $\exp\colon \Lie(\gU)\to \gU$ for abelian $\gU$, and inducting on the central series of $\gU$. 

For the purposes of this introduction, we say that a crystalline point $(\rho,\delta)$ of $X_{\tri}(\rhobar)$ is ``sufficiently nice'' if it satisfies the conditions in \Cref{def_acc_property} (it is noncritical, with regular semisimple crystalline Frobenius, and sufficiently spaced Hodge--Tate weights). Non-criticality can actually be deduced from the weights of $\delta$ being sufficiently spaced, thanks to \Cref{lemma:noncritical}.

\begin{maintheorem}[\Cref{acc_property}, Accumulation property]
If $x$ is a sufficiently regular point of $X_{\tri}(\rhobar)$ with algebraic weight, then sufficiently nice crystalline points accumulate at $x$.
\end{maintheorem}

\subsection*{Crystalline points in local deformation spaces}

Our original motivation for writing this paper is to make progress towards the following conjecture.

\begin{mainconj}\label{conj_cryst} The Zariski closure of the set of regular crystalline points in $\cX_{\rhobar}$ is a union of irreducible components. 
\end{mainconj}

Special cases of this conjecture have been proved by Colmez, Kisin, Nakamura, B\"ockle for $\gG=\GL_2$ \cite{ColmezTri,Kisin,Nakamura_TwoDim,Nakamura,Boeckle3}, by Chenevier for $\gG=\GL_n$, $K=\Q_p$ and absolutely irreducible $\rhobar$ \cite{chenevier2010sur}, and by  Iyengar \cite{Iyengar_2020} for $\gG=\GL_n$ and general $K$ and $\rhobar$. For $\gG=\GSp_{2n}$, this was proved by Aoki under certain hypotheses on $\rhobar$ \cite{aoki}.
By \cite[Corollary 15.29]{defG} $\cX_{\rhobar}$ is normal, so the conjecture may equivalently be stated in terms of connected components.

In order to prove \Cref{conj_cryst} for general $\gG$, one can follow the strategy used in the $\GL_n$ case. Using our result on the accumulation property of crystalline points, one needs to generalize \cite[Corollary 3.13]{HellmannMargerinSchraen} in order to translate this accumulation property from the trianguline variety to the Galois deformation space by comparing tangent spaces, as explained in the proof of  \cite[Proposition 5.10]{Iyengar_2020}. 

For $\gG = \GL_n$ \cite{BIP2}, and $\gG = \GSp_{2n}$ under certain hypotheses on $\rhobar$ \cite{aoki}, it is in fact known that the Zariski closure of the set of regular crystalline points is all of $\cX_{\rhobar}$. Given \Cref{conj_cryst},
this follows by understanding the irreducible components of $\cX_{\rhobar}$ and then constructing a crystalline lift on each one using a result of Emerton--Gee \cite{emertongee}.
By \cite[Corollary 15.30, Theorem 16.4]{defG}, the irreducible components are understood for most $\gG$.
Further progress on the lifting problem has been made by Lin \cite{LinIrredLifting, LinReducibleLifting}, but certain critical cases of reducible $\rhobar$ remain open.
We plan to address these problems in future work.


\subsection*{Further possible applications}

Along the lines of \cite{BHS19}, the local results of this paper should have global consequences. 
For instance, when $\gG$ admits an eigenvariety $\cE$ carrying a family of Galois representations, then one can typically apply an analogue of \Cref{prop:boundary} to show that such a family is everywhere trianguline. Then one might be able to apply our work to the description of ``companion points'' on $\cE$, i.e. points whose associated local Galois representations are isomorphic up to twist. With this purpose in view, we already introduce in \Cref{def:wsat} an analogue of the element $w_x$ appearing in the description of companion points from \cite[Theorem 4.2.3]{BHS19}, that we denote by $w_\sat$ following \cite{mowlavi}: it is the relative position of the two flags belonging to $x_{\pdR}$. We prove that $w_\sat \preceq w_{\delta}$.

If furthermore the Taylor--Wiles method can be used to construct a ``patched eigenvariety'' for $\gG$, then one could hope for a classicality criterion along the lines of \cite[Theorem 3.9]{bhssmooth}. These ingredients are now available for many classical groups. 

We remark that many of our results could be applied, or generalized with little effort, to the case of a disconnected reductive group, which would be relevant when considering Galois representations valued in the $L$- or $C$-group of a reductive group. For example Theorem A can be generalized using the notion of cohomological regularity. Theorem C follows from it by the same argument.


\smallskip

\subsection*{Proof of Theorem A}
We explain our proof of \Cref{mainA} in more detail.
For an affinoid $L$-algebra $A$ we fix a map $\Sp(A) \to \cT_{0,\gB}^{\gT}$, which defines a $\gT\hyphen(\varphi, \Gamma_K)$-module $D_{0}$.
We wish to describe the pullback of $\cS_{\gB}$ to $\Sp(A)$, so we fix an affinoid $A$-algebra $A'$.
The $A'$-points are given by the space of lifts $\Lift^\gB_{\gT}(D_{0, A'})$ of $D_{0, A'}=D_0\otimes_A A'$ to $\gB$-$\PG$-modules.
We will establish a functorial bijection $\Lift^\gB_{\gT}(D_{0, A'}) \cong H^{1,\star}_{\varphi, \Gamma_K}(\gU(\cR_{K,A'}))$ (\Cref{H1_Ext_U}), where the set $H^{1,\star}_{\varphi, \Gamma_K}(\gU(\cR_{K,A'}))$ is a non-abelian twisted version of $\PG$-cohomology where actions on $\gU(\cR_{K,A'})$ are induced by $D_{0, A'}$ from the conjugation of $\gT$ on $\gU$.

To explain this isomorphism, let us assume for simplicity that we were working with representations of an abstract group $\Gamma$ instead of $\PG$-modules.
Let $f \colon \Gamma \to \gT(L)$ be a homomorphism.
A lift of $f$ to a homomorphism $\tilde f \colon \Gamma \to \gB(L)$ consists of a pair $(f', f) \colon \Gamma \to \gU(L) \rtimes \gT(L)$, where $f' \in Z^1(\Gamma, \gU(L))$ is a non-abelian $1$-cocycle. 
The set of equivalence classes of such lifts is given by $H^1(\Gamma, \gU(L))$.
The conjugation action of $\gT$ on $\gU$ induces actions of $\varphi$ and $\Gamma_K$ on $\gU(\cR_{K,A'})$, which we denote by $\star$ (see \Cref{sec_form_smooth1} for detailed definitions).
We can think of $\gU(\cR_{K,A'})$ as a `non-abelian $\PG$-module' analogous to the non-abelian $\Gamma$-module $\gU(L)$ in the case of representations. Beware that this is not the same as a $\gU$-valued $\PG$-module! Moreover, we do not make an attempt at defining what a `non-abelian $\PG$-module' should be. The evident computations lead to the definition of the cohomology set $H^{1,\star}_{\varphi, \Gamma_K}(\gU(\cR_{K,A'}))$ and an isomorphism to $\Lift^\gB_{\gT}(D_{0, A'})$.

We are left to establish the representability of the functor $A' \mapsto H^{1,\star}_{\varphi, \Gamma_K}(\gU(\cR_{K,A'}))$.
This is achieved in \Cref{better_unipotent_H1} by constructing a functorial bijection $H^{1,\star}_{\varphi, \Gamma_K}(\gU(\cR_{K,A'})) \cong H^{1}_{\varphi, \Gamma_K}(\eta_{D_{0,A'}}(\Lie \gU))$, where $\eta_{D_{0,A'}}\colon \Rep_{L}(\gT)\to \PGcat^+_{K,A'}$ is the fiber functor corresponding to $D_{0,A'}$, and where $\Lie(\gU)$ is viewed as a $\gT$-representation via the adjoint action. If $\gU$ is abelian, then this simply given by the exponential map. In general, the proof proceeds by induction over the central series of $\gU$, so that each step is given by a central extension $\gU^{n+1}\to \gU_{n+1}\to \gU_n$. Using an idea in \cite{Kim}, we show that $H^{1,\star}_{\varphi, \Gamma_K}(\gU_{n+1}(\cR_{K,A'}))$ is an $H^{1}_{\varphi, \Gamma_K}(\eta_{D_{0,A'}}(\Lie \gU^{n+1}))$-torsor  over $H^{1,\star}_{\varphi, \Gamma_K}(\gU_n(\cR_{K,A'}))$. To show that the torsor is non-empty, a key input is the vanishing of $H^{2}_{\varphi, \Gamma_K}(\eta_{D_{0,A'}}(\Lie \gU))$, which by the calculations of \Cref{sec_obs_thy} gives the existence of a lift at each step. 
We note that the difference of our argument with that of de Daruvar, is that he inducts on a different stratification of $\gB$ in which each step is a central extension that admits a group-theoretic section. This is only possible under the hypotheses on $\gG$ of de Daruvar. 


\subsection*{Proof of Theorem D}
We explain the proof of \Cref{mainthm:local} in more detail, which is equivalent to establishing the correspondence \eqref{correspondence}. The essential work consists of two parts:
\begin{enum}
\item Proving that the map $X_{D, \Mtri}^{\square} \rightarrow X^{\square}_{W^+, \Wtri}$ (before restricting to $w$-components) is formally smooth (\Cref{lem356}).
    \item Showing that the natural map $\wh{X_{\tri}(\rhobar)}_x \to X_\rho$ to the framed deformation space 
    factors over an isomorphism $\wh{X_{\tri}(\rhobar)}_x \xrightarrow{\sim} X_{\rho, \Mtri}^{w_\delta}$ (\Cref{three_seven_eight}).
\end{enum}
\subsubsection*{Part (1): formal smoothness of $X_{D, \Mtri}^{\square} \rightarrow X^{\square}_{W^+, \Wtri}$.}
We explain the proof of \Cref{lem356}. 
First, it is sufficient to prove formal smoothness of the map $X_{D, \Mtri} \rightarrow X_{W^+, \Wtri}$ of unframed deformation problems.
This map can be factored as
$$ X_{D, \Mtri} \rightarrow \widehat \cT^\gT_{\delta} \times_{\that} X_{W^+, \Mtri} \rightarrow X_{W^+, \Wtri}, $$
where the second map is easily shown to be formally smooth.
So the result follows by base-change from the following Theorem.

\begin{theorem}[{\Cref{thm_formal_smoothness}}]\label{thm:fsmooth}
    Let $\Mtri$ be a $\gB$-$(\varphi, \Gamma_K)$-module over $\cR_{K,L}[\tfrac{1}{t}]$. 
    Assume that the parameter $\delta$ of $\Mtri$ is locally algebraic and $\gB$-regular.
    Then the morphism of pseudofunctors
    \begin{equation}
        X_{\Mtri} \longrightarrow  \widehat \cT^\gT_{\delta} \times_{\widehat \frakt} X_{\Wtri} \label{formally_smooth_morphism_intro}
    \end{equation}
    is formally smooth.
\end{theorem}

We prove this theorem in the following abstract setting, from which the result follows by inducting on the central series of $\gU$. Consider a short exact sequence of algebraic groups over $L$
\begin{align}
    0 \to \gU \to \gP \to \gM \to 0, \label{the_extension_of_groups}
\end{align} 
where we now assume that $\gU$ is abelian.
We let $\mathcal{M}_0$ be an $\gM$-trivial $\gM$-$\PG$-module over $\cR_{K,L}[\tfrac{1}{t}]$. Let $\cM$ be a lift of $\cM_0$ to $\gP$. We define $W_0$ and $W$ by the corresponding almost de Rham $\gG$-$\BdR$-representations. We want to show the formal smoothness of the morphism
\begin{align}
    X_{\cM} \to X_{\cM_0} \times_{X_{W_0}} X_{W} \label{pres_explanation}
\end{align} 
in \Cref{prescribed_lifting}. Here, by $X_?$ we denote a corresponding deformation functor.

For a surjection $A_2 \twoheadrightarrow A_1$ of artinian local $L$-algebras with residue field $L$, an $A_i$-valued point of the target of \eqref{pres_explanation} amounts to a compatible pair $(\cM_{0,A_i}, W_{A_i})$ lifting $(\cM_0, W)$ and mapping to the same lift $W_{0, A_i}$ of $W_0$. Checking formal smoothness of \eqref{pres_explanation} amounts to completing the cube
\begin{center}
    \begin{tikzcd}[column sep=small, row sep=small, every arrow/.append style={scale=0.6}]
    & \cM_{A_1} \ar[dd, |->] \ar[rr, |->] && W_{A_1} \ar[dd, |->] \\
    \cM_{A_2} \ar[dd, dashed, |->] \ar[ur, dashed, |->] \ar[rr, dashed, |->] && W_{A_2} \ar[dd, |->] \ar[ur, |->] \\
    & \cM_{0, A_1} \ar[rr, |->] && W_{0, A_1} \\
    \cM_{0, A_2} \ar[ur, |->] \ar[rr, |->] && W_{0, A_2} \ar[ur, |->]
    \end{tikzcd}
\end{center}

\smallskip

\noindent for a given lift $\cM_{A_1}$ making the back square commute. Here, the arrows pointing right are given by the $W_{\dR}$-functor, the arrows pointing down are given by applying the map $\gP \to \gM$ and the arrows pointing to the back are given by reduction modulo the kernel of $A_2 \twoheadrightarrow A_1$.

We first focus on finding a lift $\cM_{A_2}$ of $\cM_{0, A_2}$, setting aside the rest of the cube. As in the proof of \Cref{mainA}, the set of equivalence classes of such lifts forms a set-theoretic torsor under $H^{1}_{\varphi, \Gamma_K}(\eta_{\cM_{0,A_2}}(\Lie \gU))$, which may a priori be empty.
The main difficulty in establishing non-emptiness is that we operate in a setting where the map $\gP \to \gM$ need not admit a group-theoretic section.
However, since $\gU$ is a product of additive groups, a scheme-theoretic section does exist.
This section allows us to write $\gP = \gU \times \gM$ as schemes, with the multiplication on $\gP$ determined by an algebraic $2$-cocycle $z \colon \gM \times \gM \to \gU$. In \Cref{sec_obs_thy}, we develop a general obstruction theory for non-abelian $(\varphi, \Gamma_K)$-modules, and apply it here. Using $z$, we construct a class in $ H^{2}_{\varphi, \Gamma_K}(\eta_{\cM_{0,A_2}}(\Lie \gU))$ associated to $\cM_{0,A_2}$ (which may be viewed as an element of a non-abelian cohomology group $H^{1}_{\varphi, \Gamma_K}(\gM(\cR_{K,A_2}[\tfrac{1}{t}]))$). We show that a lift exists if and only if
this class vanishes, and the regularity assumption on $\cM_0$ ensures that $H^{2}_{\varphi, \Gamma_K}(\eta_{\cM_{0,A_2}}(\Lie \gU))$ vanishes.

We have a similar obstruction theory for almost de Rham $B_{\dR}$-representations, via $D_{\pdR}$ and working with algebraic group cohomology of $\Ga$.
Here we use the fact that cohomology of $\Ga$ in characteristic $0$ vanishes in degrees $\geq 2$.
Combining these facts enables us to complete the cube.



\subsubsection*{Part (2): the isomorphism $\wh{X_{\tri}(\rhobar)}_x \cong X_{\rho, \Mtri}^{w_\delta}$.}
We construct in \Cref{factor_over_trianguline_deformations} a map $\iota_x^\triangle\colon\wh{X_{\tri}(\rhobar)}_x \to X_{\rho,\Mtri}$  by interpolating triangulations using the $\gG$-version of results of \cite[\S 6.3]{KPX} (see \Cref{blowupLambda,blowupr}). Then we argue similarly to \cite[Section 3]{BHS19} to show that $\iota_x^\triangle$ is a closed immersion with image $X_{\rho,\Mtri}^\wdelta$. We thus reduce ourselves to showing the pro-representability of the latter functor, which is the main output of the work of \Cref{sec35}. We use \Cref{thm:fsmooth} in a key way to prove this representability.

\smallskip

\noindent\textbf{Organization of the paper.} In \Cref{sec:alggrp}, we provide some preliminaries on algebraic groups, the Tannakian formalism and torsors. In \Cref{sec:HPG}, we introduce $\PG$-modules with $\gG$-structure and their cohomology and the various regularity conditions for parameters and triangulable $\PG$-modules with $\gG$-structure.
In \Cref{secBdRBdRplusnew}, we extend the theory of almost de Rham $B_\dR$ and $B_\dR^+$-representations to arbitrary $\gG$, introducing the group $\bG$ and the scheme $X$ along the way. 
In \Cref{sec_form_smooth} we prove the key technical result, a general lifting theorem for $\PG$-modules with a prescribed almost de Rham $\BdR^+$-representation. 
In \Cref{sec35}, we introduce some deformation functors for triangulable $\PG$-modules, and compare them to completions of the scheme $X$.
In \Cref{sec:cryscrit}, we prove the crystallinity criterion for $\gG$-$\PG$-modules.
In \Cref{sec:trivar} we prove the main results on the geometry of the trianguline variety. 
In \Cref{sec:application}, we use our results to remove an assumption from the main theorem of \cite{Fakhruddin_2022} on the existence of trianguline lifts of global mod $p$ representations.

\smallskip

\noindent\textbf{Acknowledgments.} 
The third-named author wants to thank Vytautas Pa\v{s}k\={u}nas for helpful conversations about \cite{BIP2}. We also wish to thank Toby Gee, Eugen Hellmann, Gautier Ponsinet, and Otmar Venjakob for useful exchanges. The second-named author would like to thank Bogdan Ion and Carl Wang-Erickson for many helpful discussions and suggestions. 

\smallskip

\noindent\textbf{Notation.} We will use the following notation:
\begin{enumerate}
    \item $K$ is a finite extension of $\Qp$ contained in a fixed algebraic closure $\Qpbar$, with ring of integers $\mathcal{O}_K$ and fixed uniformizer $\varpi_K$. $K_0 \subseteq K$ is the maximal subfield unramified over $\Qp$, $K_{\infty} = K(\mu_{p^{\infty}})$, $\Gamma_K = \Gal(K_{\infty}/K)$. We write $\chi^\cyc$ for the $p$-adic cyclotomic character of $G_\Q$ (and its restriction to various subgroups) and take its Hodge--Tate weight to be 1, following the convention in \cite{bhssmooth}.
    \item $L$ is a finite extension of $\Q_p$ contained in $\Qpbar$, such that $L$ contains the Galois closure of $K$ over $\Q_p$. We write $\cO_L$ for the ring of integers of $L$, and $k$ for its residue field. We write $\Sigma$ for the set of field homomorphisms $K \hookrightarrow L$.
    \item For a field $E$, we denote by $\Art_E$ the category of local Artinian $E$-algebras with residue field $E$. We denote by $\widehat{\Art}_E$ the pro-category of $\Art_E$. We write $\ovl E$ for a fixed algebraic closure of $E$, and $\cG_E$ for the absolute Galois group $\Gal(\ovl E/E)$. 
    \item Algebraic groups will be denoted with a serif font $\gG,\gB,\gT$, except for the specific groups $\bG,\bB,\bT$ defined in \Cref{sec:HTflagsnew} and denoted in boldface. For a reductive group $\gG$ over $L$ with Borel subgroup $\gB$ and maximal torus $\gT$, we write $\Phi(\gB,\gT)$ (respectively $\Phi^+(\gB,\gT),\Delta(\gB,\gT)$) for the set of roots (respectively positive roots, simple roots) attached to $(\gB,\gT)$. 
    \item If $\gH' \to \gH$ is a morphism of algebraic groups over $L$, we write $\Res_{\gH'}^{\gH}\colon \Rep_{L}(\gH)\to \Rep_{L}(\gH')$ for the corresponding restriction functor.
    \item For a commutative ring $A$, $\Vect_A$ is the category of finitely generated projective $A$-modules, and $\CRing_A$ is the category of commutative $A$-algebras. For $B\in\CRing_A$, we denote by $\phi_B\colon \Vect_A\to\Vect_B$ the base-change functor. Similarly for a scheme $X$, the category $\Vect_X$ is the category of $\cO_X$-modules which are locally finite free.
    \item For a commutative $L$-algebra $A$ and an $L$-group scheme $\gG$, $\gG\hyphen\Tors_A$ is the groupoid of right étale $\gG$-torsors on $A$. We always work in the big étale topos of $\Qp$.
    \item $\Gpd$ denotes the $(2,1)$-category of essentially small groupoids. We identify $\Set$ with a full sub-$2$-category of $\Gpd$ by taking a set to a discrete groupoid.
    \item We use Tate's theory of rigid analytic spaces. Following \cite{ConradIrrComp,BHSmodulaire} we define the (analytic) Zariski closed subsets of a rigid space $X$ as the vanishing loci of coherent ideal sheaves.
\end{enumerate}


\section{Preliminaries on algebraic groups}
\label{sec:alggrp}

For the present work, we need generalizations of many representation-theoretic objects to a linear algebraic group $\gH$ over $L$.
This is achieved using the Tannakian formalism.

\subsection{The Tannakian formalism}

We closely follow \cite{DeligneMilne} in our notations and conventions.
A \emph{tensor category}\footnote{The notion of tensor category, up to fixing an identity object, is the same as the notion of a symmetric monoidal category, see \cite[Remark 1.4(b)]{DeligneMilne}.} $(\cC, \otimes)$ is a category $\cC$ together with a bifunctor $\otimes \colon \cC \times \cC \to \cC$ and compatible associativity and commutativity constraints, such that there merely exists an identity object, \cite[Definition 1.1]{DeligneMilne}.
A \emph{tensor functor} $(\cC, \otimes) \to (\cC', \otimes')$ is a functor $F \colon \cC \to \cC'$ together with a natural isomorphism $F(X) \otimes F(Y) \eqto F(X \otimes Y)$ satisfying certain compatibilities, \cite[Definition 1.8]{DeligneMilne}. A \emph{morphism of tensor functors} is a natural transformation satisfying the compatibilities of \cite[Definition 1.12]{DeligneMilne}.

Let $A$ be a commutative ring.
The prototypical example of a tensor category is the category of vector bundles $(\Vect_A, \otimes_A)$ over $A$.
It is not in general abelian, but comes with a notion of exact sequence and its Hom-sets have the structure of $A$-modules. We will axiomatize this situation as follows.

An \emph{exact category} is a pair $(\cA, \cE)$ consisting of an additive category $\cA$ and a distinguished class $\cE$ of composable morphisms $X \to Y \to Z$ satisfying some axioms which we will refer to as \emph{exact sequences}, \cite[Definition 2.1]{Buehler}. If $\cA$ is an $A$-linear category, i.e. a category enriched in $A$-modules, we speak of an \emph{$A$-linear exact category}. If $\cA$ is an $A$-linear exact category endowed with the structure of a tensor category, such that composition and tensor product of morphisms are $A$-bilinear, we speak of an \emph{$A$-linear exact tensor category}.

If $(\cC, \otimes)$ is a tensor category and $\eta_1, \eta_2 \colon (\cC, \otimes) \to (\Vect_A, \otimes_A)$ are tensor functors, we denote by $\uHom^{\otimes}(\eta_1, \eta_2) \colon \CRing_{A} \to \Set$ the functor which associates to a commutative $A$-algebra $B$ the set of morphisms $\phi_B \circ \eta_1 \to \phi_B \circ \eta_2$ of tensor functors $(\cC, \otimes) \to (\Vect_B, \otimes_B)$, where $\phi_B \colon \Vect_A \to \Vect_B$ is the natural scalar extension functor, \cite[§1, (1.13.1)]{DeligneMilne}.
If $(\cC, \otimes)$ is rigid, see \cite[Definition 1.7]{DeligneMilne} for a definition, then every morphism $\phi_B \circ \eta_1 \to \phi_B \circ \eta_2$ is invertible by \cite[Proposition 1.13]{DeligneMilne}.
Then the functor $\uAut^{\otimes}(\eta_1) \colonequals \uHom^{\otimes}(\eta_1, \eta_1)^{\times}$ agrees with $\uHom^{\otimes}(\eta_1, \eta_1)$.

We denote by $\Rep_L(\gH)$ the $L$-linear tensor category of algebraic representations of $\gH$ on finite-dimensional $L$-vector spaces.
We denote by $\omega_{\gH} \colon \Rep_L(\gH) \to \Vect_L$ the forgetful functor.

\subsection{Torsors}
\label{sec:tors}

Throughout the paper, by a \emph{torsor} we mean a right torsor for the étale topology of schemes or rigid analytic spaces.
In this section, $X$ will denote an affine scheme or a rigid analytic space over $L$.


We will use the following reinterpretation of torsors. An \emph{$L$-linear fiber functor} between two $L$-linear exact tensor categories is an $L$-linear, exact and faithful tensor functor.

\begin{theorem}[Saavedra Rivano, 1972]\label{tannakian_fundamental_equivalence}
     The functor $\eta \mapsto \uHom^{\otimes}(\omega_L \otimes_L \cO_X, \eta)$ defines an equivalence between the groupoid of $L$-linear fiber functors $\Rep_L(\gH) \to \Vect_X$ and the groupoid of right $\gH$-torsors over $X$. A quasi-inverse is given by sending a $\gH$-torsor $\cQ$ to the fiber functor
    \begin{align*}
        \eta_{\cQ}\colon \Rep_L(\gH)&\longrightarrow \Vect_X, \quad\quad V \mapsto \cQ\times^\gH V.
    \end{align*}
    In case $X$ is a rigid analytic space, $\gH$ is understood as a rigid analytic group.
\end{theorem}

The notation $\cQ\times^\gH V$ means the following:
We first take the product of $\cQ$ with the geometric vector bundle over $X$ associated with $V$ and then take the quotient by the free action of $\gH$.
The resulting geometric vector bundle has a natural action of $\gH$ which carries over to the corresponding locally free $\cO_X$-module.

\begin{proof}
    If $X$ is a scheme over $L$ the equivalence is given by \cite[Théorème 2.2]{SaavedraRivano}, see \cite[Theorem 3.2]{DeligneMilne} when $X = \Spec A$ is affine.
    If $A$ is an affinoid $L$-algebra, then by \cite[Proposition 4.7.2]{FresnelVanDerPut} the categories of vector bundles on $\Spec A$ and $\Sp A$ are canonically equivalent.
    One verifies, that this establishes the equivalence in the analytic case, see \cite[Theorem 2.7]{dedar2020}. The general case follows by gluing.
\end{proof}
We record the following lemma.
\begin{lemma}\label{trivialtorsoroverA}
    Let $A\in \Art_L$, and let  $\cQ$ be an $\gH$-torsor over $A$. Then $\cQ$ is a trivial $\gH$-torsor if and only if $\cQ \otimes_A L$ is trivial.
\end{lemma}
\begin{proof}
    By smoothness of $\cQ$, $\cQ(A) \to \cQ(L)$ is surjective. We then note that $\cQ$ is trivial if and only if $\cQ(A) \neq \emptyset$.
\end{proof}

We now return to the setting of the introduction, i.e. $\gG$ is a split connected reductive group over $L$ and $\gB$ is a Borel subgroup, which contains a torus $\gT$.

\begin{definition}
    Let $\cQ$ be a $\gG$-torsor over $X$. A $\gB$\emph{-reduction} of $\cQ$ is the data $(\cQ^{\triangle},i)$ of a $\gB$-torsor $\cQ^{\triangle}$ over $X$ together with an isomorphism $i \colon \cQ^{\triangle}\times^\gB \gG \xrightarrow{\sim} \cQ$ of $\gG$-torsors. This is equivalent to the datum of a factorization
\begin{center}
    \begin{tikzcd}
         \Rep_L(\gG) \arrow[rr,"\eta_{\cQ}"] \arrow[d] && \Vect_X
        \\ \Rep_L(\gB) \arrow[urr,swap,"\eta_{\cQ^\triangle}"].
    \end{tikzcd}
\end{center}
\end{definition}

\begin{lemma} \label{Bredandflgs}
    The functor sending an $X$-scheme $Y$ to the set of $\gB$-reductions of $\cQ\times_X Y$ is represented by $\cQ/\gB$. The analogous statement holds in the category of rigid analytic spaces.
\end{lemma}

Denote by $X_+^*(\gT)$ the set of dominant weights of $\gT$ with respect to $\gB$. For all $\lambda\in X^*_+(\gT)$, denote by $V_\lambda\in \Rep_{L}(\gG)$ the irreducible representation of $\gG$ of highest weight $\lambda$, and by $V_\lambda^{\mu}$ the $\mu$-isotypic subspace of $V_\lambda$ seen as a representation of $\gT$.

Let $\gG^{\der}$ be the derived subgroup of $\gG$, $\gG^{\sic}$ be the simply-connected cover of $\gG^{\der}$, and $\gG^{\ad}$ be the adjoint group of $\gG$. We denote by $\gT^{\ad}$ the image of $\gT$ in $\gG^{\ad}$, it is a maximal torus. Let $\gT^{\der} \colonequals \gT \cap G^{\der}$ and let $\gT^{\sic}$ be the preimage of $\gT^{\der}$ in $\gG^{\sic}$. 

\begin{lemma}\label{T_mod_Tder} \phantom{a}
    \begin{enumerate}
        \item The group $\gT^{\der} = \gT \cap \gG^{\der}$ is connected, hence a maximal torus of $\gG^{\der}$.
        \item The surjection $\gT \to \gG^{\ab}$ descends to an isomorphism $\gT/\gT^{\der} \cong \gG^{\ab}$.
        \item We have $\gT = \gT^{\der} \cdot Z(\gG)^0$. In particular, $\gT^{\ad}$ is also the image of $\gT^{\der}$ in $\gG^{\ad}$.
        \item The group $\gT^{\sic}$ is connected, hence a maximal torus of $\gG^{\sic}$.
    \end{enumerate}
\end{lemma} 

\begin{proof}
    We have a surjection $\pi \colon \gH \colonequals [\gG,\gG] \times Z(\gG)^0 \to \gG$.
    Since $Z(\gG) \subseteq \gT$, we have $\pi^{-1}(\gT) = (\gT \cap \gG^{\der}) \times Z(\gG)^0$.
    Moreover, $\pi^{-1}(\gT)^0 = (\gT \cap \gG^{\der})^0 \times Z(\gG)^0$ is a maximal torus of $\gH$.
    As $\ker(\pi) = \gG^{\der} \cap Z(\gG)^0$ is central in $\gG$, it is central in $\gG^{\der}$, and so is central in $\gH$.
    Hence, $\ker(\pi) \subseteq \pi^{-1}(\gT)^0$. As $\pi^{-1}(\gT)/\ker(\pi) \cong \gT$ is connected, this implies that $\pi^{-1}(\gT)$ is connected.
    We conclude that $\gT \cap \gG^{\der}$ is connected. (2) and (3) follow immediately.
    As $\pi' \colon \gG^{\sic} \to \gG^{\der}$ is a central extension, we have $\ker(\pi') \subseteq Z(\gG^{\sic}) \subseteq (\gT^{\sic})^0$.
    Since $\gT^{\sic}/\ker(\pi') \cong \gT^{\der}$ is connected, $\gT^{\sic}$ is connected.
\end{proof}

There are natural inclusions
$$ X^*(\gT)\supseteq X^*(\gT^{\ad})\subseteq X^*(\gT^{\der})\subseteq X^*(\gT^{\sic}) $$
such that $X^*(\gT^{\sic})/X^*(\gT^{\ad})$ is finite. 
Therefore we have $X^*(\gT^{\sic})_{\Q} = X^*(\gT^{\ad})_{\Q} \subseteq X^*(\gT)_{\Q}$.
Denote by $\alpha_1,\dots,\alpha_n\in X^*_+(\gT)$ the simple roots of $\gG$ associated to $\gB$, by $\alpha_1^{\vee}, \dots, \alpha_n^{\vee} \in X_*(\gT)_{\Q}$ the corresponding coroots and let $\beta_1,\dots,\beta_n\in X^*(\gT^{\sic})_{\Q}$ be a basis dual to the coroots.
We call $\beta_1,\dots,\beta_n$ \emph{fundamental weights}.

\begin{definition}\label{def_quasi_fund_wt}
    For all $i\in \{1,\dots,n\}$, denote by $\omega_i\in X^*_+(\gT^{\ad})\subseteq X^*_+(\gT)$ the smallest positive multiple of $\beta_i$ which is in $X^*(\gT^{\ad})$. The $\omega_i$'s are called the \emph{quasi-fundamental weights} of $\gG$ with respect to $\gB$. 
\end{definition}

\begin{lemma}\label{decomp_weights} We have a natural identification $X^*(\gG) = \{\lambda \in X^*(\gT) \mid \langle \lambda, \alpha_i^{\vee} \rangle = 0\}$.
    In particular, every $\lambda \in X^*(\gT)$ can be written uniquely as $\lambda = \sum_{i=1}^n \langle \lambda, \alpha_i^{\vee} \rangle \beta_i + \lambda_0$, where $\langle \lambda, \alpha_i^{\vee} \rangle \in \Z$ and $\lambda_0 \in X^*(\gG)$. Moreover, the set $X_0 \colonequals \{\lambda \in X^*(\gT) \mid \lambda_0 = 0\}$ maps isomorphically onto $X^*(\gT^{\der})$. This provides us with canonical decompositions $X^*(\gT) = X^*(\gT^{\der}) \oplus X^*(\gG)$ and $X^*_+(\gT) = X^*_+(\gT^{\der}) \oplus X^*(\gG)$.
\end{lemma}

\begin{proof}
    By \Cref{T_mod_Tder} we have a natural identification $X^*(\gG) = \{\lambda \in X^*(\gT) \mid \lambda(\gT^{\der}) = 1\}$.
    As the coroots generate $\gT^{\der}$, the right hand side identifies with $\{\lambda \in X^*(\gT) \mid \langle \lambda, \alpha_i^{\vee} \rangle = 0\}$, as claimed. It follows that $\langle \lambda_0, \alpha_i^{\vee}\rangle = 0$ for all $i$.
    Since $X_0 \cap X^*(\gG^{\ab}) = 0$ and $X_0$ has the correct rank, to prove that $X_0$ maps isomorphically to $X^*(\gT^{\der})$, it suffices to show that $X_0$ is saturated. This is clear by definition. The decomposition $X^*(\gT) = X_0 \oplus X^*(\gG)$ is canonical and the rest follows easily.
\end{proof}

Let $\Lambda \subset X_+^*(\gT^{\der})$ be a finite set of $\Q$-generators of $X^*_+(\gT^{\der})_\Q$. Consider the map $\wp\colon \bbN^{\Lambda}\to X_+^*(\gT^{\der}), ~(n_{\lambda})_{\lambda\in \Lambda}\mapsto \sum_{\lambda\in \Lambda} n_{\lambda} \lambda$, and let $R_{\Lambda}\subseteq \bbN^{\Lambda}\times \bbN^{\Lambda}$ be a finite set of generators\footnote{Every finitely generated commutative monoid is finitely presented.} of the equivalence relation given by $\wp$.

\begin{theorem}\label{pluckerdatum}
    Let $\cQ$ be a $\gG$-torsor over $X$ with associated fiber functor $\eta_{\cQ}$. Seeing the $\lambda$-isotypic subspace $V_\lambda^\lambda$ of $V_\lambda$ as a $\gB$-representation, the map $(\cQ^\triangle,i)\mapsto \big(\eta_{\cQ^\triangle}(V^{\lambda}_\lambda)\subseteq \eta_{\cQ^\triangle}(V_\lambda)\xrightarrow[i]{\sim}\eta_{\cQ}(V_\lambda)  \big)_{\lambda\in X^*_+(\gT)}$ induces a bijection between the following data: 
    \begin{enumerate}
        \item[(1)] A $\gB$-reduction $(\cQ^\triangle,i)$ of $\cQ$.
        \item[(2)] A Plücker datum: a family of saturated sub-line bundles $\mathscr{L}_\lambda$ of $\eta_{\cQ}(V_\lambda)$ for $\lambda\in X^*_+(\gT)$ such that for all $\lambda,\mu \in  X_+^*(\gT)$, the image of $\mathscr{L}_{\lambda+\mu}$ in $\eta_{\cQ}(V_{\lambda})\otimes \eta_{\mathcal{Q}}(V_{\mu})$ via the inclusion $\eta_{\cQ}(V_{\lambda+\mu})\hookrightarrow \eta_{\cQ}(V_{\lambda})\otimes \eta_{\cQ}(V_{\mu})$ coincides with $\mathscr{L}_\lambda\otimes \mathscr{L}_\mu$.
    \item[(3)] A family of saturated sub-line bundles $\mathscr{L}_{\lambda}\subseteq \eta_{\cQ}(V_{\lambda})$ for $\lambda\in \Lambda$ satisfying the following conditions:
    \begin{itemize}
        \item For every $(n_\lambda)_{\lambda}\in \bbN^{\Lambda} $, $\bigotimes_{\lambda \in \Lambda} \mathscr{L}_{\lambda}^{\otimes n_\lambda} \subseteq \eta_{\cQ}(V_{\sum_\lambda n_\lambda \lambda})\subseteq \bigotimes_{\lambda \in \Lambda}\eta_{\cQ}(V_\lambda)^{\otimes n_\lambda};$
        \item For every $ ((n_\lambda)_{\lambda},(m_\lambda)_{\lambda})\in R_{\Lambda}$, $\bigotimes_{\lambda \in \Lambda} \mathscr{L}_{\lambda}^{\otimes n_\lambda}=\bigotimes_{\lambda \in \Lambda} \mathscr{L}_{\lambda}^{\otimes m_\lambda}.$
    \end{itemize}
    \end{enumerate}
    Moreover, given two $\gG$-torsors $\cQ_1$ and $\cQ_2$ over $X$ with $\gB$-reductions $(\cQ_1^\triangle,i_1)$ and $(\cQ_2^\triangle,i_2)$, a morphism of $\gG$-torsors $f\colon \cQ_1 \to \cQ_2$ is induced from a morphism of $\gB$-torsors $f^\triangle\colon \cQ_1^\triangle \to \cQ_2^\triangle$ if and only if the corresponding morphism of fiber functors respects the Plücker data of $\cQ_1$ and $\cQ_2$.
\end{theorem}

\begin{proof}
    The proof of the equivalences between (1) and (2) can be found in \cite[\S 2.3]{dedar2020} (see also \cite[Proposition 5.7]{Haines}). The implication $(2)\Rightarrow (3)$ is immediate, note that $\scrL_0 = \cO_X$.
    It remains to show $(3)\Rightarrow (2)$, that is to construct a Plücker datum from a datum as in $(3)$. As explained in the proof of \cite[Lemma A.2]{Var04}, this can be reduced to the case where $\mathcal{Q}$ is trivial. Let $\mu \in X_+^*(\gT)$, then by \Cref{decomp_weights} we can write $\mu=\mu'+\lambda_0$ for $\mu'\in X_+^*(\gT^{\der})$ and $\lambda_0\in X^*(\gG)$. Moreover, there exists $k\in \mathbb{N}$ such that $k\mu'=\wp((n_\lambda)_\lambda)$ for some $(n_\lambda)_\lambda\in \mathbb{N}^{\Lambda}$. We define $$\mathscr{L}_{(n_{\lambda})}\colonequals \bigotimes_{\lambda \in \Lambda} \mathscr{L}_{\lambda}^{\otimes n_\lambda} \subseteq V_{k\mu'}.$$ We first show that this is independent of the choice of $(n_\lambda)_\lambda$. Let us consider elements of $\bbN^{\Lambda}$ satisfying $\wp((n_\lambda)_\lambda)=\wp((m_\lambda)_\lambda)$. Since $R_{\Lambda}$ generates all relations on $\Lambda$, there exists a chain of equivalences 
    $$ (n_\lambda)_\lambda \sim (n_\lambda^1)_\lambda \sim \cdots \sim  (n_\lambda^k)_\lambda \sim (m_\lambda)_\lambda, $$
    where $a\sim b$ means that there exists $w\in \bbN^{\Lambda}$ such that $(a-w,b-w)\in R_{\Lambda}$, or $(b-w,a-w)\in R_{\Lambda}$, or $a=b$ (see \cite{Her70}). The defining relations of our datum imply that $\mathscr{L}_{a}=\mathscr{L}_{b}$ whenever $a\sim b$. Hence $\scrL_{k\mu'} \colonequals \mathscr{L}_{(n_{\lambda})}$ is well defined. We also define $\scrL_{k\mu} \colonequals\scrL_{k\mu'} \otimes \eta_{\cQ}(\lambda_0^k)$. It remains to construct a line $\mathscr{L}_{\mu}\subseteq V_\mu$. Since $\scrL_{k\mu'}\otimes \scrL_{k\mu'}\subset V_{2k\mu'}\subset V_{k\mu'}\otimes V_{k\mu}$, a theorem of Kostant and Lichtenstein (see \cite[Corollary p. 369] {ProcesiLG} and \cite{Li82}) gives the existence of $g\in \gG^{\der}$ and a highest weight vector $v\in V_{k\mu'}$ such that $\mathscr{L}_{k\mu'}$ is generated by $gv$. Moreover, up to scaling $v$, there exists a highest weight vector $v_{\mu'}\in V_{\mu'}$ such that $v_{\mu'}^{\otimes k}=v$. We then set $\mathscr{L}_{\mu'}$ to be the line generated by $gv_{\mu'}$, and $\mathscr{L}_{\mu}\colonequals \mathscr{L}_{\mu'}\otimes \eta_{\cQ}(\lambda_0) \subset V_{\mu}$. One verifies easily that the family $(\mathscr{L}_{\mu})_{\mu\in X_+^*(\gT)}$ satisfies the conditions of a Plücker datum, completing the proof. 
\end{proof}

\begin{example}\label{examples_for_Lambda}
The two main examples of $\Lambda$ we work with are the following.
\begin{itemize}
    \item The set $\Omega = \{\omega_1, \dots, \omega_n\} \subset X^*(\gT)_{\Q}$ of quasi-fundamental weights. This is the setting of \cite{dedar2020} for which the author proves the above theorem. In this case we can take $R_{\Omega} = \emptyset$.
    \item The \emph{Hilbert basis} $\mathfrak{S}$ of $X_+^*(\gT^{\der})$, namely the minimal generating set of the monoid $X_+^*(\gT^{\der})$. Since $X_+^*(\gT^{\der})$ is a strongly convex rational cone (meaning $x, -x \in X_+^*(\gT^{\der})$ implies $x = 0$) this generating set is finite and unique. When $\gG^{\der}$ is simply connected, $\mathfrak{S}$ consists precisely of the fundamental weights.
\end{itemize}
\end{example}

The advantage of the description $(3)$ is that it allows us to work with a finite input.

\begin{example}\label{example_plucker_datum}
    The Plücker datum associated to the flag $g\gB$ for $g\in \gG$ is given by $\{gV_{\lambda}^{\lambda}\subseteq V_{\lambda}\}_{\lambda\in X^*_+(\gT)}$. Indeed, when $g=1$ then the flag $\gB$ has stabilizer $\gB$ and it is easy to check that $\gB$ stabilizes $V_{\lambda}^{\lambda}\subseteq V_{\lambda}$. So, $g \gB$ has stabilizer $g\gB g^{-1}$, which stabilizes the above system of lines.
\end{example}

\subsection{$\gH$-filtrations}\label{sec:Gfil}

Let $A$ be an $L$-algebra. We refer to \cite[Section 6.1]{Lev13} for the basics on $\gH$-gradings and $\gH$-filtrations that we only briefly recall here.
Let $\Fil_A$ (resp. $\Gr_A$) be the category of decreasing, separated, and exhaustive $\Z$-filtered (resp. graded) $A$-modules $M$, where $M$ and $\gr^i(M)$ are finitely generated projective $A$-modules. 
Note that $\Fil_A$ and $\Gr_A$ are rigid tensor categories, but $\Fil_A$ is never abelian, and neither is $\Gr_A$ if $A$ is not a field. We say that a sequence $0\to M'\to M\to M''\to 0$ in $\Fil_A$ is exact if
\[ 0\to\gr^i(M')\to\gr^i(M)\to\gr^i(M'')\to 0\] 
is an exact sequence of projective $A$-modules for every $i$. 
There are forgetful functors $\omega_{\Fil}$ and $\omega_{\Gr}$ from $\Fil_A$ and $\Gr_A$ to $\Vect_A$, and an exact tensor functor $\fil\colon \Gr_A\to \Fil_A, (M,\gr^i)\mapsto(M,\cF^i\colonequals \oplus_{j\ge i}\gr^j$).

An \emph{$\gH$-filtration} (respectively, an \emph{$\gH$-grading}) over $A$ is an $L$-linear fiber functor $\cF\colon\Rep_L(\gH)\to\Fil_A$ (respectively, $\cG\colon \Rep_L(\gH)\to\Gr_A$). We say that a $\gH$-filtration $\cF$ is \emph{neutral} if $\omega_\Fil\circ\cF=\omega_\gH\otimes_LA$. We say that $\cF$ is \emph{splittable} if there exists a $\gH$-grading $\cG$ such that $\cF=\fil\circ\cG$.
Given an $L$-linear fiber functor $\eta : \Rep_L(\gH) \to \Vect_A$, a $\gH$-filtration (respectively, a $\gH$-grading) \emph{on} $\eta$ is a factorization through $\Fil_A$ (respectively, $\Gr_A$) via the forgetful functor.
An $\gH$-grading $\mathcal{G}$ of $\eta$ corresponds to a morphism of group schemes $\nu_{\mathcal{G}}\colon \G_{\mathrm m, A} \to \uAut^\otimes(\eta)$. 


\subsection{Cocharacter associated to a $\gG$-filtration}\label{sec_cocharacters}

Let $\gG,\gB$ be our chosen connected reductive group and Borel subgroup. 
For the rest of this subsection, we stick to the case when $A=L$. 
Let $\eta\colon \Rep_L(\gG)\to \Vect_{\Qpbar}$ be a fiber functor equipped with a $\gG$-filtration $\cF$.
Since the coefficient ring is an algebraically closed field, $\eta$ is neutralizable.
Since $\gG$ is reductive, the filtration is splittable (see \cite[Theorem 4.2.13]{DatOrlikRapoport} or \cite[Ch. IV, Proposition 2.2.5]{SaavedraRivano}).
We can associate to $\cF$ a conjugacy class of cocharacters $[\nu_{\cF}]\in X_*(\gG_{\Qpbar})/\gG(\Qpbar)$ as follows:
\begin{enum}
    \item[(1)] Choose a neutralization $\tau_0\colon \eta\xrightarrow{\sim} \omega_\gG\otimes_L \Qpbar$, and let $\tau_0(\cF)$ the induced neutral filtration on $\omega_\gG \otimes_L  \Qpbar$.
    \item[(2)] Choosing a splitting of $\tau_0(\cF)$, we get a cocharacter $\nu_{\cF}\colon \mathbb{G}_{\mathrm m,\Qpbar}\to \gG_{\Qpbar}$. 
\end{enum}
By \cite[Proposition 6.1.15]{Lev13} the conjugacy class $[\nu_{\cF}]$ is independent of all choices. Given the isomorphism $X_*(\gT)/\gW \cong X_*(\gG_{\Qpbar})/\gG(\Qpbar)$ induced by the natural inclusion, this conjugacy class corresponds to a unique dominant cocharacter $\varpi\in X_\ast(\gT)$ (i.e. such that $\langle\alpha,\varpi\rangle \ge 0$ for all positive roots $\alpha$). 

\begin{definition}\label{regcochar}
We say that a dominant cocharacter $\varpi \colon \Gm \to \gG_L$ is \emph{regular} if $\langle\alpha,\varpi\rangle>0$ for all positive roots $\alpha$ of $\gG$. 
\end{definition}

Conversely, let $\nu\colon \mathbb{G}_{\mathrm m,\Qpbar}\to \gG_{\Qpbar}$ be a cocharacter. Then \cite[Proposition 6.1.4]{Lev13} attaches to $\nu$ a neutral $\gG$-grading to $\nu$, hence a $\gG$-filtration $\fil^\bullet_\nu$ defined over a finite extension of $L$ (i.e. a filtration of $\omega_{\gG}\otimes_L\Qpbar\colon \Rep_L(\gG)\to\Vect_{\Qpbar}$, $\omega_{\gG}$ the forgetful functor) whose associated conjugacy class of cocharacters is $[\nu]$.


\begin{lemma}\label{laststepfiltration}
For every $\lambda\in X^*_+(\gT)$,
$$ \langle  \lambda,\varpi\rangle=\inf\{\ i\in \Z \ | \ \eta_{\cF}^i(V_{\lambda})\neq 0\ \}, $$
and if $\varpi$ is regular, then $\eta_{\cF}^{\langle \lambda,\varpi\rangle}(V_{\lambda})$ is 1-dimensional.
\end{lemma}

\begin{proof}
Choosing $\tau_0$ and the splitting of $\tau_0(\cF)$ accordingly, we can assume that $\eta_{\cF}=\fil\circ\eta_{\varpi}$, where $\eta_{\varpi}\colon \Rep_L(\gG)\to \Gr_{\Qpbar}$ is given by the cocharacter $\varpi$. In particular, for $\lambda\in X^*_+(\gT)$, we have $\eta_{\cF}^i(V_{\lambda})=\oplus_{\langle \mu,\varpi\rangle\ge i} V_\lambda^{\mu}$, where $V_{\lambda}=\oplus_{\mu}V_\lambda^{\mu}$ is the decomposition into weight spaces. We note that $\langle \lambda,\varpi \rangle \ge \langle \mu,\varpi \rangle$ for any occurring $\mu \neq \lambda$, with a strict inequality if $\varpi$ is regular. Since $\dim V_{\lambda}^{\lambda}=1$, we get the result.
\end{proof}

We can associate to $\varpi$ an $L$-point $\delta(\varpi)$ on $\Lie(\gT)\GIT{\gW} $ by taking the image of $1$ under the map 
$$\mathrm{d}\varpi\colon L =\Lie(\Gm)\to \Lie(\gT)\to  \Lie(\gT)\GIT\gW.$$ 
Let $A$ be an $L$-algebra, $m\in \Lie(\gG)(A)$, and denote by $\overline{m}$ its image in $\Lie(\gG)\GIT \gG\cong \Lie(\gT)\GIT\gW$. Let $\mathcal{R}(\gG)$ denote the representation ring of $\gG$, which can be identified as $\mathcal{R}(\gG) = \Z[X^*(\gT)]^{\gW}$ (see for instance \cite[20.28]{MilneRed}), where $\Z[X^*(\gT)]$ is the group algebra generated by the symbols $\mathrm{e}^\lambda$ for $\lambda\in X^*(\gT)$. There is a $A$-linear map $A[X^*(\gT)]^{\gW}\to A[\Lie(\gT)\GIT \gW], \lambda \mapsto \mathrm{d}\lambda^*(t) $, where $\mathrm{d}\lambda \colon \Lie(\gT)\to \Ga$ is the derivative map and $A[\Ga]=A[t]$ is the ring of regular functions. Therefore, we can see $\overline{m}$ as a group homomorphism $\mathcal{R}(\gG)\to A$, or as a map $\Rep_L(\gG)\to A$ that satisfies the required compatibilities. 

\begin{lemma}\label{effectonrhoV}
    The image of $(V,\rho_V)\in \Rep_L(\gG)$ under $\overline{m}$ is equal to $\tr(\mathrm{d}\rho_V(m))$. In particular, the image of $(V,\rho_V)$ under $\delta(\varpi)$ is given by the image of $1$ under the map
    $$ A=\Lie(\Gm)(A)\xrightarrow{\mathrm{d}\varpi}\Lie(\gG)(A)\xrightarrow{\mathrm{d}\rho_V}\Lie(\GL(V))(A)\xrightarrow{\tr}A. $$
\end{lemma}


\begin{proof}
The derivation gives an isomorphism $X_*(\gT)\otimes A\cong \Lie(\gT)(A)$ given by sending a cocharacter $\alpha$ to $\mathrm{d}\alpha(1)$. As $m$ can be written as a sum of rank one elements, we can assume by linearity, and  up to conjugation,  that $m=\mathrm{d}\alpha(1)$ for some $\alpha\in X_*(\gT)$. The identification $\mathcal{R}(\gG)\to \Z[X^*(\gT)]^{\gW}$ is given by $[V,\rho_V]\mapsto \sum_{\chi \in X^*(\gT)}\dim V^{\chi}\cdot\mathrm{e}^\chi$, where $V^{\chi}$ is the $\chi$-isotypic component of $V$. Therefore, the image of $[V,\rho_V]$ under $\overline{\alpha}$ is $\sum_{\chi\in X^*(\gT)}\dim V^{\chi}\cdot \langle \chi,\alpha\rangle$. But this is exactly the trace of $\mathrm{d}\alpha(1)$ acting on $V$.
\end{proof}

\subsection{Continuous semilinear actions on torsors}
\label{sec:cont_semilin_act}

Let $\gH$ be a linear algebraic group over $L$, let $R$ be an $L$-algebra and let $A$ be a topological $R$-algebra with a continuous action of a topological group $\Gamma$.
For $\gamma \in \Gamma$, let us write $A_{\gamma}$ for the $A$-algebra $A$ with structure map $\gamma \colon A \to A$.
Let $\cQ$ be an $\gH$-torsor over $A$.
We can define $\gamma^*\cQ \colonequals \Spec(A) \times_{\gamma, \Spec(A)} \cQ$.
A \emph{semilinear action} of $\Gamma$ on $\cQ$ is a family of isomorphisms $a(\gamma) \colon \gamma^*\cQ \to \cQ$ for $\gamma \in \Gamma$, such that the cocycle condition $a(\gamma_1\gamma_2) = a(\gamma_1) \circ \gamma_1^*a(\gamma_2)$ holds.
We say that the action of $\Gamma$ on $T$ is \emph{continuous}, if for every $V \in \Rep_L(\gH)$, the action of $\Gamma$ on the finite projective $A$-module $\cQ \times^\gH V$ is continuous for the quotient topology induced by any $A$-linear surjection $A^n \twoheadrightarrow \cQ \times^\gH V$.\footnote{One can see this as the topology of \cite{ConradTopologies} on the closed subscheme $\{(1-e)X=0\} \subseteq M_n$, where $e \in M_n(A)$ is an idempotent with image isomorphic to $\cQ \times^\gH V$. This shows independence of the presentation and every $A$-linear map between finite projective $A$-modules is continuous.}

If $\cQ$ happens to be the base change $T_A$ of a torsor $T$ defined over $R$, then we have a canonical isomorphism $\gamma^* T_A = T_{A_{\gamma}} \cong T_A$ induced by the equality $A = A_{\gamma}$. A map $\gamma^* T_A \to T_A$ over $A$ thus corresponds to a map of $R$-schemes $T_A \to T_A$ compatible with $\gamma$.

Now assume in addition, that $\cQ = H_A$ is a trivial torsor.
Since $G_A = \gamma^*H_A$, the maps $a(\gamma)$ assemble into a map $a \colon \Gamma \to H(A)$, which satisfies the cocycle condition $a(\gamma_1\gamma_2) = a(\gamma_1)H(\gamma_1)(a(\gamma_2))$, showing that $a \in Z^1(\Gamma, H(A))$. It turns out that $a$ is continuous for the topology on $H(A)$ explained in \cite[Proposition 2.1]{ConradTopologies} if and only if the action of $\Gamma$ is continuous in the above sense: One direction is clear, for the other direction consider a faithful representation $H \hookrightarrow \GL(V)$ and use that $H(A)$ carries the subspace topology of $\GL(V)$.

For the purpose of the proofs in this paper we may often assume that the torsor is trivial, in which case the characterization by continuous cocycles comes in handy.

\section{$\gH$-$\PG$-modules}
\label{sec:HPG}

The goal of this section is to introduce $\PG$-modules valued in a linear algebraic group $\gH$, extending work of de Daruvar \cite{dedar2020}.
In \Cref{sec:reminders_robba_rings} we recall how $\PG$-modules over various Robba rings are defined, and we introduce the notion of cohomological regularity.
In \Cref{secHPGmod} we give the definition of an $\gH$-$\PG$-module over the Robba ring $\cR_{K,A}$ for an algebraic group $\gH$ and an affinoid algebra $A$, and over $\cR_{K,A}[\tfrac{1}{t}]$ for a finite-dimensional $\Qp$-algebra $A$.
In \Cref{trivialphigammamodules} we give an algebraic description of $\gH$-$\PG$-modules with trivial underlying torsor.
In \Cref{sec_param}, we define parameters for a split reductive group $\gG$, and relate them to $\gT$-valued $\PG$-modules for a maximal split torus $\gT$.

\subsection{Robba rings and $\PG$-modules}
\label{sec:reminders_robba_rings}



In order to fix the notation, we recall some definitions from the theory of $\PG$-modules over Robba rings. We mostly follow \cite[Section 2]{KPX} and \cite{Berger2002}.

Let $K$ be our fixed $p$-adic local field and $K_0$ the maximal subfield of $K$ unramified over $\Q_p$. We write $K_\infty \colonequals K(\mu_{p^\infty})$ and $K_{0,\infty} \colonequals K_0(\mu_{p^\infty})$, and we define $K_0'$ as the maximal subfield of $K_{0,\infty}$ unramified over $\Q_p$. We also let $e_K \colonequals [K_\infty : K_{0,\infty}]$ and $H_K\colonequals\Gal(\Qpbar/K_\infty)$. We let $\Gamma_K \colonequals \GK/H_K = \Gal(K_{\infty}/K)$ and let $\Delta_K \subseteq \Gamma_K$ be the $p$-torsion subgroup. Note that $\Delta_K\subseteq \{\pm 1\}$ and is trivial if $p\neq 2$. The group $\Gamma_K' \colonequals \Gamma_K/\Delta_K$ is pro-cyclic and we choose $\gamma_K \in \Gamma_K$ such that its image in $\Gamma_K'$ is a topological generator.

For $s, r \in\R_{>0}$ with $0<s\le r$, we write $A^1[s,r]$ for the rigid analytic annulus of center $0$ over $K_0'$ and variable $T$ defined by $p^{\tfrac{-r}{p-1}}\le\lvert T\rvert\le p^{\tfrac{-s}{p-1}}$, and $\cR_K^{[s,r]}$ for the ring of rigid analytic functions on $A^1[s,r]$. We use the obvious notation for semi-open annuli, e.g. $A^1(0,r] = \bigcup_{0 < s \leq r} A^1[s,r] $. 



For fixed $r>0$, we let $\cR_K^r=\bigcap_{s\to 0}\cR_K^{[s,r]}$  (i.e. the interior radius is fixed, the exterior radius approaches 1), equipped with the structure of a Fréchet space with respect to the collection of the Gauss norms on the $\cR_K^{[s,r]}$, $0<s\le r$. 
We define the \emph{Robba ring} by $\cR_K \colonequals \bigcup_{r\to 0}\cR_K^r$  (i.e. the interior radius also approaches 1). We simply consider $\cR_K$ as a ring; we do not need to specify a topology on it, since all of the $\cR_K$-valued objects we consider will arise from a $\cR_K^r$-valued object for some $r$. 


We explain the relation of the Robba ring $\cR_K$ defined in an abstract variable $T$ with certain period rings.
For this we follow \cite{Berger2002}.
In \cite[§1.3]{Berger2002} the ring $\bB^{\dagger,R}_{K}$ is defined for every $R>0$.
We use a capital $R$ to distinguish this radius from our radius $r$, since the two are normalized differently.
They are related by the formula $pR = r$.
By \cite[Proposition 1.4]{Berger2002}, there exists some $n(K) \in \Z_{\geq 0}$ and $\pi_K \in \bB_K^{\dagger, R_{n(K)}}$, for $R_n \colonequals p^{n-1}(p-1)$ such that for every $R \geq R_{n(K)}$ every element $x \in \bB_K^{\dagger, R}$ can be written\footnote{In \cite{Berger2002} this is stated with $F$, which is our $K_0$, but there is an erratum saying that $F$ should be replaced with $K_0'$ in the definition of $\bB_K^{\dagger}$, hence also in the definition of $\bB_K^{\dagger, R}$, see \cite{BergerErrata}.} as $x = \sum_{i \in \Z} a_i \pi_K^i$, where $a_i \in K_0'$ and $\sum_{i \in \Z} a_i T^i$ converges on the semi-open annulus $\{p^{-1/(e_K R)} \leq |T| < 1\}$.

The ring $\bB^{\dagger,R}_K$ inherits a Fréchet topology from $\widetilde{\bB}^{\dagger, R}_{\rig}$ defined in \cite[Definition 2.16]{Berger2002}.
In \cite[§2.6]{Berger2002} the ring $\bB^{\dagger,R}_{\rig, K}$ is defined as the completion of $\bB^{\dagger,R}_K$ with respect to the Fréchet topology. It inherits a natural action of $\Gamma_K$.
There is a Frobenius map $\varphi \colon \bB_{\rig, K}^{\dagger,R} \to \bB_{\rig, K}^{\dagger,R/p}$ commuting with the action of $\Gamma_K$.

For every $r>0$, the ring $\cR_K^r$ (respectively, $\cR_K$) is denoted\footnote{Again, $F$ should be replaced with $K_0'$ in the definition of $\cH_F^{\alpha}$ and $\cH_F$.} $\cH_F^\alpha$ with $\alpha=p^{\tfrac{-r}{p-1}}$ (respectively, $\cH_F$) in \cite[Section 2]{Berger2002}. By \cite[Proposition 2.31]{Berger2002}, translated into our notation, for $r \le C(\pi_K)\colonequals\tfrac{1}{e_Kp^{n(K)-1}}$ the evaluation map $T\mapsto\pi_K$ defines an isomorphism
\begin{equation}\label{Ttopi} \cR_K^r\xrightarrow{\sim}\bB_{\rig,K}^{\dagger,r/p}. \end{equation}
The choice of $\pi_K$ is not unique, but the theory of $\PG$-modules over the Robba ring does not depend on it : two different choices simply give a different isomorphism \eqref{Ttopi}, and the restriction on $r$ is not relevant since we are interested in $\PG$-modules over $\cR_K$, i.e. all of our objects need only be defined for $r$ arbitrarily close to 0. We fix a choice of $\pi_K$ once and for all. 




We equip $\cR_K^r$ with the $\Gamma_K$-action and the Frobenius map $\cR_K^r\to\cR_K^{r/p}$ induced by the respective structures on $\bB_{\rig,K}^{\dagger,r/p}$ via the isomorphism \eqref{Ttopi}. These structures depend on the choice of $\pi_K$ (i.e. they might be defined by different formulas in the variable $T$), but $\cR_K^r$ with the extra structure is uniquely determined up to isomorphism via the isomorphism \eqref{Ttopi}.

\subsubsection{$(\varphi, \Gamma_K)$-modules over $\cR_{K,A}$ and $\cR_{K,X}$}

In the following definitions, let $0 < s \leq r \leq C(\pi_K)$ and let $A\in \Aff_L$. 
We set $\cR^{[s,r]}_{K,A} \colonequals \cR_K^{[s,r]}\wh\otimes_{\Q_p}A$, $\cR_{K,A}^r \colonequals \bigcap_{0 < s \leq r}\cR_{K,A}^{[s,r]}$ and $\cR_{K,A} \colonequals \bigcup_{0 < r \leq C(\pi_K)}\cR_{K,A}^r$. By \cite[Lemme 1.3 (i)]{chenevier2010sur} $\cR^{[s,r]}_{K,A}$ is the ring of functions on $A^1[s,r] \times \Sp(A)$ and $\cR^r_{K,A}$ is the ring of functions on $A^1(0,r] \times \Sp(A)$.


\begin{definition}\phantom{a}
    \begin{enum}
        \item A \emph{$(\varphi,\Gamma_K)$-module over $\cR_{K,A}^{r}$} is a finite projective module $D^{r}$ over $\cR^r_{K,A}$ endowed with an isomorphism $\Phi_{D^r}\colon \varphi^* D^{r}\xrightarrow{\sim} D^{r/p} \colonequals D^{r} \otimes_{\cR_{K,A}^{r}} \cR_{K,A}^{r/p}$ of $\cR_{K,A}^{r/p}$-modules together with a continuous semilinear action of $\Gamma_K$ commuting with $\Phi_{D^r}$. 
        \item A \emph{$(\varphi,\Gamma_K)$-module over $\cR_{K,A}$} is a finite projective module $D$ over $\cR_{K,A}$ endowed with a morphism $\Phi_D\colon \varphi^* D \to D$ together with a semilinear action of $\Gamma_K$ commuting with $\Phi_D$, which is the base change to $\cR_{K,A}$ of a $(\varphi,\Gamma_K)$-module over $\cR_{K,A}^{r}$ for some $r \in(0,C(\pi_K)]$.
    \end{enum}
\end{definition}

\begin{remark}
    Note that a free $\cR_{K}$-module $D$ equipped with a commuting, semilinear, continuous action of $\varphi$ and $\Gamma_K$ descends to a $(\varphi,\Gamma_K)$-module over $\cR_{K}^{r}$ for some $r>0$ by \cite[Theorem I.3.3]{Ber08}. Hence $D$ is a $\PG$-module over $\cR_{K}$.
\end{remark}

Now let $X$ be a rigid analytic space over $L$.

\begin{definition}[{cf. \cite[Definition 6.2.1]{KPX}, \cite[Definition 3.3]{dedar2020}}]
For $? = r,\varnothing$, a $\PG$-module over $\cR^?_{K,X}$ is the datum of a $\PG$-module $D_A$ over $\cR^?_{K,A}$ for each open affinoid subdomain $\Sp A$ of $X$, together with isomorphisms $i_{AB}\colon D_B\otimes_{\cR^?_{K,B}}\cR^?_{K,A}\cong D_A$ for every $B$ with $\Sp B\subset \Sp A$, satisfying the usual cocycle conditions.
\end{definition}

Note that we do not define the objects $\cR^?_{K,X}$ themselves; they just appear as a convenient notation. 
One checks that this definition recovers the previous one in the case when $X$ is affinoid. However, as remarked in \cite{KPX}, for $X$ not quasi-compact a $\PG$-module over $\cR_{K,X}$ might not arise as a base change from $\cR^r_{K,X}$ for some $r$, as the required $r$ is not necessarily bounded away from zero as one ranges over the affinoid subdomains of $X$.

Throughout the text, we almost always work over affinoid $X$, with the exception of \Cref{sec:trivar} in which we invoke \cite[Theorem 6.10]{dedar2020} for the $\PG$-module attached to a universal deformation.

We denote by $\Phi\Gamma_{K,A}^+$ the category of $(\varphi,\Gamma_K)$-modules over $\cR_{K,A}$, with the obvious choice of morphisms.
For a rigid space $X$, we use the notation $\Phi\Gamma_{K,X}^+$.

\subsubsection{$(\varphi, \Gamma_K)$-modules over $\cR_{K,A}[\tfrac{1}{t}]$}\label{sec:phigammat}


We now turn to the definition of $\PG$-modules after inverting $t \in \cR_K$. We first recall a few definitions from $p$-adic Hodge theory.

We choose and fix a system $\varepsilon = (\zeta_{p^n})_{n\geq 0}$ of $p^n$-th roots $\zeta_{p^n}$ of $1$ in $\Qpbar$ such that $\zeta_{p^{n+1}}^p = \zeta_{p^n}$, and denote by $t$ the associated ``Fontaine's $2i\pi$ element''\footnote{See the discussion in \cite[after Definition 4.4.7]{BrinonConrad} on the choice of $t$.}. More precisely, if $R=\lim_{x\mapsto x^p}\cO_{\C_p}$ is the tilt of $\cO_{\C_p}$, $\varepsilon\in R$ is a compatible system of roots of unity with $\varepsilon^{(0)}=1$ and $\varepsilon^{(1)}\ne 1$, and $[ \cdot ]\colon R\to W(R)$ denotes the Teichm\"uller lift, we choose $t=\log([\varepsilon])$, which is defined as a power series in $[\varepsilon]-1$ that converges in $B_\dR$. 

We use the same notation both for $t\in\cR_{K}$ and $t\in B_\dR^+$, as is the case in \cite{Berger2002}, since this should not create confusion in what follows. In \cite[Section 2.2]{Berger2002} Berger defines an injective homomorphism $\iota_0:\tilde\bB^{\dagger,r_0}_{K}\to B_\dR^+$ for $r_0=\tfrac{p-1}{p}$ (that amounts to writing elements as power series that converge in $B_\dR$). Note that, for every $r$, $\tilde\bB^{\dagger,r}_{K}$ contains elements obtained as Teichm\"uller lifts of elements of the field $\mathbf E$ from \cite[Section 1.1]{Berger2002}; in particular, we can define an element $t=\log([\varepsilon])\in \tilde\bB^{\dagger,r}_{K}$ exactly as we did for $t\in B_\dR$. We denote the two $t$ with the same letter, since $\iota_0$ maps one to the other when $r\le r_0$.

For $r\le\min\{p-1,C(\pi_K)\}$, the preimage of $t\in\bB^{\dagger,r/p}_{K}$ under \eqref{Ttopi} gives an element $t\in\cR_K^{r}$, and composing $\iota_0$ with \eqref{Ttopi} gives a map $\cR_K^{r}\to B_{\dR}$ sending $t$ to $t$. 



\begin{definition}\label{def_PG_RKAoneovert}
    Let $A\in \Art_L$.
    A \emph{$(\varphi,\Gamma_K)$-module over $\cR_{K,A}[\tfrac{1}{t}]$} is a finite free $\cR_{K,A}[\tfrac{1}{t}]$-module $\cM$ equipped with a semilinear endomorphism $\varphi$ and a semilinear action of  $\Gamma_K$ commuting with $\varphi$, such that there exists an $\cR_K$-lattice $D$ of $\cM$ (i.e. $D$ is a finite projective $\cR_K$-submodule of $\cM$ with $D[\tfrac{1}{t}] \cong \cM$) stable under $\varphi$ and $\Gamma_K$, which is a $(\varphi,\Gamma_K)$-module over $\cR_K$.
    \\ We denote by $\Phi\Gamma_{K,A}$ the category of $(\varphi,\Gamma_K)$-modules over $\cR_{K,A}[\tfrac{1}{t}]$.
\end{definition}

\begin{remark} We do not know whether a $\PG$-module over $\cR_{K,A}[\tfrac{1}{t}]$ always admits an $\cR_{K,A}$-lattice, that is stable under the action of $\varphi$ and $\Gamma_K$, compare \cite[Remark 3.3.1]{BHS19}. We could also have imposed existence of an $\cR_{K,A}$-lattice as a condition, but we chose this weaker condition in order to be compatible with the definition given in \cite[§3.3]{BHS19}.
\end{remark}

In \Cref{def_PG_RKAoneovert} we intentionally do not impose any continuity condition on the action of $\Gamma_K$.
It turns out that the existence of an $\cR_K$-lattice implies that a $(\varphi,\Gamma_K)$-module over $\cR_{K,A}[\tfrac{1}{t}]$ always satisfies the continuity condition in \Cref{def_cont_one_over_t}. It will be important for studying the relationship between extensions and cohomology in \Cref{lattice_lemma}.

We will first discuss continuous maps from a more general perspective.
For the next lemma, we remind the reader that on a finite free $\cR_{K,A}^r$-module $M$ there is a canonical topology by choosing an isomorphism $M \cong (\cR_{K,A}^r)^N$ and taking the product topology on the target. This topology does not depend on the choice of a basis.

\begin{lemma}\label{topol_RKAoneovert}
    Let $A\in \Art_L$, let $X$ be an affine scheme of finite type over $L$, let $S$ be a profinite set, and let $c \colon S \to X(\cR_{K,A}[\tfrac{1}{t}])$ be a map. The following are equivalent.
    \begin{enum}
        \item For all $f \in \cO(X)$ the map $f \circ c \colon S \to X(\cR_{K,A}[\tfrac{1}{t}]) \to \bbA^1(\cR_{K,A}[\tfrac{1}{t}]) = \cR_{K,A}[\tfrac{1}{t}]$ takes values in $t^{-n}\cR_{K,A}^r$ for some $0 < r \leq C(\pi_K)$ and some $n \geq 0$ and is continuous.
        \item There exists a closed immersion $f \colon X \to \bbA^N$, such that the map $f \circ c \colon S \to X(\cR_{K,A}[\tfrac{1}{t}]) \to \bbA^N(\cR_{K,A}[\tfrac{1}{t}]) = (\cR_{K,A}[\tfrac{1}{t}])^N$ takes values in $(t^{-n}\cR_{K,A}^r)^N$ for some $0 < r \leq C(\pi_K)$ and some $n \geq 0$ and is continuous.
    \end{enum}
    Moreover, this condition is functorial in $S$ and $X$.
\end{lemma}

We denote by $\cC(S, X(\cR_{K,A}[\tfrac{1}{t}]))$ the set of maps $c$ satisfying the equivalent conditions of \Cref{topol_RKAoneovert}.

\begin{proof}
The idea here is similar to \cite[§4]{defG}. 
    The implication from (1) to (2) is clear, since $X$ is of finite type, so we can find a closed immersion into some $\bbA^N$ and then use the projections to $\bbA^1$. From (2) to (1), we use a surjection $\cO(\bbA^N) \to \cO(X)$, so we lift $f$ to some $\tilde f : \bbA^N \to \bbA^1$. To deduce continuity of $f \circ c$ it is enough to show, that the polynomial $\tilde f$ induces a continuous map $(t^{-n}\cR_{K,A}^r)^N \to t^{-m}\cR_{K,A}^r$ for some $m \geq 0$. This is clear, since $\cR_{K,A}^r$ is a topological ring and the exponent of $t$ is bounded from below. 
    Functoriality in $S$ is clear, and functoriality in $X$ follows from part (1).
\end{proof}

\begin{definition}\label{def_cont_one_over_t}
    Let $A\in \Art_L$, and let $\cM$ be a finite free $\cR_{K,A}[\tfrac{1}{t}]$-module equipped with a semilinear $\Gamma_K$-action.
    We say that the action is \emph{continuous} if after choosing an $\cR_{K,A}[\tfrac{1}{t}]$-basis $e_1, \dots, e_d \in \cM$ the map $c \colon \Gamma_K \to M_d(\cR_{K,A}[\tfrac{1}{t}])$ given by the matrix representing the operator $\gamma^*\cM \to \cM$ for each $\gamma \in \Gamma_K$ satisfies the following condition:
    There exists some $0 < r \leq C(\pi_K)$ and some $n \in \Z_{\geq 0}$, such that $c$ takes values in $t^{-n}M_d(\cR_{K,A}^r)$ and is continuous for the natural topology\footnote{We take an isomorphism $t^{-n}M_d(\cR_{K,A}^r) \cong (\cR_{K,A}^r)^{d^2}$ and the topology induced by the product topology. This does not depend on the choice of a basis of $t^{-n}M_d(\cR_{K,A}^r)$.} as an $\cR_{K,A}^r$-module.
\end{definition}

One verifies that this condition does not depend on the choice of a basis of $\cM$.

\subsubsection{Cohomology of $\PG$-modules (Herr complex versus standard complex)}
\label{sec:Herr_vs_std}

In \Cref{sec_form_smooth} we face the phenomenon that extensions of $\PG$-modules are naturally classified by $1$-cocycles that are defined in terms of standard $1$-cocycles for group cohomology.
In this section we will define the standard cohomology groups of a $\PG$-module and relate it to cohomology of the Herr complex.

Let $A\in \Art_L$, and let $\cM$ be a $\PG$-module over $\cR_{K,A}[\tfrac{1}{t}]$.
There are obvious analogs of the definitions and statements discussed in this subsection for a $\PG$-module over $\cR_{K,A}$ for an affinoid $\Qp$-algebra $A$, which are left to the reader.
We recall that $\cM = D[\tfrac{1}{t}]$ for a $\PG$-module $D$ over $\cR_K$. Moreover, $D = \bigcup_{0 < r \leq C(\pi_K)} D^r$ and $D^r = \bigcap_{0 < s \leq r} D^{[s,r]}$, where $D^{[r,s]}:=D\widehat\otimes_{\cR_K^r}\cR_K^{[r,s]}$. 

The \emph{Herr complex} is a three-term complex in cohomological degrees $0,1,2$, given by
\begin{align}
    \big[\cM^{\Delta_K} \xrightarrow{(\varphi-1, \gamma_K-1)} \cM^{\Delta_K} \oplus \cM^{\Delta_K} \xrightarrow{(\varphi-1)\pr_2 - (\gamma_K-1)\pr_1} \cM^{\Delta_K} \big]. \label{Herr_complex}
\end{align}
We write $B^i_{\varphi,\gamma_K}(\cM)$, $Z^i_{\varphi,\gamma_K}(\cM)$ and $H^i_{\varphi,\gamma_K}(\cM)$ for the groups of coboundaries, cocycles and cohomology classes defined by the Herr complex.

In the spirit of \Cref{topol_RKAoneovert}, we define $\cC(\Gamma_K^i, \cM) := \bigcup_{m \geq 0} \bigcup_{0 < r \leq C(\pi_K)} \Cont(\Gamma_K^i, t^{-m} D^r)$. As in \Cref{sec_std_cplx}, we define the standard complex for $\PG$-cohomology by 
\begin{align}
    C^{\bullet}_{\varphi, \Gamma_K}(\cM) \colonequals \Tot \big[\cC(\Gamma_K^{\bullet},\cM) \xrightarrow{\varphi-1} \cC(\Gamma_K^{\bullet},\cM) \big].
\end{align}

\begin{lemma}\label{group_to_two_term}
    The map 
    \begin{align}
        \cC((\Gamma_K')^{\bullet}, \cM^{\Delta_K}) \longrightarrow \big[\cM^{\Delta_K} \xrightarrow{\gamma_K-1} \cM^{\Delta_K} \big] 
    \end{align}
    given by the identity in degree $0$ and evaluation at $\gamma_K$ in degree $1$ is a chain homotopy equivalence.
\end{lemma}

\begin{proof}
    Write $\cM = \bigcup_{m \geq 0} \bigcup_{0 < r \leq C(\pi_K)} t^{-m} D^r$ and pass to the colimit in \Cref{comparison_herr_Gamma_modules}.
\end{proof}


\begin{lemma}\label{Delta_trick}
    The chain map
    \begin{equation} \label{chain_inflation}
    \begin{tikzcd}
        \cM^{\Delta_K} \arrow[d] \arrow[r, "\delta^0"] & \cC(\Gamma_K',\cM^{\Delta_K}) \arrow[d] \arrow[r, "\delta^1"]  & \cC((\Gamma_K')^2,\cM^{\Delta_K}) \arrow[d] \arrow[r, "\delta^2"] & \cdots \\
        \cM \arrow[r, "\delta^0"] & \cC(\Gamma_K,\cM) \arrow[r, "\delta^1"]  & \cC(\Gamma_K^2,\cM) \arrow[r, "\delta^2"] & \cdots
    \end{tikzcd}
    \end{equation}
    induces a quasi-isomorphism $\cC((\Gamma_K')^{\bullet}, \cM^{\Delta_K}) \eqto \cC(\Gamma_K^{\bullet}, \cM)$.
\end{lemma}

\begin{proof}
    We have
    \begin{align*}
        \cC(\Gamma_K^i,\cM) &= \bigcup_{m \geq 0} \bigcup_{0 < r \leq C(\pi_K)} \Cont(\Gamma_K^i, t^{-m} D^r) \\
        \cC((\Gamma_K')^i, \cM^{\Delta_K}) &= \bigcup_{m \geq 0} \bigcup_{0 < r \leq C(\pi_K)} \Cont((\Gamma_K')^i, (t^{-m} D^r)^{\Delta_K})
    \end{align*}
    So we may prove the claim for $t^{-m}D^r$ in place of $\cM$.
    Since the chain map \eqref{chain_inflation} induces the inflation maps $H^i(\Gamma_K', t^{-m}D^r) \to H^i(\Gamma_K, (t^{-m}D^r)^{\Delta_K})$ and $H^i(\Delta_K, t^{-m}D^r) = 0$ for $i \geq 1$,
    the claim follows from the inflation-restriction sequence \cite[Proposition 1.6.7]{NSW2e}.
\end{proof}

\begin{proposition}\label{H1_comparison}
    There are canonical isomorphisms $H^i_{\varphi, \gamma_K}(\cM) \cong H^i_{\varphi, \Gamma_K}(\cM)$ for all $i \geq 0$.
\end{proposition}

\begin{proof} 
    The idea is to apply \Cref{Delta_trick} to the complex $[\cM \xrightarrow{\varphi - 1} \cM]$. We consider the resulting double complexes: Let us write ${}_1\mathscr D = [\cC(\Gamma_K^{\bullet},\cM) \xrightarrow{\varphi-1} \cC(\Gamma_K^{\bullet},\cM)]$ and ${}_2\mathscr D = [\cC((\Gamma_K')^{\bullet},\cM^{\Delta_K}) \xrightarrow{\varphi-1} \cC((\Gamma_K')^{\bullet},\cM^{\Delta_K})]$ with ${}_1\mathscr D^{p,q} = \cC(\Gamma_K^q, \cM)$ for $q \in \{0,1\}$, zero otherwise, and similarly for ${}_2\mathscr D$. We have a map ${}_1\mathscr D \to {}_2\mathscr D$ induced by \eqref{chain_inflation}.

    We have associated spectral sequences with ${}_iE_0^{p,q} = {}_i\mathscr D^{p,q}$.
    We have ${}_1E_1^{0,q} = {}_1E_1^{1,q} = H^q(\cC(\Gamma_K^{\bullet}, \cM))$ and ${}_2E_1^{0,q} = {}_2E_1^{1,q} = H^q(\cC((\Gamma_K')^{\bullet}, \cM^{\Delta_K}))$.
    It follows from \Cref{Delta_trick}, that the induced map ${}_1E_1^{p,q} \to {}_2E_1^{p,q}$ is an isomorphism for all $q \geq 0$.
    Since ${}_iE_2^{0,q} = \ker({}_iE_1^{0,q} \xrightarrow{\varphi-1} {}_iE_1^{0,q})$ and ${}_iE_2^{1,q} = \coker({}_iE_1^{1,q} \xrightarrow{\varphi-1} {}_iE_1^{1,q})$, we deduce that the map  ${}_1E_2^{p,q} \to {}_2E_2^{p,q}$ is an isomorphism for all $q \geq 0$. The spectral sequence degenerates on the second page, so the map $$H_{\varphi, \Gamma_K}^q(\cM) \cong H^q(\Tot({}_1\mathscr D)) \to H^q(\Tot({}_2\mathscr D)) \cong H_{\varphi, \Gamma_K'}^q(\cM^{\Delta_K})$$ is an isomorphism for the evident definition of $H_{\varphi, \Gamma_K'}^q(\cM^{\Delta_K})$. We can now run a similar argument using \Cref{group_to_two_term} to see that $H_{\varphi, \Gamma_K'}^q(\cM^{\Delta_K}) \cong H_{\varphi, \gamma_K}^q(\cM^{\Delta_K})$.
\end{proof}

\subsubsection{Cohomological regularity}
\label{sec_coh_reg}

We would like to impose regularity conditions on $\PG$-modules in a purely cohomological way.
We anticipate that this formulation will be convenient when working with $\PG$-modules with values in $L$- and $C$-groups.
We first study the cohomology of rank $1$ objects.
The following lemma is standard.

\begin{lemma}\label{coh_char} Let $\delta \colon K^{\times} \to L^{\times}$ be a continuous character. Then:
\begin{enum}
    \item\label{character_coh_0} $H^0_{\varphi, \gamma_K}(\cR_{K,L}(\delta)) \cong L$ if and only if $\delta = z^{\mathbf k}$ for some $\mathbf k \in \Z_{\leq 0}^{\Sigma}$. Otherwise $H^0_{\varphi, \gamma_K}(\cR_{K,L}(\delta)) = 0$.
    \item\label{character_coh_2} $H^2_{\varphi, \gamma_K}(\cR_{K,L}(\delta)) \cong L$ if and only if $\delta = z^{\mathbf k}|N_{K/\Qp}(z)|$ for some $\mathbf k \in \Z_{\geq 1}^{\Sigma}$. Otherwise $H^2_{\varphi, \gamma_K}(\cR_{K,L}(\delta)) = 0$. 
    \item $\dim_L H^1_{\varphi, \gamma_K}(\cR_{K,L}(\delta)) = [K:\Qp]+1$ if and only if either $\delta = z^{\mathbf k}$ for some $\mathbf k \in \Z_{\leq 0}^{\Sigma}$ or $\delta = z^{\mathbf k}|N_{K/\Qp}(z)|$ for some $\mathbf k \in \Z_{\geq 1}^{\Sigma}$. Otherwise $\dim_L H^1_{\varphi, \gamma_K}(\cR_{K,L}(\delta)) = [K:\Qp]$. 
\end{enum}
\end{lemma}

\begin{proof}
    Denote by $B_L$ the trivial $L$-$B$-pair.
    For an $L$-$B$-pair $W$, we have $H^0(\GK, W) \cong \Hom(B_L, W)$ by \cite[Proposition 2.2 (1)]{NakClass}.
    The latter is canonically isomorphic to $\Hom_{\PGcat^+_{K,L}}(\cR_{K,L}, D(W))$ by \cite[Theorem A]{Berger_2008constr}.
    This in turn is the same as $H^0_{\varphi, \gamma_K}(D(W))$, which follows easily by inspection of the Herr complex.
    Part $(1)$ now follows from \cite[Proposition 2.14]{NakClass}.
    
    By Tate duality, \cite[Theorem 5.7]{LiuCohAndDuality}, $H^2_{\varphi, \gamma_K}(\cR_{K,L}(\delta)) \cong H^0_{\varphi, \gamma_K}(\cR_{K,L}(\delta^{-1}N_{K/\Qp}(z)|N_{K/\Qp}(z)|))$ and as $N_{K/\Qp}(z) = z^{\mathbf 1}$ part $(2)$ also follows.

    Part $(3)$ follows from the Euler characteristic formula \cite[Theorem 1.2]{LiuCohAndDuality} and by combining $(1)$ and $(2)$.
\end{proof}

We have a generalization of \Cref{coh_char} with affinoid coefficients.
Let $A \in \Aff_L$.

\begin{lemma}\label{coh_char_A}
    Let $\delta \colon K^{\times} \to A^{\times}$ be a continuous character and let $D \colonequals \cR_{K,A}(\delta)$. Then:
    \begin{enum}
        \item $H^0_{\varphi, \gamma_K}(D) = 0$ if and only if for all maximal ideals $\frakm$ of $A$, the character $\delta \colon K^{\times} \to (A/\frakm)^{\times}$ is not $z^{\mathbf k}$ for some $\mathbf k \in \Z_{\leq 0}^{\Sigma}$.
        \item $H^2_{\varphi, \gamma_K}(D) = 0$ if and only if for all maximal ideals $\frakm$ of $A$, the character $\delta \colon K^{\times} \to (A/\frakm)^{\times}$ is not $z^{\mathbf k}|N_{K/\Qp}(z)|$ for some $\mathbf k \in \Z_{\geq 1}^{\Sigma}$.
        \item $H^1_{\varphi, \gamma_K}(D)$ is a finite projective $A$-module of rank $[K:\Qp]$ if and only if for all maximal ideals $\frakm$ of $A$, the character $\delta \colon K^{\times} \to (A/\frakm)^{\times}$ is neither $z^{\mathbf k}$ for some $\mathbf k \in \Z_{\leq 0}^{\Sigma}$ nor $z^{\mathbf k}|N_{K/\Qp}(z)|$ for some $\mathbf k \in \Z_{\geq 1}^{\Sigma}$.
    \end{enum}
\end{lemma}


\begin{proof}
    Part (1) for $A = L$ is part (1) of \Cref{coh_char}.
    If $A$ is artinian, then part (1) follows easily by induction.
    We now extend part (1) to affinoid coefficients.
    Assume first that $\delta$ is not $z^{\mathbf k}$ for some $\mathbf k \in \Z_{\leq 0}^{\Sigma}$ after composing with every projection $A \to A/\frakm$.
    The map $A \to \prod_{\frakm, n} A/\frakm^n$, where the product varies over all maximal ideals of $A$ and all positive integers $n$ is injective. By explicit inspection of the rings this also holds for $\cR^{[s,r]}_{K,A} \to \prod_{\frakm, n} \cR^{[s,r]}_{K, A/\frakm^n}$.
    By passing to the limit and then to the colimit also $\cR_{K,A} \to \prod_{\frakm, n} \cR_{K, A/\frakm^n}$ is injective.
    We conclude that the map $H^0_{\varphi, \gamma_K}(\cR_{K,A}(\delta)) \to \prod_{\frakm, n} H^0_{\varphi, \gamma_K}(\cR_{K,A/\frakm}(\delta)) = 0$ is injective.
    Again, by Tate duality \cite[Theorem 2.3.11]{KPX} part (2) follows.

    For part (3), first assume that for all maximal ideals $\frakm$ of $A$, the character $\delta : K^{\times} \to (A/\frakm)^{\times}$ is neither $z^{\mathbf k}$ for some $\mathbf k \in \Z_{\leq 0}^{\Sigma}$ nor $z^{\mathbf k}|N_{K/\Qp}(z)|$ for some $\mathbf k \in \Z_{\geq 1}^{\Sigma}$. We know by \cite[Theorem 4.4.5 (1)]{KPX} that $C_{\varphi, \gamma_K}^{\bullet}(D)$ is quasi-isomorphic to a perfect complex of $A$-modules concentrated in cohomological degrees $[0,2]$. So, by parts (1) and (2), $C_{\varphi, \gamma_K}^{\bullet}(D)$ is quasi-isomorphic to a finite projective $A$-module concentrated in degree $1$. Its rank follows from the Euler characteristic formula.

    Assume that $H^1_{\varphi, \gamma_K}(D)$ is a finite projective $A$-module of rank $[K:\Qp]$.
    Then by the Euler characteristic formula, \cite[Theorem 4.4.5 (2)]{KPX}, $H^0_{\varphi, \gamma_K}(D) = H^2_{\varphi, \gamma_K}(D) = 0$.
    We conclude by parts (1) and (2).
\end{proof}

This motivates the following definition for general $\PG$-modules.

\begin{definition}\label{def_coh_reg}
    Let $D$ be a $\PG$-module over $\cR_{K,A}$.
    We say that $D$ is \emph{cohomologically regular} if $H^1_{\varphi, \gamma_K}(D)$ is a finite projective $A$-module of constant rank $[K:\Qp] \cdot \rank D$.
    We make the analogous definition for $\PG$-modules over $\cR_{K,A}[\tfrac{1}{t}]$ when $A \in \Art_L$.
\end{definition}

\begin{lemma}\label{coh_reg_ses}
    Let $0 \to D_1 \to D_2 \to D_3 \to 0$
    be a short exact sequence of $\PG$-modules over $\cR_{K,A}$.
    If $D_1$ and $D_3$ are cohomologically regular, then so is $D_2$.
    The analogous statement holds for $\PG$-modules over $\cR_{K,A}[\tfrac{1}{t}]$ under the assumption that $A\in \Art_L$.
\end{lemma}

\begin{proof}
    This follows from \Cref{coh_char_A} and the long exact sequence.
\end{proof}

\begin{lemma}\label{coh_reg_param}
    Let $\cM = \cR_{K,L}(\delta)[\tfrac{1}{t}]$ for some continuous character $\delta \colon K^{\times} \to L^{\times}$.
    If neither $\delta$ nor $|N_{K/\Qp}(z)|\delta^{-1}$ is algebraic, then $\cM$ is cohomologically regular.
\end{lemma}

\begin{proof}
    Let $D= \cR_{K,L}(\delta)$. We have $H^i_{\varphi, \gamma_K}(\cM) = \varprojlim_j
 H^i_{\varphi, \gamma_K}(t^{-j}D)$, as filtered colimits are exact, and an inclusion $H^0_{\varphi, \gamma_K}(D) \subseteq H^0_{\varphi, \gamma_K}(t^{-j}D)$.
    So $H^0_{\varphi, \gamma_K}(\cM)=0$ if and only if for all $j \geq 0$, $t^{-j}D$ does not have character of the form $\delta = z^{\mathbf k}$ for some $\mathbf k \in \Z_{\leq 0}^{\Sigma}$, i.e. if $\delta$ is not algebraic. For $H^2_{\varphi, \gamma_K}(\cM)$, we apply Tate duality \cite[Theorem 5.7]{LiuCohAndDuality} to see that $H^2_{\varphi, \gamma_K}(t^{-j}D) = H^0_{\varphi, \gamma_K}((t^{-j}D)^{\vee}(|N_{K/\Qp}(z)|)) = 0$ for all $j \geq 0$, since the parameter $z^j\delta^{-1}|N_{K/\Qp}(z)|$ of $(t^{-j}D)^{\vee}(|N_{K/\Qp}(z)|)$ is not algebraic.
\end{proof}

The functor
$$ \cT \colon \Aff_L \longrightarrow \Set, \quad A \mapsto \Hom_{\cont}(K^{\times}, A^{\times}) $$
is representable by a smooth quasi-Stein rigid analytic space over $L$ of dimension $[K:\Qp]+1$, \cite[Proposition 6.1.1]{KPX}. 

Recall that we have chosen $L$ so that it contains the images of all maps $K \hookrightarrow \Qpbar$.

\begin{definition}\label{def_algchar} A continuous character $\delta : K^{\times} \to A^{\times}$ is \emph{algebraic} if it is of the form $\delta(z) = z^{\mathbf k} = \prod_{\tau \in \Sigma} \tau(z)^{k_{\tau}}$ for some $\mathbf k = (k_{\tau})_{\tau \in \Sigma} \in \Z^{\Sigma}$.
\end{definition}



We define a character $\varepsilon \colon K^{\times} \to \Qp^{\times} \subseteq L^{\times}, ~z \mapsto \prod_{\tau \in \Sigma} \tau(z) \cdot |z|_K$, where the absolute value $|\cdot|_K \colon K \to \R_{\geq 0}$ is normalized such that $|\pi_K|_K = |\cO_K/\pi_K|^{-1}$.
We let $\rec_K \colon K^{\times} \to \cG_K^{\ab}$ be the reciprocity map, normalized so that the uniformizer is sent to a geometric Frobenius element.
Then, for the cyclotomic character $\chi^\cyc \colon \cG_K^{\ab} \to \Qp^{\times}$, we have $\varepsilon = \rec_K \circ \chi^\cyc$.

\begin{definition}\label{def_reg}
We define $\cT_0 \subseteq \cT$ to be the Zariski open complement of the set of $\delta, \varepsilon\delta$ for $\delta : K^{\times} \to L^{\times}$ an algebraic character.
The points of $\cT(L)$ given by $z \mapsto z^{-\mathbf k}$ and $z \mapsto \varepsilon(z)z^{\mathbf k}$ for $\mathbf k \in \Z^{\Sigma}_{\geq 0}$ form a discrete subspace $\cT_{\irreg} \subseteq \cT$. We define $\cT_{\reg}$ to be its complement.\footnote{Our space $\cT_{\reg}$ is the same as that in \cite[§3.7]{BHS19}, \cite[§5.2]{dedar2020} and \cite[§3]{aoki}. When $K=\Qp$ it is the same as that in \cite[Définition 2.28]{chenevier2010sur}.}
\end{definition}



\subsection{$\gH$-$(\varphi, \Gamma_K)$-modules}
\label{secHPGmod}

In the entire \Cref{secHPGmod}, we allow $\gH$ to be an arbitrary linear algebraic group over $L$, not necessarily reductive and not necessarily connected. For example, we will sometimes take $\gH$ to be a Borel subgroup of a reductive group in the following definition.

\begin{definition}\label{def_HPG_module}
    Let $A\in \Aff_L$ (resp. $A\in \Art_L$). An \emph{$\gH$-$(\varphi, \Gamma_K)$-module over $\cR_{K,A}$} (resp. $\cR_{K,A}[\tfrac{1}{t}]$) is an $L$-linear fiber functor $\eta_D\colon \Rep_L(\gH) \to \PGcat_{K,A}^+$ (resp. $\eta_{\cM}\colon \Rep_L(\gH) \to \PGcat_{K,A}$). We denote by $\PGcat_{K,A}^{\gH,+}$ (resp. $\PGcat_{K,A}^\gH$) the groupoid of $\gH$-$\PG$-modules over $\cR_{K,A}$ (resp. $\cR_{K,A}[\tfrac{1}{t}]$).
\end{definition}


\begin{proposition}\label{torsor_HPG_modules}
    Let $A\in \Aff_L$. The functor $\eta_D \mapsto \uHom^{\otimes}(\omega_L \otimes_L \cR_{K,A}, \eta_D)$ defines an equivalence between the groupoid of $\gH$-$\PG$-modules over $\cR_{K,A}$ and the groupoid of triples $(D, \Phi, C)$, where $D$ is a right $\gH$-torsor on $\Spec(\cR_{K,A})$, 
    $\Phi \colon \varphi^* D \to D$ is an isomorphism of $\gH$-torsors and $C = (C(\gamma))_{\gamma \in \Gamma_K}$ is a family of isomorphisms of $\gH$-torsors $C(\gamma) \colon \gamma^* D \to D$, such that the following conditions hold.
    \begin{enum}
        \item $C(\gamma_1\gamma_2) = C(\gamma_1) \circ \gamma_1^* C(\gamma_2)$ for all $\gamma_1, \gamma_2 \in \Gamma_K$ (cocycle condition); \label{HPG_def_1}
        \item $\Phi \circ \varphi^* C(\gamma) = C(\gamma) \circ \gamma^* \Phi$ for all $\gamma \in \Gamma_K$ (commutativity of Frobenius with $\Gamma_K$ action); \label{HPG_def_2}
        \item For all $V \in \Rep_L(\gH)$ the $\cR_{K,A}$-module $D \times^\gH V$ equipped with the induced $\varphi$ and $\Gamma_K$-actions is a $(\varphi, \Gamma_K)$-module over $\cR_{K,A}$ in the usual sense. \label{HPG_def_3}
    \end{enum}
    In \eqref{HPG_def_2} we are implicitly using the fact that $\varphi^* \gamma^* D \cong \gamma^* \varphi^* D$, since the actions of $\varphi$ and $\gamma$ on $\cR_{K,A}$ commute. The same holds with $\cR_{K,A}$ replaced by $\cR_{K,A}[\tfrac{1}{t}]$ and $\PGcat_{K,A}^+$ replaced by $\PGcat_{K,A}$.
\end{proposition}

\begin{proof}
    This follows from \Cref{tannakian_fundamental_equivalence}. 
    See also \cite[Proposition 3.6]{dedar2020}.
\end{proof}

Let $D$ be an $\gH_1$-$\PG$-module over either $\cR_{K,A}$ or $\cR_{K,A}[\tfrac{1}{t}]$, and let $f\colon \gH_1 \to \gH_2$ be any morphism of linear algebraic groups. We will denote by either $f \circ D$ or $D \times^{\gH_1} \gH_2$ the $\gH_2$-$\PG$-module obtained by composition of the fiber functor with the functor $\Rep_L(\gH_2) \to \Rep_L(\gH_1)$.
On the level of torsors, it is formation of the right $\gH_2$-torsor $D \times^{\gH_1} \gH_2$ equipped with the $\varphi$ and $\Gamma_K$-actions inherited from $D$. \Cref{torsor_HPG_modules} tells us, that the $\varphi$ and $\Gamma_K$ actions on $D \times^{\gH_1} \gH_2$ satisfy conditions \eqref{HPG_def_1}, \eqref{HPG_def_2} and \eqref{HPG_def_3}.

\begin{lemma}\label{lattice_lemma}
     Let $A\in \Art_L$. Let $0\to \cM_1 \to \cM \to \cM_2\to 0$ be an exact sequence of finite free $\cR_{K,A}[\tfrac{1}{t}]$-modules equipped with commuting, semilinear, 
     actions of $\varphi$ and $\Gamma_K$, such that the action of $\Gamma_K$ is continuous in the sense of \Cref{def_cont_one_over_t}.
     Then $\cM$ is a $\PG$-module over $\cR_{K,A}[\tfrac{1}{t}]$ if and only if $\cM_1$ and $\cM_2$ are $\PG$-modules over $\cR_{K,A}[\tfrac{1}{t}]$.
\end{lemma}

\begin{proof}
    We need to show that if $\cM_1$ and $\cM_2$ admit $\cR_{K}$-lattices stable under $\varphi$ and $\Gamma_K$, then so does $\cM$.
    We let $D_1$ and $D_2$ be such lattices which are the base change of $D_1^{r_0}$ and $D_2^{r_0}$ for some $0 < r_0 \leq C(\pi_K)$, and we choose an $\cR_K^{r_0}$-basis $e_1,\dots,e_n$ of $D_2^{r_0}$.
    Given an $\cR_K$-linear section $s\colon \cM_2\to \cM$, we define maps $f_i\colon \Gamma_K\to \cM_1, ~\gamma \mapsto \gamma \cdot s(e_i)-s(\gamma\cdot e_i)$.
    By the continuity condition on the actions of $\Gamma_K$ on $\cM_1$ and $\cM_2$, there exist $n_i \geq 0$ and $0 < r_i \leq C(\pi_K)$, such that $f_i$ takes values in $t^{-n_i}D_1^{r_i}$.
    Now we let $n_0\ge 0$ be such that $\varphi(D_1)\subset t^{-n_0}D_1$, and we set $n=\max_{0\le i \le n}n_i$ and $r = \min_{0 \le i \le n} r_i$. Then the $\cR_K$-module $D^r \colonequals t^{-n} D_1^r\oplus s(D_2^r)$ is a lattice inside $\cM$ that is stable under $\varphi$ and $\Gamma_K$.


    Assume that $\cM$ is a $\PG$-module and write $\cM = D[\tfrac{1}{t}]$ for a $\PG$-module $D$ over $\cR_{K}$.
    The image $D_2$ of $D$ in $\cM_2$ is a finitely generated, torsion-free $\cR_K$-submodule, hence free (see \cite[§2.2.3]{BC}).
    It is stable under the actions of $\varphi$ and $\Gamma_K$, hence $\cM_2$ is a $\PG$-module.
    It follows that the surjection $D \to D_2$ has an $\cR_K$-linear section, which proves that $D_1 := \cM_1 \cap D$ is a finitely generated $\cR_K$-module.
    Hence $D_1$ is a $\PG$-module and so is $\cM_1$.
\end{proof}

\begin{lemma}\label{three_prime}
    In \Cref{torsor_HPG_modules} condition \eqref{HPG_def_3} can be replaced by
    \begin{enum} 
        \item[(3')] There is a faithful representation $V \in \Rep_L(\gH)$, such that the $\cR_{K,A}$-module $D \times^\gH V$ equipped with its natural Frobenius and $\Gamma_K$-actions is a $(\varphi, \Gamma_K)$-module over $\cR_{K,A}$ in the usual sense.
    \end{enum}
    The analogous result holds over $\cR_{K,A}[\tfrac{1}{t}]$.
\end{lemma}

\begin{proof}
    If $W \in \Rep_L(\gH)$ is an arbitrary representation, by \cite[Proposition 3.1]{DeligneHodgeCycles} $W$ can be realized as a subrepresentation of a representation $V'$ obtained from $V$ by applying the operations $\otimes$, $\oplus$ and $(-)^{\vee}$. Since $D \times^\gH V'$ is obtained from $D \times^\gH V$ by applying these operations in the same way, we see that $D \times^\gH V'$ is the base change of a $\PG$-module $D^r \times^\gH V'$ over $\cR_{K,A}^r$. Since we have an inclusion $D^r \times^\gH W \subseteq D^r \times^\gH V'$, we get continuity of the $\Gamma_K$-action on $D^r \times^\gH W$, hence $D \times^\gH W$ is the base change of a $\PG$-module.

    Let $(\cM, \Phi, C)$ be a triple over $\cR_{K,A}[\tfrac{1}{t}]$ as in \Cref{torsor_HPG_modules} satisfying \eqref{HPG_def_1}, \eqref{HPG_def_2} and assume that $\cM \times^{\gH} V$ is a $\PG$-module over $\cR_{K,A}[\tfrac{1}{t}]$, that is, it has an $\cR_K$-lattice $D$ which is a $\PG$-module.
    Applying the operations $\otimes$, $\oplus$ and $(-)^{\vee}$ to an $\cR_K$-lattice gives us an $\cR_K$-lattice in $\cM \times^{\gH} V'$. It follows from \Cref{lattice_lemma} that $\cM \times^{\gH} W$ is a $\PG$-module.    
\end{proof}


\subsection{$\gH$-trivial $\gH$-$(\varphi,\Gamma_K)$-modules}
\label{trivialphigammamodules}

\begin{definition}
    Let $A\in \Aff_L$ (resp. $A\in \Art_L$). We say that an $\gH$-$(\varphi,\Gamma_K)$-module over $\cR_{K,A}$ (resp. $\cR_{K,A}[\tfrac{1}{t}]$) is \emph{$\gH$-trivial} if it is neutralizable as a fiber functor, equivalently by \Cref{torsor_HPG_modules}, if its associated $\gH$-torsor is trivial.
\end{definition}
\begin{lemma}\label{trivialforB}
    Let $\gB$ be a Borel subgroup of a reductive group $\gG$ over $L$, and let $A\in \Art_L$. Then any $\gB$-$\PG$-module over $\cR_{K,A}$ (resp. $\cR_{K,A}[\tfrac{1}{t}]$) is $\gB$-trivial.
\end{lemma}
\begin{proof}
    Since $\cR_{K,L}$ is a Bézout domain, we have $H^1_{\et}(\Spec(\cR_{K,L}),\Gm)=\Pic(\cR_{K,L}) = 1$. It follows that $H^1_{\et}(\Spec(\cR_{K,A}),\Gm)=1$, since $\Spec L \to \Spec A$ is a universal homeomorphism \stackcite{0BR6}. Since $H^1_{\et}(\Spec(\cR_{K,L}),\Ga)=0$, we get the result by the long exact sequence in non-abelian cohomology.
\end{proof}
We will prove in \Cref{G_PG_Pair_is_triv_PG_mod} that $\gH$-trivial $\gH$-$(\varphi,\Gamma_K)$-modules can be described in terms of a variant of nonabelian cocycles defined as follows.

\begin{definition}\label{def_H_PG_pair} Let $A\in \Aff_L$, let $\cZ^1(\Gamma_K, \gH(\cR_{K,A})) \colonequals \bigcup_{0 < r \leq C(\pi_K)} Z^1_{\cont}(\Gamma_K, \gH(\cR_{K,A}^r))$, and let $0 < r_0 \leq C(\pi_K)$.
    \begin{enum}
        \item An \emph{$\gH$-$(\varphi,\Gamma_K)$-pair} over $\cR^{r_0}_{K,A}$ is a pair $(M,c^{r_0}) \in \gH(\cR^{r_0/p}_{K,A})\times Z^1_{\cont}(\Gamma_K, \gH(\cR^{r_0}_{K,A}))$ such that for all $\gamma\in \Gamma_K$
        \begin{align}
            M \varphi(c^{r_0}(\gamma))=c^{r_0/p}(\gamma)\gamma(M), \label{Mc_comp}
        \end{align}
        where $c^{r_0/p}$ is the image of $c^{r_0}$ under the natural map $\gH(\cR^{r_0}_{K,A}) \to \gH(\cR^{r_0/p}_{K,A})$.
        \item An \emph{$\gH$-$(\varphi,\Gamma_K)$-pair} over $\cR_{K,A}$ is a pair $(M, c) \in \gH(\cR_{K,A}) \times \cZ^1(\Gamma_K, \gH(\cR_{K,A}))$, such that
        \begin{align}
            M \varphi(c(\gamma))=c(\gamma)\gamma(M).
\label{Mc_comp_easy}
        \end{align}

    \end{enum}
\end{definition}

Since $\gH(\cR_{K,A}) = \bigcup_{0 < r \leq C(\pi_K)} \gH(\cR_{K,A}^r)$, we have
\begin{align}
    \{\text{$\gH$-$\PG$-pairs over }\cR_{K,A}\} = \bigcup_{0 < r \leq C(\pi_K)} \{\text{$\gH$-$\PG$-pairs over }\cR^r_{K,A}\}. \label{pair_comparison}
\end{align}
In particular, our definition of $\gH$-$\PG$-pair over $\cR_{K,A}$ agrees with \cite[Definition 3.23]{dedar2020}.

\begin{example}
    Let $V$ is a finite dimensional $L$-vector space, then a $\GL(V)$-$\PG$-pair over $\mathcal{R}_{K,A}$ is equivalent to the data of a $\PG$-module structure on $V\otimes_L\cR_{K,A}$. For such a pair $(M,c)$, this is given by defining
    \begin{align*}
        \varphi\colon V\otimes_L\mathcal{R}_{K,A}&\longrightarrow V\otimes_L \mathcal{R}_{K,A} 
        \\ v\otimes r &\mapsto M(v\otimes \varphi(r)),
    \end{align*}
    and for all $\gamma\in \Gamma_K$,
    \begin{align*}
        \gamma\colon V\otimes_L\mathcal{R}_{K,A}&\longrightarrow V\otimes_L \mathcal{R}_{K,A} 
        \\ v\otimes r &\mapsto c(\gamma)(v\otimes \gamma(r)).
    \end{align*}
    Conversely, given a $\PG$-structure, we recover $(M,c)$  using the equalities $\varphi(v\otimes 1)=M(v\otimes 1)$ and $\gamma(v\otimes 1)=c(\gamma)(v\otimes 1)$ (for more details, see \cite[Lemma 3.24]{dedar2020}).
\end{example}

\begin{definition}\label{def_HPG_pair_inv_t}
    Let $A\in \Art_L$. An \emph{$\gH$-$\PG$-pair} over $\cR_{K,A}[\tfrac{1}{t}]$ is a pair $(M,c) \in \gH(\cR_{K,A}[\tfrac{1}{t}]) \times Z^1(\Gamma_K, \gH(\cR_{K,A}[\tfrac{1}{t}]))$, such that there exists a faithful representation $(\rho_V,V)\in\Rep_L(\gH)$ for which the pair $(\rho_V(M),\rho_V(c))$ induces a $\PG$-module structure over $\mathcal{R}_{K,A}[\tfrac{1}{t}]$ on the tensor product $V\otimes_L \cR_{K,A}[\tfrac{1}{t}]$.
    In particular, the pair $(M,c)$ satisfies \eqref{Mc_comp_easy}.
\end{definition}

\begin{definition}\label{H1triv} 
    Let $A\in \Art_L$. We define $Z^1_{\varphi,\Gamma_K}(\gH(\cR_{K,A}))$ to be the pointed set of $\gH$-$(\varphi,\Gamma_K)$-pairs over $\cR_{K,A}$. We define $H^1_{\varphi,\Gamma_K}(\gH(\cR_{K,A}))$ to be its quotient by the right action of $\gH(\cR_{K,A})$ given by 
    $$(M, c) \cdot h \colonequals (h^{-1}M\varphi(h), ~ [\gamma \mapsto h^{-1}c(\gamma)\gamma(h)]).$$
    We similarly define the set $Z^1_{\varphi,\Gamma_K}(\gH(\cR_{K,A}[\tfrac{1}{t}]))$ of $\gH$-$\PG$-pairs over $\cR_{K,A}[\tfrac{1}{t}]$ and its quotient $H^1_{\varphi,\Gamma_K}(\gH(\cR_{K,A}[\tfrac{1}{t}]))$ by the right action of $\gH(\cR_{K,A}[\tfrac{1}{t}])$.
\end{definition}

\begin{construction}\label{construction}
    Let $A\in \Aff_L$ (resp. $A\in\Art_L$), and let $(M,c) \in \gH(\cR_{K,A}) \times Z^1(\Gamma_K, \gH(\cR_{K,A}))$ $($resp. $\in \gH(\cR_{K,A}[\tfrac{1}{t}]) \times Z^1(\Gamma_K, \gH(\cR_{K,A}[\tfrac{1}{t}])))$. We attach to $(M,c)$ a trivial $\gH$-torsor $D_{(M,c)}$ $($resp. $\mathcal{M}_{(M,c)})$ over $\cR_{K,A}$ $($resp. $\cR_{K,A}[\tfrac{1}{t}])$ equipped with a commuting action of $\varphi$ and $\Gamma_K$ as follows. The isomorphism $\Phi \colon \varphi^*\gH_{\cR_{K,A}} \xrightarrow{\sim} \gH_{\cR_{K,A}}$ is defined by the cartesian diagram
    \begin{center}
        \begin{tikzcd}
            \varphi^* \gH_{\cR_{K,A}} \arrow[r] \arrow[d] \arrow[rr,bend left =20, "\Phi"] & \gH_{\cR_{K,A}} \arrow[r, swap, "\cdot M"] \arrow[d] & \gH_{\cR_{K,A}}
            \\ \Spec(\cR_{K,A} )\arrow[r, "\varphi"] & \Spec(\cR_{K,A}),
        \end{tikzcd}
    \end{center}
    and the morphisms $C(\gamma) \colon \gamma^* \gH_{\cR_{K,A}} \to \gH_{\cR_{K,A}}$ are obtained from $c$ in the same way, replacing $M,\varphi,$ and $\Phi$ by $c(\gamma),\gamma$, and $C(\gamma)$ respectively. The construction of $\cM_{(M,c)}$ is analogous.
    In this definition, we do not impose any continuity conditions.
\end{construction}

\begin{proposition}\label{G_PG_Pair_is_triv_PG_mod}
    The map $(M,c) \mapsto D_{(M,c)}$ $($resp. $(M,c) \mapsto \cM_{(M,c)})$ from $Z^1_{\varphi,\Gamma_K}(\gH(\cR_{K,A}))$ $($resp. $Z^1_{\varphi,\Gamma_K}(\gH(\cR_{K,A}[\tfrac{1}{t}]))$ to the set of $\gH$-trivialized $\gH$-$\PG$-modules over $\cR_{K,A}$ $($resp $\cR_{K,A}[\tfrac{1}{t}])$ is a bijection. It induces a bijection from $H^1_{\varphi,\Gamma_K}(\gH(\cR_{K,A}))$ $($resp. $H^1_{\varphi,\Gamma_K}(\gH(\cR_{K,A}[\tfrac{1}{t}]))$ to the set of isomorphism classes of $\gH$-trivial $\gH$-$\PG$-modules over $\cR_{K,A}$ $($resp. $\cR_{K,A}[\tfrac{1}{t}])$.
\end{proposition}

\begin{proof}
    The statement before inverting $t$ is proved in \cite[Proposition 3.27]{dedar2020}.
    The part about $Z^1_{\varphi,\Gamma_K}(\gH(\cR_{K,A}))$ is not explicitly stated there, but follows from the proof. Although the proof after inverting $t$ follows the same lines, we spell out some details for the convenience of the reader.

    We first need to check that \Cref{construction} yields a $\PG$-module.
    Let $V$ be a faithful representation, such that the $\cR_{K,A}[\tfrac{1}{t}]$-module $V \otimes_L \cR_{K,A}[\tfrac{1}{t}]$ equipped with the actions of $\varphi$ and $\Gamma_K$ induced by $(M,c)$ is a $\PG$-module.
    By \Cref{torsor_HPG_modules} and \Cref{three_prime} the resulting fiber functor takes values in $\PGcat_{K,A}$.

    Given an $\gH$-trivial $\gH$-$\PG$-module $(\mathcal{M},\Phi,C)$ over $\mathcal{R}_{K,A}[\tfrac{1}{t}]$, we choose a section $m\in \mathcal{M}(\cR_{K,A}[\tfrac{1}{t}])$ and we let $\varphi^*m\in \varphi^*\mathcal{M}(\cR_{K,A}[\tfrac{1}{t}])$ be its preimage under the isomorphism $i$ given by the following pullback diagram
    \begin{center}
    \begin{tikzcd}
            \varphi^*\mathcal{M} \arrow[r, "i"]\arrow[d] &\mathcal{M}\arrow[d] 
            \\ \Spec(\cR_{K,A})\arrow[r,"\varphi"]& \Spec(\cR_{K,A}).
        \end{tikzcd}
    \end{center}
    Since $\mathcal{M}$ is an $\gH$-torsor, there is a unique $M\in \gH(\cR_{K,A}[\tfrac{1}{t}])$ such that $\Phi(\varphi^*m)=M\cdot m$. For $\gamma\in \Gamma_K$, we similarly obtain a unique $c(\gamma)\in \gH(\cR_{K,A}[\tfrac{1}{t}])$ satisfying $C(\gamma)(\gamma^*m)=c(\gamma)\cdot m$, where $\gamma^*m$ is defined as above. Now given a faithful representation $(V,\rho_V) \in \Rep_L(\gH)$, the section $m$ induces an isomorphism $\mathcal{M}\times^{\gH}V\cong V\otimes_L\cR_{K,A}[\tfrac{1}{t}]$. One checks that the induced $\PG$-module structure on $V\otimes_L\cR_{K,A}[\tfrac{1}{t}]$ is given by the pair $(\rho_V(M),\rho_V(c))$. This shows that the pair $(M,c)$ is an $\gH$-$\PG$-pair over $\cR_{K,A}[\tfrac{1}{t}]$, which is the inverse of \Cref{construction}.

    If we choose a different section $m'\in \mathcal{M}(\cR_{K,A}[\tfrac{1}{t}])$, then there exists an element $h\in \gH(\cR_{K,A}[\tfrac{1}{t}])$ such that $m'=m\cdot h$. One can verify that the pair obtained from $m'$ is conjugate, under the action given in \Cref{H1triv}, to the pair given by $m$ via the element $h$.    
\end{proof}

We establish a continuity condition on the cocycles arising in pairs in the spirit of \Cref{topol_RKAoneovert} and \Cref{def_cont_one_over_t}.

\begin{lemma}\label{pair_is_ind_cont}
    Let $(M,c) \in Z^1_{\varphi,\Gamma_K}(\gH(\cR_{K,A}[\tfrac{1}{t}]))$.
    Then $c \colon \Gamma_K \to \gH(\cR_{K,A}[\tfrac{1}{t}])$ satisfies the equivalent conditions of \Cref{topol_RKAoneovert}.
\end{lemma}

\begin{proof}
    Let $V$ be a faithful representation of $\gH$ of dimension $d$.
    By forming $W = V \oplus (\bigwedge^d V)^*$, we achieve that $\gH$ takes values in $\SL(W)$.
    Choose a basis $e_1, \dots, e_{d+1}$ of $W$.
    The resulting cocycle $\Gamma_K \to M_{d+1}(\cR_{K,A}[\tfrac{1}{t}])$ for the action on $\cM_{(M,c)} \times^{\gH} W$ is continuous in the sense of \Cref{def_cont_one_over_t}.
    Since $\gH \subseteq \SL(W) \subseteq M_{d+1}$ is a closed immersion, we have established (2) of \Cref{topol_RKAoneovert}.
\end{proof}

\subsection{Parameters}
\label{sec_param}

We go back to our choice of split connected reductive group $\gG$ over $L$ with Borel subgroup $\gB$ and split maximal torus $\gT$. We define parameter spaces for $\gG$ and their $\gB$-regular and very regular subspaces.

\subsubsection{Parameters for $\gT$}

We describe the relation between $\gT$-$\PG$-modules and characters, following \cite[Example 3.13]{dedar2020}. 
We denote by $\gT^\vee$ the dual torus of $\gT$, i.e. the unique split torus $\gT^\vee$ equipped with an isomorphism $X^\ast(\gT^\vee) \cong X_\ast(\gT)$. 
Let $A\in \Aff_L$, and let $\delta \colon \gT^\vee(K) \to A^{\times}$ be a continuous character.
After choosing an isomorphism $\Gm^n \cong \gT^\vee$, we have a tuple $\underline\delta = (\delta_1, \dots, \delta_n)$ and $(\varphi, \Gamma_K)$-modules of character type $\cR_{K,A}(\delta_i)$, see \cite[Construction 6.2.4]{KPX}. By passing to the associated $\gT$-torsor, we obtain a $\gT$-trivial $\gT$-$(\varphi, \Gamma_K)$-module, which we will denote by $\cR_{K,A}(\delta)$ and which is independent of the choice of isomorphism. 

Conversely, let $D$ be a $\gT$-$\PG$-module over $\cR_{K,A}$. Even if $D$ is not $\gT$-trivial, we will attach to it a continuous character $\gT^\vee(K)\to A^\times$ as follows. For every $\chi\in X^\ast(\gT)$, the rank-one $\PG$-module $\eta_D(\chi)$ over $\cR_{K,A}$ is, by \cite[Theorem 6.2.14(1)]{KPX}, of the form $\cR_{K,A}(\delta_\chi)\otimes_A\mathscr{L}$ for a continuous character $\delta_{\chi} \colon K^\times\to A^\times$ and a line bundle $\mathscr{L}$ over $A$ carrying trivial $\varphi$ and $\Gamma_K$-actions. The association $\chi\mapsto\delta_\chi$ defines a homomorphism from $X^\ast(\gT)$ to $\Hom_{\mathrm{cont}}(K^{\times},A^\times)$, which in turn corresponds to a continuous character $\delta \colon \gT^\vee(K)\to A^\times$ (see the calculation (3.3) of \cite[Section 3]{dedar2020}). 
\begin{definition}
    Let $D$ be a $\gT$-$\PG$-module over $\cR_{K,A}$. The continuous character $\delta \colon \gT^\vee(K)\to A^\times$ obtained as above is called the \emph{parameter} of $D$.

    If $D$ is a $\gB$-$(\varphi, \Gamma_K)$-module over $\cR_{K,A}$, we say that $\delta \colon \gT^\vee(K)\to A^\times$ is the \emph{parameter} of $D$, if it is the parameter of $D \times^\gB \gT$.
\end{definition}



\begin{remark}
Choosing an isomorphism $\gT\cong\G_m^n$ and applying \cite[Theorem 6.2.14 (1)]{KPX}, one can show that every $\gT$-$\PG$-module over $\cR_{K,A}$ is Zariski-locally on $A$ of the form $\cR_{K,A}(\delta)$.
\end{remark}

\begin{definition}
    We define $\cT^\gT$ by
\begin{align}
    \cT^\gT \colon \Aff_L \longrightarrow \Set, \quad A \mapsto \Hom_{\cont}(\gT^\vee(K), A^{\times}), \label{TT}
\end{align}
which is smooth quasi-Stein of dimension $\dim\gT \cdot ([K:\Qp]+1)$.
\end{definition} 

\begin{example}
    When $\gG = \GL_n$, $\gB = \gB_n$ is the Borel subgroup of upper triangular matrices and $\gT=\gT_n$ is the maximal torus of diagonal matrices, then we have a natural choice of an isomorphism $\Gm^n \cong \gT_n$. In particular, $\cT^{\gT_n} \cong \cT^n$ and a parameter for a $\gT_n$-$(\varphi, \Gamma_K)$-module amounts to a tuple $\underline\delta = (\delta_1, \dots, \delta_n)$ of characters $\delta_i \colon K^{\times} \to A^{\times}$. In this paper we try to describe parameters in a basis free way, which means that for a general split reductive group $\gG$ with maximal torus $\gT$ we prefer not to choose an isomorphism between $\Gm^n$ and $\gT$.
\end{example}

In the following lemma, we establish an analogue of \cite[Lemma 6.2.13]{KPX} for $\PG$-modules over $\cR_{K,A}[\tfrac{1}{t}]$, giving a criterion for a rank-one object to be of character type. 

\begin{lemma}\label{class_rk_one_after_inv_t}
    Let $A\in \Art_L$, and let $\cM_A$ be a $\PG$-module over $\cR_{K,A}[\tfrac{1}{t}]$ of rank one such that $\cM_A\otimes_A L\cong \cR_{K,L}[\frac{1}{t}](\delta)$. Then there exists a continuous character $\delta_A\colon K^\times \to A^\times$ lifting $\delta$ such that $\cM_A\cong \cR_{K,A}[\tfrac{1}{t}](\delta_A)$.
\end{lemma}

\begin{proof}
    We proceed by induction on the nilpotency index $e$ of $\mathfrak{m}_A$. If $e=1$, then there is nothing to prove. 
    Otherwise, assume that the result holds for nilpotency index $e-1$, so that
    $$ \cM_A\otimes_A A/\mathfrak{m}_A^{e-1}\cong \cR_{K,A/\mathfrak{m}_A^{e-1}}[\tfrac{1}{t}](\delta_{e-1}). $$
    Let $\delta_e\colon K^\times \to A^\times$ be a lift of $\delta_{e-1}$ (which exists by smoothness of the space of characters of $K^\times$). Twisting $\cM_A$ by $\delta_e^{-1}$, we may assume that $\delta_{e-1}$ is trivial. We then obtain an $A$-linear exact sequence
     $$ 0 \to  \mathfrak{m}_A^{e-1} \cM_A \to \cM_A \to \cR_{K,A/\mathfrak{m}_A^{e-1}}[\tfrac{1}{t}] \to 0. $$
     Using the $A$-module structure on $\cM_A$, we find that the options for $\cM_A$ are parametrized by $H^1_{\varphi,\Gamma_K}(\mathfrak{m}_A^{e-1} \cM_A)$. Moreover, there is an isomorphism $\mathfrak{m}_A^{e-1} \cM_A\cong \cR_{K,A}[\tfrac{1}{t}]\otimes_A \mathfrak{m}_A^{e-1}$, and 
     $$ H^1_{\varphi,\Gamma_K}(\cR_{K,A}[\tfrac{1}{t}]\otimes_A \mathfrak{m}_A^{e-1})=\colim_n H^1_{\varphi,\Gamma_K}(t^{-n}\cR_{K,A}\otimes_A \mathfrak{m}_A^{e-1})\cong H^1_{\varphi,\Gamma_K}(\cR_{K,A}\otimes_A \mathfrak{m}_A^{e-1}), $$
     where we use \cite[Lemma 3.3.3]{BHS19} for the last isomorphism. By the proof of \cite[Lemma 6.2.13]{KPX}, the lifts $\delta_A\colon K^\times \to A^\times$ of the trivial character with coefficients in $(A/\mathfrak{m}_A^{e-1})^\times$ are parametrized by $H^1(\GK,\mathfrak{m}_A^{e-1})$, and the map $\delta_A\mapsto \cR_{K,A}(\delta_A)$ induces an isomorphism $$H^1(\GK,\mathfrak{m}_A^{e-1})\xrightarrow{\sim} H^1_{\varphi,\Gamma_K}(\cR_{K,A}\otimes_A\mathfrak{m}_A^{e-1}).$$
     This completes the proof.
\end{proof}

\subsubsection{Deformations of parameters}\label{defofparameters} If $A\in \Aff_L$, we say that a continuous character $\delta\colon \gT^\vee(K)\to A^\times$ is \emph{algebraic} if every $\delta_i$ as above is algebraic. It is easy to see that this does not depend on the chosen splitting, and the following remark provides a more intrinsic definition.

\begin{remark}\label{algdual}
A character $\delta\colon\gT^\vee(K)\to A^\times$ is algebraic if and only if it factors as
\[ \gT^\vee(K)\to\gT^\vee(K\otimes_{\Q_p}L)=\bT^\vee(L) \xto{\mu} L^\times \hookrightarrow A^\times, \]
where the first map is induced by $K\to K\otimes_{\Q_p}L, ~x\mapsto x\otimes 1$, and we require that the second one be obtained by evaluating an algebraic homomorphism $\mu : \bT^\vee\to\Gm$ on $L$ points. In particular, we can attach to $\delta$ the algebraic homomorphism $\mu^\vee\colon \Gm\to\bT$ dual to $\mu$. 
\end{remark}

\begin{definition}\label{def_weightpar}
Let $A\in \Art_L$ and $\delta_A \in \cT^\gT(A)$. We define the \emph{weight} of $\delta_A$ as the $\Qp$-linear map $\wt(\delta_A) : \Lie(\gT^{\vee})(K) \to A$ given by $\wt(\delta_A)(x) \colonequals \tfrac{d}{dt}|_{t=0}\delta_A(\exp(tx))$.

We say that $\delta_A \in \cT^\gT(A)$ is smooth if $\wt(\delta_A)=0$, and we say that it is locally algebraic if it is a product of a smooth and an algebraic character.
\end{definition}
For $\delta \in \cT^\gT(L)$, the functor
$$ A\in \Art_{L} \mapsto \{ \delta_A\in \cT^\gT(A) \ | \ \delta_A\equiv \delta \bmod \mathfrak{m}_A \} $$
is prorepresentable by the completion $\wh\cT^\gT_{\delta}$ of $\cT^\gT$ at $\delta$. We let $\widehat{\frakt}$ be the completion of $\frakt$ at $0\in \frakt(L)$. Since $$\Hom_{\Qp}(\Lie(\gT^{\vee})(K), A) = \Lie(\gT)(K \otimes_{\Qp} A) =\frakt (A) , $$
we get a morphism of formal schemes $\wt - \wt(\delta) \colon \wh\cT^\gT_{\delta} \to \wh\frakt$. The following is proved in the same way as \cite[Lemma 3.5.5]{BHS19}.

\begin{lemma}\label{lem355}
    The map $\wt - \wt(\delta) \colon \widehat\cT^\gT_{\delta} \to \that$ is formally smooth of relative dimension $\dim \gT_L$.
\end{lemma}

\subsubsection{$\gB$-regular parameters}\label{sec:Breg}

\begin{definition}\label{def_Breg}
Given a character $\alpha \colon \gT \to \Gm$, its dual $\alpha^{\vee} \colon \Gm \to \gT^{\vee}$ induces a map $\cT^\gT \to \cT$.
We define $\cT_{0,\gB}^{\gT} \subseteq \cT^\gT$ as the preimage of $\prod_{\alpha \in \Phi^+(\gB,\gT)} \cT_0$ under the map $\cT^\gT \to \prod_{\alpha \in \Phi^+(\gB,\gT)} \cT$. \\
For $A\in \Aff_L$, a parameter $\gT^{\vee}(K) \to A^{\times}$ is called \emph{$\gB$-regular} if it lies in $\cT_{0,\gB}^{\gT}(A)$.
\end{definition}
We note that $\gT$ acts on the root subgroups of $\gB$ by a nontrivial character $\alpha$, which induces a surjection $\cT^\gT \to \cT$. Therefore, $\cT_{0,\gB}^{\gT}$ is Zariski open and Zariski dense in $\cT^\gT$ whose complement is of codimension $(\dim \gT - 1) \cdot ([K:\Qp]+1)$.

\begin{remark}\label{rmk_B_solvable}
    The notion of ``regular parameter'' of \cite[Definition 5.4]{dedar2020} is strictly stronger, as there $\alpha$ varies over all roots. We also remark that the notion of $\gB$-regularity only depends on the pair $(\gB, \gT)$. It makes sense for every solvable group $\gB$ with $\gB/\Ru(\gB) \cong \gT$ such that the induced roots of $\gT$ are non-trivial, in particular all quotients of a Borel. We do not pursue this generalization here.
\end{remark}

\begin{example}
    Let $(\gB_n,\gT_n)$ be the standard torus of $\GL_n$, so that $\cT^{\gT_n} = \cT^n$.
    Then $\cT^{\gB_n,\gT_n}_0 \subseteq \cT^n$ is the complement of those $\delta = (\delta_1, \dots, \delta_n) \colon (K^{\times})^n \to L^{\times}$ for which for some $1 \leq i < j \leq n$ the character $\delta_i\delta_j^{-1}$ or $\varepsilon\delta_j\delta_i^{-1}$ is algebraic. Our set $\cT^{\gB_n, \gT_n}_0$ is strictly larger than the set $\cT^n_0$ defined in the paragraph after \cite[Lemma 3.4.3]{BHS19}, which is defined in a similar fashion with $<$ replaced by $\neq$. 
\end{example}

\subsection{Triangulable $\GPG$-modules}
\label{sectriPG}

We go back to our chosen reductive group $\gG$ over $L$, with Borel subgroup $\gB$ containing a maximal torus $\gT$. Let $A\in \Aff_L$, and let $D$ be a $\gG$-$(\varphi, \Gamma_K)$-module over $\cR_{K,A}$, and let $\cM$ be a $\gG$-$(\varphi, \Gamma_K)$-module over $\cR_{K,A}[\tfrac{1}{t}]$ if $A$ is finite-dimensional.

\begin{definition}\label{def_triangulation} \phantom{a}
    \begin{enumerate}
        \item A ($\gB$-)\emph{triangulation} of $D$ is a $\gB$-$(\varphi, \Gamma_K)$-module $\Dtri$ over $\cR_{K,A}$ together with an isomorphism of $\gG$-$(\varphi, \Gamma_K)$-modules $\iota\colon \Dtri \times^\gB \gG\xrightarrow{\sim} D$ over $\cR_{K,A}$. We can associate to $\Dtri$ a $\gT$-$(\varphi, \Gamma_K)$-module $\Dtri \times^\gB \gT$, which determines via the discussion in \Cref{sec_param} a continuous character $\gT^{\vee}(K) \to A^{\times}$, called the \emph{parameter} of the triangulation.
        
        \item A ($\gB$-)\emph{triangulation} of $\cM$ is a $\gB$-$(\varphi, \Gamma_K)$-module $\Mtri$ over $\cR_{K,A}[\tfrac{1}{t}]$ together with an isomorphism of $\gG$-$(\varphi, \Gamma_K)$-modules $\iota\colon \Mtri \times^\gB \gG\xrightarrow{\sim} \cM$ over $\cR_{K,A}[\tfrac{1}{t}]$ such that the associated $\gT$-$(\varphi, \Gamma_K)$-module $\Mtri \times^\gB \gT$ \emph{is of character type}. That is, $\Mtri \times^\gB \gT\cong \cR_{K,A}[\frac{1}{t}](\delta)$, where $\delta\colon \gT^{\vee}(K) \to A^{\times}$ is a continuous character. We refer to $\delta$ as a \emph{parameter} of the triangulation.
    \end{enumerate}
    We say that $D$ (resp. $\cM$) is ($\gB$-)\emph{triangulable} if it admits a $\gB$-triangulation.
\end{definition} 

This is equivalent to \cite[Definition 4.9]{dedar2020}. If $D$ is equipped with a triangulation $\Dtri$, then $D$ is determined by $\Dtri$ up to a unique isomorphism. Consequently, instead of using the term $\gG$-$\PG$-modules equipped with a triangulation, we will use $\gB$-$\PG$-modules.

\begin{remark}
When $\gG=\GL_n$, \Cref{def_triangulation} is equivalent to the classical definition of triangulable $(\varphi,\Gamma_K)$-modules of rank $n$: a $\GL_n$-$(\varphi,\Gamma_K)$-module over $\cR_{K,A}$ is triangulable if and only if it admits a filtration whose nontrivial subquotients are rank 1 $(\varphi,\Gamma_K)$-modules over $\cR_{K,A}$. This follows from \Cref{torsor_HPG_modules} and \cite[Example 3.11]{dedar2020}. 
\end{remark}

\begin{remark}\label{borelindependent}
Since all of the Borel subgroups of $\gG$ over $L$ are $\gG(L)$-conjugate, the property of $D$ being $\gB$-triangulable is independent of the choice of $\gB$, so we simply say that $D$ is \emph{triangulable}.
\end{remark}

\begin{lemma}
    Let $A\in \Art_L$, and let $\cM_A^{\triangle}$ be a $\gB$-$\PG$-module over $\cR_{K,A}[\tfrac{1}{t}]$ deforming $\Mtri=\cM_A^{\triangle}\otimes_A L$. If $\Mtri\times^{\gB}\gG$ is triangulable, then so is $\Mtri_A\times^{\gB}\gG$.
\end{lemma}
\begin{proof}
    This is an immediate consequence of \Cref{class_rk_one_after_inv_t}.
\end{proof}

\begin{definition}\label{def_trianguline_rep}
    A $\gG$-valued Galois representation $\rho \colon \GK \to \gG(L)$ is \emph{trianguline} if there is a finite extension $L'/L$ such that the associated $\GPG$-module $D_{\rig}\circ \eta_{\rho_{L'}}$ is triangulable.
\end{definition}

\begin{remark}
    If $D$ is a triangulable $\gG$-$\PG$-module over $\cR_{K,A}$, then for any $V \in \Rep_L(\gG)$, the $\PG$-module $\eta_D(V)$ is trianguline.
    Indeed, the image of $\gB$ in $\GL(V)$ is contained in a Borel subgroup $\gB'$ of $\GL(V)$. If $D^{\triangle}$ is a $\gB$-$\PG$-module, such that $D^{\triangle} \times^{\gB} \gG \cong D$, then $D^{\triangle} \times^\gB \gB'$ is a triangulation of $D \times^\gG \GL(V)$. In particular, $D$ satisfies the ``Tannakian definition'' of being trianguline.
    The converse implication was proved in \cite[Proposition 5.10]{conti2022lifting}. 
    For our purposes we only need that a crystalline $\gG$-valued representation gives rise to a triangulable $\PG$-module, which was shown by de Daruvar \cite[Corollary 4.34]{dedar2020}.
\end{remark}

\begin{lemma}\label{pluckercoordinatesphi-gamma}
    Let $D$ be a $\gG$-$(\varphi, \Gamma_K)$-module over $\cR_{K,A}$, and let $\Lambda\subset X_+^*(\gT^\der)$ be a finite set of $\Q$-generators of $X^*_+(\gT^{\der})_\Q$. Then the following are equivalent:
\begin{enumerate}
    \item The datum of a triangulation $\Dtri$.
    \item The datum of Plücker coordinates, that is, a family of saturated $(\varphi, \Gamma_K)$-submodules of rank one 
    $$ \big(\mathscr{L}_{\lambda} \subseteq \eta_{D}(V_{\lambda})\big)_{\lambda\in X_+^*(\gT)}, $$
    such that for all $\lambda,\mu \in  X_+^*(\gT)$, the image of $\mathscr{L}_{\lambda+\mu}$ in $\eta_{D}(V_{\lambda})\otimes \eta_{D}(V_{\mu})$, via the inclusion $\eta_{D}(V_{\lambda+\mu})\subseteq \eta_{D}(V_{\lambda})\otimes \eta_{D}(V_{\mu})$, coincides with $\mathscr{L}_\lambda\otimes \mathscr{L}_\mu$.
    \item The datum of a family of saturated $(\varphi, \Gamma_K)$-submodules of rank one $\mathscr{L}_{\lambda}\subseteq \eta_{D}(V_{\lambda})$ for $\lambda\in \Lambda$ satisfying the following conditions:
    \begin{itemize}
        \item For every $(n_\lambda)_{\lambda}\in \bbN^{\Lambda} $, $\bigotimes_{\lambda \in \Lambda} \mathscr{L}_{\lambda}^{\otimes n_\lambda} \subseteq \eta_{D}(V_{\sum_\lambda n_\lambda \lambda})\subseteq \bigotimes_{\lambda \in \Lambda}\eta_{D}(V_\lambda)^{\otimes n_\lambda},$
        \item For every $ ((n_\lambda)_{\lambda},(m_\lambda)_{\lambda})\in R_{\Lambda}$, $\bigotimes_{\lambda \in \Lambda} \mathscr{L}_{\lambda}^{\otimes n_\lambda}=\bigotimes_{\lambda \in \Lambda} \mathscr{L}_{\lambda}^{\otimes m_\lambda}.$
        \end{itemize}
\end{enumerate}
    If $A\in \Art_L$,  the analogous statement holds for a $\gG$-$(\varphi, \Gamma_K)$-module $\cM$ over $\cR_{K,A}[\tfrac{1}{t}]$, where we further assume that the rank-one $(\varphi,\Gamma_K)$-modules $\mathscr{L_\lambda}$ over $\cR_{K,A}[\tfrac{1}{t}]$ are of character type.
\end{lemma}
\begin{proof}
    A proof is provided in \cite[Theorem 4.13, Proposition 4.14]{dedar2020}. One simply applies \Cref{pluckerdatum}, and observes that since the actions of $\varphi$ and $\Gamma_K$ preserve the Plücker coordinates $\mathscr{L}_\lambda$, they induce corresponding actions on $D^\triangle$. Note that the condition $\mathscr{L}_0\cong \cR_{K,A}$ is automatic.
\end{proof}

\subsection{Regularity conditions}\label{sec:regularity}
We look for regularity conditions on a parameter $\delta$ that guarantee the uniqueness of a triangulation of a parameter $\delta$. The conditions we consider are of two kinds ``$\Lambda$-regularity'' and ``$r$-regularity''.
\subsubsection{$\Lambda$-regular parameters}\label{sec:Lambda_reg}
The following generalizes the notion of ``very regularity'' from \cite{dedar2020}.

\begin{definition}
    Let $\Lambda\subset X_+^*(\gT^\der)$ be a finite set of $\Q$-generators of $X^*_+(\gT^{\der})_\Q$. We define the subspace $\cT_{\Lambda}^\gT \subset \cT^\gT$ (resp. $\cT_{0,\Lambda}^\gT\subset \cT^\gT$) to be the space of parameters $\delta$ such that for all $\lambda\in \Lambda$ and all weights $\mu\neq \lambda$ of $V_{\lambda}$, one has $\delta\circ(\mu-\lambda)^{\vee}\in \cT_{\reg}$ (resp.  $\delta\circ(\mu-\lambda)^{\vee}\in \cT_{0}$).  
\end{definition}
Let us fix  a finite set of $\Q$-generators $\Lambda$ of $X^*_+(\gT^{\der})_\Q$ as above. The following is proven as in \cite[Lemma 6.3]{dedar2020}.
\begin{lemma}\label{rank_oner}
Let $D$ be a $\gG$-$(\varphi, \Gamma_K)$-module over $\cR_{K,A}$ 
with triangulation $D^{\triangle}$ of parameter $\delta\in \cT_{\Lambda}^\gT(A)$.
Then for all $\lambda \in \Lambda$, the space $\Hom_{\PGcat_{K,A}^{+}}(\cR_{K,A}(\delta \circ \lambda^{\vee}), \eta_{D}(V_{\lambda}))$ is a projective $A$-module of rank $1$. The analogous statement holds for a $\gG$-$(\varphi, \Gamma_K)$-module $\cM$ over $\cR_{K,A}[\tfrac{1}{t}]$ with triangulation of parameter $\delta\in \cT_{0,\Lambda}^\gT(A)$. 
\end{lemma}
\begin{proof} We only give the proof over $\cR_{K,A}$, the proof over $\cR_{K,A}[\tfrac{1}{t}]$ being the same. Let $n = \dim V_{\lambda}$. The $\gB$-representation $V_{\lambda}|_{\gB}$ admits a complete flag
$$  0 = \Fil^0 V_{\lambda} \subset \Fil^1 V_{\lambda} \subset \dots \subset \Fil^n V_{\lambda} = V_{\lambda}|_{\gB}, $$
where $\gr^i V_{\lambda} = V_{\lambda}^{\mu_i}$ for some character $\mu_i \in X^*(\gT)$ restricted to $\gB$.
Therefore,  $$\eta_{D^{\triangle}}(\gr^i V_{\lambda}) = \gr^i \eta_{D^{\triangle}}(V_{\lambda}) = \cR_{K,A}(\delta \circ \mu_i^{\vee}) \otimes_A \scrL_i$$ 
for some finite projective $A$-module $\scrL_i$, see \cite[Theorem 6.2.14]{KPX}.
We have $\Hom_{\PGcat_{K,A}^{+}}(\cR_{K,A}(\delta \circ \lambda^{\vee}), \eta_{D^{\triangle}}(V_\lambda)) = H^0_{\varphi, \gamma_K}(\eta_{D^{\triangle}}(V_{\lambda})(- \delta \circ \lambda^{\vee}))$.
The short exact sequence
$$ 0 \to V_{\lambda}^{\lambda} \to V_{\lambda} \to V_{\lambda}/V_{\lambda}^{\lambda} \to 0 $$
induces a short exact sequence 
$$ 0 \to \scrL_1 \to \eta_{D^{\triangle}}(V_{\lambda})(-\delta \circ \lambda^{\vee}) \to \eta_{D^{\triangle}}(V_{\lambda}/V_{\lambda}^{\lambda})(-\delta \circ \lambda^{\vee}) \to 0 $$
after applying $\eta_{D^{\triangle}}$ and twisting with $-\delta \circ \lambda^{\vee}$. So the long exact sequence gives
$$ 0 \to \scrL_1 \to H^0_{\varphi, \gamma_K}(\eta_{D^{\triangle}}(V_{\lambda})(-\delta \circ \lambda^{\vee})) \to H^0_{\varphi, \gamma_K}(\eta_{D^{\triangle}}(V_{\lambda}/V_{\lambda}^{\lambda})(-\delta \circ \lambda^{\vee})). $$
Since $\delta \circ (\mu_i - \lambda)^{\vee}\in \cT_{\reg}$, we have  $H^0(\cR_{K,A}(\delta \circ (\mu_i - \lambda)^{\vee})) = 0$ by \Cref{coh_char_A}.
Arguing by induction, we find that $H^0_{\varphi, \gamma_K}(\eta_{D^{\triangle}}(V_{\lambda}/V_{\lambda}^{\lambda})(- \delta \circ \lambda^{\vee})) = 0$.
We conclude that $\scrL_1 = H^0_{\varphi, \gamma_K}(\eta_{D^{\triangle}}(V_{\lambda})(-\delta \circ \lambda^{\vee}))$.
\end{proof}

Combining \Cref{pluckercoordinatesphi-gamma} and \Cref{rank_oner}, we obtain the following.

\begin{corollary}\label{uniqueness_Lambda_1}
    Let $\delta \in \cT^{\gT}_{\Lambda}(A)$.
    Then any $\gG$-$(\varphi, \Gamma_K)$-module $D$ over $\cR_{K,A}$ has at most one triangulation of parameter $\delta$.
    \end{corollary}

\begin{corollary}\label{uniqueness_Lambda_2}
    Let $D$ be a $\gG$-$\PG$-module over $\cR_{K,L}$ equipped with a triangulation $\Dtri$ of parameter $\delta\in \cT^{\gT}_{\Lambda}(L)$, and let $A\in \Art_L$. If $D_A$ is a $\gG$-$\PG$-module over $\cR_{K,A}$ deforming $D$, then $D_A$ admits at most one triangulation deforming $\Dtri$. The analogous statement holds for a $\gG$-$(\varphi, \Gamma_K)$-module $\cM$ over $\cR_{K,L}[\tfrac{1}{t}]$ with triangulation of parameter $\delta\in \cT_{0,\Lambda}^\gT(L)$. 
\end{corollary}
\begin{proof}
      We prove the statement for $D$, the proof for $\cM$ being essentially the same. Let $\Dtri_A$ be a triangulation  of $D_A$ deforming $\Dtri$, we will show that it is unique. By \Cref{pluckercoordinatesphi-gamma}, it suffices to show that for each $\lambda\in \Lambda$, the associated saturated $(\varphi,\Gamma_K)$-submodule of rank-one $\mathscr{L}_{\lambda,A} \subset \eta_{D_A}(V_{\lambda})$ deforming $\mathscr{L}_{\lambda}\subset \eta_{D}(V_{\lambda})$ is uniquely determined. Up to twisting, we may assume that the parameter of $\Dtri_A$ is trivial, so that in particular, $\mathscr{L}_{\lambda}\cong \mathcal{R}_{K,A}[\tfrac{1}{t}]$. Suppose that there exists another saturated rank-one $\PG$-module $\mathscr{L}'_{\lambda,A} \subseteq \eta_{D_A}(V_{\lambda})$ with $\mathscr{L}'_{\lambda,A}\equiv \mathscr{L}_{\lambda} \bmod \mathfrak{m}_A$. By \Cref{rank_oner}, we have 
     $$ H^0_{\varphi,\Gamma_K}(\eta_{D}(V_{\lambda})/\mathscr{L}_{\lambda})=0. $$
     By dévissage and the left exactness of $H^0$, as in \cite[Lemma 2.3.7]{BC}, we find that
    $$ H^0_{\varphi,\Gamma_K}(\eta_{D_A}(V_{\lambda})/\mathscr{L}_{\lambda,A})=H^0_{\varphi,\Gamma_K}(\eta_{D_A}(V_{\lambda})/\mathscr{L}_{\lambda,A}')=0.$$
    It follows that $H^0_{\varphi,\Gamma_K}(\mathscr{L}_{\lambda,A})=H^0_{\varphi,\Gamma_K}(\mathscr{L}_{\lambda,A}')\subset \mathscr{L}'_{\lambda,A}$, hence  $\mathscr{L}_{\lambda,A}\subseteq \mathscr{L}'_{\lambda,A}$. Since the modules are saturated, they must be equal.
\end{proof}

\subsubsection{$r$-regular parameters}
The conditions on $\delta$ resulting from \Cref{sec:Lambda_reg} are quite strong and require some additional work to make explicit.
We want to present a second method of ensuring uniqueness of the triangulation.

\begin{lemma}\label{emb_lemma_asdfjkl}
    Let $r\colon \gG_1 \to \gG_2$ be an embedding of split connected reductive groups over $L$, and let $\gT_i \subseteq \gB_i \subseteq \gG_i$ be tori and Borel subgroups for $i=1,2$, satisfying $r(\gT_1)\subseteq \gT_2$ and $r(\gB_1) \subseteq \gB_2$. Let $D$ be a $\gG_1$-$\PG$-module over $\cR_{K,A}$, $\delta \colon \gT_1^{\vee}(K) \to A^{\times}$ be a continuous character. If $D \times^{\gG_1} \gG_2$ admits at most one triangulation of parameter $\delta \circ r^{\vee}$, then $D$ admits at most one triangulation of parameter $\delta$.
\end{lemma}

\begin{proof}
    We have $\gB_1 \subseteq \gB_2 \cap \gG$.
    To show equality, we may work over $\Qpbar$. 
    The intersection $\gB_2 \cap \gG$ is a parabolic subgroup of $\gG$ by \cite[Theorem 6.2.7]{springerLAG}, hence connected by \cite[Corollary 6.4.10]{springerLAG}. It is also solvable, hence we have equality.

    Triangulations of $D$ of parameter $\delta$ correspond to factorizations of fiber functors $\Rep_L(\gG_1) \to \Rep_L(\gB_1) \to \PGcat_{K,A}^+$. Composing with the restriction functor $\Rep_L(\gG_2) \to \Rep_L(\gG_1)$, we have at most one factorization $\Rep_L(\gB_2) \to \PGcat_{K,A}^+$, corresponding to a triangulation of $D \times^{\gG_1} \gG_2$ of parameter $\delta \circ r^{\vee}$.
    But since $\gB_1 = \gG_1 \cap \gB_2$, the category $\Rep_L(\gB_1)$ arises as a pushout of $\Rep_L(\gG_1)$ and $\Rep_L(\gB_2)$ over $\Rep_L(\gG_2)$ (\cite[Lemma 5.11]{conti2022lifting}), so there is at most one factorization $\Rep_L(\gB_1) \to \PGcat_{K,A}^+$.
\end{proof}


By induction from \Cref{rank_oner} one deduces the following.

\begin{lemma}\label{GLner}
    Let $D$ be a $\GL_n$-$(\varphi, \Gamma_K)$-module over $\cR_{K,A}$ (resp. $\cR_{K,A}[\tfrac{1}{t}]$) and let $\delta : \gT_n^{\vee}(K) \to A^{\times}$ be a continuous character.
    Assume that $\delta \circ \mu^{\vee}\in \cT_{\reg}(A)$  for all $\mu \in \Phi^-(\GL_n, \gB_n)$.    
    Then $D$ has at most one triangulation of parameter $\delta$.
\end{lemma}

It is also clear that \Cref{GLner} has optimal assumptions.
We expect that the natural generalization of \Cref{GLner} holds:
\begin{conjecture}\label{conj_Ger}
    Let $D$ be a $\gG$-$(\varphi, \Gamma_K)$-module over $\cR_{K,A}$, and let $\delta \colon \gT^{\vee}(K) \to A^{\times}$ be a continuous character.
    Assume that $\delta \circ \lambda^{\vee}\in \cT_{\reg}(A)$ for all $\lambda \in \Phi^-(\gG, \gB)$.       
    Then $D$ has at most one triangulation of parameter $\delta$.
\end{conjecture}

\begin{definition}\label{def:rregular}
    Let $r\colon \gG\to \GL_m$ be a faithful representation satisfying the hypotheses of \Cref{emb_lemma_asdfjkl}. We define the subspaces $\cT^{\gT}_r\subset \cT^{\gT}$ and $\cT^{\gT}_{0,r}\subset \cT^{\gT}$ to be the spaces of parameters $\delta$ such that $\delta \circ r^{\vee}\in \cT^m_{\reg}$ and $\delta \circ r^{\vee}\in \cT^m_0$, respectively. These spaces are non-empty if and only if $r$ is multiplicity-free as a $\gT$-representation.
    In that case $\cT^{\gT}_r$ and $\cT^{\gT}_{0,r}$ are Zariski open and Zariski dense in $\cT^{\gT}$.
\end{definition}

The following is a reformulation of \Cref{emb_lemma_asdfjkl}.

\begin{corollary}\label{uniqueness_r_1}Let $r\colon \gG\to \GL_m$ be a faithful representation satisfying the hypotheses of \Cref{emb_lemma_asdfjkl} and $\delta \in \cT^{\gT}_{r}(A)$.
    Then any $\gG$-$(\varphi, \Gamma_K)$-module $D$ over $\cR_{K,A}$ has at most one triangulation of parameter $\delta$.
    \end{corollary}

\begin{lemma}\label{uniqueness_r_2}
    Let $r\colon \gG\to \GL_m$ be a faithful representation satisfying the hypotheses of \Cref{emb_lemma_asdfjkl}. Let $D$ be a $\gG$-$\PG$-module over $\cR_{K,L}$ equipped with a triangulation $\Dtri$ of parameter $\delta\in \cT^{\gT}_r(L)$, and let $A\in \Art_L$. If $D_A$ is $\gG_1$-$\PG$-module over $\cR_{K,A}$ deforming $D$, then $D_A$ itself admits at most one triangulation deforming $\Dtri$.
\end{lemma}
\begin{proof}
    This follows from the proof of \Cref{emb_lemma_asdfjkl} and \cite[Lemma 2.3.7]{BC} applied to $\eta_{D}(V_r)$.
\end{proof}

\begin{example}\label{ex:rregular} We will indicate how to calculate the regularity condition of \Cref{def:rregular} for some classical groups. We need to look for irreducible weight multiplicity-free representations $r\colon \gG\to \GL(V)$. It is sufficient that the kernel of $r$ is contained in $Z(\gG)$, since the flag varieties of $\gG$ and $\gG^{\ad}$ coincide and the uniqueness of a triangulation modulo $Z(\gG)$ is sufficient for the uniqueness of the triangulation for $\gG$. Hence, the assumptions of \Cref{emb_lemma_asdfjkl} will be satisfied. 

There is an exhaustive list of irreducible weight multiplicity-free representations in \cite[Theorem 4.6.3]{howe} for all almost-simple groups $\gG$. In addition, since $V$ is weight multiplicity-free, there is a unique complete flag in $V|_{\gB}$. Our notation refers to \cite[Planches II-IV, IX]{planches}.
    \begin{enum}
        \item Consider $\gG = \SO_{2l+1}$ (or $\mathrm{Spin}_{2l+1}$) for $l \geq 2$ (type $B_l$). The root system is defined in the vector space $\mathbb R^l$ with basis $\varepsilon_1, \dots, \varepsilon_l$. The positive roots are $\varepsilon_i$ for $1 \leq i \leq l$ and $\varepsilon_i \pm \varepsilon_j$ for $1 \leq i < j \leq l$.
        Let $V$ be the standard representation of $\gG$.
        The unique flag in $V|_{\gB}$ has weights ordered as $\varepsilon_1, \dots, \varepsilon_l, 0, -\varepsilon_l, \dots, -\varepsilon_1$. This gives us the positive weight differences $\varepsilon_i \pm \varepsilon_j$ for $1 \leq i < j \leq l$ and $2\varepsilon_i$ for $1 \leq i \leq l$.
        The resulting regularity condition is strictly stronger than the one formulated in \Cref{conj_Ger}: In addition to asking that $\delta \circ (-\varepsilon_i^{\vee})$ is not of the form $z^{\mathbf k}$ for $\mathbf k \in \Z_{\leq 0}^{\Sigma}$, we also ask this for $(\delta \circ \varepsilon_i^{\vee})^2$.
        \item Consider $\gG = \Sp_{2l}$ for $l \geq 2$ (type $C_l$). The root system is defined in the vector space $\mathbb R^l$ with basis $\varepsilon_1, \dots, \varepsilon_l$. The positive roots are $2 \varepsilon_i$ for $1 \leq i \leq l$ and $\varepsilon_i \pm \varepsilon_j$ for $1 \leq i < j \leq l$. Let $V$ be the standard representation of $\gG$.
        We can pick as a Borel representation a full flag in $V$ such that the weights of $V$ are ordered as $\varepsilon_1, \dots, \varepsilon_l, -\varepsilon_l, \dots, -\varepsilon_1$. The resulting weight differences are exactly the positive roots. So \Cref{conj_Ger} holds in this case!
        \item Consider $\gG = \GSp_{2l}$ for $l \geq 2$ (type $C_l$). The root system is defined in the vector space $\mathbb R^{l+1}$ with basis $\varepsilon_0, \dots, \varepsilon_{l}$ (with $\varepsilon_0$ corresponding to the similitude character). 
        The positive roots are $\vareps_i-\vareps_j,\vareps_i+\vareps_j-\vareps_0$, and $2\vareps_i-\vareps_0$. 
        We can pick as a Borel representation a full flag in $V$ such that the weights of $V$ are ordered as $\varepsilon_1, \dots, \varepsilon_l, \varepsilon_{0}-\varepsilon_l, \dots, \varepsilon_{0}-\varepsilon_1$. The resulting weight differences are exactly the positive roots, so we obtain \Cref{conj_Ger} once more.
        \item Consider $\gG = \SO_{2l}$ (or $\mathrm{Spin}_{2l}$) for $l \geq 3$ (type $D_l$). The root system is defined in the vector space $\mathbb R^l$ with basis $\varepsilon_1, \dots, \varepsilon_l$. The positive roots are $\varepsilon_i \pm \varepsilon_j$ for $1 \leq i  < j \leq l$. Let $V$ be the standard representation of $\gG$. We can pick as a Borel representation a full flag in $V$ such that the weights of $V$ are ordered as $\varepsilon_1, \dots, \varepsilon_l, -\varepsilon_l, \dots, -\varepsilon_1$. The weight differences are the positive roots and $2 \varepsilon_i$ for $1 \leq i \leq l$. Also in this case the resulting regularity condition is stronger than the one predicted by \Cref{conj_Ger}, and of the same form as for $B_l$.
        \item Consider $\gG = \mathrm G_2$. We define the root system in the plane $\{x \in \mathbb R^3 \mid x_1 + x_2 + x_3 = 0\}$. The base is $\alpha_1 = \varepsilon_1 - \varepsilon_2$ (short root), $\alpha_2 = -2\varepsilon_1 + \varepsilon_2 + \varepsilon_3$ (long root). The positive roots are $\alpha_1, \alpha_2, \alpha_1+\alpha_2,2\alpha_1+\alpha_2,3\alpha_1+\alpha_2,3\alpha_1+2\alpha_2$. Let $V$ be the (faithful) $7$-dimensional representation of $\gG$ of highest weight $2\alpha_1+\alpha_2$. We can pick as a Borel representation a full flag in $V$ such that the weights of $V$ are ordered as $2\alpha_1+\alpha_2,\alpha_1+\alpha_2,\alpha_1,0,-\alpha_1,-\alpha_1-\alpha_2,-2\alpha_1-\alpha_2$. The weight differences are the positive roots and $2\alpha_1$, $2\alpha_1+2\alpha_2$ and $4\alpha_1+2\alpha_2$. So the resulting regularity condition is stronger than the one predicted by \Cref{conj_Ger}, but it is of a similar mild form as in the case $B_l$.
    \end{enum}
   The above computations can also be done for the exterior and symmetric powers of the standard representation of $\SL_n$ and its dual, the spin representation of $\mathrm{Spin}_{2l+1}$, the two half-spin representations of $\mathrm{Spin}_{2l}$, the two $27$-dimensional representations of $\mathrm E_6$ and the $56$-dimensional representation of $\mathrm E_7$. For $\mathrm F_4$ and $\mathrm E_8$ there are no weight multiplicity-free representations.
\end{example}


\section{$\gH$-$B_{\dR}$- and $\gH$-$B_{\dR}^+$-representations}
\label{secBdRBdRplusnew}

In \Cref{secBdRnew} below, we let $\gH$ be an arbitrary linear algebraic group over $L$, not necessarily reductive.
Let $A$ be a finite-dimensional $L$-algebra.

\subsection{Definitions}
\label{secBdRnew}

We use Fontaine's period rings $\BdR^+$ and $\BdR$ equipped with their natural topology and $\GK$-action, that we introduced in \Cref{sec:phigammat}.
A \emph{$\BdR$-representation} (resp. \emph{$\BdR^+$-representation}) of $\GK$ is a finite-dimensional $\BdR$-vector space $W$ (resp. finite free $\BdR^+$-module $W^+$) with a continuous semilinear action of $\GK$.

An \emph{$A \otimes_{\Qp} \BdR$-representation} of $\GK$ is a $\BdR$-representation $W$ together with a map of $\Qp$-algebras $A \to \End_{\Rep_{\BdR}(\GK)}(W)$, such that $W$ is a finite free\footnote{As a finite $\BdR$-algebra $A \otimes_{\Qp} \BdR$ is semi-local, so finite projective modules of constant rank are free.} $A \otimes_{\Qp} \BdR$-module. We denote by $\Rep_{A \otimes_{\Qp} \BdR}(\GK)$ the category of such representations. We similarly define the category $\Rep_{A \otimes_{\Qp} \BdR^+}(\GK)$ of $A \otimes_{\Qp} \BdR^+$-representations of $\GK$.

\begin{definition}
    An \emph{$\gH$-$A \otimes_{\Qp} \BdR^+$-representation} of $\GK$ is an $L$-linear fiber functor $$\eta_W\colon \Rep_L(\gH)\longrightarrow \Rep_{A \otimes_{\Qp} \BdR^+}(\GK).$$ We denote by $\Rep^{\gH}_{A \otimes_{\Qp} \BdR^+}(\GK)$ the category of such representations together with isomorphisms. \\ We define \emph{$\gH$-$(A \otimes_{\Qp} \BdR)$-representations} in the same way by replacing $\BdR^+$ with $\BdR$. 
\end{definition}


We introduce the period rings $B^+_{\pdR} \colonequals \BdR^+[\log t]$ and $B_{\pdR} \colonequals \BdR \otimes_{\BdR^+} B_{\pdR}^+$ (c.f. \cite[§4.3]{FontaineAst}): they are simply polynomial rings in the variable $\log t$, equipped with a continuous $\GK$-action that extends that on $\BdR$ and $\BdR^+$ compatibly with the notation $\log t$, i.e. satisfying $g(\log(t))=\log(t)+\log(\varepsilon(g))$ for $g\in \GK$. We equip $B_{\pdR}$ with a $B_{\dR}$-derivation $\nu_{B_{\pdR}}$ satisfying $\nu_{B_{\pdR}}(\log(t))=-1$ (i.e. the derivation $-\delta/\delta\log t$ if we consider elements of $B_\pdR$ as polynomials in $\log t$ with $\BdR$-coefficients). 

We say that a $\BdR$-representation $W$ of $\GK$ is \emph{almost de Rham} if $D_\pdR(W) \colonequals (B_{\pdR} \otimes_{\BdR} W)^{\GK}$ is a $K$-vector space of dimension $\dim_{\BdR} W$. The almost de Rham $\BdR$-representations form a full tensor subcategory $\Rep_{\pdR}(\GK)$ of $ \Rep_{\BdR}(\GK)$ stable under kernels, cokernels, extensions and duals.\footnote{See \cite[Section 3.7]{FontaineAst}, where $\Rep_{\pdR}(\GK)$ is noted $\Rep_{\dR,\Z}(\GK)$.}

Recall (cf. \cite[Proposition 3.1.1]{BHS19}) that $\Ga$ acts on $D_\pdR(W)$ via the endomorphism induced by $\nu_{B_\pdR}\otimes 1$ on $B_{\pdR} \otimes_{\BdR} W$. This gives rise to an equivalence of $K$-linear abelian tensor categories
\begin{align}
    D_\pdR \co \Rep_{\pdR}(\GK) \eqto \Rep_K(\Ga). \label{equiv_pdR_Ga}
\end{align}

Let $A$ be a finite-dimensional $\Qp$-algebra. 
We say that an $A \otimes_{\Qp} \BdR$-representation $W$ of $\GK$ is \emph{almost de Rham} if it is so as a $\BdR$-representation. We denote the category of almost de Rham $A \otimes_{\Qp} \BdR$-representations by $\Rep_{\pdR, A}(\GK)$. By \cite[Proposition 3.1.4]{BHS19} the equivalence \eqref{equiv_pdR_Ga} can be upgraded to an equivalence of $A$-linear tensor categories
\begin{align}
    D_{\pdR,A} \colon \Rep_{\pdR, A}(\GK) \eqto \Rep_{A \otimes_{\Qp} K}(\Ga). \label{equiv_pdR_Ga_A}
\end{align}
Recall that, for a ring homomorphism $B\to C$, $\phi_C:\Vect_B\to\Vect_C$ denotes the extension of scalars.

\begin{definition}
    Let $A$ be a finite dimensional $\Qp$-algebra. An $\gH$-$A\otimes_{\Qp}\BdR$-representation $\eta_{W_A}$ of $\GK$ is called \emph{almost de Rham} if for every $V\in \Rep_L(\gH)$, the $A\otimes_{\Qp}\BdR$-representation  $\eta_{W_A}(V)$ of $\GK$ is almost de Rham. An $\gH$-$A\otimes_{\Qp}\BdR^+$-representation $\eta_{W_A^+}$ of $\GK$ is called \emph{almost de Rham} if $\phi_{\BdR}\circ \eta_{W_A^+}$ is almost de Rham. 
\end{definition}
Recall that $\omega_{\gH}:\Rep_L(\gH)\to\Vect_L$ is the forgetful functor. We will also write 
$$ \omega_{\mathbb{G}_{a,A\otimes K}} \colon \Rep_{A\otimes_{\Qp}K}(\Ga) \longrightarrow \Vect_{A\otimes_{\Qp}K}$$ for the canonical forgetful functor. 

Later, we will construct $\gB$-flags on the period modules $D_{\pdR}$. For this, will need to trivialize the corresponding fiber functors, in the sense of the following definition. 
\begin{definition}\label{deftrivialization}
     Let $A$ be a finite dimensional $L$-algebra. A \emph{trivialization} of an $\gH$-$A\otimes_{\Qp}\BdR$-representation $\eta_{W_A}$ is an equivalence
     $$ \alpha_A\colon \phi_{A\otimes_{\Qp}K}\circ \omega_{\gH}\xrightarrow{\sim} \omega_{\mathbb{G}_{a,A\otimes K}} \circ D_{\pdR,A}\circ \eta_{W_A}.  $$
\end{definition}

\begin{lemma}\label{pdR_charanew}
    The data $(W_A,\alpha_A)$ of an almost de Rham  $\gH$-$A\otimes_{\Qp}\BdR$-representation equipped with a trivialization is equivalent to the datum of a nilpotent element $N_{W_A}\in \Lie(\gH)(A\otimes_{\Qp}K)$.
\end{lemma}
\begin{proof}
    Composing with $D_{\pdR,A}$, we see that $(W_A,\alpha_A)$ is equivalent to the datum of a fiber functor 
    $$ \eta' \colon \Rep_L(\gH) \longrightarrow \Rep_{A\otimes_{\Qp}K}(\Ga) $$
    equipped with an equivalence $\alpha_A' \colon \phi_{A\otimes_{\Qp}K} \circ \omega_{\gH}\xrightarrow{\sim} \omega_{\mathbb{G}_{a,A\otimes K}} \circ \eta'$.
    For each $V \in \Rep_L(\gH)$ the datum of a $\Ga$-action on $(A \otimes_{\Qp} K) \otimes_{L} V$ is equivalent to a nilpotent endomorphism $N_V \in \End_{A \otimes_{\Qp} K}((A \otimes_{\Qp} K) \otimes_{L} V)$. A standard calculation shows that $N_V$ is natural in $V$ and $N_{V \otimes W} = N_V \otimes \id_W + \id_V \otimes N_W$ if and only if the family $(N_V)_{V \in \Rep_L(\gH)}$ assembles to a fiber functor $\eta'$ as above.
    It is known that the set of natural transformations $E \colon \eta' \to \eta'$ satisfying $E(V \otimes W) = E(V) \otimes \id_W + \id_V \otimes E(W)$ is in bijection with $\Lie (\gH)(A \otimes_{\Qp} K)$.
\end{proof}

\subsection{$A$-$B$-pairs and the functors $W_\dR^+,W_\dR,W_e$}
We recall the definition of $B$-pair following \cite[§2]{Berger_2008constr}.
We introduce the ring $B_e \colonequals B_{\cris}^{\varphi=1}$, which carries a continuous $\GK$-action.

\begin{definition}
    Let $A\in \Art_L$.
    \begin{enum}
        \item An \emph{$A \otimes_{\Qp} B_e$-representation} is a free $A \otimes_{\Qp} B_e$-module of finite rank equipped with a continuous semilinear action of $\GK$. We write $\Rep_{A\otimes_{\Qp}B_e}$ for the category of $A \otimes_{\Qp} B_e$-representations of $\GK$, with the evident choice of morphisms.
        \item An \emph{$A$-$B$-pair} is a pair $(M_e, M_{\dR}^+)$, where $M_e$ is an $A \otimes_{\Qp} B_e$-representation and $M_{\dR}^+$ is a $\GK$-stable $A \otimes_{\Qp} B_{\dR}^+$-lattice in $M_{\dR} \colonequals B_{\dR} \otimes_{B_e} M_e$.
    \end{enum}
\end{definition}

Berger's equivalence between $\PGcat^+_K$ and the category of $B$-pairs is refined in \cite[Proposition 3.5.1]{BHS19} to an equivalence between the category of $\PGcat_{K,A}^+$ and the category of $A$-$B$-pairs, mapping a $\PG$-module $D_A$ over $\cR_{K,A}$ to an $A$-$B$-pair $(W_{e,A}(D_A),W_{\dR,A}^+(D_A))$. The two components of this equivalence give functors
\begin{align*} W_{\dR,A}^+\colon \PGcat_{K,A}^+\longrightarrow \Rep_{A\otimes_{\Qp}\BdR^+}(\GK), \quad 
W_{e,A} \colon \PGcat_{K,A}^+ \longrightarrow \Rep_{A\otimes_{\Qp}B_e}(\GK), \end{align*}
that extend the evident functors from $A$-linear $\GK$-representations to $A$-$B_\dR^+$- and $A$-$B_e$-representations obtained by extending scalars to the respective period rings. 
For  $\mathcal{M}_A\in \PGcat_{K,A}$, we define 
\[ W_{e,A}(\mathcal{M}_A)\colonequals W_{e,A}(D_A), \quad W_{\dR,A}(\mathcal{M}_A)\colonequals B_\dR\otimes_{B_\dR^+}W_{\dR,A}^+(D_A), \]
where $D_A\subset \cM_A$ is a $\PG$-module over $\cR_{K,A}$ such that $D_A=\mathcal{M}_A[\tfrac{1}{t}]$. These definitions are independent of the choice of $D_A$.

\begin{lemma}\label{Btensoreq}
    The functors $W_{\dR,A}^+$ and $W_{e,A}$ are $A$-linear tensor functors, and $W_{e,A}$ yields an exact tensor equivalence
    \begin{equation} 
        W_{e,A}\colon \PGcat_{K,A}\xrightarrow{\sim} \Rep_{A\otimes_{\Qp}B_e}(\GK). 
    \end{equation}
\end{lemma} 

\begin{proof}
The only part of the statement that is not formal is the fact that a free $\cR_{K,A}$-module is mapped by $W_{\dR}^+$ to a free $\BdR^+$-module and by $W_{e,A}$ to a free $B_{e,A}$-module. This is proved for $W_{\dR,A}^+$ in \cite[Lemma 3.3.5]{BHS19}, and for $W_{e,A}$ in \cite[Proposition 3.5.1]{BHS19}. 
\end{proof}



Now let $\cM$ be an $\gH$-$(\varphi, \Gamma_K)$-module over $\cR_{K,A}[\tfrac{1}{t}]$. The composition $W_{\dR}(\eta_{\mathcal{M}})\colon \Rep_L(\gH) \to \Rep_{A \otimes_{\Qp} B_\dR}(\GK)$ defines an $\gH$-$A \otimes_{\Qp} \BdR$-representation of $\GK$. We similarly define an $\gH$-$A \otimes_{\Qp} \BdR^+$-representation $ W^{+}_{\dR}(\eta_D)$ for a $\gH$-$\PG$-module $D$ over $\cR_{K,A}$.

\begin{proposition}\label{localgpdR}
Let $\cM$ be a triangulable $\gG$-$\PG$-module over $\cR_{K,L}[\tfrac{1}{t}]$, and let $\delta\colon \gT^\vee(K)\to L^\times$ be the parameter of a $\gB$-triangulation of $\cM$. Then the following are equivalent:
\begin{enumerate}[label=(\roman*)]
\item $\cM$ is Hodge--Tate;
\item $W_\dR\circ\eta_\cM$ is almost de Rham;
\item $\delta$ is locally algebraic.
\end{enumerate}
Moreover, if $\cM_A$ is a triangulable $\gG$-$\PG$-module over $\cR_{K,A}[\tfrac{1}{t}]$, then $W_\dR(\eta_{\cM_A})$ is almost de Rham if and only if $W_\dR(\eta_{\cM_A\otimes_AL})$ is almost de Rham. 
\end{proposition} 


\begin{proof}
By definition, $W_\dR\circ\eta_\cM$ is almost de Rham if and only if $W\colonequals W_\dR(\eta_\cM(V))$ is almost de Rham for every $V\in\Rep_L(\gG)$. Since $\cM$ admits a $\gB$-triangulation, $\eta_\cM(V)$ admits a $\gB'$-triangulation for a Borel subgroup $\gB'\subset\GL(V)$, with parameter $\delta_V$ obtained by composing $\delta$ with a map $\gT^{',\vee}\to\gT^\vee$ where $\gT'\subset\gB'$ is a maximal torus. Since the category of almost de Rham representations is stable under taking subquotients and extensions, $\eta_\cM(V)$ is almost de Rham if and only if $W_\dR(\cR_{K,L}(\delta_V))$ is almost de Rham. By \cite[Lemma 3.3.6 (ii)]{BHS19} and the preceding discussion, $W_\dR(\cR_{K,L}(\delta_V))$ is almost de Rham if and only if $\delta_V$ is locally algebraic. Therefore, $W_\dR\circ\eta_\cM$ is almost de Rham if and only if $\delta_V$ is locally algebraic for every $V$. This is equivalent to $\delta\circ\lambda^\vee$ being locally algebraic for every $\lambda\in X^\ast(\gT)$, hence to $\delta$ itself being locally algebraic. This gives the equivalence of (ii) and (iii). Since for a character $K^\times\to L^\times$ being locally algebraic is equivalent to being Hodge--Tate, and being Hodge--Tate is also stable under taking subquotients and extensions, the equivalence of (i) and (iii) also follows.

The last statement follows by dévissage, as in \cite[Remark 3.1.5]{BHS19}.
\end{proof}

\subsection{The schemes $X$ and $X^w$}\label{sec:Xnew}

Following \cite[\S 2]{BHS19}, we define the reductive group
\begin{align}
    \bG \colonequals &\Spec(L) \times_{\Spec(\Qp)} \Res_{K/\Qp}\gG_K \cong \underbrace{\gG_L \times \dots \times \gG_L}_{[K:\Qp] \text{ times}}. \label{def_phat_Gnew}
\end{align}
Similarly we define $\mathbf B$ and $\mathbf T$.
The Lie algebras of $\bG$, $\mathbf B$ and $\mathbf T$ are denoted by $\frakg$, $\frakb$ and $\frakt$ respectively. We write $\gW$ and $\bW$ for the Weyl groups of $(\gG,\gT)$ and $(\bG,\bT)$, respectively. The group $\bG$ will appear naturally in the following discussion when considering Hodge--Tate cocharacters attached to Hodge--Tate representations $\GK\to\gG(L)$, and weights of characters $\gT^\vee(K)\to L^\times$.

\begin{remark}\label{fiberfunctorforbG}
    Given an $L$-algebra $A$, an $L$-linear fiber functor $\eta\colon \Rep_L(\bG)\to \Vect_A$ is equivalent to the data of an $L$-linear fiber functor $\eta'\colon \Rep_L(\gG)\to \Vect_{A\otimes_{\Qp}K}$.
\end{remark}

We will use the following fact in the proof of \Cref{Twdelta}.

\begin{lemma}\label{invariantlemmanew}
For $A\in\Art_L$, the projection $\frakt\onto\frakt\GIT\bW$ induces a bijection $(\frakt\GIT\bW)(A)=\frakt(A)/\bW$. In other words, two points $x_1,x_2\in\frakt(A)$ belong to the same $\bW$-orbit if and only if their images under $\frakt\onto\frakt\GIT\bW$ coincide.
\end{lemma}

\begin{proof}
Since $\Lbar[\frakt]$ is a polynomial ring, the Chevalley--Shephard--Todd theorem implies that $\Lbar[\frakt]^\bW$ is also a polynomial ring. 
Since $\Lbar[\frakt]^\bW$ and $\Lbar[\frakt]$ are polynomial rings of the same finite dimension, the extension $\Frac(\Lbar[\frakt]^\bW)\subset\Frac(\Lbar[\frakt])$ of fraction fields is finite, and separable since $L$ is of characteristic 0. 
Therefore, $\Lbar[\frakt]$ is finite free over $\Lbar[\frakt]^\bW$.

The $\Frac(\Lbar[\frakt]^\bW)$-algebra $\Frac(\Lbar[\frakt]^\bW)\otimes_{\Lbar[\frakt]^\bW}\Lbar[\frakt]$ is contained in $\Frac(\Lbar[\frakt])$, hence of dimension at most $\lvert\bW\rvert$ as a $\Frac(\Lbar[\frakt]^\bW)$-vector space. Specializing at a regular point of $\Lbar[\frakt]$ shows that the $\Lbar[\frakt]^\bW$-rank of $\Lbar[\frakt]$ is at least $\lvert\bW\rvert$, which implies that $\Frac(\Lbar[\frakt])^\bW=\Frac(\Lbar[\frakt]^\bW)$.

Let $A'=\Z[\frakt]\otimes_{\Z[\frakt]^\bW}A$. Since $\Z[\frakt] \otimes_{\Z[\frakt]^\bW}L\cong L^{\lvert\bW\rvert}$, we obtain a map $A'\to L^{\lvert\bW\rvert}$. Lifting the corresponding orthogonal system of idempotents to $A$ gives us an isomorphism $A'\to A^{\lvert\bW\rvert}$ where the action of $\bW$ permutes the copies of $A$ on the right-hand side, so that every $\bW$-orbit of $L$-points of $A'$ has $\lvert\bW\rvert$ elements. The fiber above an $L$-point $\Z[\frakt]^\bW$ consists of $\lvert\bW\rvert$ points because of the previous paragraph, so it must coincide with a $\bW$-orbit.
\end{proof}

Following the approach of \cite{BHS19}, we will give a local description of the deformation space of trianguline $\gG$-representations of $\GK$ in terms of localizations of a certain variety $X$ attached to $\bG$. We first summarize some definitions and results of \cite[§2]{BHS19} and specialize them to our situation. Keep the notation of the previous subsection, and write $\Ad : \bG \to \GL(\frakg)$ for the adjoint action.

\begin{definition}[{cf. \cite[(2.1), (2.2)]{BHS19}}]\label{def:gtilde}
    We define 
    $$ \widetilde\frakg \colonequals \bG \times^{\bB} \frakb =  \{(g\bB,\psi) \in \bG/\bB \times \frakg \mid \Ad(g^{-1})\psi \in \frakb\}, $$
    where $\bG \times^{\bB} \frakb$ denotes the quotient of $\bG \times_{\Spec L} \frakb$ by the free right action of $\mathbf B$ defined as $(g,\psi) \cdot b \colonequals (gb,\Ad(b^{-1})\psi)$.
\end{definition}
The scheme $\widetilde\frakg$ is smooth and irreducible of dimension $\dim_L \frakg = \dim \gG \cdot [K:\Qp]$.

We define morphisms with source $\wtl\frakg=\bG \times^{\mathbf B} \frakb$ by giving a formula for a $\mathbf B$-equivariant morphism with source $\bG \times \frakb$. There is a map 
\begin{equation}\label{kappat}
    \wtl\frakg\longrightarrow\frakt, \quad (g\bB,\psi)\mapsto \ovl{\Ad(g^{-1}\psi)},
\end{equation}
where $\ovl{b}$ denotes the image of $b \in \frakb$ under the projection $\frakb \to \frakt$. Note that replacing $g$ with a different representative for $g\bB$ amounts to conjugating the right-hand side with an element of $\bB$, which does not modify the image of the projection to $\frakt$, so that the map is well-defined. 

\begin{definition}[{cf. \cite[(2.3)]{BHS19}}]\label{def_X} We define an $L$-scheme
$$ X \colonequals \widetilde \frakg \times_{\frakg} \widetilde \frakg = \{(g_1\bB,g_2\bB,\psi) \in \bG/\bB \times \bG/\bB \times \frakg \mid \Ad(g_1^{-1})\psi, \Ad(g_2^{-1})\psi \in \frakb\}.$$
\end{definition}
The points of $X$ parametrize two flags compatible with an element of $\frakg$.

We write $\pi \colon X \to \bG/\bB \times \bG/\bB$ for the inclusion $X \hookrightarrow \bG/\bB \times \bG/\bB \times \frakg$ followed by the projection to the first two factors.
For $i \in \{1,2\}$, we define 
\begin{equation}\label{kappa}
\kappa_i \colon X\longrightarrow \frakt, \quad
(g_1\bB, g_2\bB,\psi) \mapsto \ovl{\Ad(g_i^{-1})\psi}, 
\end{equation}
where $\ovl{b}$ denotes again the image of $b \in \frakb$ under the projection $\frakb \to \frakt$. These maps are well-defined for the same reason as \eqref{kappat}. 
The pair $(\kappa_1,\kappa_2)$ defines a map from $X$ to the variety
\begin{equation}
T=\frakt\times_{\frakt\GIT\bW}\frakt
\end{equation}
introduced in \cite[Section 2.5]{BHS19}. The irreducible components $T^w$ of $T$ are indexed by $w\in\bW$ as in \cite[Lemma 2.5.1]{BHS19}.

For $w \in \bW$ and some lift $\dot w \in N_{\bG}(\bT)(L) \subseteq \bG(L)$, we define $U^w \colonequals \bG(1,\dot w)(\bB \times \bB) \subseteq \bG/\bB \times \bG/\bB$ and $V^w \colonequals \pi^{-1}(U^w)$.
We let $X^w$ be the Zariski closure of $V^w$ in $X$ and let $\kappa_{i,w} : X^w \to \frakt$ denote the restriction of $\kappa_i$ to $X^w$. Then $(\kappa_{1,w},\kappa_{2,w})$ maps $X^w$ onto $T^w$. The morphisms $\kappa_i$ and $\kappa_{i,w}$ are flat by \cite[Proposition 2.3.3]{BHS19}. This will be used to establish another flatness result \Cref{flatness_result}.

By \cite[Proposition 2.2.5]{BHS19}, the scheme $X$ is locally a complete intersection, and its irreducible components are the subschemes $X^w$ for $w\in\bW$. In particular, each $X^w$ is Cohen--Macaulay of dimension $\dim X=\dim\frakg=\dim\bG$. 
We refer to \cite[\S 2]{BHS19} for further properties of $X$.


\subsection{The Hodge--Tate flag}\label{sec:HTflagsnew}
Let $A\in\Art_L$, and let $(W_A^+,\alpha_A)$ be a pair where $W_A^+$ is an almost de Rham $\gG$-$B^+_{\dR,A}$-representation, and $\alpha_A \colon \phi_{A\otimes_{\Qp}K}\circ \omega_{\gG}\xrightarrow{\sim} \omega_{\mathbb{G}_{a,A\otimes K}}\circ D_{\pdR,A}\circ \eta_{W_A}$ is a trivialization of $\eta_{W_A}\colonequals \phi_{\BdR}\circ \eta_{W_A^+}$. Given $V\in \Rep_L(\gG)$, we follow  \cite[(3.5) and (3.9)]{BHS19} and define a decreasing Hodge--Tate filtration $\Fil_{\eta_{W_A^+}(V)}^\bullet$ on $D_{\pdR}(\eta_{W_A}(V))$
as
$$ \Fil^i_{\eta_{W_A^+}(V)}\big(D_\pdR(\eta_{W_A}(V))\big) \colonequals \big(t^iB_{\pdR}^+\otimes_{B_\dR^+}\eta_{W_A^+}(V)\big)^{\GK}.$$
Via $\alpha^{-1}_A$, the filtration $\Fil_{\eta_{W_A^+}(V)}^{\bullet}$ induces a $\gG$-filtration $\mathcal{F}_{\HT,A}$ on $\phi_{A\otimes_{\Qp}K}\circ \omega_{\gG}$ which, under the decomposition $A\otimes_{\Qp} K\cong \oplus_{\tau\in \Sigma} A$, breaks into $\gG$-filtrations $\mathcal{F}_{\HT,A,\tau}$ on $\phi_A \circ \omega_{\gG}$ indexed by $\tau\in \Sigma$.

Write $W^+=W_A^+\otimes_AL$ and $W=W_A\otimes_AL$. When $A = L$, we drop $A$ from the notation.
\begin{definition}\label{def_HT}
The construction in \Cref{sec_cocharacters} attaches to the above $\gG$-filtration $\mathcal{F}_{\HT,\tau}$ a dominant cocharacter $-\varpi_{\HT,\tau}\in X_*(\gT)$. The tuple $\varpi_{\HT}=(\varpi_{\HT,\tau})_\tau\in X_*(\bT)$ of opposite cocharacters is called the \emph{Hodge--Tate cocharacter} of $W^+$.
\end{definition}

Note that $\varpi_\HT$ does not depend on the choice of $\alpha$. 

\begin{definition}\label{def:WA+reg}
    We say that the almost de Rham $\gG$-$B^+_{\dR,A}$-representation $W_A^+$ is \emph{regular} if the associated Hodge--Tate cocharacter $\varpi_\HT$ is regular (see \Cref{regcochar}).
\end{definition}


\begin{definition}\label{firstflag}
 Assume that $W^+$ is regular. We define $x_{W^+_A,\HT}\in (\bG/\bB)(A)$ to be the element specified by the Plücker data indexed by $\tau\in \Sigma$
 $$ \left(\big( \eta_{\mathcal{F}_{\HT,A,\tau}}^{-\langle\lambda,\varpi_{\HT,\tau}\rangle}(V_{\lambda}\otimes_L A)\subseteq V_{\lambda} \otimes_L A \big)_{\lambda\in X^*_+(\gT)}\right)_{\tau\in \Sigma}. $$
 This is well defined by \Cref{laststepfiltration}.
\end{definition}

For every $V\in\Rep_L(\gG)$ and embedding $\tau:K\into L$, we define the $\tau$-Hodge--Tate weights $h_{\tau,1},\ldots,h_{\tau,n}$ of $\eta_{W^+}(V)$ as the opposites of the jumps in $\Fil^\bullet_{\eta_{W^+}(V)}(D_\pdR(\eta_W(V)))\otimes_{K,\tau} L$ (i.e. the opposites of the integers $i$ such that $\Fil^{i+1}\subsetneq\Fil^{i}$). Note that this corresponds to our convention that the cyclotomic character of $\cG_{\Q_p}$ has Hodge--Tate weight 1.

\begin{remark}
We show that \Cref{firstflag} recovers the flag defined in \cite[after Definition 3.2.4]{BHS19} in the case $\gG=\GL_n$. Indeed, if $\mathrm{Std}$ denotes the standard representation of $\GL_n$, then the flag in \emph{loc. cit.} is the direct sum over $\tau$ of the flags on $D_{\pdR,\tau}(\eta_{W_A}(\mathrm{Std}))$ whose associated Hodge--Tate weights are strictly increasing. In particular, the Plücker datum attached to this flag corresponds to the choice of the line of lowest Hodge--Tate weight in $\eta_{W_A}(V_\lambda)$ for every $\lambda$. This coincides with the Pl\"ucker datum from \Cref{firstflag} since $\varpi_{\HT, \tau}$ is antidominant.
\end{remark}

Let $N_{W_A}$ be the nilpotent operator attached to $(W_A^+,\alpha_A)$ by \Cref{pdR_charanew}.

\begin{lemma}\label{pointgtilde1}
    Assume that $W^+$ is regular. The pair $(x_{W_A^+,\HT},N_{W_A})\in (\bG/\bB)(A) \times \frakg(A)$ defines an $A$-point of $\gtilde$.
\end{lemma}

\begin{proof}
 Let $g\in \bG(A)$ be an element corresponding to $x_{W^+,\HT}\in (\bG/\bB)(A)=\bG(A)/\bB(A)$, where the last equality holds since $A$ is Artinian.
 
    Let $(V,\rho_V)\in \Rep_L(\gG)$. By construction, the nilpotent operator $\rho_V(N_{W_A})$ preserves the filtration $\Fil_{\eta_{W_A^+}(V)}^{\bullet}$ on $D_{\pdR}(\eta_{W_A}(V))$. In particular, for each $\tau \in \Sigma$ it preserves the Plücker datum $$g\cdot ( V_{\lambda}^{\lambda}\otimes_L A)=\eta_{\mathcal{F}_{\HT,A,\tau}}^{-\langle\lambda,\varpi_{\HT,\tau}\rangle}(V_{\lambda}\otimes_L A)\subseteq V_{\lambda} \otimes_L A, \quad \lambda\in X^*_+(\gT).$$
Consequently, $\Ad(g^{-1})N_{W_A}$ preserves $V_\lambda^{\lambda}\otimes_L A$ for all $\lambda\in X^*_+(\gT)$, and therefore, $\Ad(g^{-1})N_{W_A}\in \mathfrak{b}(A)$.
\end{proof}

There is a Sen operator $\Theta({\eta_{W_A^+}(V)})$ acting on the $A\otimes_{\Q_p}K_\infty$-module $\Delta_\Sen(\eta_{W_A^+}(V)/t\eta_{W_A^+}(V))$ defined in \cite{sencont}. Its characteristic polynomial has coefficients in $A\otimes_{\Q_p}K$: this is due to \cite[Theorem 5]{sencont} when $A=L$, and follows for instance from \cite[Lemma 3.7.5]{BHS19} for arbitrary $A\in\Art_L$. 
We define a map
\[ \pi_{W^+_A,\Sen}\colon \Rep_{L}(\gG)\longrightarrow A\otimes_{\Q_p}K \]
sending $V$ to the trace of $\Theta(\eta_{W_A^+}(V))$. It is a standard fact that for $V,V'\in\Rep_L(\gG)$,
\begin{itemize}
    \item $\Theta(\eta_{W_A^+}(V\oplus V'))=\big(\Theta(\eta_{W_A^+}(V)),\Theta(\eta_{W_A^+}(V'))\big)$,
    \item $\Theta(\eta_{W_A^+}(V\otimes_L V'))=\Theta(\eta_{W_A^+}(V))\otimes \id + \id \otimes \Theta(\eta_{W_A^+}(V'))$.
\end{itemize} 
By the Tannakian formalism, once we choose a trivialization of $\Delta_{\Sen}\circ \phi_{\mathbb{C}_p}\circ \eta_{W_A^+}$, the Sen operator produces an element $\Theta^{\gG}_{W_A^+}\in \Lie(\gG)(A\otimes_{\Q_p}K_{\infty})$. 


\begin{definition}\label{def_Senpar}
    The element $y_{W_A^+,\Sen}\in (\frakt\GIT\bW)(A)$ obtained this way is called the \emph{Sen parameter} associated to $W_A^+$.
\end{definition}

\subsection{Deformations of almost de Rham representations}\label{sec_adRdefs}

\begin{definition}\label{def_XW} Let $W$ be an almost de Rham $\gH$-$(L \otimes_{\Qp} \BdR)$-representation of $\GK$. We define a pseudofunctor $X_W \colon \Art_L \to \Gpd$ as follows:
\begin{enum}
    \item For every $A \in \Art_L$, let the objects of $X_W(A)$ be pairs $(W_A, \iota_A)$, where $W_A$ is an almost de Rham $\gH$-$A \otimes_{\Qp} \BdR$-representation of $\GK$, and $\iota_A \colon \phi_{L}\circ \eta_{W_A} \eqto \eta_W$ is an isomorphism of $\gH$-$L \otimes_{\Qp} \BdR$-representations of $\GK$. A morphism $(W_1, \iota_1) \to (W_2, \iota_2)$ is an isomorphism of $\gH$-$A \otimes_{\Qp} \BdR$-representations $\varphi \colon \eta_{W_1} \to \eta_{W_2}$ such that $\iota_1 = \iota_2 \circ \phi_L \circ \varphi$.
    \item For a morphism $f : A_1 \to A_2$ in $\Art_L$, the functor $X_W(f) \colon X_W(A_1) \to X_W(A_2)$ maps $(\eta_{W_{A_1}}, \iota_{A_1})$ to $(\phi_{A_2}\circ \eta_{W_{A_1}}, \iota_{A_1})$.
\end{enum}

\end{definition}
We could have defined the deformation problem without requiring $W_A$ to be almost de Rham, since this is automatic when $W$ is almost de Rham.

We have the following framed version of \Cref{def_XW}.

\begin{definition}
Let $(W,\alpha)$ be an almost de Rham $\gH$-$(L \otimes_{\Qp} \BdR)$-representation of $\GK$ equipped with a trivialization $\alpha$. We define a pseudofunctor $X_W^{\square} \colon \Art_L \to \Gpd$ as follows:
    \begin{enumerate}
        \item For every $A \in \Art_L$, let the objects of $X_W^{\square}(A)$ be triples $(W_A, \iota_A, \alpha_A)$, where $(W_A, \iota_A)$ is an object of $X_W(A)$ and $\alpha_A$ is a trivialization of $W_A$, such that the following diagram commutes:
        \begin{center}
            \begin{tikzcd}
        \phi_{A\otimes_{\Qp}K}\circ \omega_{\gH} \arrow[d,"{\phi_{L\otimes_{\Qp}K}}"] \arrow[rr,"{\alpha_A}"] & & \omega_{\mathbb{G}_{a, A\otimes K}}\circ D_{\pdR,A}\circ \eta_{W_A} \arrow[d, "{\iota_A\circ\phi_{L}}"]
                \\ \phi_{L\otimes_{\Qp}K}\circ \omega_{\gH} \arrow[rr,"\alpha"] & & \omega_{\Ga}\circ D_{\pdR}\circ \eta_{W}
            \end{tikzcd}
        \end{center}
        A morphism $(W_1, \iota_1, \alpha_1) \to (W_2, \iota_2, \alpha_2)$ is an isomorphism $\varphi \colon (W_1, \iota_1) \eqto (W_2, \iota_2)$ in $X_W(A)$ that is compatible with the trivializations.
        \item For a morphism $f \colon A_1 \to A_2$ in $\Art_L$, the functor $$X_W^{\square}(f) \colon X_W^{\square}(A_1) \longrightarrow X_W^{\square}(A_2)$$ maps $(\eta_{W_{A_1}}, \iota_{A_1}, \alpha_{A_1})$ to $(\phi_{A_2}\circ\eta_{W_{A_1}}, \iota_{A_1}, \alpha_{A_1} \otimes_{A_1} A_2)$.
    \end{enumerate}
\end{definition}

In fact, $X_{W}^{\square}$ is valued in setoids, so we will regard it as a functor taking values in sets.
There is a canonical map $X_{W}^{\square} \to X_W$ obtained by forgetting the trivialization.

For the rest of the section, we fix  an almost de Rham $\gG$-$(L \otimes_{\Qp} \BdR)$-representation $W$ of $\GK$ equipped with a trivialization $\alpha$. Given $(W_A, \iota_A, \alpha_A) \in X_{W}^{\square}(A)$, we have an element $N_{W_A} \in \frakg \otimes_L A$ obtained by \Cref{pdR_charanew}.
We denote by $\widehat \frakg$ the formal completion of $\frakg$ at $N_W$, i.e., as a functor on $A\in \Art_L$, we have 
$$\widehat\frakg(A) = \{x \in \frakg \otimes_L A \mid x \equiv N_W \,\mathrm{mod}\, \frakm_A\}.$$
\begin{proposition}[{cf. \cite[Corollary 3.1.6]{BHS19}}]\label{XW_and_g}
    The map $X^{\square}_{W}(A) \to \widehat \frakg(A), ~(W_A, \iota_A, \alpha_A) \mapsto N_{W_A}$ is an equivalence.
    In particular, the functor $X^{\square}_{W}$ is pro-representable by a complete noetherian local $L$-algebra $R^{\square}_{W} \cong L\br{x_1, \dots, x_{\dim\gG \cdot [K:\Qp]}}$.
\end{proposition}

\begin{proof}
    By \Cref{pdR_charanew}, 
    it suffices to show that any deformation of $N_W$ along $\frakg \otimes_L A \to \frakg$ is a nilpotent element of $\frakg \otimes_L A$.
    The kernel of this map is $\frakg \otimes_L \frakm_A$ and thus consists only of nilpotent elements. This establishes the equivalence. Moreover, subtracting $N_W$ establishes a bijection $\widehat \frakg(A) \eqto \frakg \otimes_L \frakm_A$.
    The deformation problem is thus formally smooth of dimension $\dim_L \frakg = \dim\gG \cdot [K:\Qp]$.
\end{proof}

\subsubsection{Deformations of $W^+$}
Given an almost de Rham $\gG$-$(L \otimes_{\Qp} \BdR^+)$-representation $W^+$ such that $\eta_W=\phi_{\BdR}\circ \eta_{W^+}$, we define a pseudofunctor $X_{W^+} \colon \Art_L \to \Gpd$ in the same way as \Cref{def_XW}. Then the base change $\phi_{\BdR}$ induces a map $X_{W^+} \to X_W$, and we define $$X_{W^+}^{\square} \colonequals X_{W^+} \times_{X_W} X_W^{\square}.$$

\begin{proposition}[{cf. \cite[Theorem 3.2.5]{BHS19}}]\label{very_key_proposition_W+}
    Assume that $W^+$ is regular. Then the groupoid $X_{W^+}^{\square}$ is set-valued and is pro-representable and the functor 
\begin{align} \label{key_map+}
    X_{W^+}^{\square} \longrightarrow \gth, \quad (W^{+}_A, \iota_A, \alpha_A) \mapsto (x_{W_A^+,\HT}, ~N_{W_A}),
\end{align}
given by \Cref{pointgtilde1}, is an isomorphism. In particular, $X_{W^+}^{\square}$ is pro-representable by a formally smooth complete noetherian local $L$-algebra with residue field $L$ of dimension $\dim \gG_L \cdot [K:\Qp] = \dim \widetilde{\frakg}$.
\end{proposition}
\begin{proof}
 It remains to show that given $(x_A,N_A)\in \widehat{\widetilde{\frakg}}(A)$, there exists a unique $(W^{+}_A, \iota_A, \alpha_A)$ mapping to it. By \Cref{XW_and_g}, there exists $(W_A,\iota_A,\alpha_A)$ mapping to $N_A$. We need to show that $W_A$ is the base change of some $W_A^+$ with $x_{W_A^+,\HT}=x_A$. We do this by showing that $x_A$ induces compatible filtrations on $\eta_W(V)$ for every $V\in \Rep_L(\gG)$, and then applying \cite[Lemma 3.2.2]{BHS19} to identify this $\gG$-filtration with such a $W^+_A$.

Let $\tau\in \Sigma$. Given $V
\in \Rep_L(\gG)$, the Hodge-Tate cocharacter $\varpi_{\HT,\tau}$ induces a filtration $$\fil^i V =\sum_{\mu,\langle \mu,\varpi_{\HT,\tau} \rangle\ge i} V^{\mu},$$ where $V=\oplus_{\mu}V^{\mu}$ is the decomposition of $V$ as a $\gT$-representation. Let $g\in \gG(A)$ be the element such that $x_A=g\gB(A)\in \gG/\gB(A)$. Then the filtration we construct is $\Fil^i V\otimes_{L} A =g\cdot (\fil^i V\otimes_{L}A).$

If $x_A=x_{W_A^+,\HT}$, this recovers the Hodge-Tate filtration since by definition of $g$,
$$ \eta_{\mathcal{F}_{\HT,A,\tau}}^{-\langle\lambda,\varpi_{\HT,\tau}\rangle}(V_{\lambda}\otimes A)=g\cdot(V_\lambda^\lambda\otimes_L A)= g\cdot \fil^{-\langle\lambda,\varpi_{\HT,\tau}\rangle}(V\otimes_L A), $$
for $\lambda\in X_+^*(\gT)$.
\end{proof}

Write $\that$ for the formal completion of $\frakt$ at $0$.
The weight map $\kappa$ from \eqref{kappat} sends
$(x_{W^+,\HT},N_W)$ to $0 \in \frakt(L)$ (since $N_{W}$ is nilpotent), 
hence induces a morphism $\wh\kappa : \wh{\widetilde\frakg} \to \wh\frakt$. The composition
\begin{align*}
     X_{W+}^{\square} \overset{\eqref{key_map+}}{\eqto} \gth \xrightarrow{\wh\kappa} \wh{\frakt}
\end{align*}
factors through the projection $X_{W^+}^\square\to X_{W^+}$, inducing a map $\kappa_{W^+} \colon X_{\Wtri} \to \widehat\frakt$.
\subsubsection{Deformations of $\Wtri$}
We now assume that $\eta_{W}\colon \Rep_L(\gG)\to \Rep_{\GK}(L \otimes_{\Qp} \BdR)$ factors through an almost de Rham $\gB$-$(L \otimes_{\Qp} \BdR)$-representation $\eta_{\Wtri}\colon \Rep_L(\gB)\to \Rep_{\GK}(L \otimes_{\Qp} \BdR)$.
We define a pseudofunctor $X_{ \Wtri}^{\square} \colon \Art_L \to \Gpd$ as the fiber product
\begin{align}\label{XWWtridef}
    X_{\Wtri}^{\square} \colonequals X_{\Wtri} \times_{X_W} X_W^{\square}.
\end{align}
The $A$-points of $X_{ \Wtri}^{\square}$ can be described as triples $(W^{\triangle}_A, \iota_A, \alpha_A)$, where $W^{\triangle}_A$ is an almost de Rham $\gB$-$A \otimes_{\Qp} \BdR$-representation, $\iota_A \colon \phi_L\circ\eta_{W_A^{\triangle}} \eqto \eta_{\Wtri}$ is an isomorphism of $\gB$-$L \otimes_{\Qp} \BdR$-representations, and $\alpha_A$ is a trivialization of $W_A$ that is compatible with $\alpha$. Here, $\eta_{W_A}= \eta_{\Wtri_A}\circ\Res_{\gB}^{\gG}$.

By \Cref{Bredandflgs}, giving an $A\otimes_{\Qp}K$-valued point of $\gG/\gB$ is equivalent to giving a factorization of $\phi_{A\otimes_{\Qp}L}\circ\omega_{\gG}$ through $\Rep_L(\gB)$.
\begin{definition}\label{secondflag}
    We define $x_{\Wtri_A,\tri}\in (\bG/\bB)(A)$ to be the $A$-valued point of $\bG/\bB$ corresponding to the factorization $\alpha_A^{-1}\circ \omega_{\mathbb{G}_{a,A\otimes K}}\circ D_{\pdR,A}\circ \eta_{\Wtri_A}$.
\end{definition}
Recall that the $L$-scheme $\gtilde$ was introduced in \Cref{def:gtilde}.
 
\begin{lemma}\label{pointgtilde2}
    The pair $(x_{\Wtri_A,\tri}, N_{W_A}) \in (\bG/\bB)(A) \times \frakg(A)$ defines an $A$-point of $\gtilde$. 
\end{lemma}

\begin{proof}   Let $g\in \bG(A)$ be an element corresponding to $x_{\Wtri_A,\tri}\in (\bG/\bB)(A)=\bG(A)/\bB(A)$. The fiber functor $$\eta_x\colonequals \alpha_A^{-1}\circ \omega_{\mathbb{G}_{a,A\otimes K}}\circ D_{\pdR,A}\circ \eta_{\Wtri_A}$$ is a factorization of the forgetful functor $\omega_{\gG,A\otimes K}$ fitting in the following commutative diagram \begin{center}  \begin{tikzcd}[column sep=large, row sep=large]\Rep_{L}(\gG)  \arrow[r] \arrow[rr, bend left=20, "\omega_{\gG,A\otimes K}"]   \arrow[d, equals] & \Rep_{L}(\gB)         \arrow[r, "\eta_x"]       \arrow[d, equals] & \Vect_{A\otimes K}     \arrow[d, "{g^{-1}\cdot -}"] \\   \Rep_{L}(\gG)     \arrow[r]  & \Rep_{L}(\gB)   \arrow[r, swap, "\omega_{\gB,A\otimes K}"'] & \Vect_{A\otimes K}. \end{tikzcd} \end{center} As in \Cref{pdR_charanew}, this diagram implies that $\Ad(g^{-1})N_{\Wtri_A}$ corresponds to a compatible family of endomorphisms $\rho_V(g^{-1}) N_V \rho_V(g) \in \End(V\otimes_L({A\otimes_{\Qp} K}))$ for all $(\rho_V,V) \in \Rep_L(\gB)$. Therefore, $\Ad(g^{-1})N_{W_A}\in \mathfrak{b}(A)$. \end{proof}


We denote the completion of $\gtilde$ at $(x_{\Wtri}, N_W)$ by $\gth$. 
We then have $(x_{\Wtri_A}, ~N_{W_A}) \in \gth(A)$.

\begin{corollary}[{cf. \cite[Corollary 3.1.9]{BHS19}}]\label{319} The groupoid $X_{\Wtri}^{\square}$ is set-valued and is pro-representable. The functor 
\begin{align}
    X_{ \Wtri}^{\square} \longrightarrow \gth, \quad (W^{\triangle}_A, \iota_A, \alpha_A) \mapsto \big(x_{\Wtri_A}, ~N_{W_A}\big)
    \label{key_map}
\end{align}
is an isomorphism. In particular, $|X_{ \Wtri}^{\square}|$ is pro-representable by a formally smooth complete noetherian local $L$-algebra with residue field $L$ of dimension $\dim \gG_L \cdot [K:\Qp] = \dim \widetilde{\frakg}$.
\end{corollary}

As before, we write $\kappa_{\Wtri}$ for the composition
\begin{align*}
     X_{\Wtri}^{\square} \overset{\eqref{key_map}}{\eqto} \gth \xrightarrow{\wh\kappa} \wh{\frakt},
\end{align*}
which factors through the projection $X_{W^\triangle}^\square\to X_{\Wtri}$, inducing a map still denoted $\kappa_{\Wtri} \colon X_{\Wtri} \to \widehat\frakt$.

\subsection{Comparing deformations of $\gG$-$\PG$-modules and $\gG$-$\BdR$-representations}
\label{def_of_triangulations}

\begin{definition}\label{def_XM}
   Let $\cM$ be an $\gH$-$\PG$-module over $\cR_{K,L}[\tfrac{1}{t}]$. We define a pseudofunctor $X_{\cM} \colon \Art_L \to \Gpd$
    as follows: for $A \in \Art_L$, the objects of $X_{\cM}(A)$ are pairs $(\cM_A, j_A)$, consisting of an $\gH$-$(\varphi, \Gamma_K)$-module $\cM_A$ over $\cR_{K,A}[\tfrac{1}{t}]$, and an isomorphism $j_A \colon \cM_A \otimes_A L \eqto \cM$. The morphisms of $X_{\cM}(A)$ are given by the evident notion of isomorphisms of such pairs. For a map $A \to B$ in $\Art_L$ the map $X_{\cM}(A) \to X_{\cM}(B)$ is given by the evident base extension functor.
\end{definition}

We now let $\cM$ be a $\gG$-$\PG$-module over $\cR_{K,L}[\tfrac{1}{t}]$ that is equipped with a triangulation $\Mtri$ of parameter $\delta \colon \gT^{\vee}(K) \to L^{\times}$.
Recall from \Cref{defofparameters} that $\widehat\cT^\gT_{\delta}$ denotes the formal completion of $\cT^\gT$ at $\delta$, and that there is a morphism of formal schemes $\wt - \wt(\delta) \colon \wh\cT^\gT_{\delta} \to \wh\frakt$.

 By \Cref{class_rk_one_after_inv_t}, for any $A\in\Art_L$, every deformation $\Mtri_A\in X_{\Mtri}(A)$ is trianguline of some parameter $\delta_{A} \in \cT^\gT(A)$ (in the sense of \Cref{def_triangulation}). This allows us to define a morphism of groupoids $\omega_{\delta} \colon  X_{\Mtri} \to \widehat\cT^\gT_{\delta}$ of pseudofunctors by
$$ \omega_{\delta}\colon X_{\Mtri}(A) \longrightarrow \widehat\cT^\gT_{\delta}(A), \quad (\Mtri_A, j_A) \mapsto \delta_A. $$

\begin{lemma}\label{BHS3.3.9}
    The diagram of groupoids over $\Art_L$
    \begin{center}
        \begin{tikzcd}
            X_{\Mtri} \arrow[rr] \arrow[d,"\omega_\delta"] & & X_{W^{\triangle}} \arrow[d,"\kappa_{\Wtri}"]
             \\ \widehat\cT^\gT_{\delta} \arrow[rr, "\wt - \wt(\delta)"] & & \widehat{\frakt}
        \end{tikzcd}
    \end{center}
    is commutative.
\end{lemma}

\section{Lifting nonabelian $\PG$-modules, and a formally smooth morphism}
\label{sec_form_smooth}

\subsection{Lifting nonabelian $\PG$-modules}
\label{sec_form_smooth1}

Let
\begin{equation} \label{SES_fs_new}
1 \to \gU \to \gP \xrightarrow{\pi} \gM \to 1
\end{equation}
be a short exact sequence of linear algebraic groups over $L$. We will consider two situations:
\begin{enumerate}
    \item[(I)] $\gU$ is isomorphic to a product of copies of $\Ga$. Since $\gP$ is a $\gU$-torsor over the affine scheme $\gM$, there exists a scheme-theoretic section $\sigma \colon \gM \to \gP$. This induces an adjoint action of $\gM$ on $\gU$ defined by $m \cdot u := \sigma(m)u\sigma(m)^{-1}$. It also induces a decomposition $\gP \cong \gU \times \gM$ as $L$-schemes.
    \item[(II)] $\gU$ is unipotent, $\gM$ is reductive, and $\pi$ admits a group-theoretic section $\sigma\colon \gM\to \gP$. Then $\gM$ also acts on $\gU$ by $m \cdot u := \sigma(m)u\sigma(m)^{-1}$, and $\gP$ is a semidirect product $\gU \rtimes \gM$.
\end{enumerate}

In both situations, we fix the identification $\gP \cong \gU \times \gM$ of $L$-schemes.

Let $A \in \Aff_L$, and fix an $\gM$-trivial $\gM$-$(\varphi,\Gamma_K)$-module $D_0$ over $\mathcal{R}_{K,A}$, with a corresponding $\gM$-$(\varphi,\Gamma_K)$-pair $(M_0,c_0)$ given by  \Cref{G_PG_Pair_is_triv_PG_mod}. If $A\in \Art_L$, we will also consider the $\gM$-$(\varphi,\Gamma_K)$-module $\cM_0 \colonequals D_0[\tfrac{1}{t}]$ over $\cR_{K,A}[\tfrac{1}{t}]$. Our goal is to study the lifts of $D_0$ (and of $\cM_0$) to $\gP$-$\PG$-modules.

\begin{definition}
    A \emph{lift} of $D_0$ to $\gP$ is a pair $(D,\iota)$, where $D$ is a $\gP$-$(\varphi,\Gamma_K)$-module over $\cR_{K,A}$ and $\iota\colon D \times^{\gP} \gM \xrightarrow{\sim} D_0$ is an isomorphism of $\gM$-$(\varphi,\Gamma_K)$-modules.
    An \emph{equivalence} between two lifts $(D_1,\iota_1)$ and $(D_2,\iota_2)$ is an isomorphism $f\colon D_1\to D_2$ of $\gP$-$(\varphi,\Gamma_K)$-modules such that $\iota_2 \circ p_*f=\iota_1$. We write $\Lift_{\gM}^{\gP}(D_0)$ for the set of equivalence classes of lifts of $D_0$ to $\gP$.
\end{definition}
The adjoint action of $\gM$ on $\gU$ given by $\sigma$ induces a twisted $\varphi$ and $\Gamma_K$ action on $\gU(\cR_{K,A})$ given by
$$\varphi\star u \colonequals M_0\cdot \varphi(u)\quad \text{and,}\quad  \gamma\star u\colonequals c_0(\gamma)\cdot \gamma(u),$$
for all $u\in \gU(\cR_{K,A})$ and $\gamma \in \Gamma_K$. We also get a structure of a $(\varphi,\Gamma_K)$-module on $V \colonequals \Lie \gU(\cR_{K,A})$ given by 
$$ \varphi \star v\colonequals\ad\big(\sigma(M_0)\big)(\varphi(v))\quad \quad  \text{ and, }\quad  \quad \ \gamma\star v \colonequals \ad\big(\sigma(c_0(\gamma))\big)(\gamma(v)) $$
for all $v\in V,$ and $\gamma\in \Gamma_K$,
where the $\varphi(v)$ and $\gamma(v)$ are given by the action on $\cR_{K,A}$ and the trivial action on $\Lie \gU$. We note that this is the structure of the $\PG$-module $\eta_{D_0}(\Lie\gU)$, seeing $\Lie \gU$ as an $\gM$-representation.
Our convention is that ``$\star$'' binds stronger than ``$\cdot$''.

For our purposes, we define nonabelian $(\varphi,\Gamma_K)$-cohomology, in analogy with \Cref{trivialphigammamodules}.

\begin{definition}\label{def_H1U_new} \phantom{a}
\begin{enumerate}
    \item[(1)] We define the pointed set
$$ H^{0,\star}_{\varphi, \Gamma_K}\big(\gU(\cR_{K,A})\big) \colonequals \left\{ M\in \gU(\mathcal{R}_{K,A}) \ \middle| \ M= \varphi \star M, \ \text{ and } \  \forall \gamma\in\Gamma_K, \ \ M = \gamma \star M \right\}.$$
    
     \item[(2)] Given $r>0$, we let $Z^{1,\star}_{\cont}(\Gamma_K,\gU(\cR_{K,A}^r))$ be the set of continuous $\Gamma_K$-cocycles $c\colon \Gamma_K\to \gU(\cR_{K,A}^r)$ with respect to the twisted action $\star$, that is,
\begin{align}\label{U_cocycle_new}
           c(\gamma_1\gamma_2) &= c(\gamma_1) \cdot \gamma_1 \star c(\gamma_2), \ \forall \gamma_1,\gamma_2\in \Gamma_K .
\end{align}
    We set $$\cZ^{1,\star}\big(\Gamma_K, \gU(\cR_{K,A})\big) := \bigcup_{r \to 0} Z^{1,\star}_{\cont}\big(\Gamma_K,\gU(\cR_{K,A}^r)\big).$$ 
    \item[(3)] We let $Z^{1,\star}_{\varphi, \Gamma_K}(\gU(\cR_{K,A}))$ be the pointed set of pairs $(M, c) \in \gU(\cR_{K,A}) \times \cZ^{1,\star}(\Gamma_K,\gU(\cR_{K,A}))$ satisfying
    \begin{align}
        M \cdot \varphi \star c(\gamma) = c(\gamma) \cdot \gamma \star M. \label{Mc_comp_U_new}
    \end{align}
    for all $\gamma\in \Gamma_K$. There is a natural right action of $\gU(\cR_{K,A})$ on $Z^{1,\star}_{\varphi, \Gamma_K}(\gU(\cR_{K,A}))$ via
    $$ (M,c) \cdot h \colonequals \big(h^{-1} \cdot M \cdot \varphi \star h, ~ [\gamma \mapsto h^{-1} \cdot c(\gamma) \cdot \gamma \star h]\big), $$
    and we define $H^{1,\star}_{\varphi, \Gamma_K}(\gU(\cR_{K,A}))$ as the quotient by this action.

    \item[(4)] If $A\in \Art_L$, we similarly define $H^{0,\star}_{\varphi, \Gamma_K}(\gU(\cR_{K,A}[\tfrac{1}{t}])) \subseteq \gU(\cR_{K,A}[\tfrac{1}{t}])$. We define $$\cZ^{1,\star}\big(\Gamma_K,\gU(\cR_{K,A}[\tfrac{1}{t}])\big) \subseteq Z^{1,\star}\big(\Gamma_K,\gU(\cR_{K,A}[\tfrac{1}{t}])\big)$$ as the subset of cocycles for the $\star$-action that satisfy the equivalent conditions of \Cref{topol_RKAoneovert}.
    We define $Z^{1,\star}_{\varphi, \Gamma_K}(\gU(\cR_{K,A}[\tfrac{1}{t}]))$ and $H^{1,\star}_{\varphi, \Gamma_K}(\gU(\cR_{K,A}[\tfrac{1}{t}]))$ analogously, replacing $\gU(\cR_{K,A})$ by $\gU(\cR_{K,A}[\tfrac{1}{t}])$.
\end{enumerate}
\end{definition}

Since $\gU$ is unipotent, and we are working in characteristic $0$, there is an isomorphism $\exp\colon \Lie(\gU)\xrightarrow{\sim} \gU$ of $L$-schemes which commutes with the adjoint action of $\gP$ (\cite[IV, \S 2, Proposition 4.1]{DG70}).

\begin{lemma}\label{expH0_new} 
    The exponential map induces a bijection of pointed sets
    $$ H^0_{\varphi, \Gamma_K}(V) \eqto H^{0,\star}_{\varphi, \Gamma_K}\big(\gU(\cR_{K,A})\big). $$
The analogous statement holds if $A\in \Art_L$, with $V$ replaced by $V[\tfrac{1}{t}]$ and $\cR_{K,A}$ by $\cR_{K,A}[\tfrac{1}{t}]$.
\end{lemma}

\begin{lemma}\label{H1_Ext_U} We have bijections $H^{1,\star}_{\varphi, \Gamma_K}(\gU(\cR_{K,A})) \cong \Lift^\gP_{\gM}(D_0)$ and $H^{1,\star}_{\varphi, \Gamma_K}(\gU(\cR_{K,A}[\tfrac{1}{t}])) \cong \Lift^\gP_{\gM}(\cM_0)$ determined by a choice of lifting.
\end{lemma}

\begin{proof}
    Any lift of $D_0$ to $\gP$ is $\gP$-trivial. Indeed, it is a $\gP$-torsor over an affine scheme whose pushforward along $\gP\to\gM$ recovers a trivial $\gM$-torsor. Since $H^1_{\et}(X,\gU) = 0$ by affineness, the exact sequence of pointed sets in étale cohomology implies that the $\gP$-torsor itself is trivial (see \cite[Lemma 3.17]{dedar2020}). 
    
    It follows from \Cref{G_PG_Pair_is_triv_PG_mod} that such a lift corresponds to a pair $(M,c) \in \gU(\cR_{K,A}) \times \Map(\Gamma_K, \gU(\cR_{K,A}))$ such that $((M, M_0), (c,c_0))$ is a $\gP$-$\PG$-pair. We first verify the algebraic properties of this pair:
    \begin{align*}
        \big(c(\gamma_1), c_0(\gamma_1)\big) \cdot \big(\gamma_1(c(\gamma_2)), \gamma_1(c_0(\gamma_2))\big) &= \big(c(\gamma_1)\cdot \gamma_1\star c(\gamma_2), c_0(\gamma_1) \gamma_1(c_0(\gamma_2))\big),
    \end{align*}
    and given that $c_0$ is a cocycle, the cocycle condition for $(c,c_0)$ is equivalent to $c$ satisfying relation \eqref{U_cocycle_new}.
    We further have
    \begin{align*}
        (M, M_0) \cdot \big(\varphi(c(\gamma)), \varphi(c_0(\gamma))\big) &= \big(M\cdot\varphi\star c(\gamma), M_0\varphi(c_0(\gamma))\big), \\
        (c(\gamma), c_0(\gamma)) \cdot \big(\gamma(M), \gamma(M_0)\big) &= \big(c(\gamma)\cdot\gamma\star M, c_0(\gamma)\gamma(M_0)\big).
    \end{align*}
    Hence, condition \eqref{Mc_comp_easy} is equivalent to \eqref{Mc_comp_U_new}.
    
    By \eqref{pair_comparison}, there is some $r > 0$ such that $(c, c_0) \in Z^1_{\cont}(\Gamma_K, \gP(\cR_{K,A}^r))$. Therefore, we have that $c \in Z^{1,\star}_{\cont}(\Gamma_K, \gU(\cR_{K,A}^r))$, which implies that $(M, c) \in Z^{1,\star}_{\varphi,\Gamma_K}(\gU(\cR_{K,A}))$.
    Conversely, we already know that $c_0 \in \cZ^1(\Gamma_K, \gM(\cR_{K,A}))$ and if we assume that $c \colon \Gamma_K \to \gU(\cR_{K,A})$ factors over some $\gU(\cR_{K,A}^r)$ and is continuous, then the same is true for $(c, c_0)$.
    This establishes a canonical bijection between $Z^{1,\star}_{\varphi,\Gamma_K}(\gU(\cR_{K,A}))$ and the set of $\gP$-trivialized lifts of $D_0$.

    It is an easy computation to verify that two $\gP$-trivialized lifts are isomorphic by an isomorphism which induces the identity on $\gM$ if and only if the corresponding pairs belong to the same orbit in $Z^{1,\star}_{\varphi,\Gamma_K}(\gU(\cR_{K,A}))$ under the right action of $\gU(\cR_{K,A})$ in \Cref{def_H1U_new}. This completes the proof of the first bijection.

    Let $\cM$ be a lift of $\cM_0$ to $\gP$. As in the previous case, $\cM$ is $\gP$-trivial.
    By \Cref{G_PG_Pair_is_triv_PG_mod} and the same calculation as above, $\cM$ corresponds to a pair $(M,c) \in \gU(\cR_{K,A}[\tfrac{1}{t}]) \times Z^{1,\star}(\Gamma_K, \gU(\cR_{K,A}[\tfrac{1}{t}]))$, such that $((M, M_0), (c,c_0))$ is a $\gP$-$\PG$-pair.
    By \Cref{pair_is_ind_cont} this implies that $(c,c_0) : \Gamma_K \to \gP(\cR_{K,A}[\tfrac{1}{t}])$ satisfies the equivalent conditions of \Cref{topol_RKAoneovert}. By the functoriality part of \Cref{topol_RKAoneovert}, this implies that $c$ satisfies these conditions.
    Therefore, we get that $(M,c) \in Z^{1,\star}_{\varphi,\Gamma_K}(\gU(\cR_{K,A}[\tfrac{1}{t}]))$.

    For the converse, assume that $(M,c) \in Z^{1,\star}_{\varphi,\Gamma_K}(\gU(\cR_{K,A}[\tfrac{1}{t}]))$, and 
    let $(\rho_V,V) \in \Rep_L(\gP)$. We need to show that the pair $(\rho_V(M, M_0), \rho_V(c,c_0))$ induces a $\PG$-module structure on $V\otimes_L \cR_{K,A}[\tfrac{1}{t}]$. Choose a maximal $\gP$-stable flag $0 = V_0 \subseteq \dots \subseteq V_t = V$ in $V$.
    Then $\gU$ acts trivially on the semisimplification $V^{\sss} = \bigoplus_{i=0}^{t-1} V_{i+1}/V_i$. Indeed, since $\gU$ is normal in $\gP$, the subspace $(V_{i+1}/V_i)^{\gU}$ is $\gP$-invariant. Given that $\gU$ is unipotent, this subspace is nonzero. This not being all of $V_{i+1}/V_i$ would allow us to refine the flag. Hence every constituent of $V^{\sss}$ is in fact an $\gM$-representation. In particular, the action of $(\rho_V(M, M_0), \rho_V(c,c_0))$ on each graded piece $V_{i+1}/V_i$ factors through $(\rho_V(M_0),\rho_V(c_0))$, and therefore induces a $\PG$-module structure on $(V_{i+1}/V_i)\otimes_L \cR_{K,A}[\tfrac{1}{t}]$.
    By induction and \Cref{lattice_lemma}, we conclude that $(\rho_V(M, M_0), \rho_V(c,c_0))$ induces a $\PG$-module structure on $V\otimes_L \cR_{K,A}[\tfrac{1}{t}]$.
    
    One checks again that isomorphism classes of lifts of $\cM_0$ are given by orbits under the right action of $\gU(\cR_{K,A}[\tfrac{1}{t}])$.
    Therefore, we obtain the bijection $H^{1,\star}_{\varphi,\Gamma_K}(\gU(\cR_{K,A}[\tfrac{1}{t}])) \cong \Lift^\gP_{\gM}(\cM_0)$ as in the previous case.
\end{proof}

\subsubsection{The case (I) of abelian $\gU$}

\begin{lemma}\label{expH1_new} 
    The exponential map induces isomorphisms
    $$ Z^1_{\varphi, \Gamma_K}(V) \eqto Z^{1,\star}_{\varphi, \Gamma_K}\big(\gU(\cR_{K,A})\big), \quad\quad H^1_{\varphi, \Gamma_K}(V) \eqto H^{1,\star}_{\varphi, \Gamma_K}\big(\gU(\cR_{K,A})\big). $$
    The analogous statement holds if $A\in \Art_L$ with $V$ replaced by $V[\tfrac{1}{t}]$ and $\cR_{K,A}$ by $\cR_{K,A}[\tfrac{1}{t}]$.
\end{lemma}

\begin{proof}
    Since $\gU$ is abelian,   $\exp$ is an isomorphism of group schemes. Therefore, it induces a bijection $Z^1(\Gamma_K, V) \eqto Z^{1,\star}(\Gamma_K, \gU(\cR_{K,A}))$.
    A pair $(v,c) \in V \times Z^1(\Gamma_K, V)$ satisfies $\gamma(v) - v = \varphi(c(\gamma)) - c(\gamma)$ for all $\gamma \in \Gamma_K$ if and only if $(\exp(v), \exp(c))$ satisfies \eqref{Mc_comp_U_new}.
    Moreover, $c \in \cZ^1(\Gamma_K, V)$ if and only if $\exp(c) \in \cZ^{1,\star}(\Gamma_K,\gU(\cR_{K,A}))$. This gives the two isomorphisms.
    
    For the case where $t$ is inverted the argument is identical. We apply \Cref{topol_RKAoneovert} with the closed immersion $\gU \cong \bbA^n$ to conclude that $c \in \cZ^1(\Gamma_K, V[\tfrac{1}{t}])$ if and only if $\exp(c) \in \cZ^{1,\star}(\Gamma_K,\gU(\cR_{K,A}[\tfrac{1}{t}]))$.
\end{proof}

\begin{proposition}\label{existenceoflift}
    Assume that $H^{2}_{\varphi, \Gamma_K}(V)=0$. Then there exists a lift of $D_0$ to $\gP$. In particular, there is a lift of $\cM_0$ to $\gP$. 
\end{proposition}

\begin{proof}
    $D_0$ is obtained by base change from some $\gM$-$\PG$-module $D_0^r$ over $\cR_{K,A}^r$ for $r$ small enough.
    Since $H^{2}_{\varphi, \Gamma_K}(V) = \varinjlim_r H^{2}_{\varphi, \Gamma_K}(V^r)$ the obstruction class to lifting $D_0^r$ to $\gP$ defined in \Cref{is_two_cocycle_PG} vanishes for suitable $r$ and by \Cref{final_obs_lemma} there exists a lift $D^r$ of $D_0^r$.
    So $D = D^r \otimes_{\cR_{K,A}^r} \cR_{K,A}$ is a lift of $D_0$.
\end{proof}



\subsubsection{The case (II) of nonabelian $\gU$}
Let 
\begin{align}
    \gU = \gU^0 \supseteq \gU^1 \supseteq \gU^2 \supseteq \cdots \label{central_series_of_P_new}
\end{align}
be the central series of $\gU$, so that $\gU^{n+1}=[\gU,\gU^n]$. 
We set $\gU_n \colonequals \gU/\gU^{n+1}$.
Clearly the conjugation action preserves the central series, so we also have homomorphisms $\gM \to \uAut(\gU^n)$, $\gM \to \uAut(\gU_n)$ and $\gM \to \uAut(\gU^n/\gU^{n+1})$ and $\gU^n$ is normal in $\gP$.

We have a central extension 
\begin{align}
    1 \to \gU^{n+1}/\gU^{n+2} \to \gU_{n+1}\to \gU_n \to 1, \label{U_extension_new}
\end{align} 
which realizes $\gU_{n+1}$ as a trivial $\gU^{n+1}/\gU^{n+2}$-torsor over $\gU_n$.
The short exact sequence \eqref{U_extension_new} is $\gM$-equivariant.
Passing to Lie algebras, we obtain a short exact sequence
\begin{align}
    0 \to \Lie \gU^{n+1} / \Lie \gU^{n+2} \to \Lie \gU_{n+1} \to \Lie \gU_n \to 0 \label{split_Lie_seq_new}
\end{align}
of $\gM$-representations. 
Since $\gM$ is reductive, \eqref{split_Lie_seq_new} is split as a sequence of $\gM$-representations, so we can exponentiate a section to a scheme-theoretic and $\gM$-equivariant section of \eqref{U_extension_new}. We have a structure of $\PG$-modules on $V_n\colonequals \Lie\gU_n(\cR_{K,A})$ and $V^n\colonequals \Lie\gU^n(\cR_{K,A})$ given by the $\star$-action.

\begin{proposition}\label{better_unipotent_H1}
    Assume that $V^n/V^{n+1}$ is cohomologically regular for all $n \geq 0$.
    Then the functors
    \begin{alignat*}{4}
        F_\gU\colon \Aff_A &\longrightarrow \Set \quad \quad \quad \quad \quad \quad  F_V\colon \Aff_A &&\longrightarrow \Set
            \\ A' &\mapsto H^{1,\star}_{\varphi,\Gamma_K}(\gU(\cR_{K,A'})) \quad \quad \quad \quad  A' &&\mapsto H^1_{\varphi, \Gamma_K}(V \widehat{\otimes}_A A')
    \end{alignat*}
    are isomorphic and representable by a quasi-Stein rigid analytic space, which is a vector bundle (in particular smooth) of relative dimension $\dim \gU \cdot [K:\Qp]$ over $\Sp(A)$.
\end{proposition}

\begin{proof}
    The proof is inspired by that of \cite[Proposition 2]{Kim}.
    
    By \Cref{coh_reg_ses}, cohomological regularity for all $V^n/V^{n+1}$ implies cohomological regularity for all $V_n$ and all $V^n$. By cohomological regularity of $V$, $H^1_{\varphi, \Gamma_K}(V)$ is a projective $A$-module of rank $\dim \gU \cdot [K:\Qp]$ and the natural map $H^1_{\varphi, \Gamma_K}(V) \otimes_A A' \to H^1_{\varphi, \Gamma_K}(V \wtimes_A A')$ is an isomorphism. Therefore, the functor $F_V$ is representable by a vector bundle of the claimed dimension.

    We now construct the isomorphism between $F_{\gU}$ and $F_V$. We proceed by induction over the central series \eqref{central_series_of_P_new} of $\gU$.
    In particular, we will use cohomological regularity of all $V_n$, although in the following text it appears to only be used for $V_0$. For $n=0$, we have that $\gU = \gU_0$ is abelian, so by \Cref{H1_comparison} and \Cref{expH1_new}, we get an isomorphism $H^1_{\varphi, \Gamma_K}(V_0 \wtimes_A A') \cong H^{1,\star}_{\varphi, \Gamma_K}(\gU_0(\cR_{K,A'}))$, which is functorial in $A'$.
    We assume now that the claim holds when $\gU=\gU_n$, and we show that the same holds for $\gU = \gU_{n+1}$. In particular, $\gU^{n+1}$ is abelian and $\gU^{n+2} = 1$.

    Recall from \eqref{U_extension_new} that we have a central extension $\gU_{n+1} \to \gU_n$, which admits a scheme-theoretic section $e \colon \gU_n \to \gU_{n+1}$.
    Given that $\gU^{n+1}$ is in the center of $\gU_{n+1}$, the group $H^{1,\star}_{\varphi, \Gamma_K}(\gU^{n+1}(\cR_{K,A'}))$ acts on the set $H^{1,\star}_{\varphi, \Gamma_K}(\gU_{n+1}(\cR_{K,A'}))$ by multiplication. Let us prove that this action is free. 
    
    Suppose that $(M',c')\in H^{1,\star}_{\varphi, \Gamma_K}(\gU^{n+1}(\cR_{K,A'}))$, $(M,c)\in H^{1,\star}_{\varphi, \Gamma_K}(\gU_{n+1}(\cR_{K,A'}))$ are such that $(M'M,c'c)$ and $(M,c)$ are equivalent. Then there exists $u\in \gU_{n+1}(\mathcal{R}_{K,A'})$ such that
    $$ M' \cdot M = u^{-1} \cdot M \cdot \varphi \star u \quad \text{ and } \quad \ c'(\gamma) \cdot c(\gamma)= u^{-1} \cdot c(\gamma) \cdot \gamma \star u, \  \forall \gamma\in \Gamma_K. $$
    Projecting this onto $\gU_{n}(\mathcal{R}_{K,A'})$, the first equation gives
    $$ \overline{u} \cdot (\varphi \star \overline{u})^{-1} = [\overline{M}, \varphi \star \overline u], $$
    where $[\cdot,\cdot]$ denotes the commutator, and we obtain a similar equality 
    $$ \overline u \cdot (\gamma \star \overline u)^{-1} = [\overline{c(\gamma)}, \gamma \star \overline u] $$
    for the second equation.
    After projecting onto $\gU_0$, the right hand side vanishes by definition of $\gU_0$. Therefore, the image $\overline{\overline{u}}$ of $u$ in $\gU_0$ lies in $ H^{0,\star}_{\varphi, \Gamma_K}(\gU_0(\cR_{K,A'})) = H^0_{\varphi, \Gamma_K}(V_0 \otimes_A A')$, where the equality holds by \Cref{expH0_new}.
    This group vanishes as $V_0$ is cohomologically regular.
    Hence $\overline{u}\in (\gU^1/\gU^{n+1})(\mathcal{R}_{K,A'})$.
    Arguing by induction and using the vanishing $ H^{0,\star}_{\varphi, \Gamma_K}((\gU^i/\gU^{i+1})(\cR_{K,A'}))=H^0_{\varphi, \Gamma_K}((V^i/V^{i+1}) \wtimes_A A')=0$, we find $\overline{u}=1$, which means that $u\in \gU^{n+1}(\mathcal{R}_{K,A'})$. It follows that the class $(M',c')$ is trivial in $H^{1,\star}_{\varphi, \Gamma_K}(\gU^{n+1}(\cR_{K,A'}))$, and therefore the action is free. 
    We have a commutative diagram
\begin{center}    \begin{tikzcd}
            H^{1,\star}_{\varphi, \Gamma_K}\big(\gU_{n+1}(\cR_{K,A'})\big) \arrow[r]\arrow[d,"\cong"] &H^{1,\star}_{\varphi, \Gamma_K}\big(\gU_n(\cR_{K,A'})\big) \arrow[d,"\cong"] \\
            \Lift_{\gM}^{\gP}(D_0\widehat{\otimes}A') \arrow[r] & \Lift_{\gM}^{\gP/\gU^{n+1}}(D_0\widehat{\otimes}A'),\end{tikzcd}
    \end{center}
    where the vertical maps are isomorphisms by \Cref{H1_Ext_U}. For any $(M',c')\in \Lift_{\gM}^{\gP/\gU^{n+1}}(D_0\widehat{\otimes}A')$, the induced $\star$ action on $\gU^{n+1}$ coincides with that induced by $(M,c)$. Since $H^2_{\varphi, \Gamma_K}(V^{n+1}) = 0$ by cohomological regularity, we may apply \Cref{existenceoflift} to the extension $\gU^{n+1}\to \gP\to \gP/\gU^{n+1}$ and to any element of $\Lift_{\gM}^{\gP/\gU^{n+1}}(D_0)$. This yields the surjectivity of the map $$ H^{1,\star}_{\varphi, \Gamma_K}(\gU_{n+1}(\cR_{K,A'})) \longrightarrow H^{1,\star}_{\varphi, \Gamma_K}(\gU_n(\cR_{K,A'})).$$

    By \Cref{expH1_new}, the rigid analytic group $\cG(A') \colonequals H^{1,\star}_{\varphi, \Gamma_K}(\gU^{n+1}(\cR_{K,A'}))$ is locally isomorphic to $\mathbb G_{\mathrm a, A}^i$ for some $i\ge 1$, so torsors under it are $\mathbb G_{\mathrm a,A}^i$-torsors. The freeness of the action shows that $F_{\gU_{n+1}}$ is a $\mathcal{G}$-pseudotorsor over $F_{\gU_n}$.  For an affinoid $\Sp(A') \subseteq F_{\gU_n}$,  the surjectivity just established provides a lift of the corresponding element of $F_{\gU_n}(A')$ to $F_{\gU_{n+1}}(A')$, which defines a local section $\Sp(A') \to F_{\gU_{n+1}}$.  Hence $F_{\gU_{n+1}}$ is a torsor for the G-topology over $F_{\gU_n}$.
    Since $\Ga$-torsors over quasi-Stein spaces are split by \stackcite{03AJ} and Theorem B \cite[Theorem 2.4.2]{KiehlAB}, $F_{\gU_{n+1}}$ is a trivial $\mathbb{G}_{a,A}^i$-torsor over $F_{\gU_n}$.
    So we get an isomorphism
    $$ H^{1,\star}_{\varphi, \Gamma_K}\big(\gU^{n+1}(\cR_{K,A'})\big) \times H^{1,\star}_{\varphi, \Gamma_K}(\gU_{n}(\cR_{K,A'})) \xrightarrow{\sim} H^{1,\star}_{\varphi, \Gamma_K}\big(\gU_{n+1}(\cR_{K,A'})\big), $$
    that is functorial in $A'$.
    
    The vanishing of $H^2_{\varphi, \Gamma_K}(V^{n+1})$ also implies that $H^1_{\varphi, \Gamma_K}(V_{n+1}) \to H^1_{\varphi, \Gamma_K}(V_n)$ is surjective.
    Thus we also find that $F_{V_{n+1}}$ is a trivial $\cG$-torsor over $F_{V_n}$, and we get an isomorphism
    $$ H^1_{\varphi,\Gamma_K}
\big(V^{n+1}\widehat{\otimes}A'\big) \times H^1_{\varphi,\Gamma_K}(V_n\widehat{\otimes}A') \xrightarrow{\sim} H^1_{\varphi,\Gamma_K}(V_{n+1}\widehat{\otimes}A'), $$
    that is functorial in $A'$.
    By the induction hypothesis, we have isomorphisms $F_{\gU_n} \cong F_{V_n}$ and $F_{\gU^{n+1}} \cong F_{V^{n+1}}$, so we get a (non-canonical) isomorphism $F_{\gU_{n+1}} \cong F_{V_{n+1}}$.
\end{proof}

\subsection{Lifting $\PG$-modules with prescribed $B_{\dR}$-representation}
\label{sec_lift_PG_with_BdR_new}
In this section, we place ourselves in situation (I), where $\gU$ is abelian. 

Let $A \in \Art_L$.
Let $\eta_{W_0}$ be a trivializable almost de Rham $\gM$-$A \otimes_{\Qp} \BdR$-representation. 
This means that there exists an isomorphism of fiber functors $\alpha_A \colon \phi_{A \otimes_{\Qp} K}\circ \omega_{\gM}\eqto \omega_{\mathbb{G}_{a,A\otimes K}}\circ D_{\pdR,A}\circ \eta_{W_0}$. Seeing $\Lie(\gU)$ as a representation of $\gM$, we may consider the $\G_{\mathrm a, A \otimes_{\Qp} K}$-representation $\mathbb{V} \colonequals D_{\pdR,A}(\eta_{W_0}(\Lie(\gU)))$.

As for $\PG$-modules, we make the following definition.
\begin{definition}
    A \emph{lift} of $\eta_{W_0}$ to $\gP$ is a pair $(\eta_W,\iota)$, where $\eta_W$ is a $\gP$-$A \otimes_{\Qp} \BdR$-representation and $\iota\colon \eta_W \circ \Res_{\gP}^{\gM} \xrightarrow{\sim} \eta_{W_0}$ is an isomorphism of $\gM$-$A \otimes_{\Qp} \BdR$-representations.
    An \emph{equivalence} between two lifts $(\eta_{W_1},\iota_1)$ and $(\eta_{W_2},\iota_2)$ is an isomorphism $f \colon \eta_{W_1} \to \eta_{W_2}$ of $\gP$-$A \otimes_{\Qp} \BdR$-representations such that $\iota_2 \circ p_*f=\iota_1$. We write $\Lift_{\gM}^{\gP}(\eta_{W_0})$ for the set of equivalence classes of lifts of $\eta_{W_0}$ to $\gP$.
\end{definition}
The assumption that $\eta_{W_0}$ is trivializable ensures that every lift is trivializable as well: the fiber functor $\eta_{W_0}$ corresponds to an $\gM$-torsor over an affine scheme $X$, and any lift corresponds to $\gP$-torsor whose pushforward along $\gP\to\gM$ recovers the $\gM$-torsor. Given that $H^1_{\et}(X,\gU) = 0$ by affineness, the exact sequence of pointed sets in étale cohomology implies that the $\gP$-torsor is trivial.

\begin{lemma}\label{H1_Ext_U_Ga_case} There exists a lift of $\eta_{W_0}$ to $\gP$ and the choice of such a lift determines a bijection $H^1(\G_{\mathrm a, A \otimes_{\Qp} K}, \mathbb{V}) \cong \Lift_{\gM}^{\gP}(\eta_{W_0})$.
\end{lemma}

\begin{proof}
   Fixing a trivialization $\alpha_A$ of $\eta_{W_0}$ corresponds, by \Cref{pdR_charanew}, to a homomorphism $h_0\colon \G_{\mathrm a, A \otimes_{\Qp} K} \to \gM_{A \otimes_{\Qp} K}$. It also fixes an isomorphism $\mathbb{V}\xrightarrow{\sim} \Lie(\gU)(A\otimes_{\Qp}K)$ induced by $\alpha_A^{-1}$, so we can see the latter space as a $\mathbb{G}_{a,A\otimes_{\Qp}K}$ representation. 
    Moreover, any lift of $\eta_{W_0}$ is trivializable, and a choice of trivialization yields a homomorphism $ h : \G_{\mathrm a, A \otimes_{\Qp} K} \to \gP_{A \otimes_{\Qp} K}$.

    By a standard computation in group theory, the set of lifts $h$ of $h_0$ can be identified with the set of algebraic cocycles $Z^1(\G_{\mathrm a, A \otimes_{\Qp} K}, \Lie \gU( A \otimes_{\Qp} K))$ once we fix a lift $\tilde h$. The existence of a lift is guaranteed, as $H^2(\G_{\mathrm a, A \otimes_{\Qp} K}, \mathbb{V}) = 0$ in characteristic $0$. Identifying such lifts with conjugation by elements of $\gU(A \otimes_{\Qp} K)$, the set of equivalence classes of lifts is identified with $H^1(\G_{\mathrm a, A \otimes_{\Qp} K}, \Lie \gU( A \otimes_{\Qp} K))$. 
    
    We are left to see that the map 
    $$ Z^1\big(\G_{\mathrm a, A \otimes_{\Qp} K}, \Lie \gU( A \otimes_{\Qp} K)\big) \cong \Hom_{A \otimes_{\Qp} K}(\Ga,\gP) \times_{\Hom_{A \otimes_{\Qp} K}(\Ga,\gM)} \{h_0\} \to \Lift_{\gM}^{\gP}( \eta_{W_0}) $$
    induces a bijection $H^1(\G_{\mathrm a, A \otimes_{\Qp} K}, \Lie \gU( A \otimes_{\Qp} K)) \eqto \Lift_{\gM}^{\gP}( \eta_{W_0})$.
    One notes that the precomposition with $H^1(\G_{\mathrm a, A \otimes_{\Qp} K},\mathbb{V})\xrightarrow{\sim} H^1(\G_{\mathrm a, A \otimes_{\Qp} K}, \Lie \gU( A \otimes_{\Qp} K))$ is independent of the choice of trivialization.

    For surjectivity, suppose that $\eta_W$ is a lift of $\eta_{W_0}$.
    Any trivialization of $\eta_W$ induces a trivialization of $\eta_{W_0}$.
    It differs from $\alpha_A$ by a torsor automorphism of $\gM_{A \otimes_{\Qp} K}$, i.e.,  an element of $\gM(A \otimes_{\Qp} K)$.
    Since the map $\gP(A \otimes_{\Qp} K) \to \gM(A \otimes_{\Qp} K)$ is surjective, we can arrange that the trivialization of $\eta_W$ recovers $\alpha_A$. This means that $\eta_W$ corresponds to a lift $h$ of $h_0$ as above.

    For injectivity, suppose $h_1$ and $h_2$ induce lifts $\eta_{W_1}$ and $\eta_{W_2}$ of $\eta_{W_0}$ which are equivalent.
    An equivalence $\eta_{W_1} \eqto \eta_{W_2}$ inducing the identity over $\eta_{W_0}$ is necessarily induced by an element of $\gU(A \otimes_{\Qp} K)$. Hence $h_1$ and $h_2$ define the same class in $H^1$.
\end{proof}

We let $\mathcal{M}_0$ be an $\gM$-trivial $\gM$-$\PG$-module  over $\cR_{K,L}[\tfrac{1}{t}]$, to which we associated 
a $\PG$-module $V[\tfrac{1}{t}]$. Let $\cM$ be a lift of $\cM_0$ to $\gP$. Then it is a $\gP$-trivial $\gP$-$(\varphi, \Gamma_K)$-module.  We set $\eta_{W_0}\colonequals W_{\dR}\circ\eta_{\mathcal{M}_0}$ and $\eta_{W}\colonequals W_{\dR}\circ\eta_{\mathcal{M}}$. 

\begin{theorem}\label{prescribed_lifting}
Assume that 
\begin{enumerate}
    \item[(1)] $V[\tfrac{1}{t}]$ is cohomologically regular,
    \item[(2)] the map $H^1_{\varphi,\Gamma_K}(V[\tfrac{1}{t}]) \to H^1(\GK, W_{\dR}(V[\tfrac{1}{t}]))$ is surjective,
    \item[(3)] $\eta
    _W$ is almost de Rham.
\end{enumerate}
 Then the morphism of pseudofunctors
    \begin{align}
        X_{\cM} \longrightarrow X_{\cM_0} \times_{X_{W_0}} X_{W} \label{general_formally_smooth}
    \end{align}
    is formally smooth.
\end{theorem}
\begin{proof}
    Let $f\colon A_2\twoheadrightarrow A_1$ be a surjection in $\Art_L$. Let $(\cM_{A_1},j_{A_1})\in X_{\cM}(A_1)$. For simplicity, in this proof we will suppress $j_{A_1}$ from the notation. Let $\cM_{0,A_1}\in X_{\cM_0}(A_1)$, $\eta_{W_{A_1}}\in X_{W}(A_1)$, and $\eta_{W_{0,A_1}}\in X_{W_0}(A_1)$ be the images of $\cM_{A_1}$ under \eqref{general_formally_smooth}. Let $(\cM_{0,A_2},\eta_{W_{A_2}})\in (X_{\cM_0} \times_{X_{W_0}} X_{W})(A_2)$ be a lift of $(\cM_{0,A_1},\eta_{W_{A_1}})$ along $f$. We will also suppress the equivalence $W_{\dR}\circ \eta_{\cM_{0,A_2}}\cong \eta_{W_{A_2}}\circ \Res_{\gP}^{\gM}$ lifting the corresponding equivalence on $A_1$.

    We need to find $\cM_{A_2}\in X_{\cM}(A_2)$ lifting $\cM_{A_1}$ along $f$ and $(\cM_{0,A_2},\eta_{W_{A_2}})$ along \eqref{general_formally_smooth} making the evident cube commute. Let $\mathbb{V}_{A_1}$ and $\mathbb{V}_{A_2}$ be the $\mathbb{G}_a$-representations obtained by applying $\eta_{W_{0,A_1}}$ and $\eta_{W_0,A_2}$ to $\Lie(\gU)$. We need to show that the top map of the following diagram is surjective:
    \begin{equation*}
    \begin{tikzcd}
    \Lift^\gP_{\gM}(\cM_{0,A_2}) \arrow[rrr, "(W_{\dR}\circ -) \times (\phi_{A_1}\circ -)"] \ar[d] &&& \Lift^\gP_{\gM}(\eta_{W_{0,A_2}}) \times_{\Lift^\gP_{\gM}(\eta_{W_{0,A_1}}) } \Lift^\gP_{\gM}(\cM_{0,A_1}) \ar[d] \\
    H^{1,\star}_{\varphi,\Gamma_K}(\gU(\cR_{K,A_2}[\tfrac{1}{t}])) \arrow[rrr] &&& H^1(\mathbb{G}_{a,A_2\otimes K}, \mathbb{V}_{A_2})\times_{H^1(\mathbb{G}_{a,A_1\otimes K}, \mathbb{V}_{A_1})} H^{1,\star}_{\varphi,\Gamma_K}(\gU(\cR_{K,A_1}[\tfrac{1}{t}])).
    \end{tikzcd} 
    \end{equation*}
    The left vertical map is the bijection of \Cref{H1_Ext_U}.
    The same applies to the right coordinate of the right vertical map.
    The left coordinate of the right vertical map is the bijection of \Cref{H1_Ext_U_Ga_case}, hence the right vertical map is bijective.
    Since the bijection constructed in \Cref{H1_Ext_U} is functorial, the diagram is commutative. 

    The situation is now very close to that in the proof of \cite[Theorem 3.4.4]{BHS19}. Let $V_{A_1}[\tfrac{1}{t}]$ and $V_{A_2}[\tfrac{1}{t}]$ be the $\PG$-modules associated to $\cM_{0,A_1}$ and $\cM_{0,A_2}$. 
    By \Cref{expH1_new} and \Cref{H1_comparison}, we have $H^{1,\star}_{\varphi,\Gamma_K}(\gU(\cR_{K,A_i}[\tfrac{1}{t}])) \cong H^1_{\varphi,\Gamma_K}(V_i[\tfrac{1}{t}])$, and by cohomological regularity of $V_i[\tfrac{1}{t}]$, we have an isomorphism
    \begin{align*}
       H^1_{\varphi,\Gamma_K}(V_{A_2}[\tfrac{1}{t}]) \otimes_{A_2} A_1 \cong H^1_{\varphi,\Gamma_K}(V_{A_1}[\tfrac{1}{t}]).
    \end{align*}    
    We have that $H^1(\mathbb{G}_{a,A_2\otimes K}, \mathbb{V}_{A_2}) \otimes_{A_2} A_1 \cong H^1(\mathbb{G}_{a,A_1\otimes K}, \mathbb{V}_{A_1})$, since the first $\Ga$-cohomology group is the cokernel of a nilpotent endomorphism, which is compatible with base extension. 

     The map $ H^1_{\varphi,\Gamma_K}(V_{A_2}[\tfrac{1}{t}]) \to H^1(\mathbb{G}_{a,A_2\otimes K}, \mathbb{V}_{A_2})$ is surjective since it is surjective after the base change $-\otimes_{A_2} L$ by our assumption. Therefore, we can apply \cite[Lemma 3.4.5]{BHS19} with $M = H^{1,\star}_{\varphi,\Gamma_K}(\gU(\cR_{K,A_2}[\tfrac{1}{t}]))$ and $N = H^1(\mathbb{G}_{a,A_2\otimes K}, \mathbb{V}_{A_2})$ to deduce that the horizontal maps in the diagram are surjective. This completes the proof of the theorem.
\end{proof}

\subsection{Application to $(\gG, \gB, \gT)$}
\label{application_to_GBT}

We apply the work of the previous sections to a generalization of the main results of \cite[§3.4]{BHS19}.
In the above notation, we choose $\gP = \gB$ and $\gM = \gT$.





\begin{lemma}\label{Breg_coh_reg}
    Let $A \in \Aff_L$. If $D_0$ is a $\gB$-regular $\gT$-$\PG$-module over $\cR_{K,A}$, then $V^{n+1}/V^{n+2}$ is cohomologically regular for all $n \geq 0$. The same statement holds after inverting $t$.
\end{lemma}

\begin{proof}
    As a $\gT$-representation, $\Lie \gU^{n+1}/\gU^{n+2}$ is a direct sum of characters $\alpha \colon \gT \to \Gm$ in $\Phi^+(\gB,\gT)$.
    Hence $V^{n+1}/V^{n+2}$ is a direct sum of $\cR_{K,A}(\delta_{\alpha})$ for $\alpha \in \Phi^+(\gB,\gT)$. By our assumption on $\gB$-regularity, $\delta_{\alpha}$ is a regular parameter. Therefore, the statement follows from \Cref{coh_char_A} and \Cref{coh_reg_param}.
\end{proof}

\begin{proposition}\label{unipotentH1}
Let $A \in \Aff_L$, and let $D_0$ be a $\gB$-regular $\gT$-$\PG$-module over $\cR_{K,A}$.
Then the functors
\begin{alignat*}{4}
    F_U\colon \Aff_A &\longrightarrow \Set \quad \quad \quad \quad \quad \quad  &F_V\colon \Aff_A &\longrightarrow \Set
        \\ A' &\mapsto H^{1,\star}_{\varphi,\Gamma_K}(\gU(\cR_{K,A'})) \quad \quad \quad \quad \quad \quad &A' &\mapsto H^1_{\varphi, \Gamma_K}(V) \otimes_A A'
    \end{alignat*}
are isomorphic and representable by a quasi-Stein rigid analytic space, which is a vector bundle, in particular smooth, of relative dimension $\dim \gU \cdot [K : \Qp]$ over $\Sp(A)$.
\end{proposition}

\begin{proof}
    We want to apply \Cref{better_unipotent_H1}, so we need to verify the hypotheses of \Cref{coh_char_A} for $V^{n+1}/V^{n+2}$.
    This follows from \Cref{Breg_coh_reg}.
\end{proof}

\begin{theorem}\label{thm_formal_smoothness}
    Let $\Mtri$ be a $\gB$-trivial $\gB$-$(\varphi, \Gamma_K)$-module over $\cR_{K,L}[\tfrac{1}{t}]$.
    Assume that the parameter $\delta$ of $\Mtri$ is locally algebraic and lies in $\cT^{\gT}_{0,\gB}(L)$.
    Let $\eta_{\Wtri} \colonequals W_{\dR}\circ \eta_{\Mtri}$.
    Then the morphism of pseudofunctors
    \begin{equation}
        X_{\Mtri} \longrightarrow  \widehat \cT^\gT_{\delta} \times_{\widehat \frakt} X_{\Wtri} \label{formally_smooth_morphism}
    \end{equation}
    is formally smooth.
\end{theorem}

\begin{proof}
    For $n\ge 0$, let $\cM_n\colonequals \Mtri\times^{\gB}\gB/\gU^n$, and let $\eta_{W_n}\colonequals W_{\dR}\circ \eta_{\cM_n}$. 
    Suppose that the map $X_{\Mtri}\to X_{\cM_{n+1}}\times_{X_{W_{n+1}}}X_{\Wtri}$ is formally smooth. Then assuming that $X_{\cM_{n+1}}\to X_{\cM_n}\times_{X_{W_n}}X_{W_{n+1}}$, it follows by base change and composition, that $X_{\Mtri}\to X_{\cM_n}\times_{X_{W_n}}X_{\Wtri}$. The conclusion of the theorem then follows by noting that $\widehat\cT^\gT_{\delta} \cong X_{\cM_0}$ and $\widehat{\frakt}\cong X_{W_0}$.

    To get the formal smoothness of the maps $X_{\cM_{n+1}}\to X_{\cM_n}\times_{X_{W_n}}X_{W_{n+1}}$, we apply \Cref{prescribed_lifting}. It remains to verify the hypothesis.
    Since $\delta$ is locally algebraic, \cite[Lemma 3.4.3]{BHS19} implies that the map $H^1_{\varphi,\gamma_K}(V^n/V^{n+1}[\tfrac{1}{t}]) \to H^1(\GK, W_{\dR}(V^n/V^{n+1}[\tfrac{1}{t}]))$ is surjective.
    Since $\delta \in \cT^{\gT}_{0,\gB}(L)$, cohomological regularity of $V^n/V^{n+1}[\tfrac{1}{t}]$ follows from \Cref{Breg_coh_reg}.
\end{proof}

\subsection{The space $\cS_\gB$ of $\gB$-regular rigidified $\gB$-$\PG$-modules}
\label{sec_SB}

\begin{definition}\label{def_reg_rig_PG}
Let $A \in \Aff_L$.
Let $\Dtri$ be a $\gB$-$\PG$-module over $\cR_{K,A}$.
\begin{enum}
    \item A \emph{rigidification} of $\Dtri$ is an isomorphism $\theta\colon\cR_{K,A}(\delta)\cong\Dtri\times^\gB\gT$.
    \item A \emph{$\gB$-regular rigidified} $\gB$-$\PG$-module over $\cR_{K,A}$ is a pair $(\Dtri,\theta)$ of a $\gB$-regular $\gB$-$\PG$-module $\Dtri$ over $\cR_{K,A}$ and a rigidification $\theta$ of $\Dtri$.\footnote{This notion is weaker than that of \cite[Definition 5.6]{dedar2020}, as our notion of $\gB$-regularity is weaker than the notion of regularity of \emph{loc. cit.}}
    \item An \emph{equivalence} between two $\gB$-regular rigidified $\gB$-$\PG$-modules $(\Dtri_1,\theta_1)$ and $(\Dtri_2,\theta_2)$ is an isomorphism of $\gB$-$\PG$-modules $f\colon\Dtri_1\cong\Dtri_2$ such that the isomorphism $f\times^\gB\gT\colon\Dtri_1\times^\gB\gT\cong\Dtri_2\times^\gB\gT$ satisfies $\theta_1=\theta_2\circ f\times^\gB\gT$.
\end{enum}
\end{definition}

\begin{definition}\label{def:SB}
    We define $\cS_\gB \colon \Aff_L \to \Gpd$ as the pseudofunctor which sends an affinoid $L$-algebra $A$ to the groupoid of $\gB$-regular rigidified $\gB$-$\PG$-modules over $\cR_{K,A}$.
\end{definition}

\begin{lemma}\label{SBtrivialAut}
    Let $A\in \Aff_L$. The groupoid $\cS_\gB(A)$ has trivial automorphisms.
    In particular, $\cS_\gB(A)$ is the set of equivalence classes of $\gB$-regular rigidified $\gB$-$\PG$-modules over $\cR_{K,A}$.
\end{lemma}

\begin{proof}
    Let $(D^{\triangle}, \theta) \in \cS_\gB(A)$.
    By the proof of \Cref{H1_Ext_U}, $\Dtri$ is $\gB$-trivial. 
    Let $(MM_0,cc_0)$ be the $\gB$-$\PG$-pair corresponding to $\Dtri$. Then by \Cref{G_PG_Pair_is_triv_PG_mod}, an automorphism $f$ of $\Dtri$ is given by an element $b\in \gB(\cR_{K,A})$ satisfying $$ b^{-1}MM_0\varphi(b)=MM_0,\quad \text { and, }\quad b^{-1}c(\gamma)c_0(\gamma)\gamma(b)=c(\gamma)c_0(\gamma),\  \forall\gamma\in \Gamma_K.$$
    These equations may be rewritten as
    $$ b^{-1}\cdot \varphi\star b=[b^{-1},M^{-1}], \quad \text{  and, } \quad   b^{-1}\cdot \gamma\star b= [b^{-1},c_0(\gamma)^{-1}].$$
    The condition $\theta\circ \pi_*f=\theta$ amounts to $\pi(b)=1$, that is $b\in \gU(\cR_{K,A})$.  
    Projecting the above equalities onto $\gU_0=\gU/[\gU,\gU]$, we find that $\overline{b}\in H^{0,\star}_{\varphi,\Gamma_K}(\gU_0(\cR_{K,A}))=\{1\}$ by $\gB$-regularity. Hence, $b\in \gU^1(\cR_{K,A})$, and for example $[b^{-1},M^{-1}]\in \gU^2(\cR_{K,A})$. We can then project again the equations onto $\gU_1$ and use $H^{0,\star}_{\varphi,\Gamma_K}(\gU_1(\cR_{K,A}))=\{1\}$ to deduce that $b\in \gU^2(\cR_{K,A})$. Arguing inductively, we find that $b\in \gU^n(\cR_{K,A})$  for all $n\ge 0$, and therefore $b=1$. 
\end{proof}

We obtain the following generalization of \cite[Theorem 5.8]{dedar2020}.\footnote{Our space $\cS_\gB$ contains the space of \cite[Definition 5.7]{dedar2020} as an open subspace. In particular, we recover his Theorem 5.8. Our result is more general, as we remove the hypothesis (T) of loc. cit.. We could also allow $\gB$ to be a solvable group with $\gB/\Ru(\gB) \cong \gT$ such that the roots of $\gT$ are non-trivial, compare \Cref{rmk_B_solvable}.}
\begin{theorem}\label{SB_representable}
    $\cS_\gB$ is representable by a rigid analytic space over $L$, which is a vector bundle, in particular smooth, of relative dimension $\dim \gU \cdot [K:\Qp]$ over $\cT_{0,\gB}^{\gT}$.
\end{theorem}

\begin{proof}
    There is a natural map $\cS_\gB \to \cT_{0,\gB}^{\gT}$ sending a $\gB$-regular rigidified $\gB$-$\PG$-module to its parameter.
    Let $\Sp(A) \subseteq \cT_{0,\gB}^{\gT}$ be an affinoid subdomain.
    Then by \Cref{SBtrivialAut}, $\cS_\gB \times_{\cT_{0,\gB}^{\gT}} \Sp(A) \to \Sp(A)$ represents the functor 
    $$\Aff_A \longrightarrow \Set, ~ A' \mapsto \Lift_{\gT}^{\gB}(D_{A'}),$$ where $D_{A'}$ is the universal $\gT$-trivial $\gT$-$\PG$-module over $\cT_{0,\gB}^{\gT}$ specialized to $A'$. By \Cref{H1_Ext_U}, we have $\Lift_{\gT}^{\gB}(D_{A'}) \cong H^{1,\star}_{\varphi,\Gamma_K}(\gU(\cR_{K,A'}))$. Hence by \Cref{unipotentH1}, the functor $\cS_\gB \times_{\cT_{0,\gB}^{\gT}} \Sp(A)$ is representable by a rigid analytic space, which is smooth of relative dimension $\dim \gU \cdot [K : \Qp]$ over $\Sp(A)$.
    The result follows by gluing and the smoothness of $\cT_{0,\gB}^{\gT}$.
\end{proof}

\section{Deformations of triangulable $\PG$-modules}\label{sec35}

In this section we will study the consequences of the formal smoothness result established in the previous section. We work in the setting of deformations of a $\gG$-$\PG$-module over $\cR_{K,L}$, as well as in the setting of deformations of a $\gG$-valued Galois representation, both equipped with a triangulation over $\cR_{K,L}[\tfrac{1}{t}]$.

\subsection{The case of a $\gG$-$\PG$-module over $\cR_{K,L}$}\label{deformationsofPG}

We fix a $\gG$-$\PG$-module $D$ over $\cR_{K,L}$. We define the groupoid $ X_{D}\colon \Art_L \to \Gpd$ of deformations of $D$ exactly as in \Cref{def_XM}. Setting $\cM=D[\tfrac{1}{t}]$, we obtain a natural morphism 
$$ X_D\longrightarrow X_{\cM} $$
of groupoids over $\Art_L$. We assume that $D$ equipped with a triangulation $D^{\triangle}$ with locally algebraic parameter $\delta\colon \gT^\vee(K)\to L^\times$, and set  $\Mtri=\Dtri[\tfrac{1}{t}]$. We also define the groupoid $X_{\Dtri}$ of deformations of $\Dtri$ as before. 

We fix a finite set $\Lambda\subset X_+^*(\gT^\der)$ of $\Q$-generators of $X^*_+(\gT^{\der})_\Q$, and a faithful representation $r\colon \gG\to \GL_m$ satisfying the hypotheses of \Cref{emb_lemma_asdfjkl} (if one exists). 
\begin{proposition}\label{triimmersion}
    Assume that either $\delta\in \cT^{\gT}_{\Lambda}(L)$ (resp. $\delta\in \cT^{\gT}_{0,\Lambda}(L)$), or $\delta\in \cT^{\gT}_{r}(L)$ (resp. $\delta\in \cT^{\gT}_{0,r}(L)$). Then the natural functor $X_{\Dtri} \to X_{D}$ (resp. $X_{\Mtri} \to X_{\cM}$) is relatively representable and is a closed immersion.
\end{proposition}
\begin{proof}
 By \Cref{uniqueness_Lambda_2} and \Cref{uniqueness_r_2}, we have an equivalence of groupoids 
\begin{equation}\label{eqgroupoids}
X_{\Dtri}\xrightarrow{\sim} X_{D}\times_{|X_{D}|}|X_{\Dtri}|
    \end{equation}
    over $\Art_L$. We show that $|X_{\Dtri}|\to |X_{D}|$ is relatively representable following the proof of \cite[Proposition 2.3.9]{BC} (see  \cite[Proposition 6.6]{dedar2020}). The only non-trivial condition to verify is the following: given an injective morphism $A\hookrightarrow A'$ in $\Art_L$ and a deformation $D_{A}\in |X_D|(A)$ such that $D_{A'}\colonequals D_A\otimes_A A'\in |X_{\Dtri}|(A')$, one has $D_{A}\in |X_{\Dtri}|(A)$. We first assume that $\delta\in \cT_{\Lambda}^{\gT}(L)$. For each $\lambda\in \Lambda$, let $\scrL_{\lambda,A'}\subseteq \eta_{D_{A'}}(V_\lambda)$ be the rank-one $\PG$-module associated to the triangulation of $D_{A'}$, and define
    $$ \mathscr{L}_{\lambda,A}\colonequals \scrL_{\lambda,A'} \cap \eta_{D_A}(V_{\lambda}) $$
    inside $\eta_{D_{A'}}(V_{\lambda})$ for $\lambda\in \Lambda$. By the argument in the proof of \cite[Proposition 2.3.9]{BC}, the submodules $\mathscr{L}_{\lambda,A}$ are saturated (here we use the fact that the triangulation on $\eta_{D_{A'}}(V_{\lambda})$ has regular parameters). To verify the compatibilities required in \Cref{pluckerdatum}, it suffices to check them after tensoring with $A'$, where they hold by assumption. Hence, the $\mathscr{L}_{\lambda,A}$ define a triangulation of $D_A$. 
    
    We now assume that $\delta\in \cT^{\gT}_r(L)$. Since $D_{A'}$ is trianguline, so is $\eta_{D_{A'}}(V_r)$. As $\delta\circ r^{\vee}\in \cT^m_{\reg}(L)$, the same argument as in \cite[Proposition 2.3.9]{BC} shows that $\eta_{D_A}(V_r)$ is trianguline. We therefore obtain a commutative diagram
   \begin{equation}\label{pushoutdiag}
  \begin{tikzcd}
    \Rep_{L}(\GL_m) \arrow[r,"r^*"] \arrow[d] & \Rep_{L}(\gG) \arrow[d,dashed] \arrow[r, "\eta_{D_A}"] & \PGcat_A^+ \\
    \Rep_{L}(\gB_m) \arrow[r,dashed] \arrow[urr] & \Rep_L(\gB) \arrow[ur,dashed]
  \end{tikzcd}
\end{equation}
By \cite[Lemma 5.11]{conti2022lifting}, the pushout of the left part of the diagram is $\Rep_L(\gB)$, and hence $\eta_{D_A}$ factors through $\Rep_L(\gB)$ as desired. 
\end{proof}

Since we assume that the parameter $\delta$ is locally algebraic, $\eta_W\colonequals W_{\dR}\circ \eta_\cM$ (resp. $\eta_{W^+} \colonequals W^{+}_{\dR}\circ \eta_D$) is an almost de Rham $\gG$-$(L\otimes_{\Qp} \BdR)$-representations (resp. $\gG$-$(L\otimes_{\Qp} \BdR^+)$-representation) of $\GK$. We set $\eta_{\Wtri}\colonequals W_{\dR}\circ \eta_{\Mtri}$, and fix a trivialization $\alpha$ of $\eta_W$ as in \Cref{deftrivialization}.

The functor $W^{+}_{\dR,A}$ induces a map $X_D \to X_{W^+}$.
Since $\eta_W = \phi_{\BdR}\circ \eta_{W^+}$, we have a natural map $X_{W^+} \to X_W$.
The situation is summarized by a $2$-commutative square
\begin{equation}\label{deformation_square}
\begin{tikzcd}
    X_D \arrow[r] \arrow[d] & X_{W^+} \arrow[d] \\
    X_{\cM} \arrow[r] & X_W.
\end{tikzcd}
\end{equation}

As a consequence of \Cref{Btensoreq}, we have the following analog of \cite[Proposition 3.5.1]{BHS19}. 

\begin{lemma}\label{G_351} The map 
\begin{align}
    X_D \longrightarrow X_{\cM} \times_{X_W} X_{W^+} \label{defo_equiv}
\end{align} 
induced by \eqref{deformation_square} is an equivalence.
\end{lemma} 

We define the following fiber products of groupoids over $\Art_L$
\begin{align*}
 X_{D, \Mtri} \colonequals X_D \times_{X_{\cM}} X_{\Mtri}, \quad \quad  X_{W^+, \Wtri} \colonequals X_{\Wtri} \times_{X_W} X_{W^+}.
\end{align*}
Recall that we introduced in \Cref{sec_adRdefs} the groupoids $X_W^{\square}$, $X_{\Wtri}^\square=X_{\Wtri}\times_{X_W}X_W^{\square}$, and $X_{W^+}^\square=X_{W^+}\times_{X_W}X_W^{\square}$. We use them to define the following fiber products
\begin{flalign*}
     &X_{\cM}^{\square} \colonequals X_{\cM} \times_{X_W} X_W^{\square}, \quad\quad X_{D}^{\square}  \colonequals X_{D} \times_{X_W} X_W^{\square}, \quad\quad  X_{\Mtri}^{\square} \colonequals X_{\Mtri} \times_{X_W} X_W^{\square},
    \\   &X_{D, \Mtri}^{\square} \colonequals X_{D, \Mtri} \times_{X_W} X_W^{\square}, \quad X_{W^+,\Wtri}^{\square} \colonequals X_{W^+,\Wtri}\times_{X_W} X_W^{\square}.
\end{flalign*}

\begin{lemma}[{cf. \cite[Lemma 3.5.3]{BHS19}}]\label{rel_rep_XM_XW}
    The morphisms $X_{\cM} \to X_W$ and $X_{\cM^{\triangle}} \to X_{\Wtri}$ are relatively representable.
\end{lemma}

\begin{proof} Let $A \in \Art_L$ and $W_A \in X_W(A)$.
    We need to show that $F \colonequals \Spf(A) \times_{X_W} X_{\cM}$ is set-valued and pro-representable by a noetherian ring in $\widehat \Art_L$.
    For $A' \in \Art_L$ an object in $F(A')$ is given by a map $A \to A'$, an object $\cM_{A'} \in X_{\cM}(A')$ and an isomorphism $\phi_{A'}\circ\eta_{W_A} \cong W_{\dR,A}\circ\eta_{\cM_{A'}}$.
    Using the equivalence between $\PGcat_{K,A}$ and the category of $A \otimes_{\Qp} B_e$-representations, an object of $F(A')$ is equivalently given by a map $A \to A'$, a Tannakian deformation $M_e$ of $W_e^{\gG}(\cM)$ over $A'$ and an isomorphism $W_A \otimes_A A' \cong B_{\dR} \otimes_{B_e} M_e$. This description implies that every automorphism of this groupoid is trivial. Hence, $F$ is set-valued.
    Using that $\eta_{W_A}(V)$ is free over $A$ and freeness of $\eta_{M_e}(V)$ for all $V \in \Rep_L(\gG)$, we see that $F$ preserves fiber products and, as a deformation functor, preserves the final object, so it preserves finite limits. Hence, by Grothendieck's criterion, $F$ is pro-representable.

    For finiteness of the tangent space $F(L[\varepsilon])$, we observe that the tangent space is base changed from the tangent space of $\Spf(A) \times X_{\cM}$.
    A standard computation shows that the tangent space of $X_{\cM}$ is given by $H^1_{\varphi, \gamma_K}(\ad \cM)$, which is finite-dimensional by a dévissage, using triangulinity of $\cM$ and cohomology of rank $1$ objects in \Cref{coh_char}, passing to a colimit along multiplication by $t^{-1}$.

    The argument for the second functor is essentially the same.
\end{proof}

\begin{corollary}
[{cf. \cite[Corollary 3.5.4]{BHS19}}]\label{lem_354} The morphisms $X_{\Mtri}^{\square} \to X_{\Wtri}^{\square}$, $X_{D, \Mtri} \to X_{W^+, \Wtri}$ and $X_{D, \Mtri}^{\square} \to X_{W^+, \Wtri}^{\square}$ are relatively representable.
\end{corollary}

\begin{proof}
    This follows from \Cref{G_351} and \Cref{rel_rep_XM_XW} as in the proof of \cite[Corollary 3.5.4]{BHS19}.
\end{proof}

\begin{lemma}[{cf. \cite[Corollary 3.5.6]{BHS19}}]\label{lem356}
    The morphisms $X_{\Mtri} \to X_{\Wtri}$, $X_{\Mtri}^{\square} \to X_{\Wtri}^{\square}$, $X_{D, \Mtri} \to X_{W^+, \Wtri}$, and $X_{D, \Mtri}^{\square} \to X_{W^+, \Wtri}^{\square}$ are formally smooth.
\end{lemma}

\begin{proof}
    By base changing $\wt - \wt(\delta) \colon \wh\cT^\gT_{\delta} \to \wh\frakt$, we obtain maps $\wh\cT^\gT_{\delta} \times_{\wh\frakt} X_{\Wtri} \to X_{\Wtri}$ and $\wh\cT^\gT_{\delta} \times_{\wh\frakt} X_{W^+, \Wtri} \to X_{W^+, \Wtri}$, which are formally smooth by \Cref{lem355}. By composing with the formally smooth map $X_{\Mtri} \to \widehat \cT^\gT_{\delta} \times_{\widehat \frakt} X_{\Wtri}$ \eqref{formally_smooth_morphism} of \Cref{thm_formal_smoothness}, we find that $X_{\Mtri} \to X_{\Wtri}$ is formally smooth.
    By base changing \eqref{formally_smooth_morphism} along $X_{W^+} \to X_W$, and using \Cref{G_351}, we obtain a formally smooth map $X_{D, \Mtri} \to \wh\cT^\gT_{\delta} \times_{\wh\frakt} X_{W^+, \Wtri}$, so formal smoothness of the third map follows by composition.
    The statements for the second and the last map follow by base change.
\end{proof}

We introduce a groupoid of \emph{versal} (in the terminology of \cite[Proposition 3.5.7]{BHS19}) or \emph{rigidified} (in the terminology of \cite[Section 5.2]{dedar2020} and \cite[Section 3.1]{chenevier2010sur}) deformations of $\PG$-modules. 

We fix a rigidification $\theta\colon\cR_{K,A}[\frac{1}{t}](\delta)\cong\Mtri\times^\gB\gT$ of the $\gB$-$\PG$-module $\Mtri$, and we assume that $\delta\in \cT^{\gT}_{0,\gB}$. 
For $A\in\Art_L$, we define $X_{\Mtri}^\ver(A)$ to be the groupoid of triples $(\Mtri_A,\theta_A,\iota_A)$ where $(\Mtri_A,\theta_A)$ is a rigidified $\gB$-$\PG$-module over $\cR_{K,A}[\tfrac{1}{t}]$, and $\iota_A$ an isomorphism $\Mtri_A\otimes_AL\cong\Mtri$ such that $\theta\circ\iota_A\times^\gB\gT=\theta_A\otimes_AL$, as in \Cref{def_reg_rig_PG}.


\begin{proposition}[{cf. \cite[Proposition 3.5.7]{BHS19}}]\label{lem_357}
    The pseudofunctor $X_{ \Mtri}^{\square}$ is pro-representable. 
    The functor $|X_{ \Mtri}^{\square}|$ is pro-represented by a formally smooth complete noetherian local ring with residue field $L$ of dimension $[K : \Qp]\cdot(\dim \gG_L + \dim \gB_L)$.
\end{proposition}


\begin{proof} 
The pro-representability of $X_{ \Mtri}^{\square}$, and $\lvert X_{ \Mtri}^{\square}\rvert$, is immediate from \Cref{319} and \Cref{lem_354}. We now compute its dimension. By the same argument as in the proof of \Cref{SB_representable} and \Cref{SBtrivialAut} (but with inverting $t$), together with the fact that $\dim \cT_{0,\gB}^{\gT}=\dim\cT^{\gT}=([K:\Qp]+1)\cdot\dim \gT_L$, the dimension of $X_{\Mtri}^\ver$ is $[K:\Q_p]\cdot\dim\gB_L+\dim\gT_L$.
Now consider the cartesian diagram
    \begin{center}
    \begin{tikzcd}
        X_{\Mtri}^\ver\times_{X_{\Mtri}}X_{\Mtri}^\square \arrow[r] \arrow[d] & \arrow[d] X_{ \Mtri}^{\square}
        \\ X_{\Mtri}^\ver \arrow[r] & X_{\Mtri} 
    \end{tikzcd}       
    \end{center}
where the vertical arrows are the obvious ones, and the horizontal ones are obtained by forgetting the rigidification. 
 The vertical arrows are of relative dimension $[K:\Q_p]\cdot \dim\gG_L$ (coming from the choice of a framing) and the horizontal ones are of relative dimension $\dim\gT_L$ (coming from the choice of a rigidification). Therefore, the dimension of $X_{\Mtri}^\square$ is $[K:\Q_p] (\dim\gG_L+\dim\gB_L)$, as desired.
\end{proof}

\begin{corollary}[{cf. \cite[Corollary 3.5.8]{BHS19}}]\label{XW+prorep}\phantom{a}
\begin{enumerate}
\item The groupoid $X_{W^+,W^{\triangle}}^{\square}$ is a setoid and is pro-represented by the formal scheme $\wh X_{x_\pdR}$.
\item The groupoid $X_{D,\cM^\triangle}^\sq$ is a setoid and is pro-represented by a formal scheme which is formally smooth of relative dimension $[K:\Q_p]\cdot \dim \gB_L$ over $\wh X_{x_\pdR}$. 
\end{enumerate}
\end{corollary}

\begin{proof} The first statement follows directly from \Cref{very_key_proposition_W+} and \Cref{319}.
For the second, combining $(1)$ with \Cref{lem_354} shows that $X_{D,\cM^\triangle}^\sq=|X_{D,\cM^\triangle}^\sq|$ is representable. By \Cref{lem356}, the morphism $X_{D, \Mtri}^{\square} \to X_{W^+, \Wtri}^{\square}$ is formally smooth. Moreover, its relative dimension is equal to that of $X_{\Mtri}^\square \to X_{\Wtri}^\square$ since the latter is obtained from the former by base change. By \Cref{319} and \Cref{lem_357}, it is equal to $[K:\Q_p]\dim \gB_L$.
\end{proof}

\subsection{The point of $X$ attached to a regular triangulable $\PG$-module}\label{sec:defx} 

We keep the notations and assumptions of \Cref{deformationsofPG}. We further assume that $W^+$ is regular (\Cref{def:WA+reg}).

Consider:
\begin{itemize}
\item the nilpotent element $N_{W}\in \frakg(L)$ associated to $(W,\alpha)$ in \Cref{pdR_charanew},
\item the regular Hodge--Tate cocharacter $\varpi_\HT \in X_*(\bT)$ associated to $W^+$ in \Cref{def_HT},
\item the Hodge--Tate flag $x_{W^+,\HT}\in (\bG/\bB)(L)$ associated to $(W^+,\alpha)$ in \Cref{firstflag},
\item the Sen parameter $y_{W^+,\Sen}\in (\frakt\GIT\bW)(L)$ associated to $W^+$ in \Cref{def_Senpar},
\item the flag $x_{\Wtri,\tri}\in (\bG/\bB)(L)$ associated to $(\Wtri,\alpha)$ in \Cref{secondflag}, for $\eta_{\Wtri}=W_{\dR}\circ \eta_{\Mtri}$.
\end{itemize}
By \Cref{pointgtilde1} and \Cref{pointgtilde2}, we obtain an $L$-point of $X$:
$$ x_{\pdR}\colonequals (x_{\Wtri,\tri}, x_{W^+,\HT}, N_{W})\in X(L). $$
This is the $\gG$-analogue of the point defined in \cite[before Corollary 3.5.8]{BHS19}.

Given $w\in \bW$, recall that we defined $U^{w}= \bG(L)(1,\dot{w})(\bB\times \bB)(L) \subseteq (\bG/\bB\times \bG/\bB)(L)$, where $\dot{w}\in N_{\bG}(\bT)$ is some lift of $w$. We have a set-theoretic decomposition 
\[ (\bG/\bB\times \bG/\bB)(L)= \bigsqcup_{w\in \bW} U^w(L) \]
coming from the Bruhat decomposition of $\bG(L)$. 

\begin{definition}\label{def:wsat}
    We define $w_{\sat}\in \bW$ to be the element such that $(x_{\Wtri,\tri}, x_{W^+,\HT})\in U^{w_{\sat}}(L)$. 
\end{definition}

\begin{lemma}\label{Worbit}
The images of $\mathrm{d} \varpi_\HT $ and $\wt(\delta)$ in $(\frakt\GIT\bW)(L)$ coincide with $y_{W^+,\Sen}$.
\end{lemma}

\begin{proof}
Let $\pi_{1},\pi_{2}\colon \Rep_L(\gG)\to L\otimes_{\Qp}K$ be the maps associated to $\mathrm{d} \varpi_\HT$ and $\wt(\delta)$, respectively, as in \Cref{sec_cocharacters}. We need to show that $\pi_{1}=\pi_{2}=\pi_{W^+,\Sen}$. Let $(V,\rho_V)\in\Rep_{L}(\gG)$, 
then $\pi_{W^+,\Sen}(V)$ is the trace of the Sen operator associated to $\eta_{W^+}(V)/t\eta_{W^+}(V)$, which is the sum of its Hodge--Tate weights. By the definition of $\varpi_\HT$, this is equal to $\pi_2(V)$. By \Cref{effectonrhoV}, $\pi_{1}(V)$ is the trace of $\rho_V(\wt(\delta))$. This coincides with the trace of the Sen operator by \cite[Proposition 2.9]{BHSmodulaire}.
\end{proof}

\begin{lemma}\label{wsatwx}
The element $w_\sat\in\bW$ satisfies $\Ad(w_{\sat}^{-1})(\wt(\delta))=\mathrm{d}\varpi_{\HT}$.
\end{lemma}

\begin{proof}
Let us write $(x_{\Wtri,\tri}, x_{W^+,\HT})=g(w_{\sat}^{-1},1)(\bB(L)\times \bB(L))$ for some $g=(g_\tau)_{\tau\in \Sigma}\in \bG(L)$, and denote $w_{\sat}=(w_{\sat,\tau})_{\tau\in \Sigma}$. 
The Pl\"ucker datum associated with the flag $g_\tau\gB$ is given by $$(g_\tau V^{\lambda}_{\lambda}\subseteq V_\lambda)_{\lambda\in X^*_+(\gT)},$$ where $V_\lambda^\lambda$ denotes the highest weight line in $V_\lambda$ (see \Cref{example_plucker_datum}). 

By definition of the Hodge--Tate filtration $\cF_{\HT,A,\tau}$, the line in $V_\lambda$ corresponding to the Hodge--Tate flag is $gV_{\lambda}^{\lambda}=\eta_{\cF_{\HT,A,\tau}}^{-\langle \lambda,\varpi_{\HT,\tau}\rangle}(V_{\lambda})$. On the other hand, the line in $V_{\lambda}$ corresponding to the flag of the triangulation is $$gw_{\sat,\tau}^{-1}V_{\lambda}^{\lambda}= \alpha^{-1}\big(D_{\pdR}(W_{\dR}^+(\eta_{\Dtri}(V_{\lambda}^{\lambda}))\otimes_{\BdR^+}\BdR)\big)\otimes_{K, \tau} L, $$
where $\eta_{\Dtri}(V_{\lambda}^{\lambda})\cong \cR_{K,L}(\delta\circ\lambda^\vee)$ by definition of the parameter $\delta$. By \cite[Proposition 2.9]{BHSmodulaire}, its Hodge--Tate weight is $\wt_{\tau}(\delta\circ\lambda^\vee)=\langle \lambda,\delta^{\alg}_\tau\rangle$, where $\delta^{\alg}=\prod_{\tau\in \Sigma}\delta^{\alg}_{\tau}$ is the algebraic part of $\delta$. However, we have that $$gw_{\sat,\tau}^{-1}V_{\lambda}^{\lambda}= gV_{\lambda}^{w_{\sat,\tau}^{-1} \lambda},$$ which means that the jump of the induced Hodge--Tate filtration occurs at the $-\langle w_{\sat,\tau}^{-1}\lambda,\varpi_{\HT,\tau}\rangle$ step. 
We deduce that $\langle \lambda,\delta^{\alg}_\tau\rangle=\langle w_{\sat,\tau}^{-1}\lambda,\varpi_{\HT,\tau}\rangle$ for every $\lambda\in X_+^*(\gT)$. Therefore, we have $\delta^{\alg}=\Ad(w_{\sat})(\varpi_{\HT})$. 
\end{proof}


\subsection{Deformations with prescribed relative position}\label{defrelpos}

We keep the notations and assumptions of the previous two sections. In particular, recall that we associated with $D$ and the triangulation $\cM^\triangle$ a closed point $x_\pdR\in X(L)$.

Given $w\in\bW$, recall that $X^w$ is the irreducible component of $X$ defined as the closure of $V^w$. 
We denote by $\wh X_{x_\pdR}^w$ the completion of $X^w$ at $x_\pdR$ (it will be empty if $x_\pdR$ is not in the component $X^w$). We define a groupoid over $\Art_L$ by
$$ X^{\square,w}_{W^+,W^{\triangle}}\colonequals X^{\square}_{W^+,W^{\triangle}}\times_{\lvert X_{W^+,W^{\triangle}}^\square\rvert}\wh X_{x_\pdR}^w. $$
Because of the framing, all automorphisms in $X_{W^+}^\square$ are trivial, hence there are equivalences $X_{W^+,W^{\triangle}}^\square\xto{\sim}\lvert X_{W^+,W^{\triangle}}^\square\rvert$ and $X^{\square,w}_{W^+,W^{\triangle}}\xto{\sim}\lvert X^{\square,w}_{W^+,W^{\triangle}}\rvert$ of groupoids over $\Art_L$. 
From \Cref{XW+prorep} (1), we immediately deduce the following.

\begin{corollary}[{cf. \cite[Corollary 3.5.9]{BHS19}}]\label{Xsqwprorep}
The groupoid $X^{\square,w}_{W^+,W^{\triangle}}$ is pro-representable, and $\lvert X^{\square,w}_{W^+,W^{\triangle}}\rvert$ is pro-represented by the formal scheme $\wh X_{x_\pdR}^w$.
\end{corollary}

We define the groupoid $X^{w}_{W^+,W^{\triangle}}$ as the image of $X^{\square,w}_{W^+,W^{\triangle}}$ under the forgetful functor $X^{\square}_{W^+,W^{\triangle}}\to X_{W^+,W^{\triangle}}$, so that
\begin{equation}\label{Xwsq} X^{\square,w}_{W^+,W^{\triangle}} = X^w_{W^+,W^{\triangle}}\times_{X_{W^+,W^{\triangle}}}X^{\square}_{W^+,W^{\triangle}}. \end{equation}
We also define
\begin{align}\label{XDMtrisquarew}
X_{D,\cM^\triangle}^{\sq,w}\colonequals X_{D,\cM^\triangle}^\sq\times_{X_{W^+,W^{\triangle}}^\sq}X_{W^+,W^{\triangle}}^{\sq,w}, \quad\quad X_{D,\cM^\triangle}^{w}\colonequals X_{D,\cM^\triangle}\times_{X_{W^+,W^{\triangle}}}X_{W^+,W^{\triangle}}^{w}. 
\end{align}

The following is proved in the same way as \cite[Proposition 3.5.10]{BHS19}.

\begin{proposition}\label{3510}
The morphisms of groupoids
\[ X_{W^+,W^{\triangle}}^w\to X_{W^+,W^{\triangle}}, \quad X_{W^+,W^{\triangle}}^{\sq,w}\to X_{W^+,W^{\triangle}}^\sq, \quad X_{D,\Mtri}^w\to X_{D,\Mtri}, \quad X_{D,\Mtri}^{\sq,w}\to X_{D,\Mtri}^{\sq}, \]
are relatively representable and closed immersions.
\end{proposition}

We let
\begin{equation}\label{def_Wx} \bW(x_\pdR) \colonequals \{w\in\bW\,\vert\,x_\pdR\in X^w(L)\}. \end{equation}

\begin{corollary}[{cf. \cite[Corollary 3.5.11]{BHS19}}]\label{three_five_eleven}
For every $w\in\bW(x_\pdR)$, the functor $X_{D,\cM^\triangle}^{\sq,w}$ is pro-representable by a noetherian complete local domain of residue field $L$ and dimension $[K:\Q_p](\dim\gG_L+\dim\gB_L)$, whose associated formal scheme is formally smooth over $\wh X_{x_\pdR}^w$.
\end{corollary}

\begin{proof}
The pro-representability follows from \Cref{3510} and \Cref{XW+prorep}.
By base change from \Cref{lem356}, and by \Cref{XW+prorep}, we have that $X_{D,\cM^\triangle}^{\sq,w}\to X_{W^+,\Wtri}^{\sq,w}$ is formally smooth of relative dimension $[K:\Q_p](\dim \gB_L)$. The dimension estimate then follows, as $|X_{W^+,\Wtri}^{\sq,w}|\cong \wh X_{x_\pdR}^w$ has dimension $[K:\Q_p](\dim \gG_L)$. Finally, by \cite[Theorem 2.3.6]{BHS19} and the proof of \cite[Corollary 3.5.11]{BHS19}, the local ring $\widehat{\mathcal{O}}_{X_{x_\pdR}^w}$ is a complete local normal domain, and so is any local ring that is formally smooth over it. 
\end{proof}

Recall from \Cref{sec:Xnew} that the irreducible components of $T=\frakt\times_{\frakt\GIT\bW}\frakt$ are of the form $T^w=\{(z,\Ad(w^{-1})z),z\in\frakt\}$ for $w\in\bW$. Moreover, the maps $\kappa_{1},\kappa_{2}$ from \eqref{kappa} induce a map $\wh X_{x_\pdR}\to\wh T_{(0,0)}$ and, after restriction, a map $\wh X_{x_\pdR}^w\to\wh T^{w}_{(0,0)}$ for every $w\in\bW$. 
We denote by $\Theta$ the composition 
$$ \Theta\colon X_{D,\Mtri}^\sq\longrightarrow X_{W^+,W^{\triangle}}^\sq\xto{\sim}\lvert X_{W^+,W^{\triangle}}^\sq\rvert\xto{\sim}\wh X_{x_\pdR}\longrightarrow \wh T_{(0,0)}. $$
The following is deduced from \cite[Lemma 2.5.2]{BHS19} as in the proof of \cite[Corollary 3.5.12]{BHS19}.

\begin{corollary}\label{3512}
For $w\in\bW(x)$ and $w^\prime\in\bW$, the morphisms
\begin{equation}\begin{aligned}
X_{D,\Mtri}^{\sq,w}&\into X_{D,\Mtri}^\sq\longrightarrow\wh T_{(0,0)}, \\
X_{D,\Mtri}^{w}&\into X_{D,\Mtri}\longrightarrow\wh T_{(0,0)},
\end{aligned}\end{equation}
factor through $\wh T^{w^\prime}_{(0,0)}\into\wh T_{(0,0)}$ if and only if $w^\prime=w$.
\end{corollary}

\subsection{The case of a $\gG$-valued Galois representation}\label{caseofgaloisrep} 
Given $A\in \Art_L$, we let $\Rep_{A}(\GK)$ be the category of free of finite rank $A$-modules $V_A$
equipped with a continuous action of $\GK$.

Let $\rho\colon \GK\to \gG(L)$ be a continuous representation. We can associate with it a fiber functor
$$ \eta_\rho\colon \Rep_L(\gG)\longrightarrow \Rep_L(\GK). $$

\begin{definition}\label{defXrho}
     We define pseudofunctors $X_\rho,X_{\eta_\rho} \colon \Art_L \to \Gpd$ as follows:
\begin{enum}
    \item For every $A \in \Art_L$, let the objects of $X_\rho(A)$ to be continuous morphisms $\rho_A\colon \GK\to \gG(A)$ such that composing with $\gG(A)\to \gG(L)$ gives $\rho$, and the objects of $X_{\eta_\rho}(A)$ are pairs $(\eta_{\rho_A},j_A)$  where $\eta_{\rho_A}\colon \Rep_L(\gG)\to \Rep_A(\GK)$ is a fiber functor, and $j_A$ is an isomorphism $j_A\colon \phi_L\circ \eta_{\rho_A} \xrightarrow{\sim} \eta_\rho$. A morphism $(\eta_{\rho_A}, j_A) \to (\eta_{\rho_A'}, j_A')$ is an isomorphism $\varphi \colon \eta_{\rho_A} \xrightarrow{\sim} \eta_{\rho_A'}$ such that $j_A = j_A' \circ \phi_L \circ \varphi$.
    \item For a morphism $f \colon A_1 \to A_2$ in $\Art_L$, the functor $X_\rho(f) \colon X_\rho(A_1) \to X_\rho(A_2)$ maps $\rho_{A_1}$ to $\gG(f)\circ \rho_{A_1}$, and the functor $X_{\eta_\rho}(f) \colon X_{\eta_\rho}(A_1) \to X_{\eta_\rho}(A_2)$ maps $(\eta_{\rho_{A_1}},j_{A_1})$ to $(\phi_{A_2}\circ \eta_{\rho_{A_1}},j_{A_1})$.
\end{enum}
\end{definition}
Note that $X_\rho=|X_\rho|$, and is representable by the universal framed deformation ring $R_{\rho}$. 
There is a natural morphism
$$ X_\rho \longrightarrow X_{\eta_\rho} $$
that is relatively representable and formally smooth of relative dimension $\dim \gG$. 


Let $D$ be the $\gG$-$\PG$-module over $\cR_{K,L}$ whose corresponding fiber functor is $\eta_D=D_{\rig}\circ \eta_\rho$. By the arguments of \cite[Proposition 2.3.13]{BC}, the functor $D_{\rig}$ induces an equivalence $X_{\eta_\rho}\xrightarrow{\sim} X_D$. 
We assume that $D$ satisfies all the assumptions of Sections \ref{deformationsofPG} and \ref{sec:defx}, and we keep the notations of those sections. We can then define the following groupoids over $\Art_L$:
\begin{align*}
  X_{\rho}^{\square} &\colonequals X_{\rho}\times_{X_{D}} X_{D}^{\square},
\quad \quad & X_{\rho,\Mtri}&\colonequals X_\rho\times_{X_{D}}X_{D,\Mtri},
    \quad \quad &X_{\rho,\Mtri}^{\square} &\colonequals X_{\rho,\Mtri}\times_{X_\rho} X_\rho^{\square}.
\end{align*}
We similarly define the $w$-versions (for $w\in \bW$) of the groupoids in the bottom two lines as in \Cref{defrelpos}. Recall that we follow the convention of \cite{BHS19} and reserve the upper index ${}^\square$ for the framing on the associated almost de Rham representations ($X_\rho$ already corresponds to framed deformations in the usual sense).

It follows by base change that $X_\rho^{\square}\to X_{D}^{\square}$ is formally smooth of relative dimension $\dim \gG$. Similarly, given that $X_W^{\square}\to X_{W}$ is formally smooth of relative dimension $[K:\Qp]\cdot \dim \gG$, the same is true for $X_\rho^{\square} \to X_\rho$. 

\begin{theorem}\label{BHS3.6.2} Assume that either $\delta\in \cT^{\gT}_{0,\Lambda}$ or $\delta\in \cT^{\gT}_{0,r}$.
    \begin{enum} 
        \item The functor $|X_{\rho,\Mtri}|$ is pro-representable by a reduced equidimensional local complete noetherian ring $R_{\rho,\Mtri}$ of the residue field $L$ and dimension $\dim \gG+[K:\Qp]\cdot \dim \gB$.
        \item For every $w\in \bW(x_\pdR)$, the functor $|X_{\rho,\Mtri}^w|$ is pro-representable by $R_{\rho,\Mtri}^w\colonequals R_{\rho,\Mtri}/\mathfrak{p}_w$ where $\mathfrak{p}_w$ is a minimal prime ideal of $R_{\rho,\Mtri}$ and $R_{\rho,\Mtri}/\mathfrak{p}_w$ is a normal local ring.
        \item The morphism $|X_{\rho,\Mtri}^w|\to |X_{D,\Mtri}^w|\hookrightarrow |X_{D, \Mtri}| \xrightarrow{\Theta} \widehat{T}_{(0,0)}$ of groupoids over $\Art_L$ factors through $\widehat{T}^{w'}_{(0,0)}\hookrightarrow \widehat{T}_{(0,0)}$ if and only if $w'=w$.
    \end{enum}
\end{theorem}
\begin{proof}
By base change from \Cref{triimmersion}, the morphism $X_{D,\Mtri}\to X_D$ is relatively representable. The same holds for $X_{\rho,\Mtri}\to X_{\rho}$ also by base-change. Since $X_{\rho}$ is pro-representable, it follows that $X_{\rho,\Mtri}=|X_{\rho,\Mtri}|$ is pro-representable as well. We note that the local ring pro-representing $|X_{\rho,\Mtri}^\sq|$ is a formal power series ring over the one representing $|X_{\rho,\Mtri}|$. Then $(1)$ follows from $(2)$ of \Cref{XW+prorep} and the fact that $\wh X_{x_\pdR}$ is reduced and equidimensional. The latter holds by  \cite[Proposition 2.2.5, Theorem 2.2.6(i)]{BHS19}, using that $\mathcal{O}_{X,x_{\pdR }}$ is excellent, and hence these properties are stable under completion. 

By \Cref{3510}, $X_{D,\Mtri}^w\to X_{D,\Mtri}$ is relatively representable and a closed immersion. Hence so is $X_{\rho,\Mtri}^w\to X_{\rho,\Mtri}$. Since $X_{\rho,\Mtri}$ is pro-representable, $X_{\rho,\Mtri}^w$ is pro-representable by a quotient of the representing ring of $X_{\rho,\Mtri}$. As above, we find that the local ring pro-representing $|X_{\rho,\Mtri}^{\sq,w}|$ is a formal power series ring over the one representing $|X_{\rho,\Mtri}^w|$. Then $(2)$ follows from the formal smoothness of $X_{\rho,\Mtri}^{\sq,w}\to X_{D,\Mtri}^{\sq,w}$ and \Cref{three_five_eleven}. 

Finally, $(3)$ follows from the fact that $\Theta\colon X_{D,\Mtri} \to \widehat{T}_{(0,0)}$ factors through $|X_{D,\Mtri}|$, \Cref{3512}, and the formal smoothness of $X_{\rho,\Mtri}^w\to X_{D,\Mtri}^w$.
\end{proof}

Recall that in \Cref{def:wsat} we attached to $x_\pdR$ an element $w_\sat\in\bW$. 

\begin{proposition}[{cf. \cite[Proposition 3.6.4]{BHS19}}]\label{wsatpreceqw}
For every $w\in\bW(x_\pdR)$, $w_\sat\preceq w$. 
\end{proposition}

\begin{proof}
Let $w\in\bW(x_\pdR)$. By definition of $w_{\sat}$, we have $x_\pdR\in V^{w_\sat}$, so that $x_\pdR\in V^{w_\sat}\cap X^{w}$. By \cite[Lemma 2.2.4]{BHS19}, we conclude that $w_\sat\preceq w$. 
\end{proof}

\section{A crystallinity criterion}
\label{sec:cryscrit}

The goal of this section is to prove that, under some reasonable assumptions on the parameter, we can deduce that a triangulable $\PG$-module over $\cR_{K,L}$ is crystalline. 
We start by recalling how the classical construction of filtered $\varphi$-modules attached to crystalline representations can be generalized to representations valued in reductive groups.

\subsection{Filtered $\gH$-$\varphi$-modules and crystalline representations}\label{sec:Gphimod}

Let $\sigma$ be the arithmetic Frobenius automorphism of $K_0$, and let $L'$ be a finite extension of $L$. We set $\sigma_{L'}\colonequals \sigma\otimes \id_{L'}\colon K_0\otimes_{\mathbb{Q}_p}L'\to K_0\otimes_{\mathbb{Q}_p}L'$, and $\Sigma(K,L') \colonequals \Hom_{\mathbb{Q}_p}(K,L')$. 

\begin{definition}\label{def_G_varphi_module} \phantom{a}
\begin{enumerate}
\item[(1)] A \emph{$\varphi$-module} over $K_0$ with coefficients in $L'$ is a free $K_0 \otimes_{\Qp} L'$-module $M$ of finite rank equipped with a $\sigma_L'$-semilinear bijection $\varphi \colon M \to M$. 
\item[(2)] A \emph{filtered $\varphi$-module} over $K_0$ with coefficients in $L'$ is a $\varphi$-module $M$ over $K_0$ with coefficients in $L'$ endowed with a decreasing, exhaustive and separated filtration of $M \otimes_{K_0} K$ by $K\otimes_{\Q_p}L'$-submodules.
    \item[(3)] A (filtered) \emph{$\gH$-$\varphi$-module} over $K_0$ with coefficients in $L'$ is a fiber functor $\eta_M$ from $\Rep_{L'}(\gH)$ to the category of (filtered) $\varphi$-modules with coefficients in $L'$. 
    \item[(4)] We say that a (filtered) $\gH$-$\varphi$-module is \emph{$\gH$-trivial} if its associated $\gH$-torsor is trivial. 
\end{enumerate}
\end{definition}


Consider pairs $(b,\underline{\mu})$, where $b\in \gH(K_0\otimes_{\mathbb{Q}_p}L')$ and $\underline{\mu}=(\mu_\tau)\in X_*(\gH_{L'})^{\Sigma(K,L')}$. By \Cref{sec_cocharacters}, the character $(-\mu_\tau)$ determines a $\gG$-filtration $\fil^\bullet_{\mu_\tau}$. To a pair $(b,\underline{\mu})$, we associate an $\gH$-trivial filtered $\gH$-$\varphi$-module as follows: given $V\in \Rep_{L'}(\gH)$, we equip $V\otimes_{\mathbb{Q}_p}K_0$ with a $\varphi$-action given by $\varphi(v\otimes x)=b(v\otimes \sigma(x))$, and we equip $V\otimes_{\mathbb{Q}_p}K$ with a filtration induced by $(\fil_{\mu_{\tau}}^\bullet)_{\tau}$.

Two pairs $(b,\underline{\mu})$ and $(b',\underline{\mu'})$ are said to be equivalent if there exists $h\in \gH(K_0\otimes_{\mathbb{Q}_p}L')$ such that $b'=hb\sigma_{L'}(h)^{-1}$, and $\fil^\bullet_{\mu'_\tau}=\fil^\bullet_{h\mu_\tau h^{-1}}$ for all $\tau\in \Sigma(K,L')$. 



\begin{lemma}[{cf. \cite[Proposition 4.27]{dedar2020}}]\label{framedfilteredphimods}
The set of isomorphism classes of $\gH$-trivial filtered $\gH$-$\varphi$-modules is in bijection with the set of equivalence classes of pairs $(b,\underline{\mu})$ as above.
\end{lemma}

By \cite[Lemma 4.22]{dedar2020}, any $\gH$-$\varphi$-module becomes trivial after base change to a finite extension of $L'$.


\begin{definition}\label{def:crys}
A $\PG$-module $D$ of rank $d$ over $\cR_{K,L'}$ is \emph{crystalline} if the $K_0\otimes_{\Qp}L$-module $D_{\cris}(D)\colonequals (D[\tfrac{1}{t}])^{\Gamma_K}$ has dimension $[L':\Qp]\cdot d$ over $K_0$. 
We define this way a functor $D_\cris$ from crystalline $\PG$-module $D$ over $\cR_{K,L'}$ to filtered $\gH$-$\varphi$-modules over $K_0$ with coefficients in $L'$.
An $\gH$-$\PG$-module $D$ over $\cR_{K,L'}$ is crystalline if, for every $V\in\Rep_{L'}(\gH)$, $\eta_D(V)$ is crystalline.
\end{definition}

By the main result of \cite{Ber08}, as reformulated in \cite[Proposition 1.2.9]{Benois}, we have the following.

\begin{proposition}\label{equivalencecrystallinecat}
The functor $\eta_{D}\mapsto D_\cris\circ\eta_D$ 
defines an equivalence of categories between crystalline $\gH$-$\PG$-modules over $\cR_{K,L}$ and filtered $\gH$-$\varphi$-modules over $K$ with coefficients in $L$.
\end{proposition}

\subsubsection{The crystalline Frobenius} 
Let $\rho\colon\GK\to\gG(L)$ be continuous representation with associated fiber functor $\eta_\rho$, and let $\eta_D=D_\rig\circ\eta_\rho$ be the associated $\gG$-$\PG$-module over $\cR_{K,L}$.  
We say that $\rho$ is crystalline if $D$ is crystalline in the sense of \Cref{def:crys}.  
Since $D_\cris\circ \eta_D$ coincides with $D_\cris\circ\eta_\rho$, one recovers the usual definition of crystalline representation.
\begin{proposition}[{cf. \cite[Corollary 4.34]{dedar2020}}]
    Every crystalline representation $\rho
    \colon \GK\to\gG(L)$ is trianguline.
\end{proposition}

Set $f=[K_0:\Q_p]$, and assume that $\rho$ is crystalline. There exists a finite extension $L'/L$ such that $D_{\cris}\circ \eta_{\rho}$ is trivializable. Hence by \Cref{framedfilteredphimods}, we may associate to it a pair $(\varphi^f,\varpi_{\HT})$, where $\varpi_{\HT}=(\varpi_{\HT,\tau})_{\tau}$ is the Hodge--Tate cocharacter of $\rho$.


Assume that $L'$ is big enough so that the associated $\gG$-$\PG$-module $D_{L'}$ over $\cR_{K,L'}$ admits a triangulation $\Dtri$. 
We may then assume that $\varphi^f=(\varphi^f_\tau)_\tau\in \gB(K_0\otimes_{\Qp}L')=\prod_{\tau\in \Sigma(K_0,L')}\gB(L')$. We let  $\overline{\varphi}^f$ be the image of $\prod_\tau \varphi^f_\tau$ under the projection $\gB(L')\to\gT(L')$, which is independent of the choice of $\varphi^f$ in the $\sigma_{L'}$-conjugacy class by commutativity of $\gT$.

\begin{definition}\label{def:phiregsem}
We say that $\varphi^f$ is \emph{regular semisimple} if $\alpha(\overline{\varphi}^f)\ne 1$ for every root $\alpha\in\Phi(\gG,\gT)$.
\end{definition}

Given $x\in (L')^\times$, write $\unr(x)$ for the character $K^\times\to (L')
^\times$ that is trivial on $\cO_K^\times$ and maps a uniformizer of $K$ to $x$. 
\begin{proposition}[{c.f. \cite[Proposition 4.44]{dedar2020}}]
    The parameter $\delta\colon \gT^\vee(K)\to (L')^\times$ of the triangulation $\Dtri$ is given by
    $$ \delta= \unr_\gT(\overline{\varphi}^f)\cdot \delta^{\alg} $$
where $\unr_\gT(\overline{\varphi}^f)$ is the unique character satisfying $\unr_\gT(\overline{\varphi}^f)\circ\chi=\unr(\chi(\overline{\varphi}^f))$ for every $\chi\in X^\ast(\gT)$, and $\delta^{\alg}$ is the product of the characters
$$ \delta^{\alg}_{\tau} \colon \gT^\vee(K)\xrightarrow{\tau} \gT^\vee(L')\xrightarrow{\varpi_{\HT,\tau}^\vee} (L')^\times,$$
for $\tau\in \Sigma(K,L')$.
\end{proposition}
\begin{remark}\label{phireg}
The linearized Frobenius $\varphi^f$ is regular semisimple if and only if, for every simple root $\alpha\in \Delta(\gG,\gB)$, $\delta\circ\alpha^\vee$ is the product of an algebraic character with a non-trivial unramified character. 
\end{remark}

\subsection{Selmer groups for $B$-pairs and $B_e$-representations}
For convenience, we mostly work in the language of $B$-pairs in the rest of the section. At various points, we rely on the equivalence of categories between ($\gH$-)$B$-pairs and ($\gH$-)$\PG$-modules over $\cR_{K,L}$, as well as the equivalence of categories between ($\gH$-)$B_e\otimes_{\Qp}L$-representations of $\GK$, and ($\gH$-)$\PG$-modules over $\cR_{K,L}[\tfrac{1}{t}]$ (see the proof of \cite[Proposition 3.5.1]{BHS19}).

For a $p$-adic field $L$, we write $B_{\heartsuit,L}=B_\heartsuit\otimes_{\Qp}L$ for $\heartsuit\in \{e,\cris,\dR\}$, where $L$ is equipped with the trivial action of $\GK$. 
Let $W=(W_e,W_\dR^+)$ be an $L$-$B$-pair. 
We define a complex of $\GK$-modules $C^\bullet(W)$, concentrated in degrees $[0,1]$, as
\begin{align*} C^0(W)=W_e\oplus W_\dR^+&\xto{\delta_0} C^1(W)=W_\dR \\
(x,y)&\mapsto x-y,
\end{align*}
where we are implicitly identifying $W_e,W_\dR^+$ with submodules of $W_\dR$. 
For $i\in\bbN$, we define $H^i(\GK,W)=H^i(\GK,C^\bullet(W))$.

We denote by $W^\triv$ the trivial $L$-$B$-pair of rank 1. 
The discussion in \cite[Section 2.1]{NakClass} identifies $H^1(\GK,W)$ with the group of extensions of $W^\triv$ by $W$. 
Similarly, given a $B_{e,L}$-representation $W_e$, the Galois cohomology group $H^1(\GK,W_e)$ classifies extensions of $B_{e,L}$ by $W_e$.


Write $K_0$ for the maximal unramified extension of $\Q_p$ contained in $K$. 
We say that a $B_{e,L}$-representation $W_e$ is \emph{crystalline} (respectively, \emph{de Rham}) if $D_\cris^K(W_e)\colonequals(B_\cris\otimes_{B_e}W_e)^{\GK}$ is a free $K_0\otimes_{\Q_p}L$-module of the same rank as $W_e$ (respectively, $D_\dR^K(W_e)\colonequals(B_\dR\otimes_{B_e}W_e)^{\GK}$ is a free $K\otimes_{\Q_p}L$-module of the same rank as $W_e$). We say that an $L$-$B$-pair $W$ is crystalline (respectively, de Rham) if $W_e$ is. 
As for Galois representations, the implications crystalline $\implies$ de Rham $\implies$ Hodge--Tate hold.

Given a $B_{e,L}$-representation $W_e$, we define 
\[ H^1_f(\GK,W_e)\colonequals\ker\big(H^1(\GK,W_e)\longrightarrow H^1(\GK,B_\cris\otimes_{B_e}W_e)\big). \]
If $W_e$ is crystalline, one shows by standard arguments that a class in $H^1(\GK,W_e)$ corresponds to a crystalline extension of $B_{e,L}$ by $W_e$ if and only if it belongs to $H^1_f(\GK,W_e)$.

Given a $B$-pair $W$, we define, following \cite[Definition 2.4]{NakClass},
$$ H^1_f(\GK,W)\colonequals\ker\big(H^1(\GK,W)\longrightarrow H^1(\GK,B_\cris\otimes_{B_e}W_e)\big). $$
If $W$ is crystalline, then by \cite[Remark 2.5]{NakClass} a class in $H^1(\GK,W)$ corresponds to a crystalline extension of $W^\triv$ by $W$ if and only if it belongs to $H_f^1(\GK,W)$. 


Recall that to a character $\delta\colon K^\times\to L^\times$  we can attach a $\PG$-module $\cR_{K,L}(\delta)$, hence an $L$-$B$-pair $W(\delta)$, which coincides with that constructed in \cite[Section 1.4]{NakClass}. 

\begin{remark}\label{cryschar}
The $L$-$B$-pair $W(\delta)$ is crystalline if and only if $\delta\vert_{\cO_K^\times}$ is algebraic. This is clear from the construction in \cite[Section 1.4]{NakClass}: with the notation there, $$W(\delta)=W(L(\wtl\delta_0))\otimes W_0^{\otimes i}$$ where $\wtl\delta_0\colon \GK\to L^\times$ is a Galois character and $W_0$ is a crystalline $L$-$B$-pair (it is attached to a filtered $\varphi$-module $D_0$). Therefore, $W(\delta)$ is crystalline if and only if $\wtl\delta_0$ is, if and only if $\wtl\delta_0\vert_{\cO_K^\times}=\delta\vert_{\cO_K^\times}$ is algebraic, by the standard classification of crystalline Galois characters.
\end{remark}

Recall that $W_\dR^+/tW_\dR^+$ is a free $\C_p\otimes_{\Q_p}L$-module of the same rank as $W$, equipped with a Sen operator $\Theta_{W_\dR^+}$. We say that $W$ is \emph{Hodge--Tate} if $\Theta_{W_\dR^+}$ is diagonalizable and all of its eigenvalues belong to $\Z$. In such a case, we call these eigenvalues the \emph{Hodge--Tate weights} of $W$.
By \cite[Theorem 5]{sencont}, $\Theta_{W_\dR^+}$ can be represented by a matrix with coefficients in $K\otimes_{\Q_p}L$, and by writing $K\otimes_{\Q_p}L=\bigoplus_{\sigma:K\into\C_p}L$, we obtain a tuple of Hodge--Tate weights $\uk_\sigma=(k_{\sigma,i})_{i=1,\ldots,\rank W}$ for every $\sigma$. 

With the above definition, the Hodge--Tate weight of the cyclotomic character is 1 (as per our convention). Similarly to \cite[Section 6.3]{emertongee}, which uses the opposite convention, we say that the Hodge--Tate weights of a Hodge--Tate $L$-$B$-pair $W$ are \emph{slightly positive} if $k_{\sigma,i}>0$ for every $\sigma,i$ and if there exists one $\sigma$ such that $k_{\sigma,i}>1$ for every $i$. 
If $W,W'$ are two Hodge--Tate $L$-$B$-pairs with Hodge--Tate weights $\uk,\uk'$, we say that $\uk$ is \emph{slightly greater} than $\uk'$ if the Hodge--Tate weights of $W_1\otimes W_2^\vee$ are slightly positive; equivalently, if for every $i,j,\sigma$, $k_{\sigma,i}>k'_{\sigma,j}$, and there exists one $\sigma$ such that $k_{\sigma,i}>k'_{\sigma,j}+1$ for every $i,j$.

\begin{lemma}\label{H1f}
    If $W$ is a de Rham $L$-$B$-pair with slightly positive Hodge--Tate weights, then $H^1(\GK, W) = H^1_f(\GK, W)$.
\end{lemma}

\begin{proof}
By \cite[Proposition 2.7]{NakClass}, we have 
$$ \dim_L H^1_f(\GK, W) = \dim_L D^K_{\dR}(W)/\Fil^0 D^K_{\dR}(W) + \dim_L H^0(\GK, W). $$
By the Euler characteristic formula \cite[Theorem 4.3]{LiuCohAndDuality}, we have
$$ \dim_L H^1(\GK, W) = \dim_L H^0(\GK, W) + \dim_L H^2(\GK, W) + [K:\Qp] \rank W. $$
Since $W$ is de Rham, $\dim_L D_\dR^K(W)=[K:\Q_p]\rank W$, and since the Hodge--Tate weights of $W$ are all positive, $\Fil^0D_\dR^K(W)=0$. 
Hence,
$$ \dim_L H^1(\GK, W) - \dim_L H^1_f(\GK, W) = \dim_L H^2(\GK, W) = \dim_L H^0(\GK,W^\vee(1)), $$
where (1) denotes the Tate twist and the last equality follows from Tate duality \cite[Theorem 4.7]{LiuCohAndDuality}. Since there exists one $\sigma$ such that all of the $\sigma$-Hodge--Tate weights of $W$ are strictly greater than $1$, $W^\vee(1)$ cannot contain a copy of the trivial $L$-$B$-pair, so that $H^0(\GK,W^\vee(1))=0$, as desired. \end{proof}

Although \Cref{H1f} is enough for our purposes, we record the following generalization of \cite[Lemma 6.3.1]{emertongee} (keeping in mind that we are using opposite conventions for Hodge--Tate weights).

\begin{corollary}
Let $W,W'$ be two crystalline $L$-$B$-pairs with Hodge--Tate weights $\uk,\uk'$. If $\uk$ is slightly greater than $\uk'$, then every extension $0\to W\to W^{\prime\prime}\to W'\to 0$ is crystalline.
\end{corollary}

\begin{proof}
Thanks to \cite[Proposition 1.26]{nekovarheight}, it is enough to apply \Cref{H1f} to $W\otimes (W')^\vee$.
\end{proof}

    


\subsection{Crystalline extensions of $\gM$-$\PG$-modules to $\gP$}


The next two lemmas provide a characterization of crystalline $\gB$-$\PG$-modules as those for which a certain cocycle vanishes. This criterion will be central in showing that the crystalline classes in $H^1_{\varphi,\Gamma_K}(V)$ correspond to crystalline lifts.
\begin{lemma}\label{set_of_crystalline_extensions}
    The set
    \begin{align}
        \ker\big(H^1_{\varphi,\Gamma_K}(\gH(\cR_{K,L})) \longrightarrow H^1\big(\Gamma_K, \gH(\cR_{K,L}[\tfrac{1}{t}])\big)\big) \label{set_of_cris}
    \end{align}
    is contained in the set of equivalence classes of $\gH$-trivial crystalline $\gH$-$\PG$-modules. Here $H^1(\Gamma_K,-)$ denotes abstract group cohomology.
    If $\check H^1_{\et}(L, \gH) = \{*\}$, then it is equal to the set of crystalline $\gH$-$\PG$-modules.\footnote{The assumption that $\check H^1_{\et}(L, \gH) = \{*\}$ is not really necessary to give a characterization of the crystalline classes, but this is in general more complicated to formulate.}
\end{lemma}

\begin{proof}
    The first statement follows by composing with a faithful representation of $\gH$ and observing that a $\PG$-module $D$ over $\cR_{K,L}$ is crystalline if and only if $D[\tfrac{1}{t}]\cong (D[\tfrac{1}{t}])^{\Gamma_K}\otimes_{K_0}\cR_{K,L}$, that is, $D[\tfrac{1}{t}]$ is a trivial $\Gamma
    _K$-module. Now assume that $\check H^1_{\et}(L, \gH) = \{*\}$, and let $(M,c) \in Z^1_{\varphi,\Gamma_K}(\gH(\cR_{K,L}))$ be a class corresponding to a crystalline $\gH$-$\PG$-module.
    Letting  $\eta_{(M,c)}\colon \Rep_L(\gH) \to \PGcat_{K,L}^{+,\cris}$ be the associated fiber functor, we have the following diagram
\begin{center}
\begin{tikzcd}
    \Rep_L(\gH) \arrow[r, "\eta_{(M,c)}"'] \arrow[rrr, bend left=20, "\eta_c"]
    & \PGcat_{K,L}^{+,\cris} \arrow[rr] \arrow[drr, "D_{\cris}"']
    & & \Rep_{\cR_{K,L}[\frac{1}{t}]}(\Gamma_K) \\
    & & & \Vect_{K_0\otimes L} \arrow[u, "{-\otimes \cR_K[\frac{1}{t}]}"']
\end{tikzcd}
\end{center}
    where $D_{\cris}$ is the functor given by $D\mapsto (D[\tfrac{1}{t}])^{\Gamma_K}$ (forgetting the extra structure), and $\Rep_{\cR_{K,L}[\tfrac{1}{t}]}(\Gamma_K)$ is the category of $\cR_{K,L}[\tfrac{1}{t}]$-semilinear $\Gamma_K$-representations.
    The fiber functor $f\circ\eta_{(M,c)}$ corresponds to a torsor on $\Spec(K_0 \otimes_{\Qp} L)$, which is trivial by our assumption that $\check H^1_{\et}(L, \gH) = \{*\}$.
    It follows that the fiber functor $\eta_c$ is trivializable. Since $\eta_c$ lands in trivial $\Gamma_K$-representations, we conclude using the bijection between $H^1(\Gamma_K,\gH(\cR_{K,L}[\tfrac{1}{t}]))$ and the equivalence classes of $\gH$-trivial fiber functors $\Rep_L(\gH)\to \Rep_{\cR_{K,L}[\tfrac{1}{t}]}(\Gamma_K)$.
\end{proof}

\begin{lemma}\label{H1_vanishing_groups}
    Let $\gH$ be a connected solvable group over $L$ with composition factors only $\Gm$ and $\Ga$. Then $\check H^1_{\et}(L, \gH) = \{*\}$.
\end{lemma}

\begin{proof}
    By considering the unipotent radical of $\gH$ and using the long exact sequence, we reduce to the case of $\Gm$ and a unipotent group.
    For a unipotent group, we use the central series to reduce to the case of $\Ga$.
\end{proof}

Let $\gP$ be an algebraic group over $L$ of the form $\gP = \gU \rtimes \gM$, where $\gM$ is linear algebraic and $\gU$ is unipotent. 
We say that the Hodge--Tate weights of a Hodge--Tate $\gM$-$B$-pair $W_\gM$ have \emph{enough gaps with respect to $\gP$} if the Hodge--Tate weights of the $B$-$L$-pair $W_\gM\times^\gM\Lie\gU$ are slightly positive, where $\Lie\gU$ is equipped with the adjoint action of $\gM$. 

\begin{lemma}\label{crystalline_extensions_exist}
If $D_{\gB/\gU^i}$ is a $\gB/\gU^i$-trivial  crystalline $\gB/\gU^i$-$\PG$-module over $\cR_{K,L}$, then it admits a crystalline extension to $\gB/\gU^{i+1}$.
\end{lemma}

\begin{proof}
    By the equivalence of categories stated in \Cref{equivalencecrystallinecat}, it suffices to prove the statement in terms of filtered $\gB/\gU^i$-$\varphi$-modules.
    By \Cref{framedfilteredphimods}, the data of a $\gB/\gU^i$-trivial filtered $\gB/\gU^i$-$\varphi$-module is equivalent to specifying a pair $(b,\underline{\mu}=(\mu_\tau))\in \gB/\gU^i(K_0\otimes_{\mathbb{Q}_p}L)\times X_*((\gB/\gU^i)_L)^{\Sigma(K,L)}$. 
    As the category of representations of $\Gm$ is semisimple (in particular $H^2(\Gm, -) = 0$), there is no obstruction to lifting a homomorphism $\mathbb{G}_{\mathrm m,L} \to (\gB/\gU^i)_L$ to a homomorphism $\mathbb{G}_{\mathrm m,L} \to (\gB/\gU^{i+1})_L$. Alternatively, one can lift such a homomorphism directly by conjugating the image of $\Gm$ into the maximal torus of $\gB$. Moreover, by the surjectivity of the map $(\gB/\gU^{i+1})(K_0 \otimes_{\Qp} L) \to (\gB/\gU^i)(K_0 \otimes_{\Qp} L)$, we conclude that the desired lifting exists.
\end{proof}

\begin{theorem}\label{lincrys}
If $D_\gT$ is crystalline and its Hodge--Tate weights have enough gaps with respect to $\gB$, then any lift $D_{\gB}$ of $D_{\gT}$ to $\gB$ is crystalline.
\end{theorem}

\begin{proof}
    We will apply \Cref{H1_vanishing_groups} to $\gB$ and all of its subquotients.
    Let us write $D_{\gB/\gU^i}$ for the pushforward of $D_{\gB}$ via $\gB \to \gB/\gU^i$.
    By \Cref{set_of_crystalline_extensions} the class in $H^1_{\varphi, \Gamma_K}(\gT(\cR_{K,L}))$ representing $D_{\gT}$ maps to $1$ in $H^1(\Gamma_K, \gT(\cR_{K,L}[\tfrac{1}{t}]))$. Using the short exact sequence
    $$ 1\to \gU^i/\gU^{i+1}\to \gB/\gU^{i+1} \to \gB/\gU^i\to 1, $$
    we prove that $D_{\gB/\gU^{i+1}}$ is contained in \eqref{set_of_cris} assuming that the same thing holds for $D_{\gB/\gU^i}$. This would prove the theorem. 
    Consider the following diagram:
    \begin{equation} \label{six_diagram}
    \begin{tikzcd}
        H^{1,\star}_{\varphi,\Gamma_K}(\gU^i/\gU^{i+1}(\cR_{K,L})) \arrow[r, "\alpha"] \arrow[d] & H^1_{\varphi, \Gamma_K}(\gB/\gU^{i+1}(\cR_{K,L})) \arrow[r] \arrow[d] & H^1_{\varphi, \Gamma_K}(\gB/\gU^i(\cR_{K,L})) \arrow[d] \\
        H^{1,\star}\big(\Gamma_K, \gU^i/\gU^{i+1}(\cR_{K,L}[\tfrac{1}{t}])\big) \arrow[r, "\beta"] & H^1\big(\Gamma_K, \gB/\gU^{i+1}(\cR_{K,L}[\tfrac{1}{t}])\big) \arrow[r] & H^1\big(\Gamma_K, \gB/\gU^i(\cR_{K,L}[\tfrac{1}{t}])\big).
    \end{tikzcd}
    \end{equation}
    The sets on the left are identified with the fibers of the right horizontal maps above $D_{\gB/\gU^i}$ and $D_{\gB/\gU^i}[\tfrac{1}{t}]$ respectively. This identification is done by choosing a crystalline element $D'$ of the fiber in which $D_{\gB/\gU^{i+1}}$ is contained, which we know exists by \Cref{crystalline_extensions_exist}. So the map $\alpha$ is determined by mapping $0$ to $D'$. The map $\beta$ is defined by mapping $0$ to $*$.
    By \Cref{set_of_crystalline_extensions}, $D'$ maps to $*$ in $H^1(\Gamma_K, (\gB/\gU^{i+1})(\cR_{K,L}[\tfrac{1}{t}]))$, so that the left square commutes.
    Via the exponential map, we have the commutative diagram
    \begin{center}
    \begin{tikzcd}
        H^{1}_{\varphi,\Gamma_K}(V^i/V^{i+1}) \arrow[r,"\sim","\exp"'] \arrow[d] & H^{1,\star}_{\varphi,\Gamma_K}\big(\gU^i/\gU^{i+1}(\cR_{K,L})\big) \arrow[d]
        \\ H^{1}(\Gamma_K, V^i/V^{i+1}[\frac{1}{t}]) \arrow[r,"\sim", "\exp"'] & H^{1,\star}\big(\Gamma_K, \gU^i/\gU^{i+1}(\cR_{K,L}[\tfrac{1}{t}])\big),
    \end{tikzcd}       
    \end{center}
    where the top vertical map is an isomorphism by \Cref{expH1_new}, and the bottom one by a similar calculation. Consider a pair $(v,d)\in H^{1}_{\varphi,\Gamma_K}(V^i/V^{i+1})$ corresponding to an extension $0\to V^i/V^{i+1}\to E\to \cR_{K,L}\to 0$. By inverting $t$ and taking $\Gamma_K$-invariants in the short exact sequence, a standard argument shows that the extension $E$ is crystalline if and only if the image of $d$ in $H^{1}(\Gamma_K, V^i/V^{i+1}[\tfrac{1}{t}])$ is trivial. Therefore by \Cref{H1f}, the left vertical map in the diagram is zero, which means that the left vertical map in \eqref{six_diagram} is zero.
    Hence $D_{\gB/\gU^{i+1}}$ maps to $*$ in $H^1(\Gamma_K, \gB/\gU^{i+1}(\cR_{K,L}[\tfrac{1}{t}]))$, which implies that it is crystalline.
\end{proof}

We rely on \Cref{lincrys} to prove a criterion for a $\gG$-$\PG$-module over $\cR_{K,L}$ to be crystalline. This can be seen as a generalization of \cite[Lemma 2.10]{bhssmooth}. Note that, thanks to \Cref{lincrys}, we are able to remove the condition on the Frobenius eigenvalues.

Given a parameter $\delta \colon \gT^\vee(K)\to L^\times$, we write $\delta^\vee\colon K^\times\to\gT(L)$ for its dual. For every root $\alpha\in\Phi(\gG,\gB)$, we obtain a character $\alpha(\delta^\vee)\colon K^\times\to L^\times$, similarly to what we did in \Cref{sec:Breg}.

By writing $\frakt(L)=\Lie(\gT)(L\otimes_{\Q_p}K)$, $L\otimes_{\Q_p}K\cong\bigoplus_{\sigma:K\into L}L$, and choosing a splitting $\gT\cong\Gm^n$, we can identify the weight of a parameter $\delta$ with a tuple $\uk=(k_{\sigma,i})$ of elements of $L$ indexed by embeddings $\sigma\colon K\into L$ and $i\in\{1,\ldots,n\}$. 

\begin{corollary}\label{cryscriterion}
Let $D$ be a triangulable $\gG$-$\PG$-module over $\cR_{K,L}$, and let $\delta:\gT^\vee(K)\to L^\times$ be the parameter of a $\gB$-triangulation $D^\tri$ of $D$. Assume that:
\begin{enumerate}[label=(\roman*)]
\item $\delta\vert_{\gT^\vee(\cO_K)}$ is algebraic,
\item for every embedding $\sigma\colon K\into L$ and positive root $\alpha\in\Phi^+(\gG,\gB)$, $\wt_\sigma(\alpha(\delta^\vee))>0$,
\item for at least one $\sigma \in \Sigma$, $\wt_\sigma(\alpha(\delta^\vee))>1$ for every positive root $\alpha\in\Phi^+(\gG,\gB)$.
\end{enumerate}
Then $D$ is crystalline. 
\end{corollary}

\begin{proof}
By condition $(i)$ and \Cref{cryschar}, the $\gT$-$\PG$-module $\cR_{K,L}(\delta)$ is crystalline. By conditions $(ii),(iii)$, the Hodge--Tate weights of $\cR_{K,L}(\delta)$ have enough gaps with respect to $\gB$. Therefore, $\Dtri$ is crystalline by \Cref{lincrys}, and so is $D$.
\end{proof}

\section{The trianguline variety}\label{sec:trivar}

In this section, we define the trianguline variety for $\gG$.
This will recover \cite[Définition 2.4]{BHSmodulaire} in the case $\gG=\GL_n$, and the definition in de Daruvar's thesis via very regular parameters \cite[Definition 6.21]{dedar2020}, see \Cref{sec_indep_regularity}.

For basic notions regarding Zariski closed sets in rigid analytic spaces, we refer to \cite[\S 2.1]{BHSmodulaire}.

\subsection{Definition}

Consider our fixed continuous representation $\ovl\rho\colon\GK\to \gG(k)$. 
As usual, let $\Art_{\cO_L}$ be the category of local Artinian $\cO_L$-algebras with residue field $k$. 
We consider the functor $X_{\ovl\rho}\colon\Art_{\cO_L}\to\mathrm{Sets}$ of \emph{framed deformations} of $\ovl\rho$, associating with $A\in\Art_{\cO_L}$ the set of group homomorphisms $\GK\to \gG(A)$ that reduce to $\ovl\rho$ modulo the maximal ideal of $\cO_L$. By \cite{balaji_thesis}, $X_{\ovl\rho}$ is pro-representable by a Noetherian local $\cO_L$-algebra $R_{\ovl\rho}$. We still denote by $X_{\ovl\rho}$ the $\cO_L$-formal scheme $\Spf R_{\ovl\rho}$, and by $\cX_{\ovl\rho}$ its rigid analytic generic fiber in the sense of Berthelot, which is a rigid analytic space over $L$.

Let $\gG$ be our chosen split connected reductive group over $L$, with split maximal torus $\gT$ and Borel subgroup $\gB$ containing $\gT$. 

\begin{definition}
    Let $U_{\tri}(\rhobar) \subseteq \cX_{\rhobar} \times \cT^\gT$ be the subset of pairs $(\rho, \delta)$, such that $\delta$ is $\gB$-regular (\Cref{def_Breg}) and such that $D_{\rig}\circ \eta_{\rho}$ admits a triangulation with parameter $\delta$.
    We define the \emph{trianguline variety} $X_{\tri}(\rhobar)$ associated to $\rhobar$ as the Zariski closure of $U_{\tri}(\rhobar)$ in $\cX_{\rhobar} \times \cT^\gT$ endowed with the reduced subspace structure.
\end{definition}

\begin{remark}
    It follows from \cite[Corollary 15.29]{defG} that $\cX_{\rhobar}$ is normal, so that $\cX_{\rhobar} \times \cT^\gT$ is normal.
\end{remark}

We define
\begin{align}
    \omega' \colon X_{\tri}(\rhobar) \longrightarrow \cT^\gT \label{def_omega_prime}
\end{align}
as the composition of $X_{\tri}(\rhobar) \hookrightarrow \cX_{\rhobar} \times \cT^\gT$ with the projection to the second coordinate. Since $X_\tri(\ovl\rho)$ maps to $\cX_{\ovl\rho}$ by construction, we can pull-back along $\omega^\prime$ the $\GPG$-module over $\cR_{K,\cX_{\ovl\rho}}$ attached to the universal framed deformation of $\ovl\rho$, to obtain a $\GPG$-module $D$ over $\cR_{K,X_\tri(\ovl\rho)}$. By pulling back the universal character $\gT^\vee(K)\to\cO(\cT^\gT)^\times$ on $\cT^\gT$ along $\omega^\prime$, we obtain a continuous character $\delta^\tri \colon \gT^\vee(K)\to\cO(X_\tri(\ovl\rho))^\times$ 

\subsection{The space of rigidified trianguline representations}


In order to prove \Cref{thm_equidimensional} below, we define an auxiliary space of rigidified trianguline deformations of $\ovl\rho$ whose parameter satisfies one of the regularity conditions from \Cref{sec:regularity}.
The strategy is the same as in \cite{chenevier2010sur}, \cite{hellmann2012families}, \cite{HS_density_potentially}, \cite[Théorème 2.6]{BHSmodulaire}, and \cite[Theorem 6.22]{dedar2020}.

\begin{definition}
    Let $\cS_{\gB}(\rhobar) \colon \Aff_L^{\red} \to \Set$ be the functor on \emph{reduced} affinoid $L$-algebras such that $\cS_{\gB}(\rhobar)(A)$ is the set of equivalence classes of quadruples $(\rho, D^{\triangle},  \theta, \alpha)$, where $\rho \colon \GK \to \gG(A)$ is a continuous representation deforming $\rhobar$, in the sense that it corresponds to a map $\Sp(A) \to \cX_{\rhobar}$, and $(D^{\triangle}, \theta)$ is a rigidified $\gB$-$\PG$-module over $\cR_{K,A}$ with $\gB$-regular parameter equipped with an equivalence $$\alpha \colon \eta_{D^{\triangle}} \circ\Res_{\gB}^{\gG} \eqto D_{\rig, A}\circ \eta_\rho.$$
\end{definition}

If $X$ is a reduced rigid analytic space over $L$ and $D$ a $\gG$-$\PG$-module over $X$, de Daruvar \cite[§6.4]{dedar2020} following \cite{hellmann2012familiesPhiGamma} associates with $D$ a rigid analytic space $X^{\adm}$, an étale morphism $\pi \colon X^{\adm} \to X$ and a fiber functor $\eta_\cV \colon \Rep_L(\gG)\to \Rep_{\mathcal{O}(X^{\adm})}(\GK)$ such that $D_{\rig, X}\circ \eta_{\cV} \cong \eta_{\pi^* D}$. The morphism is universal among reduced spaces $\pi' \colon Y \to X$ such that $(\pi')^* D$ comes from a family of Galois representations \cite[Theorem 6.11, Corollary 6.13]{dedar2020}. We refer to $X^{\adm}$ as the \emph{admissible locus} of $D$.

\begin{proposition}\label{dimension_of_SB}
    The functor $\cS_{\gB}(\rhobar)$ is representable by a rigid analytic space which is relatively smooth of relative dimension $\dim \gG_L + \dim \gU_L \cdot [K :\Qp]$ over $\cT^{\gT}$. In particular, if $\cS_{\gB}(\rhobar) \neq \emptyset$, then it is smooth and equidimensional of dimension $\dim \gB_L \cdot ([K:\Qp]+2)$.
\end{proposition}

\begin{proof}
    By \Cref{SB_representable}, $\cS_{\gB}$ carries a universal $\gB$-regular rigidified $\gB$-$\PG$-module.
    We denote by $\cS_{\gB}^{\adm} \to \cS_{\gB}$ the admissible locus of the universal $\gB$-$\PG$-module, which is well-defined as $\cS_{\gB}$ is reduced.
    It carries a universal representation $\cV\colon \Rep_{L}(\gG)\to \Rep_{\cS_{\gB}^{\adm}}(\GK)$. Following the proof of \cite[Theorem 2.6]{BHSmodulaire} (or \cite[Proposition 6.24]{dedar2020}), we see that $\cS_{\gB}(\rhobar)$ is representable by a subspace of the underlying $\gG$-torsor of $\cV$. The statement about dimension follows from \Cref{SB_representable}. 
\end{proof}

We consider the groupoid  $X_{\rho,\Dtri}\colonequals X_{\rho}\times_{X_D}X_{\Dtri}\colon \Art_L\to \Gpd$ (see \Cref{sec35} for the relevant definitions). 
\begin{lemma}\label{trianguline_def_ring}
    Assume that either $\delta\in \cT^{\gT}_{\Lambda}(L)$ or $\delta\in \cT^{\gT}_{r}(L)$.
    Then, $|X_{\rho, D^{\triangle}}|$ is pro-representable by a quotient of the universal framed deformation ring $R_{\rho}$.
\end{lemma}
\begin{proof}
    This follows by base change from \Cref{triimmersion}, and from the pro-representability of $|X_{\rho}|$ by $R_{\rho}$.
\end{proof}

The next lemma is proved in the same way as \cite[Lemma 6.25]{dedar2020}.

\begin{lemma}\label{local_structure_of_SB}
    Let $x = (\rho, D^{\triangle},\theta, \alpha)$ be a closed point of $\cS_{\gB}(\rhobar)$ such that the parameter $\delta$ of $D^{\triangle}$ is either in $\cT^{\gT}_{\Lambda}(L)$ or in $\cT^{\gT}_{r}(L)$. Then $\widehat \cO_{\cS_{\gB}(\rhobar),x}$ is formally smooth of relative dimension $\dim \gT_L$ over $R_{\rho, D^{\triangle}}$.
\end{lemma}

\subsection{Independence of the regularity condition}
\label{sec_indep_regularity} 
In de Daruvar's thesis, the trianguline variety is defined as the closure of the subspace $U_{\tri}^{\Omega}(\rhobar) \subseteq U_{\tri}(\rhobar)$, which consists of points $(\rho, \delta)$ such that $\delta$ is very regular \cite[Definition 6.21]{dedar2020}. This definition turns out to be equivalent to ours, as a consequence of \Cref{independence_of_regularity_condition} below.

\begin{lemma}\label{density_lemma}
    Let $X$ be a smooth connected rigid analytic space.
    Assume that $\emptyset \neq U \subseteq X$ is Zariski open.
    Then $U$ is Zariski dense.
\end{lemma}

\begin{proof} Let $x \in U$.
For every point $y \in X$ there is a sequence $Z_1, \dots, Z_n$ of connected affinoids of $X$ such that $x \in Z_1$, $y \in Z_n$ and $Z_i \cap Z_{i+1} \neq \emptyset$ for all $i$ (see \cite[p. 492]{ConradIrrComp}).
Then $Z_1 \cap U$ is Zariski open and Zariski dense in $Z_1$ by smoothness. Since $Z_1 \cap Z_2 \neq \emptyset$, it contains an affinoid $V$. Since $V$ is not contained in any proper Zariski closed subset of $Z_1 \cap Z_2$, we have $V \cap U \neq \emptyset$. We continue by induction to conclude that $U \cap Z_i \neq \emptyset$ is Zariski dense in $Z_i$ for all $i$. It follows that $U$ is Zariski dense in $X$.
\end{proof}

\begin{proposition}\label{independence_of_regularity_condition}
    Let $\cT^* \subseteq \cT_{0,\gB}^{\gT}$ be a Zariski open and Zariski dense subset.
    Then its preimage $U_{\tri}^*(\rhobar)$ is Zariski dense in $U_{\tri}(\rhobar)$.
    In particular, $U_{\tri}^{\vreg}(\rhobar)$ is Zariski dense in $U_{\tri}(\rhobar)$.
\end{proposition}

\begin{proof}
    Denote by $\cS_{\gB}^*$ and $\cS_{\gB}^*(\rhobar)$ respectively the preimages of $\cT^*$ in $\cS_{\gB}$ and $\cS_{\gB}(\rhobar)$.
    By the universal property of $\pi$, we obtain a map $j : \cS_{\gB}^*(\rhobar) \to \cS_{\gB}(\rhobar)$ which fits in the following commutative square:
    \begin{center}
        \begin{tikzcd}
            \cS_{\gB}^*(\rhobar) \ar[r, "j"] \ar[d] & \cS_{\gB}(\rhobar) \arrow[d, "\pi"] \\
            \cS_{\gB}^* \ar[r] & \cS_{\gB}.
        \end{tikzcd}
    \end{center}
    By the universal property of the admissible locus, this square is cartesian. Thus $j$ is an open immersion.
    
    We first show that $\cS^*_{\gB}(\rhobar) \subseteq \cS_{\gB}(\rhobar)$ is Zariski dense.
    We may work with a fixed connected component of $\cS_{\gB}(\rhobar)$; we denote it by $\cS_{\gB}(\rhobar)^0$ and its preimage in $\cS_{\gB}^*(\rhobar)$ by $\cS_{\gB}^*(\rhobar)^0$.
    Since $\cS_{\gB}(\rhobar)^0$ is smooth by \Cref{SB_representable} and connected, it is sufficient by \Cref{density_lemma} to show that $\cS_{\gB}^*(\rhobar)^0 \neq \emptyset$.
    
    The very regular parameters are Zariski dense in $\cT_{0,\gB}^{\gT}$. 
    By \Cref{SB_representable}, $\cS_{\gB}$ is a vector bundle over $\cT_{0,\gB}^{\gT}$.
    Hence, $\cS_{\gB} \setminus \cS_{\gB}^*$ is Zariski closed of codimension $> 0$ in $\cS_{\gB}$.
    Since $\pi$ is étale and morphisms of rigid spaces are continuous for the Zariski topology (pullback of coherent ideals are coherent), $\cS_{\gB}(\rhobar)^0 \setminus \cS_{\gB}^*(\rhobar)^0$ is also Zariski closed of codimension $> 0$.
    Since $\cS_{\gB}(\rhobar)^0$ is a connected component, it is non empty, hence $\cS_{\gB}^*(\rhobar)^0 \neq \emptyset$.

    Consider the following commutative square, in which the vertical arrows are surjective by definition:
    \begin{center}
        \begin{tikzcd}
            \cS_{\gB}^*(\rhobar) \ar[r, "j"] \ar[d, twoheadrightarrow] & \cS_{\gB}(\rhobar) \arrow[d, "\pi_{\rhobar}", twoheadrightarrow] \\
            U_{\tri}^*(\rhobar) \ar[r] & U_{\tri}(\rhobar).
        \end{tikzcd}
    \end{center}
    Let $Z \subseteq U_{\tri}(\rhobar)$ be a Zariski closed subset with $U_{\tri}^*(\rhobar) \subseteq Z$.
    Then $\cS_{\gB}^*(\rhobar) \subseteq \pi_{\rhobar}^{-1}(Z)$, and by the density we have just shown, we obtain $\pi_{\rhobar}^{-1}(Z) = \cS_{\gB}(\rhobar)$.
    Hence $Z = U_{\tri}(\rhobar)$.
    We conclude that $U_{\tri}^*(\rhobar)$ is Zariski dense in $U_{\tri}(\rhobar)$.
\end{proof}

\begin{remark}\label{dimS*}
With the notation from the proof of \Cref{density_lemma}, $\cS^*_{\gB}(\rhobar)$ is Zariski open in $\cS_{\gB}(\rhobar)$, and we showed that it is also Zariski-dense. In particular, it is equidimensional of the same dimension as $\cS_{\gB}(\rhobar)$.
\end{remark}

\begin{remark}
    Suppose we apply \Cref{independence_of_regularity_condition} to the set $\cT^*$ of $\frakS$-regular and/or very regular trianguline parameters. We do not need to assume that $\frakS$-regular and/or very regular trianguline lifts exist.
    Indeed, the proof shows that every connected component of $\cS_{\gB}(\rhobar)$ also contains an $\frakS$-regular and very regular trianguline lift.
    We expect that every connected component of $\cX_{\rhobar}$ contains an $\frakS$-regular and very regular trianguline lift. We leave the verification of this to future work.
\end{remark}

\subsection{Smooth loci on the trianguline variety}


\begin{lemma}[{cf. \cite[Example 3.13]{dedar2020}}]\label{space_of_rigidifications}
    Let $X$ be a rigid analytic space over $L$ and let $D$ be a $\gT$-$\PG$-module over $\cR_{K,X}$ of parameter $\delta : \gT^{\vee}(K) \to \cO(X)^{\times}$.
    Then the functor $\Rig(D) \colon \Rig_X \to \Set$, which associates with a rigid analytic space $f \colon Y \to X$ the set of isomorphisms $\cR_{K,Y}(\delta) \to f^*D$ of $\gT$-$\PG$-modules, is representable by a $\gT$-torsor over $X$.
\end{lemma}

If $D^{\triangle}$ is a $\gB$-$\PG$-module over $\cR_{K,X}$, we also write $\Rig(D^{\triangle}) \colonequals \Rig(D^{\triangle} \times^{\gB} \gT)$. We refer to $\Rig(D^{\triangle})$ as the \emph{space of rigidifications} of $D^{\triangle}$.

The following is a generalization of results of Kedlaya--Pottharst--Xiao and de Daruvar.

For a point $x=(\rho_x,\delta_x)$ of $X_\tri(\ovl\rho)$, defined over a field $\kappa(x)$, write $\eta_{D_x}=D_{\rig}\circ\eta_{\rho_x}$ for the associated $\PG$-module over $\cR_{K,\kappa(x)}$, and  let $\cM_x=D_x[\tfrac{1}{t}]$. 


\begin{lemma}\label{Zariski_dense_set_trick}
    Let $A\in \Aff_L$, let $0 < r \leq C(\pi_K)$, and let $S \subseteq \Sp(A)$ be a Zariski dense subset.
    Then $S \times A^1(0,r]$ is Zariski dense in $\Sp(A) \times A^1(0,r]$. In particular it defines a Zariski dense set of maximal ideals in $\Spec(\cR^r_{K,A})$. The same is true if we invert $t$ on $A^1(0,r]$, i.e. consider $\Spec(\cR^r_{K,A}[\tfrac{1}{t}])$.
\end{lemma}

\begin{proposition}\label{blowupLambda}
Let $\Lambda\subset X_+^*(\gT^\der)$ be a finite set of $\Q$-generators of $X^*_+(\gT^{\der})_\Q$. There exist
\begin{itemize}
    \item a projective birational morphism $f \colon X^\prime\to X_\tri(\ovl\rho)$, which is a composition of blowups and normalizations,
    \item for every $\lambda\in \Lambda$, a line bundle $\mathscr{L}_{\lambda}$ on $X'$ and a morphism
\begin{equation}\label{ui} u_\lambda \colon \cR_{K,X^\prime}(f^*\delta^\tri\circ\lambda^\vee)\otimes_{X^\prime}\mathscr{L}_{\lambda}\longrightarrow \eta_{f^* D^\tri}(V_{\lambda}), \end{equation}
where $V_{\lambda}$ denotes the representation of $\gG$ of highest weight $\lambda$, and $D^\tri$ denotes the universal $\gG$-$\PG$-module over $X_\tri(\ovl\rho)$,
\item a Zariski open, Zariski dense subspace $U\subseteq X^\prime$ containing $f^{-1}(U_\tri^{\Lambda}(\rhobar))$,
\end{itemize}
such that:
\begin{itemize}
    \item[(1)] for every $\lambda\in \Lambda$, the cokernel of the dual map 
    \begin{equation*} u_\lambda^\vee \colon \eta_{f^* D^\tri}(V^\vee_{\lambda}) \longrightarrow \cR_{K,X^\prime}(-(f^*\delta^\tri)\circ\lambda^\vee) \otimes_{X^\prime}\mathscr{L}^\vee_{\lambda} \end{equation*}
    is locally annihilated by a power of $t$, and supported on $X^\prime\setminus U$;
    \item[(2)] every $x\in U$ admits an affinoid open neighborhood $U'$ contained in $U$ such that for every $\lambda\in \Lambda$, $u_\lambda\vert_{U'}$ defines a Pl\"ucker datum for $f^\ast D^\tri\vert_{U'}$, corresponding to a triangulation of parameter $f^\ast\delta^\tri\vert_{U'}$.
\end{itemize}
\end{proposition}
\begin{proof}
    For $x\in U_\tri^\Lambda(\overline{\rho})$, let $D_x$ be the corresponding $\gG$-$\PG$-module, which admits a triangulation $\Dtri_x$ of parameter $\delta_x$. Then by \Cref{rank_oner}, 
$$\Hom_{\PGcat_{K,\kappa(x)}^{+}}(\cR_{K,\kappa(x)}(\delta_x \circ \lambda^{\vee}), \eta_{D_x}(V_{\lambda}))=H^0_{\varphi,\Gamma_K}(\eta_{D_x}(V_{\lambda})(-\delta_x\circ \lambda^\vee))$$ is free of rank one over $\kappa(x)$. Moreover, for a choice of basis of this space, the image of $\cR_{K,\kappa(x)}$ in $\eta_{D_x}(V_{\lambda})$ is saturated. Therefore, we may apply \cite[Theorem 6.3.9]{KPX} successively to $\eta_{D_x}(V_{\lambda}^\vee)$ for all $\lambda\in \Lambda$, since $\Lambda$ is finite. This yields a proper birational morphism $f\colon X'\to X_{\tri}(\overline{\rho})$, line bundles $\mathscr{L}_{\lambda}$ over $X'$, morphisms
    \begin{equation*} u_\lambda^\vee \colon \eta_{f^\ast D^\tri}(V^\vee_{\lambda}) \longrightarrow \cR_{K,X^\prime}(-(f^\ast\delta^\tri)\circ\lambda^\vee) \otimes_{X^\prime}\mathscr{L}^\vee_{\lambda} \end{equation*}
 for $\lambda\in \Lambda$, and a Zariski open subset $U$ of $X'$ containing $f^{-1}(U_{\tri}^{\Lambda}(\overline{\rho}))$ such that $(1)$ holds, and for all $x\in U$, the specialization
 $$ u_{\lambda,x}^\vee \colon \eta_{D_x}(V^\vee_{\lambda}) \longrightarrow \cR_{K,\kappa(x)}(-\delta_x\circ\lambda^\vee) $$
 is surjective.
 We let $u_{\lambda}$ be the dual map of $u_{\lambda}^\vee$, and it remains to show that $u_{\lambda}$ satisfies $(2)$. We can work locally on $X'$, and assume that $X'=\Sp(A)$ and that the line bundles $\mathscr{L}_{\lambda}$ are trivial. Indeed, the second and the third bullet point of the statement remain true: since $U$ is Zariski open it remains Zariski dense after being cut down to an affinoid. We first show that for $\mu = \sum_{\lambda \in \Lambda} n_\lambda \lambda \in X^*_+(\gT)$, the image of 
    $$ u_{\mu} \colon \bigotimes\nolimits_{\lambda \in \Lambda}  \cR_{K,A}(-  (f^*\delta^\tri)\circ\lambda^\vee)^{\otimes n_{\lambda}} \longrightarrow \eta_{f^*D^\tri}\big(\bigotimes\nolimits_{\lambda \in \Lambda} V_\lambda^{\otimes n_\lambda}\big) $$
    is contained in $\eta_{f^*D^\tri}(V_\mu)$.
    As $V_\mu$ is a direct summand of $\bigotimes\nolimits_{\lambda \in \Lambda} V_\lambda^{\otimes n_\lambda}$, the quotient $$W \colonequals \eta_{f^*D^\tri}\left(\bigotimes\nolimits_{\lambda \in \Lambda} V_\lambda^{\otimes n_\lambda}\right)/\eta_{f^*D^\tri}(V_\mu)$$ is  free over $\cR_{K,A}$. Given the projection
    $$ \pi \colon \eta_{f^*D^\tri}\left(\bigotimes\nolimits_{\lambda \in \Lambda} V_\lambda^{\otimes n_\lambda}\right) \longrightarrow W, $$
    we have to show $\pi \circ u_{\mu} = 0$. Since the source and target are free modules over $\cR_{K,A}$,
    it is sufficient to show $(\pi \circ u_{\mu}) \otimes_{\cR_{K,A}} \cR_{K,A}/\frakm = 0$ for a Zariski dense set of maximal ideals $\frakm$ of $\cR_{K,A}$. We get such a set by applying \Cref{Zariski_dense_set_trick} to $S=f^{-1}(U_\tri^{\Lambda}(\overline{\rho}))$, so we may assume that the maximal ideal $\frakm_s$ of $A$ corresponding to a point $s \in S$ maps into $\frakm$. We get $(\pi \circ u_{\mu}) \otimes_A \kappa(s) = 0$ since by our assumption on $S$, the $\im(u_{\lambda,s})$ define a Plücker datum for $s \in S$. 

Let $U'=\Sp(B)\subseteq U$ be a non-empty admissible open. The morphisms $u_{\lambda}|_{U'}$ are injective with saturated image: indeed, since both the source and target are finite projective modules over $\cR_{K,B}$, this is equivalent to $u_\lambda^\vee|_{U'}$ being surjective, which follows from $(1)$.

To show that the maps $u_{\lambda}$ define a Plücker datum on $U'$, it remains to show that the image of $u_{\mu}|_{U'}$ in $\eta_{f^*D^\tri|_{U'}}(V_{\mu})$ is independent of the expression of $\mu$ in terms of $\Lambda$. This follows by the same argument as above since the image is saturated, and so the quotient is a free $\cR_{K,B}$-module.
\end{proof}

\begin{proposition}\label{blowupr}
   Let $r\colon \gG\to \GL_m$ be a faithful representation satisfying the hypotheses of \Cref{emb_lemma_asdfjkl}. There exists a proper birational morphism $f\colon X'\to X_{\tri}(\overline{\rho})$, which is a composition of blowups and normalizations, together with a Zariski open Zariski dense subspace $U\subseteq X'$ containing $f^{-1}(U_{\tri}^r(\overline{\rho}))$, such the pullback $f^*D^{\tri}|_{U}$ admits a triangulation of parameter $f^*\delta^{\tri}|_{U}$.
\end{proposition}
\begin{proof}
    We apply \cite[Theorem 6.3.10]{KPX} to $X_{\tri}(\overline{\rho})$ and the $\PG$-module $\eta_{D^{\tri}}(V_r)$ over $\cR_{K,X_{\tri}(\overline{\rho})}$. By the diagram \eqref{pushoutdiag}, a triangulation on $\eta_{f^*D^{\tri}|_{U}}(V_r)$ induces one on $f^*D^{\tri}|_{U}$.
\end{proof}

The morphism $\cS_{\gB}(\rhobar) \to \cX_{\rhobar} \times \cT^{\gT}$ sending $(\rho, D^{\triangle},\theta, \alpha)$ to $(\rho, \delta)$, where $\delta$ is the parameter of $\Dtri$, has set-theoretic image $U_{\tri}(\rhobar)$, and thus factors through a morphism
\begin{align}
    \pi_{\rhobar} \colon \cS_{\gB}(\rhobar) \longrightarrow X_{\tri}(\rhobar). \label{def_pi_rhobar}
\end{align}
The situation is summarized in the following diagram:
\begin{equation}
\begin{tikzcd}
    \cS_{\gB}(\rhobar) \ar[r] \ar[rr, bend left, "\pi_{\rhobar}"] \ar[dr] & U_{\tri}(\rhobar) \ar[r] \ar[d] & X_{\tri}(\rhobar) \ar[r] & \cX_{\rhobar} \times \cT^{\gT} \ar[d] \\
    & \cT^{\gT, \gB}_0 \ar[rr] & & \cT^{\gT}.
\end{tikzcd}
\end{equation}

We introduce the loci in $\cS^{*}_{\gB}(\rhobar)$ and $U_\tri(\rhobar)$ corresponding to the regularity conditions on the parameter discussed in \Cref{sec:regularity}.
\begin{definition}
For $*\in\{\Lambda,r\}$, let $\cS^{*}_{\gB}(\rhobar)$ be the subset of points $(\rho, D^{\triangle},\theta, \alpha)$ of $\cS_{\gB}(\rhobar)$ such that the parameter $\delta$ of $D^{\triangle}$ belongs to $\cT^{\gT}_{*}(\kappa(x))$. 
Let $U_{\tri}^{*}(\rhobar)$ be the subset of points $(x,\delta)$ of $U_\tri(\rhobar)$ such that $\delta$ belongs to $\cT^{\gT}_{*}(\kappa(x))$. 
\end{definition}

\begin{theorem}\label{thm_equidimensional} Given $*\in\{\Lambda,r\}$, assume that $\cS^{*}_{\gB}(\rhobar) \neq \emptyset$. Then: 
    \begin{enum}
        \item $X_{\tri}(\rhobar)$ is equidimensional of dimension $\dim \gG_L + [K:\Qp] \cdot \dim \gB_L$;
        \item $\pi_{\rhobar}|_{\cS^{*}_{\gB}(\rhobar)}$ is smooth of relative dimension $\dim \gT_L$;
        \item $U^{*}_{\tri}(\rhobar)$ is Zariski open and Zariski dense in $X_{\tri}(\rhobar)$;
        \item $U_{\tri}^{*}(\rhobar)$ is smooth and $\omega'|_{U_{\tri}^{*}(\rhobar)}$ is smooth.
    \end{enum}
\end{theorem}


\begin{proof}
    We follow closely the proof of \cite[Théorème 2.6]{BHSmodulaire} and \cite[Theorem 6.22]{dedar2020}.
    Let $\rho \colon \GK \to \gG(\cO(X_{\tri}(\rhobar)))$ be the pullback of the universal representation along $X_{\tri}(\rhobar) \to \cX_{\rhobar}$, let $D^\tri$ be the associated $\gG$-$\PG$-module, and $\delta^\tri \colon \gT^{\vee}(K) \to \cO(X_{\tri}(\rhobar))^{\times}$ the pullback of the universal parameter along $X_{\tri}(\rhobar) \to \cT^{\gT}$.
    Then Propositions \ref{blowupLambda} and \ref{blowupr} provide us with data $f \colon X' \to X_{\tri}(\rhobar)$, $u_i$, $\mathscr{L}_i$ and $U \subseteq X'$ satisfying the conditions stated there. By intersecting $U$ with the preimage of the space of parameters in $\cT^{\gT}_*$ under the map $X' \to \cT^{\gT}$, we may assume that $f^*\delta^\tri|_U$ is in $\cT^{\gT}_*$.

    By \textit{loc. cit.}, we have a triangulation $D^{\triangle}$ of $f^* D^{\tri}|_U$ of parameter $f^*\delta^\tri|_U$. 
    Let $U^{\square} \colonequals \Rig(D^{\triangle})$ be the space of rigidifications of $D^{\triangle}$ as defined after \Cref{space_of_rigidifications}, which is a $\gT$-torsor over $U$. The pullback $D^{\triangle}|_{U^{\square}}$ of $D^{\triangle}$ to $U^{\square}$ admits a universal rigidification $\theta \colon \cR_{K,U^{\square}}(\delta|_{U^{\square}}) \to f^*D^{\tri}|_{U^{\square}}$.
    Since $D^{\tri}$ comes from a representation, the universal rigidification corresponds to a map $s \colon U^{\square} \to \cS^{*}_{\gB}(\rhobar)$ such that $\pi_{\rhobar} \circ s$ is equal to the composition $U^{\square} \to U \to X_{\tri}(\rhobar)$.

    Since $f$ is a composition of blow-ups and normalizations, there is a Zariski dense and Zariski open subset $W \subseteq X_{\tri}(\rhobar)$ such that $f^{-1}(W) \subseteq U$ and $f$ induces an isomorphism between $f^{-1}(W)$ and $W$. Let $W^{\square}$ be the space of rigidifications of $D^{\triangle}|_W$ as above.
    The space $\pi_{\rhobar}^{-1}(W)$ represents the functor $\Rig_L \to \Set$, which associates with a rigid analytic space $Y$ over $L$ a quadruple $(g, D^{\triangle}_Y,\theta, \alpha)$, where $g \colon Y \to W$ is a morphism in $\Rig_L$, $D^{\triangle}_Y$ is a $\gB$-$\PG$-module over $\cR_{K,Y}$, $\alpha : D^{\triangle}_Y \times^{\gB} \gG \to g^*\rho|_W$ is an isomorphism of $\gG$-$\PG$-modules and $\theta : \cR_{K,Y}(g^*\delta^\tri) \to D^{\triangle}_Y \times^{\gB} \gT$ is a rigidification. 
    Since $g^*\delta|_W$ is in $\cT^{\gT}_*$, \Cref{uniqueness_Lambda_1} and \Cref{uniqueness_r_1} imply that $(D^{\triangle}_Y, \alpha)$ is the unique triangulation of $g^*D$, hence must coincide with $D^{\triangle}|_W$. We conclude that $W^{\square}$ and $\pi_{\rhobar}^{-1}(W)$ represent the same functor and are therefore canonically isomorphic as spaces over $X_{\tri}(\rhobar)$.

    In the following computation, we use the facts that $W^{\square}$ is a $\gT$-torsor over $W$, that $W \subseteq X_{\tri}(\rhobar)$ is Zariski open and Zariski dense, that $\dim \pi_{\rhobar}^{-1}(W) = \dim \cS_{\gB}^{*}(\rhobar)$, and that $\cS_{\gB}^{*}(\rhobar)$ is equidimensional of the same dimension as $\cS_{\gB}(\rhobar)$ (see \Cref{dimS*}), that we computed in \Cref{dimension_of_SB}. In particular, it will follow that $X_{\tri}(\rhobar)$ is equidimensional.
    The equality $\dim \pi_{\rhobar}^{-1}(W) = \dim \cS^{*}_{\gB}(\rhobar)$ follows by equidimensionality of $\cS^{*}_{\gB}(\rhobar)$ and our assumption that $\cS^{*}_{\gB}(\rhobar) \neq \emptyset$, so that $W \neq \emptyset$ and $\pi_{\rhobar}^{-1}(W) \cong W^{\square} \neq \emptyset$ as a torsor over $W$. We have
    \begin{align*}
        \dim X_{\tri}(\rhobar) &= \dim W^{\square} - \dim \gT_L = \dim \pi_{\rhobar}^{-1}(W) - \dim \gT_L \\
        &= \dim \cS^{*}_{\gB}(\rhobar) - \dim \gT_L = \dim \gG_L + \dim \gB_L \cdot [K : \Qp]
    \end{align*}
    This completes the proof of (1).

    Let $x \in \cS^{*}_{\gB}(\rhobar)$, let $y \colonequals \pi_{\rhobar}(x) \in U^{*}_{\tri}(\rhobar)$, let $A \colonequals \widehat\cO_{X_{\tri}(\rhobar),y}$ and $B \colonequals \widehat \cO_{\cS^{*}_{\gB}(\rhobar),x}$. Let $d \colonequals \dim \gT_L$. We will show that $B$ is a formal power series ring over $A$ in $d$ variables.
    The projection map $X_{\tri}(\rhobar) \to \cX_{\rhobar}$ induces a continuous map $R_{\rho_y} \cong \widehat \cO_{\cX_{\rhobar}, \rho_y} \to \widehat\cO_{X_{\tri}(\rhobar),y}$, where the first isomorphism with the universal framed deformation ring $R_{\rho_y}$ of $\rho_y$ follows from \cite[Theorem 1.5.3]{balaji_thesis}.
    Let $D_y^{\triangle}$ be the triangulation of $\rho_y$ of parameter $\delta_y$, which is unique by \Cref{uniqueness_Lambda_1} and \Cref{uniqueness_r_1}. Then by \Cref{local_structure_of_SB}, we have an isomorphism $B \cong R_{\rho_y, D^{\triangle}_y}\br{x_1, \dots, x_d}$.

    We define a continuous $A$-linear map $f\colon A\br{X_1, \dots, X_d} \to B$ by sending $X_i \mapsto x_i$.
    By \Cref{trianguline_def_ring}, the map $R_{\rho_y} \to R_{\rho_y, D^{\triangle}_y}$ is surjective, so $f$ is surjective.
    By \Cref{dimension_of_SB} and part (1), we have $\dim A + \dim \gT_L = \dim B$, so it is sufficient\footnote{If the map were not an isomorphism the kernel would contain a nonzerodivisor $z$, which would form a regular sequence, so that $\dim A\br{X_1, \dots, X_d}/(z) = \dim A + d-1$.} to show that $A$ is a domain.
    Since $X_{\tri}(\rhobar)$ is reduced, by \cite[\S 7.2.6 Proposition 3]{BG84} $\cO_{X_{\tri}(\rhobar),y}$ is reduced, so by \cite[Theorem 1.1.3]{ConradIrrComp}, \stackcite{07NZ} and \stackcite{07QV} $A$ is reduced. Hence, we need to show that $A$ has a unique minimal prime ideal.
    Let $\widetilde X_{\tri}(\rhobar)$ be the normalization of $X_{\tri}(\rhobar)$.
    It suffices to show that the fiber over $y$ in the normalization is a point.

    Since $\pi_{\rhobar}$  contains a torsor over a Zariski dense and Zariski open subset of its target, $\pi_{\rhobar}$ is dominant.
    By the universal property of the normalization, crucially using the dominance of $\pi_{\rhobar}$, and the fact that $\cS^{*}_{\gB}(\rhobar)$ is smooth, there is a unique factorization $\cS^{*}_{\gB}(\rhobar) \xrightarrow{h} \widetilde X_{\tri}(\rhobar) \xrightarrow{\nu} X_{\tri}(\rhobar)$.  
    As shown earlier, there is a canonical isomorphism of $W^{\square}$ with the subspace $\pi_{\rhobar}^{-1}(W)\subset\cS_{\gB}^{*}(\rhobar)$. In particular, since $y\in U_\tri^*(\rhobar)\subset W$,
    the fiber of $\pi_{\rhobar}$ over $y$ is isomorphic to $\gT_L$ by the uniqueness of the triangulation (\Cref{uniqueness_Lambda_1} and \Cref{uniqueness_r_1}), and it is in particular connected.
    By this factorization, we have that $\nu^{-1}(\{y\})$ is contained in the image of $h$.
    Therefore $\nu^{-1}(\{y\})$ is connected, hence a singleton. This completes the proof of (2).

    Part (3) follows from (2) by the same argument as in the proof of \cite[Théorème 2.6]{BHSmodulaire}.
    Part (4) follows as in \cite[Théorème 2.6]{BHSmodulaire}.
\end{proof}

\subsection{Points on the trianguline variety are trianguline}

We show that, even outside of the smooth locus $U_\tri^*(\rhobar)$, the Galois representation $\rho_x$ attached to a point $(\rho_x,\delta_x)\in X_{\tri}(\overline{\rho})$ admits a triangulation. However the corresponding parameter need not coincide with $\delta_x$ in general, but differs from it by an algebraic character. This is precisely why we introduced $\Lambda$-regularity in place of the the very regularity condition of \cite{dedar2020}: choosing $\Lambda$ to be the Hilbert basis of $X_+^*(\gT^{\der})$ provides the necessary control over this discrepancy.

\begin{proposition}\label{prop:boundary}
    For all $x\in X_{\tri}(\overline{\rho})$, the associated $\PG$-module $D_x$ over $\cR_{K,\kappa(x)}$ admits a triangulation with some parameter $\widetilde{\delta}_x\colon \gT^\vee(K)\to \kappa(x)^\times$ such that $\tilde \delta_x \delta_x^{-1}$ is algebraic.
\end{proposition}
\begin{proof}
    Let $f\colon X'\to X_{\tri}(\overline{\rho})$ be the morphism of \Cref{blowupLambda} applied in the case $\Lambda=\mathfrak{S}$ is the Hilbert basis of $X_+^*(\gT^{\der})$. For any $x'\in f^{-1}(x)$ and all $\lambda\in \mathfrak{S}$,  we have $\eta_{D_{x'}}(V_{\lambda})\cong\eta_{D_x}(V_{\lambda})\otimes_{\kappa(x)}\kappa(x')$, so we may work with $x'$ in what follows.
    \\ Let $u_{\lambda,x'}^\vee \colon \eta_{D_{x'}}(V_{\lambda}^\vee)\to \cR_{K,\kappa(x')}(-\delta_{x'}\circ \lambda^\vee)$ be the specialization at $x'$ of the map $u_{\lambda}^\vee$ from \Cref{blowupLambda}. Since its cokernel is killed by a power of $t$, the map is non-trivial, and so its dual
    $$u_{\lambda,x'} \colon \cR_{K,\kappa(x')}(-\delta_{x'}\circ \lambda^\vee)\to  \eta_{D_{x'}}(V_{\lambda})$$
    is injective. Let $\mathscr{L}^{\sat}_{\lambda,x'}\subset \eta_{D_{x'}}(V_{\lambda})$ denote the saturation of its image. We will show that this collection of lines define a Plücker datum for $D_{x'}$. By the proof of \Cref{blowupLambda}, for any $\mu = \sum_{\lambda \in \mathfrak{S}} n_\lambda \lambda \in X^*_+(\gT^\der)$, the image of $\otimes_{\lambda\in \mathfrak{S}}(\mathscr{L}^{\sat}_{\lambda,x'})^{\otimes n_{\lambda}
    }$ lands in $\eta_{D_x'}(V_{\mu})$. It remains to show that if $\mu=\sum_{\lambda\in \mathfrak{S}}m_{\lambda}\lambda$ is another such  expression, then $\otimes_{\lambda\in \mathfrak{S}}(\mathscr{L}^{\sat}_{\lambda,x'})^{\otimes n_\lambda}=\otimes_{\lambda\in \mathfrak{S}}(\mathscr{L}^{\sat}_{\lambda,x'})^{\otimes m_\lambda}$ as submodules of $\eta_{D_x'}(V_{\mu})$. It suffices to prove this equality after inverting $t$ by \cite[Proposition 2.2.2]{BC}. For this, we may work locally on $X'$ and assume that $X'=\Sp(A)$. Let $W'$ be the cokernel of the map
    $$u_{\mu}[\tfrac{1}{t}] \colon \bigotimes\nolimits_{\lambda \in \Lambda}  \cR_{K,A}(-  (f^*\delta^\tri)\circ\lambda^\vee)^{\otimes n_{\lambda}}[\tfrac{1}{t}] \longrightarrow \eta_{f^*D^\tri}\big(V_{\mu})[\tfrac{1}{t}].$$
    Then $W'$ is a locally free $\cR_{K,A}[\tfrac{1}{t}]$-module. We need to show that $\otimes_{\lambda\in \mathfrak{S}}(\mathscr{L}^{\sat}_{\lambda,x'})^{\otimes m_\lambda}[\tfrac{1}{t}]$ is in the kernel of the projection to $W'$. It suffices to check this on a Zariski dense set of maximal ideals of $\cR_{K,A}[\tfrac{1}{t}]$, and such a set is provided by applying \Cref{Zariski_dense_set_trick} to $S=f^{-1}(U_\tri^{\frakS}(\overline{\rho}))$. This establishes that the  equality of tensor products holds, and we obtain a triangulation of $D_x$. Let $\widetilde \delta_x$ be the corresponding parameter.
    By \cite[Lemma 3.3.4]{BHS19}, $(\widetilde \delta_x \delta_x^{-1}) \circ \lambda^{\vee}$ is algebraic for every $\lambda \in \frakS$.
    Since $\frakS$ generates $X^*(\gT^{\der})$, it follows that $\widetilde \delta_x \delta_x^{-1}$ is algebraic.
    \end{proof}

\begin{remark}\label{rem:boundary}
    The same argument as in \Cref{prop:boundary} would give a $\gG$-analogue of \cite[Corollary 6.3.13]{KPX}, i.e. a triangulinity result over any rigid analytic space equipped with a densely pointwise $\frakS$-regular trianguline $\PG$-module.
\end{remark}

\begin{remark}
The proof of \Cref{prop:boundary} is constructive, in the sense that it produces a parameter $\wtl\delta_x$ and a triangulation of $D_x$ of this parameter rather than just showing that objects with the desired properties exist.
\end{remark}

Keep the notation of \Cref{prop:boundary}. Recall that to a Hodge--Tate representation $\rho\colon\GK\to\gG(L)$ we can attach a Hodge--Tate cocharacter $\varpi_\HT\colon\Gm\to\bT$ via \Cref{def_HT}, and that to an algebraic parameter $\delta\colon\gT^\vee(K)\to L^\times$ we can attach an algebraic cocharacter $\Gm\to\bT$ as in \Cref{algdual}. 

\begin{corollary}\label{lemmaKPX6313}
Let $x=(\rho_x,\delta_x)$ be a point of $X_\tri(\rhobar)$ and let $\widetilde \delta_x$ be the parameter of some triangulation of $\rho_x$.
The following are equivalent:
\begin{enumerate}[label=(\roman*)]
    \item $\rho_x$ is Hodge--Tate;
    \item $\rho_x$ is almost de Rham;
    \item $\delta_x$ is locally algebraic;
    \item $\wtl\delta_x$ is locally algebraic.
\end{enumerate}
Moreover, if any of the above hold, then both $\delta_x$ and $\wtl\delta_x$ are the product of a smooth character with an algebraic character whose dual is conjugate to the Hodge--Tate cocharacter $\varpi_{\HT,x}$.
\end{corollary}

\begin{proof}
Since $\rho_x$ admits a triangulation of parameter $\wtl\delta_x$, $(i)$, $(ii)$ and $(iv)$ are equivalent by \Cref{localgpdR}. The equivalence between $(iii)$ and $(iv)$ follows from \Cref{prop:boundary} since the $\frakS$-regular points are Zariski-dense in $X_\tri(\rhobar)$ by \Cref{independence_of_regularity_condition}, so we only need to prove the last statement.

Assume that $\rho_x$ is almost de Rham. Since $\rho_x$ admits a triangulation of parameter $\wtl\delta_x$, by \Cref{Worbit} the weight of $\wtl\delta_x$ is conjugate to the derivative of the Hodge--Tate cocharacter, hence $\wtl\delta_x$ is the product of a smooth character with an algebraic character whose dual is conjugate to the Hodge--Tate cocharacter $\varpi_{\HT,x}$. This gives the last statement about $\wtl\delta_x$.

Assuming $(i)$ again, we prove the last statement about $\delta_x$. 
Let $U=\Spm A\subset X$ be an open affinoid neighborhood of $x$, and let $\delta_U\colon\gT^\vee(K)\to A^\times$ be the corresponding parameter. For an open dense set $S$ of points $y=(\rho_y,\delta_y)\in U$, $\delta_y$ is the parameter of a triangulation $\rho_y$, and moreover for a dense subset of such points, $\delta_y$ is locally algebraic, so that $\rho_y$ is almost de Rham by \Cref{localgpdR}. Therefore, by \Cref{Worbit}, the weight of $\delta_y$ coincides with the derivative of $\varpi_{\HT,y}$ for every $y\in S$. Since $S$ is dense in $U$, the weight of $\delta_U$ coincides with the trace of the Sen operator on $U$. By specializing at $x$, we deduce that the weight of $\delta_x$ coincides with the trace of the Sen operator at $x$. Since $\rho_x$ is almost de Rham, this implies that the weight of $\delta_x$ is the derivative of the Hodge--Tate cocharacter at $x$, i.e. that $\delta_x$ is of the required form.
%
\end{proof}



\subsection{The Weyl group elements attached to a point of the trianguline variety}\label{sec:weyl}

With the above notation, let $x=(\rho_x,\delta_x)$ be a point of $X_\tri(\ovl\rho)$ with $\delta_x$ locally algebraic. We attach to $x$ three cocharacters $\Gm\to\bT$, and two elements of the Weyl group $\bW$ measuring their distance. 

\begin{itemize}
\item[(a)] Since $\rho_x$ is Hodge--Tate, \Cref{def_HT} attaches to it an antidominant Hodge--Tate cocharacter
 $$\varpi_{\HT,x}\colon \Gm\longrightarrow \bT.$$ 

\item[(b)] 
By \Cref{lemmaKPX6313}, $\delta_x$ is the product of 
a smooth character $\delta_{x,0}:\gT^\vee(K)\to\kappa(x)^\times$ 
and an algebraic character $\delta_{x,1}\colon\gT^\vee(K)\to L^\times$. Let $\mu_{x,1}^\vee:\Gm\to\bT$ be the dual of the algebraic homomorphism underlying $\delta_{x,1}$. By \Cref{lemmaKPX6313}, there exists 
an element $\wdelta\in\bW$ such that 
\[ \mu_{x,1}^{\vee}=\Ad(\wdelta)\varpi_{\HT,x}. \]



\item[(c)] The $\PG$-module $D_x$ does not necessarily admit a triangulation of parameter $\delta_x$. 
%
%
However, \Cref{prop:boundary} produces a triangulation of a parameter $\wtl\delta_x\colon\gT^\vee(K)\to\Gm(L)$ such that $\wtl\delta_x\delta_x^{-1}$ is algebraic. 
By \Cref{lemmaKPX6313}, 
we can write $\wtl\delta_x=\wtl\delta_{x,0}\wtl\delta_{x,1}$ with $\wtl\delta_{x,0}$ smooth and $\wtl\delta_{x,1}$ algebraic. Let $\wtl\mu_{x,1}\colon\Gm\to\bT$ be the dual of the algebraic homomorphism underlying $\wtl\delta_{x,1}$. Again by \Cref{lemmaKPX6313}, there exists $w_{\sat}\in\bW$ such that 
\[ \wtl\mu_{x,1}^{\vee}=\Ad(w_{\sat})\varpi_{\HT,x}. \]
\end{itemize}

\begin{remark}
Note that $w_\delta$ only depends on $\delta$, while $w_\sat$ only depends on $\rho$ and $\delta_{x,0}$ (i.e. on the refinement of $\rho$ attached to $\delta$, in the sense of \cite[Definition 4.37]{dedar2020}). \\
The element $w_\delta$ is the analogue of the element $w$ defined before \cite[Lemma 3.7.4]{BHS19}. 
We introduced the notation $w_\delta$ in order to distinguish this element from an arbitrarily chosen $w\in\bW$. 

On the other hand, \Cref{wsatwx} shows that $w_\sat$ coincides with the element of the same name attached to the triangulation of $\rho_x$ of parameter $\wtl\delta_x$ by \Cref{def:wsat} (with $A=L$). The notation $w_\sat$ comes from \cite[Section 4.1]{mowlavi}.

As in \Cref{sec:defx}, we attach to $x$ a pair of flags $x_{\Wtri,\tri},x_{W^+,\HT}\in\bG/\bB(L)$, hence a point of $(\bG/\bB\times \bG/\bB)(L)=\bigsqcup_{w\in \bW}\bG(1,w)(\bB\times \bB)(L)$. By \Cref{wsatwx}, $w_\sat$ is the only element of $\bW$ such that $(x_{\Wtri,\tri},x_{W^+,\HT})\in g(w_\sat^{-1},1)(\bB\times \bB)(L)$ for some $g\in\bG$. In particular, $w_\sat$ is the analogue of the element $w_x$ defined in \cite[before Proposition 3.6.4]{BHS19}. 
\end{remark}

\begin{remark}\label{VX}
Analogously to \cite{BHS19}, we will show that the pair $(\wdelta,\wsat)$ locates the point $x$ over the scheme $X$ from \Cref{sec:Xnew}, namely the image of $x$ in $X$ will lie in $V^{\wsat}\cap X^{\wdelta}$: indeed, $\wsat$ determines the relative position of the Hodge--Tate filtration and the triangulation at $x$ (so that $x\in V^\sat$), while $\wdelta$ does so only over a punctured neighborhood of $x$ (so that $x$ only belongs to the Zariski closure $X^\wdelta$ of $V^{\wdelta}$). We refer to \Cref{three_seven_eight} for a proof of this fact. 
\end{remark}

\begin{definition}\label{def:dom}
We say that $x$ is \emph{dominant} (respectively, \emph{strictly dominant}) if $\mu_{x,1}^{\vee}$ (from (b) above) is dominant (respectively, strictly dominant). 
We say that $x$ is \emph{noncritical} if $w_\sat=1$. 
\end{definition}

\begin{remark}
Thanks to \Cref{wsatwx}, we know that $x$ is noncritical if and only the two flags $x_{\Wtri,\tri},x_{W^+,\HT}$ are in general position. We refer to \cite[Sections 2.2-2.3]{bhssmooth} for a discussion of the situation in the case $\gG=\GL_n$.
\end{remark}

Using our crystallinity criterion stated in \Cref{cryscriterion}, we prove a $\gG$-version of \cite[Lemma 3.12]{bhssmooth}. 
Recall that $\varpi_K$ is a uniformizer of $K$. Recall that $\Omega=\{\omega_1,\ldots,\omega_n\}$ denotes the set of quasi-fundamental weights of $\gG$, and for every $1 \leq i \leq n$ write $\omega_i=\lambda_i\omega_i'$ for a fundamental weight $\omega_i'$ of $\gG^\der$ (so that $\lambda_i\in\Z_{>0}$). Let $\alpha_1,\ldots,\alpha_n$ denote the corresponding simple roots. If $r_i$ is the simple reflection in $\gW$ attached to $\alpha_i$, then $r_i(\omega_i)=\omega_i-\lambda_i\alpha_i$ and $r_i(\omega_j)=\omega_j$ if $i\ne j$ : this follows from the fact that the fundamental weights form a basis of $X^\ast(\gT^\der)$ dual to the coroots. 
If $w\in\gW$ is arbitrary, writing it as a product of simple reflections shows that $w(\omega_i)$ is either $\omega_i-\lambda_i\alpha_i$ (if $r_i$ appears in $w$) or $\omega_i$ (if not).

Recall that the weight $\wt(\delta)$ of a parameter $\delta\colon T^\vee(K)\to L^\times$ is an element of $\frakt$. Given $\bw\in\bW$ and $t\in\frakt$, we write $\bw(t)$ for $\dot\bw(t)$, where $\dot\bw$ is a representative of $\bw$ in $N_\bG(\bT)$, with its adjoint action on $\frakt$.

\begin{lemma}\label{lemma:noncritical}
Assume that 
\begin{equation}\label{ktau} 
\wt_\tau(\delta_x\circ\alpha_i^\vee)>\max\{0, [K:K_0]v_p((\delta_x\circ\omega_i^\vee)(\varpi_K))\}.
\end{equation}
for every $\tau\colon K\into L$ and $1 \leq i \leq n$. 
Then $x$ is crystalline, strictly dominant and noncritical.
\end{lemma}



\begin{proof}
Let $V_i$ be the representation of $\gG$ of highest weight $\omega_i$. Since $D_\rig^\gG(\rho_x)$ admits a triangulation of parameter $\wtl\delta_x$ and $D_\rig^\gG(\rho_x)\times^\gG\GL(V_i)$ is étale, its subquotient $\cR_{K,L}(\wtl\delta_x\circ\omega_i^\vee)$ must have slope $\ge 0$, i.e.
\begin{equation}\label{vptilde} v_p((\wtl\delta_x\circ\omega_i^\vee)(\varpi_K))\ge 0. \end{equation}
By \Cref{prop:boundary}, $\wtl\delta_x\delta_x^{-1}$ is the algebraic character of weight $\wt(\wtl\delta_x)-\wt(\delta_x)$, and by \Cref{lemmaKPX6313}, $\wt(\delta_x)=\bw(\wt(\wtl\delta_x))$ for an element $\bw=(w_\tau)_\tau\in\bW$, with $w_\tau\in\gW$. Identifying weights with tuples of integers, \eqref{vptilde} implies 
\begin{align*} [K:K_0]v_p(\delta_x\circ\omega_i^\vee(\varpi_K))\ge\wt(\delta_x\circ\omega_i^\vee)-\wt(\wtl\delta_x\circ\omega_i^\vee)
&=\wt(\delta_x\circ\omega_i^\vee)-\bw(\wt(\delta_x\circ\omega_i^\vee)) \\ 
=\sum_\tau(\wt_\tau(\delta_x\circ\omega_i^\vee)-\wt_\tau(\delta_x\circ\omega_i^\vee)\circ w_\tau) &=
\sum_\tau\wt_\tau(\delta_x\circ(\omega_i-w_\tau(\omega_i))^\vee). \end{align*}
By the discussion before the lemma, $\omega_i-w_\tau(\omega_i)$ is either 0 or $\lambda_i\alpha_i$, so that if the reflection $r_i$ appears in $w_\tau$ the above equation amounts to 
\[ [K:K_0]v_p(\delta_x\circ\omega_i^\vee(\varpi_K))\ge\sum_\tau\lambda_i\wt_\tau(\delta_x\circ\alpha_i^\vee). \] 
By \eqref{ktau}, each term in the sum on the right-hand side is strictly positive and larger than the left-hand side, which leads us to a contradiction. We conclude that, for no choice of $i$ and $\tau$, the reflection $r_i$ appears in $w_\tau$, hence that $\bw$ is trivial.

If $\alpha$ is a positive root of $(\gG,\gB)$, then the $\tau$-weight of $\delta\circ\alpha^\vee$ is a linear combination of the $\wt_\tau(\delta_x\circ\alpha_i^\vee)$ with nonnegative coefficients, not all zero. Since each such $\tau$-weight is strictly positive by \eqref{ktau}, $x$ is strictly dominant. For the same reason, if $D^\triangle$ is a triangulation of $D$ of parameter $\delta$ and $\gU\subseteq\gB$ is the unipotent radical, then the Hodge--Tate weights of $\eta_{\Dtri}(\Lie(\gU))$ are slightly positive, hence $\rho$ is crystalline by \Cref{cryscriterion}.
\end{proof}

\subsection{Local irreducibility of the trianguline variety}\label{sec:locirr}


We combine the constructions of the previous sections to study the local geometry of $X_\tri(\rhobar)$. By \Cref{thm_equidimensional} (2), the open subspaces of the form $U_\tri^*(\rhobar)$ are smooth, so that the statements that we prove are only interesting for points in their complement.

Let $x=(\rho,\delta)$ be an element of $X_\tri(\rhobar)$ defined over a field $\kappa(x)$. We make the following assumptions on $x$, analogous to those in \cite[Remark 3.7.9]{BHS19}: 
\begin{enumerate}[label=\textup{(Reg)}]
    \item $\delta$ is locally algebraic and belongs to one of the sets $\cT^{\gT}_{*}(L)$ for $*\in\{\Lambda,r\}$, and the Hodge--Tate cocharacter $\varpi_{\HT}$ associated to $\rho$ is regular.\label{reg_hyp}
\end{enumerate}

We let $D$ be the $\gG$-$\PG$-module over $\cR_{K,L}$ whose corresponding fiber functor is $\eta_D=D_{\rig}\circ \eta_{\rho}$, and we let $\cM \colonequals D[\tfrac{1}{t}]$. By \Cref{prop:boundary}, $D$ admits a triangulation $D^{\triangle}$ of parameter $\wtl\delta$ such that $\wtl\delta\delta^{-1}$ is algebraic, and we let $\Mtri\colonequals D^{\triangle}[\tfrac{1}{t}]$. Such a triangulation is unique by  \Cref{uniqueness_Lambda_1} and \Cref{uniqueness_r_1}.

By \cite[Lemma 3.11]{defG} (compare \cite[Lemma 2.3.3 and Proposition 2.3.5]{kisinmoduli}), there is a canonical isomorphism of formal schemes
\begin{equation}\label{completerho} \wh \cX_{\ovl\rho,\rho}\cong X_\rho, \end{equation}
where $X_{\rho}$ is defined in \Cref{defXrho}. Explicitly, for $A\in\Art_{\kappa(x)}$, an $A$-point of $\wh\cX_{\ovl\rho,\rho}(A)$ corresponds to a map $\Spec A \to X_{\ovl\rho}$ sending the point $\lvert\Spec A\rvert$ to $\rho$, hence to a continuous representation $\rho_A\colon \GK\to \gG(A)$ that reduces to $\rho$ modulo $\mathfrak{m}_A$, which amounts to an $A$-point of $X_\rho$. 


The projection $X_\tri(\ovl\rho)\to\cX_{\ovl\rho}$ induces a morphism between formal completions
\begin{equation}\label{iotax} \iota_x \colon \wh{X_\tri(\ovl\rho)}_x\longrightarrow\wh\cX_{\ovl\rho,\rho}\cong X_\rho. \end{equation}

We prove the following $\gG$-valued version of \cite[Proposition 3.7.2]{BHS19}.

\begin{proposition}\label{factor_over_trianguline_deformations}
The morphism $\iota_x\colon \wh{X_\tri(\ovl\rho)}_x\to X_\rho$ factors through the closed immersion $X_{\rho,\cM^\triangle}\to X_\rho$, thus inducing a morphism 
\begin{align}
    \iota_x^\triangle \colon \wh{X_\tri(\ovl\rho)}_x\longrightarrow X_{\rho,\cM^\triangle}. \label{iotaxtri}
\end{align}
\end{proposition}

\begin{proof}
Let $f \colon X^\prime\to X$ be the map produced by \Cref{blowupLambda} and \Cref{blowupr}. 
Let $U$ be an open affinoid neighborhood of $x$ in $X_\tri(\ovl\rho)$, so that $\wh{X_\tri(\ovl\rho)}_x\cong \wh U_x$. Let $x^\prime\in X^\prime(\kappa(x))$ be a preimage of $x$ under $f$, and let $V^\prime$ be an open affinoid neighborhood of $x^\prime$ contained in $f^{-1}(U)$ and satisfying the properties of \Cref{blowupLambda,blowupr}. 
Set $D'=f^\ast(D\vert_{U})\vert_{V'}$ and $\cM'=D'[\tfrac{1}{t}]$: by our choice of $V'$, $\cM'$ admits a triangulation $\cM^{\triangle,'}$ of parameter $(f^\ast\delta)\vert_{V'}$. At $x'$, $\cM^{\triangle,'}$ specializes to the triangulation of $\cM'_{x'}\cong\cM_x$ of parameter $\delta$, which is unique by assumption \ref{reg_hyp} and \Cref{uniqueness_Lambda_1}, \Cref{uniqueness_r_1}. 
Let $\iota : \Spm A\to V^\prime$ be a rigid map with $A\in\Art_L$, mapping the unique point of $\Spm A$ to $x'$. In particular, $\iota$ factors through a map $\Spm A\to \wh{V'_{x'}}$, and via pull-back along $\Spm A\to \wh{V'_{x'}}\to U_x\to \wh{X_\tri(\ovl\rho)}_x\to X_\rho$, $\Spm A$ comes equipped with a lift $\rho_A$ of $\rho$. Pulling back the triangulation of $\cM'$ along $\iota$ gives a triangulation $\Mtri_A$ of $\iota^\ast \cM'=D_{\rig,A}^\gG(\rho_A)$ that reduces to $\cM_x$ modulo $\mathfrak{m}_A$. In particular, mapping $\wh{V'_{x'}}(A)$ to $(\rho_A,\Mtri_A)$ defines a morphisms of groupoids $\wh{V'_{x'}}\to X_{\rho,\Mtri}$. Now $\wh{V'_{x'}}\to X_{\rho,\Mtri}\to X_\rho$ factors through the dominant map $\wh{V'_{x'}}\to \wh U_x$, and since $X_{\rho,\Mtri}\to X_\rho$ is a closed immersion by base change from \Cref{triimmersion}, we conclude that $\wh U_x\to X_\rho$ also factors through a map $\wh U_x\to X_{\rho,\Mtri}$.
%
\end{proof}

The following is proved exactly as \cite[Proposition 3.7.3]{BHS19}.

\begin{proposition}\label{373}
    The maps $\iota_x$ and $\iota_x^\triangle$ from \eqref{iotax}, \eqref{iotaxtri} are closed immersions.
\end{proposition}


From now on, we assume that $x=(\rho,\delta)$ is a point of $X_\tri(\ovl\rho)$ such that $\delta$ locally algebraic. Using the map $\Theta$ defined in \Cref{defrelpos}, we define the composition
$$ \Theta_x\colon \wh{X_{\tri}(\rhobar)}_x \xrightarrow{\iota_x^{\triangle}} X_{\rho,\Mtri} \longrightarrow X_{\eta_{\rho},\Mtri}\cong X_{D,\Mtri} \xrightarrow{\Theta} \wh{T}_{(0,0)}. $$
For $w\in\bW$, let $\wh T_{(0,0)}^w$ be the completion of $T^w$ at $(0,0)$. Let $\wdelta\in\bW$ be the element attached to $(x,\delta)$ in \Cref{sec:weyl}. 

\begin{lemma}[{cf. \cite[Lemma 3.7.4]{BHS19}}]\label{Twdelta}
The map $\Theta_x$ factors through $\wh T^\wdelta_{(0,0)}\into \wh T_{(0,0)}$.
\end{lemma}
\begin{proof}
    We consider the compositions
 \begin{align*} \Theta_{x,\Wtri} &\colon \wh{X_{\tri}(\rhobar)}_x \xrightarrow{\iota_x^{\triangle}} X_{\rho,\Mtri} \to X_{\eta_{\rho},\Mtri}\cong X_{D,\Mtri} \to X_{W^+,\Wtri} \to X_{\Wtri} \xrightarrow{\kappa_{\Wtri}} \wh{\mathfrak{t}}, \\
     \Theta_{x,W^+} &\colon \wh{X_{\tri}(\rhobar)}_x \xrightarrow{\iota_x^{\triangle}} X_{\rho,\Mtri} \to X_{\eta_{\rho},\Mtri}\cong X_{D,\Mtri} \to X_{W^+,\Wtri} \to X_{W^+} \xrightarrow{\kappa_{W^+}} \wh{\mathfrak{t}}. \end{align*}
     Then by definition of $T^{w_{\delta}}$, we need to show that $\Theta_{x,W^+}=\Ad(w_{\delta}^{-1})\Theta_{x,\Wtri}$.

     Let $A\in \Art_L$, and $x_A\in \wh{X_{\tri}(\rhobar)}_x(A)$. Let $(W_A^+,\Wtri_A)$ be the corresponding point in $X_{W^+,\Wtri}(A)$. We let $\delta_A=w_\delta(x_A)$. Then by \Cref{BHS3.3.9}, we have $\Theta_{x,\Wtri}(x_A)=\wt(\delta_A)-\wt(\delta)$.

     Consider the $A$-points $y_1\colonequals\kappa_\Wtri(\Wtri_A)+\wt(\delta)$ and $y_2\colonequals\kappa_{W^+}(W_A^+)+\mathrm{d}\varpi_\HT$ of $\frakt$. Let $\ovl y_1,\ovl y_2$ be their images in $(\frakt\GIT\bW)(A)$, corresponding to two maps $\pi_1,\pi_2:\cR(\gG)\to A\otimes_{\Qp}K$. The map $\pi_1$ sends $[V,\rho_V]$ to $\tr(\rho_V(N_{W_A})+\wt(\delta))=\tr(\mathrm{d}\rho_V(\wt(\delta_A)))$, where the equality follows from \cite[Lemma 3.3.6 (ii)]{BHS19}. The map $\pi_2$ sends $[V,\rho_V]$ to $\tr(\rho_V(N_{W_A})+\pi_{W^+,\Sen}(V))$ coincides with $\pi_{W_A^+,\Sen}(V)$ by \cite[Lemma 3.7.5]{BHS19}, and the latter equals $\tr(\mathrm{d}\rho_V(\wt(\delta_A)))$ by \cite[Lemma 3.7.6]{BHS19}. 
Therefore $\overline{y}_1=\overline{y}_2$, which by \Cref{invariantlemmanew}, implies that there exists there exists $w\in\bW$ such that $\Ad(w)y_1=y_2$. Reducing modulo $\frakm_A$ gives $\Ad(w_\delta)(\wt(\delta))=\mathrm{d}\varpi_\HT$, so that $w=w_{\delta}^{-1}$. 
\end{proof}

We prove the following description of the trianguline variety at $x$.

\begin{proposition}[{cf. \cite[Corollary 3.7.8]{BHS19}}]\label{three_seven_eight}
    The closed immersion $\iota_x^\triangle \colon \wh{X_{\tri}(\rhobar)}_x \hookrightarrow X_{\rho,\Mtri}$ induces an isomorphism $\wh{X_{\tri}(\rhobar)}_x \eqto X_{\rho,\Mtri}^\wdelta$. In particular, $x\in V^{\wsat}\cap X^{\wdelta}$.
\end{proposition}

\begin{proof}
    The proof goes exactly as that of \cite[Corollary 3.7.8]{BHS19}, using the fact that the dimensions of $\wh\cO_{X_\tri(\rhobar),x}$ and $\lvert X_{\rho, \Mtri}\rvert$ coincide by the dimension part of \Cref{SB_representable} and \Cref{BHS3.6.2}. 

\end{proof}

\begin{remark}
    From the last statement of \Cref{three_seven_eight}, one deduces that $\wdelta\in\bW(x)$, hence $w_\sat\preceq\wdelta$ by \Cref{wsatpreceqw}.
\end{remark}

\begin{theorem}[{cf. \cite[Corollary 3.7.10]{BHS19}}]\label{three_seven_ten}
    If $x$ satisfies \ref{reg_hyp}, then $X_{\tri}(\rhobar)$ is normal and hence irreducible, and Cohen--Macaulay at $x$.
\end{theorem}

\begin{proof}
We deduce the result from the properties of the scheme $X$ via the following chain of morphisms:
\begin{align*}
    \wh{X_{\tri}(\rhobar)}_x \xrightarrow{\cong} X_{\rho, \Mtri}^{w_\delta} \leftarrow X_{\rho, \Mtri}^{\square, w_\delta} \rightarrow X_{\eta_\rho, \Mtri}^{\square, w_\delta} \xrightarrow{\cong} X_{D, \Mtri}^{\square, w_\delta} \rightarrow X^{\square, w_\delta}_{W^+, \Wtri} \xrightarrow{\cong} \widehat X^{w_\delta}_{x_\pdR}.
\end{align*}
The first isomorphism follows from \Cref{three_seven_eight}, and the last from \Cref{Xsqwprorep}. By \Cref{three_five_eleven}, and the discussion in \Cref{caseofgaloisrep}, every morphism in this chain is formally smooth. 

    The first part of the theorem then follows from \Cref{BHS3.6.2}(2), which asserts that $ X_{\rho, \Mtri}^{w_\delta} = |X_{\rho, \Mtri}^{w_\delta}|$ is pro-representable by a normal local ring. For the second part, since $X^{w_\delta}$ is Cohen--Macaulay by \cite[Proposition 2.3.3]{BHS19}, the completion $\wh X^{w_\delta}_{x}$ is also Cohen--Macaulay (see \stackcite{07NX}).
\end{proof}

\subsection{The accumulation property}\label{sec:acc}

Throughout this section, let $x = (\rho,\delta)$ be a closed point of $X_{\tri}(\rhobar)$. We still work under the assumption \ref{reg_hyp}.

\begin{proposition}[{cf. \cite[Proposition 4.1.1]{BHS19}}]\label{flatness_result}
    If $x$ satisfies \ref{reg_hyp}, then the map $\omega'$ from \eqref{def_omega_prime} is flat in a neighborhood of $x$.
\end{proposition}

\begin{proof}
    The proof of \cite[Proposition 4.1.1]{BHS19} goes through verbatim, using \cite[Proposition 2.3.3]{BHS19} once again and \Cref{thm_formal_smoothness} (use assumption \ref{reg_hyp}) in place of \cite[Theorem 3.4.4]{BHS19}, \Cref{G_351} in place of \cite[Proposition 3.5.1]{BHS19}, \Cref{Xsqwprorep} in place of \cite[Corollary 3.5.9]{BHS19} (the results of \Cref{sec35} requiring Hodge--Tate regularity) and \Cref{three_seven_eight} 
    in place of \cite[Corollary 3.7.8]{BHS19}.
\end{proof}

\noindent The \emph{weight space} $\cW^\gT$ is a rigid analytic space over $L$ representing the functor $A \mapsto \Hom^{\cont}(\gT^{\vee}(\cO_K), A^{\times})$. 
There is a flat morphism $\cT^\gT \to \cW^\gT$ obtained by restriction of characters $\gT^{\vee}(K) \to A^{\times}$ to $\gT^{\vee}(\cO_K)$.
We denote by $\omega$ the composition $X_{\tri}(\rhobar) \xrightarrow{\omega'} \cT^\gT \to \cW^\gT$. 

Following \cite[§3.3.1]{BC}, we say that a subset $Z \subseteq X_{\tri}(\rhobar)$ \emph{accumulates at a point $z$ of $X_{\tri}(\rhobar)$} if $z$ has a basis of open affinoid neighborhoods $U$ such that $U \cap Z$ is Zariski dense in $U$. 
Recall that in \Cref{sec:Gphimod} we attached a Frobenius operator $\varphi^f\in\gG(\Qpbar)$ to a crystalline representation $\GK\to\gG(L)$.

\begin{definition}\label{def_acc_property} Let $z$ be a point of $X_{\tri}(\rhobar)$ such that $\omega(z)$ is algebraic.
    We say that $X_{\tri}(\rhobar)$ has the \emph{accumulation property} at $z$ if for every $C > 0$ the set of crystalline strictly dominant points $z' = (\rho',\delta')$ such that
    \begin{enum}
        \item $\varphi^f$ is regular semisimple (\Cref{def:phiregsem}), 
        \item $\rho'$ is noncritical (\Cref{def:dom}), 
        \item for all positive roots $\alpha \in \Phi^+(\gB,\gT)$ the character $\omega(z') \circ \alpha^{\vee} \colon \cO_K^{\times} \to \gT^{\vee}(\cO_K) \to k(z')^{\times}$ is an algebraic character $\delta_{\mathbf k'} = \prod_{\tau \in \Sigma} \tau^{k_{\tau}'}$ with $\mathbf k' = (k_{\tau}')_{\tau \in \Sigma} \in \Z^{\Sigma}$ such that $k_{\tau}' > C$ for all $\tau \in \Sigma$,
    \end{enum}
    accumulates at $z$.
\end{definition}

\begin{proposition}[{cf. \cite[Proposition 4.1.4]{BHS19}}]\label{acc_property} Assume that $x=(\rho,\delta) \in X_{\tri}(\rhobar)$ satisfies \ref{reg_hyp} and that $\omega(x)$ is algebraic. Then $X_{\tri}(\rhobar)$ satisfies the accumulation property at $x$.
\end{proposition}

\begin{proof}
Clearly, it is enough to prove the statement for sufficiently large $C$. Recall that $\varpi_K$ is a uniformizer of $K$. After fixing a splitting $\gT\cong\Gm^n$ and writing $\delta=(\delta_1,\ldots,\delta_n)$ accordingly, assume throughout the proof that $C>[K:K_0]\sum_{j=1}^n \lvert v_p(\delta_j(\varpi_K))\rvert$. Let $U_x$ be an open affinoid neighborhood of $x$ such that $v_p(\delta_j')=v_p(\delta_j)$ for every $x'=(\rho',\delta')$ in $U_x$ and $1\le j\le n$ (with the obvious notation for the $\delta_j'$).

We first assume that $x$ is contained in the Zariski open and $p$-adically dense subset $V$ of $X_{\tri}(\rhobar)$ of \Cref{blowupLambda,blowupr}.
    The algebraic characters in $\cW^{\gT}$ that satisfy (3) of \Cref{def_acc_property} accumulate at every algebraic weight in $\cW^{\gT}$, in particular at $\omega(x)$. Among these, those that also satisfy (2) of \Cref{def_acc_property} still accumulate at $\omega(x)$ by a standard argument. 
    By \Cref{phireg}, condition (1) holds away from the set of points $(\rho',\delta')$ such that $\alpha(\delta')^f$ is an algebraic character for at least one root $\alpha$, and such a set is Zariski-closed (its image in the space of parameters $\cT^\gT$ is contained in the union of the translates of $\cT_{0,\gB}^{\gT}$ by characters of order dividing $f$). Therefore, the set of points satisfying (1) and (3) still accumulates at $x$. 
    Thanks to \Cref{cryscriterion}, \Cref{lemma:noncritical}, and our choice of large enough $C$, for every $x' = (\rho', \delta')\in U_x$ with $\omega(x') \in \cW^{\gT}$ an algebraic character as above, $\rho'$ is crystalline strictly dominant noncritical.
    By \Cref{flatness_result} and \cite[Corollary 5.11]{BoschLuetkebohmert}, $x$ admits an open affinoid neighborhood $U$ such that the map $\omega\vert_U$ is $p$-adically open. In particular, $\omega(V\cap U)$ is $p$-adically open in $\cW^{\gT}$ and contains $\omega(x)$, so that the accumulation property holds at $x$.

    We now reduce the general case to the case when $x \in V$.
    As before, there is an affinoid neighborhood $U$ of $x$ such that $\omega(U)$ is $p$-adically open in $\cW^\gT$.
    By \Cref{three_seven_ten}, $\widehat{\cO}_{U,x}$ is normal, and since the normal locus of the excellent ring $\cO(U)$ is Zariski open, we can assume that $U$ is normal after possibly shrinking it.
    Let $U' \subseteq U$ be an arbitrary connected affinoid open neighborhood of $x$.
    Since $V \cap U' \neq \emptyset$ (by the $p$-adic density in \Cref{blowupLambda,blowupr}, which then also holds in all affinoids), we may pick $v \in V \cap U'$.
    Let $Z$ be the Zariski closure in $U'$ of the set of points $x' \in U'$ as in \Cref{def_acc_property}.
    By the accumulation property at $v$, $Z$ contains a small affinoid neighborhood $V'$ of $v$.
    Since $U'$ is connected and normal, \cite[Lemma 2.1.4]{ConradIrrComp} implies $Z = U'$. (Note that ``analytic subsets'' in \emph{loc. cit.} correspond to Zariski closed subsets in our terminology.)
\end{proof}

\section{An application: Trianguline lifts of global mod $p$ representations}\label{sec:application}


As an immediate consequence of our results, we improve the main result \cite[Theorem 3.16]{Fakhruddin_2022} by removing the assumption that $\gG$ has no factors of type $G_2,F_4$ or $E_8$. 

Let $F$ be a number field, $S$ a finite set of places of $F$, $\gG$ a connected, split reductive group over the ring of integers $\cO$ of a $p$-adic field $E$, and $\ovl\rho \colon \cG_{F,S}\to G(k)$ a $t$-odd continuous representation, in the sense of \cite[Definition 3.15]{Fakhruddin_2022}. Let $\ovl\mu \colon \cG_F\to(\gG/\gG^\der)(k)$ be the homomorphism induced by $\ovl\rho$ and let $\mu\colon \cG_F\to\gG(\cO)$ be a group homomorphism lifting $\ovl\mu$.

\begin{theorem}
Assume that 
\begin{itemize}
\item $p\gg_\gG 0$, meaning that $p$ is larger than a constant depending only on the root datum of $\gG$ (see the proof of \cite[Proposition 6.8]{FKP21}),
\item $\ovl\rho$ satisfies \cite[Assumption A]{Fakhruddin_2022},
\item up to possibly replacing $E$ with a finite extension, there exists for every finite place $v$ of $F$ a lift of $\ovl\rho\vert_{\cG_{F_v}}$ to $\gG(\cO)$ with determinant $\mu\vert_{\cG_{F_v}}$, which is trianguline and very regular if $v\mid p$.
\end{itemize}
Then there exist a finite set of places $S'\supset S$, an extension $E'\supset E$ with ring of integers $\cO'$, and a lift of $\ovl\rho$ to a $t$-odd representation $\rho\colon \cG_{F,S'}\to\gG(\cO')$ with determinant $\mu$, which is very regular trianguline at all $p$-adic places of $F$. Furthermore, we can ensure that $\rho(\cG_{F,S'})$ contains an open subgroup of $\gG^\der(\cO')$.
\end{theorem}

\begin{proof}
The statement coincides with that of \cite[Theorem 3.16]{Fakhruddin_2022} without the assumption that $\gG$ has no factors of type $G_2,F_4$ or $E_8$, which only appears in \cite[Lemma 2.6]{Fakhruddin_2022}, which is quoted from \cite{dedar2020}. Without the extra assumption, statements (1), (2) of \cite[Lemma 2.6]{Fakhruddin_2022} are proved in our \Cref{thm_equidimensional}, and (3) in our \Cref{373}.
\end{proof}

%

\appendix

\section{Obstruction theory for non-abelian semilinear objects}
\label{sec_obs_thy}
The purpose of this appendix is to develop the obstruction theory for lifting non-abelian $H^1$-classes for $\PG$-modules.
We put ourselves in the following abstract setting.
Let $\Gamma$ be an arbitrary topological group, and let $\cO$ be a commutative topological\footnote{All proofs of this section also work for a condensed group $\Gamma$ and condensed coefficient rings.} ring.
Let $R$ be a commutative topological $\cO$-algebra, which has commuting $\cO$-linear actions of $\varphi$ and $\Gamma$, and assume that the action of $\Gamma$ is continuous.

Although in this paper we are only interested in the case $\Qp \subseteq \cO$, more specifically $\Gamma = \Gamma_K$ and $R=\cR^r_{K,A}$, there is no additional difficulty in formulating the results in this greater generality, and we do so for future reference.

We consider extension problems with values in algebraic groups, here the theory is very much parallel to that of group extensions and representations, but the computations are more involved.

Let $\gP$ be an $\cO$-group scheme, which sits in a short exact sequence
\begin{equation} \label{SES_lifting}
1 \to \gU \to \gP \to \gM \to 1,
\end{equation}
with $\gU$ a commutative closed subgroup scheme of $\gP$ (written additively), and we assume that there is a scheme-theoretic section $s : \gM \to \gP$.
Through that section, $\gU$ has a well-defined action $m \cdot u := s(m)us(m)^{-1}$ of $\gM$ by scheme-automorphisms and $\gP \cong \gU \times \gM$ as schemes.
We will describe the group multiplication of $\gP$ by a normalized algebraic $2$-cocycle $z : \gM \times \gM \to \gU$ by the following rule:
\begin{align}
    (u,m) \cdot (u',m') := (u + m \cdot u'+ z(m,m'), mm') \label{mult_in_P}
\end{align}

\subsection{Obstruction theory for $\Gamma$-modules}
\label{sec_obs_Gamma_mod}
Let us define an \emph{$\gM$-trivialized $\gM$-$\Gamma$-module} as an element of the set $Z^1_{\cont}(\Gamma, \gM(R))$.
Let us fix $c_0 \in Z^1_{\cont}(\Gamma, \gM(R))$.
We are interested in describing the set of liftings of $c_0$ along \eqref{SES_lifting} to an element of $Z^1_{\cont}(\Gamma, \gP(R))$.
We will see that it is a torsor under $Z^{1, \star}_{\cont}(\Gamma, \gU(R))$, where the $\star$ stands for a $c_0$-twisted action on $\gU(R)$:
\begin{align}
    \gamma \star u := c_0(\gamma) \cdot \gamma(u) \label{star_Gamma_action}
\end{align}

\begin{lemma}\label{lifting_sets_Gamma}
    For a continuous map $c' : \Gamma \to \gU(R)$ the condition that $(c',c_0) \in Z^1_{\cont}(\Gamma, \gP(R))$ is equivalent to
    \begin{align}
        \gamma_1 \star c'(\gamma_2) - c'(\gamma_1\gamma_2) + c'(\gamma_1) &= -z\left(c_0(\gamma_1), \gamma_1\big(c_0(\gamma_2)\big)\right) \label{cocyc_cond_Gamma}
    \end{align}
    for all $\gamma_1,\gamma_2 \in \Gamma$.
\end{lemma}

\begin{proof}
    For any $\gamma_1,\gamma_2\in \Gamma$, we have
    \begin{align*}
         & \big(c'(\gamma_1),c_0(\gamma_1)\big) \cdot \left(\gamma_1\big(c'(\gamma_2)\big), \gamma_1\big(c_0(\gamma_2)\big)\right) \\
         &= \left(c'(\gamma_1)+c_0(\gamma_1) \cdot \gamma_1\big(c'(\gamma_2)\big) + z\left(c_0(\gamma_1), \gamma_1\big(c_0(\gamma_2)\big)\right), c_0(\gamma_1)\gamma_1\big(c_0(\gamma_2)\big)\right)
    \end{align*}
    Using the fact that $c_0 \in Z^1_{\cont}(\Gamma, \gM(R))$ and the definition of the $\star$-action, we see that $(c',c_0) \in Z^1_{\cont}(\Gamma, \gP(R))$ if and only if for all $\gamma_1,\gamma_2 \in \Gamma$,
    $$ c'(\gamma_1\gamma_2) = c'(\gamma_1) + \gamma_1 \star c'(\gamma_2) + z\left(c_0(\gamma_1), \gamma_1\big(c_0(\gamma_2)\big)\right), $$
    which is \eqref{cocyc_cond_Gamma}.
\end{proof}

\begin{lemma}\label{fiber_is_torsor}
    Pointwise multiplication $Z^{1, \star}_{\cont}(\Gamma, \gU(R)) \times \Cont(\Gamma, \gP(R)) \to \Cont(\Gamma, \gP(R))$ preserves the fiber of $c_0$ in $Z^1_{\cont}(\Gamma, \gP(R))$ and turns it into a set-theoretic torsor under $Z^{1, \star}_{\cont}(\Gamma, \gU(R))$.
\end{lemma}

\begin{proof}
    The right hand side of \eqref{cocyc_cond_Gamma} is a function on $\Gamma \times \Gamma$ that does not depend on $c'$.
    The left hand side of \eqref{cocyc_cond_Gamma} is $\delta^{1,\star}(c')(\gamma_1,\gamma_2)$, where $\delta^{1,\star}$ is the boundary map for the $\star$-action.
    Hence multiplication of $(c',c_0)$ by an element of $Z^{1, \star}_{\cont}(\Gamma, \gU(R))$ (which amounts to adding a cocycle to $c'$) does not change this equation, and the difference of each two $c'$ satisfying \eqref{cocyc_cond_Gamma} is an element of $Z^{1, \star}_{\cont}(\Gamma, \gU(R))$.
    Hence the set of liftings is a torsor under $Z^{1, \star}_{\cont}(\Gamma, \gU(R))$.
\end{proof}

To get an obstruction theory, we need to see that the right hand side of \eqref{cocyc_cond_Gamma} is a $2$-cocycle.

\begin{lemma}\label{is_two_cocycle_Gamma}
    The function 
    $$\overline z \colon \Gamma \times \Gamma \longrightarrow \gU(R), ~(\gamma_1, \gamma_2) \mapsto z\left(c_0(\gamma_1), \gamma_1\big(c_0(\gamma_2)\big)\right)$$ is a $2$-cocycle for the $\star$-action.
\end{lemma}

\begin{proof}
    We need to show that for all $\gamma_1,\gamma_2,\gamma_3 \in \Gamma$,
\begin{align}
    \overline z(\gamma_1, \gamma_2)+\overline z(\gamma_1\gamma_2, \gamma_3)=\gamma_1 \star \overline z(\gamma_2, \gamma_3) + \overline z(\gamma_1, \gamma_2\gamma_3). \label{zbar_cocyc}
\end{align} 
The $2$-cocycle condition for $z$ says that for $m_1, m_2, m_3 \in \gM(R)$, we have
\begin{align}
    z(m_1, m_2) + z(m_1m_2, m_3) = m_1 \cdot z(m_2, m_3) + z(m_1, m_2m_3) \label{z_cocyc}
\end{align} 
in $\gU(R)$.
Taking $m_1 = c_0(\gamma_1)$, $m_2 = \gamma_1(c_0(\gamma_2))$ and $m_3 = \gamma_1\gamma_2(c_0(\gamma_3))$, we obtain
\begin{equation} \label{eq:z_cocycle_applied_Gamma} 
    \begin{split}
    &z\left(c_0(\gamma_1), \gamma_1\big(c_0(\gamma_2)\big)\right) + z\left(c_0(\gamma_1)\gamma_1\big(c_0(\gamma_2)\big), \gamma_1\gamma_2\big(c_0(\gamma_3)\big)\right) \\ 
    &= c_0(\gamma_1) \cdot z\left(\gamma_1\big(c_0(\gamma_2)\big), \gamma_1\gamma_2\big(c_0(\gamma_3)\big)\right) + z\left(c_0(\gamma_1), \gamma_1\big(c_0(\gamma_2)\gamma_2(c_0(\gamma_3))\big)\right). 
    \end{split}
\end{equation}
Using the fact that $c_0 \in Z^1_{\cont}(\Gamma, \gM(R))$, we have
\begin{align*}
    z\left(c_0(\gamma_1)\gamma_1\big(c_0(\gamma_2)\big), \gamma_1\gamma_2\big(c_0(\gamma_3)\big)\right) &= z\left(c_0(\gamma_1\gamma_2), \gamma_1\gamma_2\big(c_0(\gamma_3)\big)\right) = \overline z(\gamma_1\gamma_2,\gamma_3), \\
    z\left(c_0(\gamma_1), \gamma_1\big(c_0(\gamma_2)\gamma_2(c_0(\gamma_3))\big)\right) &= z\left(c_0(\gamma_1), \gamma_1\big(c_0(\gamma_2\gamma_3)\big)\right) = \overline z(\gamma_1, \gamma_2\gamma_3).
\end{align*}
By the definition of the $\star$-action, we have
\begin{align*}
    c_0(\gamma_1) \cdot \gamma_1(z(c_0(\gamma_2), \gamma_2(c_0(\gamma_3)))) = \gamma_1 \star z(c_0(\gamma_2), \gamma_2(c_0(\gamma_3))) = \gamma_1 \star \overline z(\gamma_2,\gamma_3).
\end{align*}
It follows that \eqref{eq:z_cocycle_applied_Gamma} is equivalent to \eqref{zbar_cocyc}, which completes the proof.
\end{proof}

We can now state the obstruction theory for $\Gamma$-modules over $R$.
The group $H^{2,\star}_{\cont}(\Gamma, \gU(R))$ is the defined as the $2$-nd cohomology group of the standard complex of continuous cocycles for the $\star$-action of $\Gamma$ on the abelian topological group $\gU(R)$. 

\begin{proposition}
    The equation \eqref{cocyc_cond_Gamma} is equivalent to $\delta^{1,\star}(c') = -\overline z$.
    Moreover, the torsor in \Cref{fiber_is_torsor} is trivial, if and only if $[\overline z] = 0$ in $H^{2,\star}_{\cont}(\Gamma, \gU(R))$.
\end{proposition}

\begin{proof}
    The first claim holds by definition.
    Clearly $[\overline z] = 0$ if and only if $\overline z$ is a boundary, if and only if the equation \eqref{cocyc_cond_Gamma} can be solved by some $c'$, if and only if the torsor is trivial.
\end{proof}

\subsection{The pro-cyclic case for $\Gamma$-modules}
\label{pro_cyclic_case}

Assume that $\Gamma$ is a pro-cyclic group which contains $\Zp$ as an open subgroup.
In particular, $\Gamma$ is a product of $\Zp$ and a finite cyclic group of order prime to $p$.
We fix a topological generator $\gamma_0\in\Gamma$.

\begin{lemma}\label{inverter_augmetation}
    The map $\gamma_0 - 1 \colon \Zp\br{\Gamma} \to \Zp\br{\Gamma}$ is injective with image equal to the augmentation ideal $I$ of $\Zp\br{\Gamma}$.
    In particular, there is a continuous inverse $I \to \Zp\br{\Gamma}, ~\gamma-1 \mapsto \tfrac{\gamma-1}{\gamma_0-1}$.
\end{lemma}

\begin{proof}
    Passing to finite quotients $\Gamma_i$ of $\Gamma$ an element of $\ker(\gamma_0-1)$ must map to $a \cdot \sum_{\gamma \in \Gamma_i} \gamma$ in $\Zp[\Gamma_i]$ for some $a\in \Zp$.
    Since $\Gamma$ contains an open pro-$p$ subgroup, it follows, that $a$ is infinitely $p$-divisible, so $a=0$.
    Since $\Zp\br{\Gamma}$ is profinite, the ideal generated by $\gamma_0 - 1$ is closed, so contains every element of the form $\gamma-1$ for $\gamma \in \Gamma$.
    It follows, that $(\gamma_0-1)$ is the augmentation ideal. As both $I$ and $\Zp\br{\Gamma}$ are profinite, the existence of a continuous inverse follows as well.
\end{proof}

\begin{lemma}\label{projectivity_lemma}
    The map $\big[\Zp\br{\Gamma} \xrightarrow{\gamma_0-1} \Zp\br{\Gamma}\big] \eqto \Zp\br{\Gamma^{\bullet+1}}$ given by the identity in degree $0$ and $\gamma \mapsto (\gamma, \gamma \gamma_0)$ in degree $1$ is a chain homotopy equivalence. 
\end{lemma}

Here $\Zp\br{\Gamma^{\bullet+1}}$ is the completed version of the usual bar resolution of the trivial $\Zp\br{\Gamma}$-module $\Zp$.

\begin{proof}
    By \Cref{inverter_augmetation} the first complex is a projective resolution of the trivial module.
    The claim follows by the usual argument once we know that $\Zp\br{\Gamma^{i+1}}$ is projective in the category of profinite $\Zp\br{\Gamma}$-modules. This follows by the adjunction property of $\Zp\br{\Gamma} \wtimes -$ once we know that $\Zp\br{\Gamma^i}$ is projective in the category of profinite $\Zp$-modules. Given a surjection $M \to M'$ of profinite $\Zp$-modules, there exists a continuous section $M' \to M$ by \cite[§I.1 Proposition 1]{SerreBook}. A continuous map $\Gamma^i \to M$ hence lifts to $M'$ and both extend uniquely to a map on $\Zp\br{\Gamma^i}$ by the universal property.
\end{proof}

Assume that we are in the setting of \Cref{sec_obs_Gamma_mod} with $\Zp \subseteq \cO$, and that $M$ is a limit of complete normed $\Zp$-modules (we do not assume the norm to be bounded).

\begin{lemma}\label{comparison_herr_Gamma_modules}
    The map 
    \begin{align}
        \Cont(\Gamma^{\bullet}, M) \longrightarrow \big[M \xrightarrow{\gamma_0-1} M\big] \label{comp_map_Gamma}
    \end{align}
    given by the identity in degree $0$ and evaluation at $\gamma_0$ in degree $1$ is a chain homotopy equivalence.
\end{lemma}

\begin{proof}
    This is an immediate consequence of \Cref{projectivity_lemma} and the isomorphism $\Hom_{\cont}(\Zp\br{\Gamma^{i+1}}, M) \cong \Cont(\Gamma^{i+1}, M)$. To prove the latter, we may assume that $M$ is a complete normed $\Zp$-module.
    By a standard argument $M$ has an open lattice.
    Since $\Zp\br{\Gamma^{i+1}}$ is compact any map $\Zp\br{\Gamma^{i+1}} \to M$ takes values in an open lattice.
    The claim now follows form the universal property of the map $\Gamma^{i+1} \to \Zp\br{\Gamma^{i+1}}$ with respect to maps into pro-discrete $\Zp$-modules.
\end{proof}

\subsection{The standard complex for $(\varphi, \Gamma)$-cohomology}
\label{sec_std_cplx}

Let $M$ be a topological $R$-module with continuous semilinear $(\varphi, \Gamma)$-action.
We define the \emph{standard complex} for $(\varphi, \Gamma)$-cohomology as
\begin{align}
    C^{\bullet}_{\varphi, \Gamma}(M) \colonequals \Tot \big[\Cont(\Gamma^{\bullet},M) \xrightarrow{\varphi-1} \Cont(\Gamma^{\bullet},M) \big]. \label{std_cplx}
\end{align}
Concretely, this means that 
$C^n_{\varphi, \Gamma}(M) = \Cont(\Gamma^{n-1}, M) \oplus \Cont(\Gamma^n, M)$ with boundary maps $$\delta^n_{\varphi,\Gamma}(c,d) = \big((\varphi-\id)(d) - \delta^{n-1}(c), \delta^n(d)\big),$$ where
\begin{align*}
    \delta^n(f)(\gamma_1, \dots, \gamma_{n+1}) &= \gamma_1 f(\gamma_2, \dots, \gamma_{n+1}) \\
    &+ \sum\limits_{i=1}^{n} (-1)^i f(\gamma_1, \dots, \gamma_{i-1}, \gamma_i \gamma_{i+1}, \gamma_{i+2}, \dots, \gamma_{n+1}) + (-1)^{n+1} f(\gamma_1, \dots, \gamma_n)
\end{align*}
is the usual boundary map of the complex $\Map(\Gamma^{\bullet}, M)$ of inhomogeneous cochains.

We write $Z^i_{\varphi, \Gamma}(M)$, $B^i_{\varphi, \Gamma}(M)$ and $H^i_{\varphi, \Gamma}(M)$ for the groups of cocycles, coboundaries and cohomology classes of the complex $C^{\bullet}_{\varphi, \Gamma}(M)$. 
For convenience we spell out $\delta^n_{\varphi,\Gamma}$ in low degrees:
\begin{align*}
    \delta^0_{\varphi,\Gamma}(d) &= \big(\varphi(d)-d, \big[\gamma \mapsto \gamma(d) - d\big]\big), \\
    \delta^1_{\varphi,\Gamma}(c,d) &= \big(\big[\gamma \mapsto \varphi(d(\gamma))-d(\gamma) -(\gamma(c)-c)\big], \big[(\gamma_1,\gamma_2) \mapsto \gamma_1d(\gamma_2) - d(\gamma_1\gamma_2)+d(\gamma_1)\big]\big), \\
    \delta^2_{\varphi,\Gamma}(c,d) &= \big(\big[(\gamma_1,\gamma_2) \mapsto \varphi(d(\gamma_1,\gamma_2)) - d(\gamma_1,\gamma_2) - \big(\gamma_1c(\gamma_2)-c(\gamma_1\gamma_2)+c(\gamma_1)\big)\big], \delta^2(d)\big).
\end{align*}

The maps $\delta^0_{\varphi,\Gamma}$ and $\delta^1_{\varphi,\Gamma}$ also make sense in a non-abelian setting, where $M$ is just a group equipped with commuting actions of $\varphi$ and $\Gamma$. 
Our convention is then $\delta^0_{\varphi,\Gamma}(d) = (d^{-1}\varphi(d), [\gamma \mapsto d^{-1}\gamma(d)])$.
If $\gH$ is an $\cO$-group scheme, we will consider $\gH(R)$ with the induced $\varphi$- and $\Gamma$-actions or some twisted action in \Cref{obstruction_PG}. If $\gH$ is commutative, the entire standard complex still makes sense and we will also use $\delta^2_{\varphi,\Gamma}$ under this assumption.

\subsection{Obstruction theory for $\gH$-$(\varphi, \Gamma)$-modules}
\label{obstruction_PG}
Let $\gH$ be an algebraic group over $\mathcal{O}$. The set of $\gH$-trivialized $\gH$-$(\varphi, \Gamma)$-modules can be described in terms of nonabelian $(\varphi,\Gamma)$-$1$-cocycles.
In the spirit of the non-abelian first boundary map $\delta^1_{\varphi,\Gamma}$ from the previous section, we make the following definition:

\begin{definition}
A \emph{nonabelian $(\varphi, \Gamma)$-$1$-cocycle} with values in $\gH(R)$ is a pair $(w, c) \in  \gH(R) \times Z^1_{\cont}(\Gamma,  \gH(R))$ satisfying 
\begin{align}
    w \varphi(c(\gamma)) = c(\gamma) \gamma(w) \label{PG_commutation_rule}
\end{align}
for all $\gamma \in \Gamma$. We denote the set of nonabelian $(\varphi, \Gamma)$-$1$-cocycles by $Z^1_{\varphi, \Gamma}(\gH(R))$.
\end{definition}

Fix $(w_0,c_0) \in Z^1_{\varphi, \Gamma}(\gM(R))$.
The set of lifts of $(w_0,c_0)$ along \eqref{SES_lifting} to an element of $Z^1_{\varphi, \Gamma}(\gP(R))$ will be a torsor under $Z^{1, \star}_{\varphi, \Gamma}(\gU(R))$. Here the $\star$ stands for the $c_0$-twisted action of $\varphi$ and $\Gamma$ on $\gU(R)$, where the $\Gamma$-action is defined as in \eqref{star_Gamma_action}, and the $\varphi$-action is given by
\begin{align}
    \varphi \star u &:= w_0 \cdot \varphi(u). \label{star_varphi_action}
\end{align}

\begin{lemma}\label{lifting_sets}
    For $w' \in \gU(R)$ and a continuous map $c' \colon \Gamma \to \gU(R)$ the condition that $((w',w_0),(c',c_0)) \in Z^1_{\varphi, \Gamma}(\gP(R))$ is equivalent to \eqref{cocyc_cond_Gamma} and
    \begin{align} 
        w' + \varphi \star c'(\gamma) - (c'(\gamma) + \gamma \star w') &= z\big(c_0(\gamma),\gamma(w_0)\big) - z\big(w_0, \varphi\big(c_0(\gamma)\big)\big) \label{PG_additive_commutation_rule}
    \end{align}
    for all $\gamma \in \Gamma$.
\end{lemma}

\begin{proof}
    For \eqref{cocyc_cond_Gamma} we proceed as in the proof of \Cref{lifting_sets_Gamma}.    
    We do the commutation rule \eqref{PG_additive_commutation_rule}.
    On the left hand side of \eqref{PG_commutation_rule} we have
    \begin{align*}
        (w',w_0) \cdot \left(\varphi\big(c'(\gamma)\big), \varphi\big(c_0(\gamma)\big)
    \right) &= \left(w' + w_0 \cdot \varphi\big(c'(\gamma)\big) + z\left(w_0, \varphi\big(c_0(\gamma)\big)\right), w_0\varphi\big(c_0(\gamma)\big)\right).
    \end{align*}
    On the right hand side, we have
    \begin{align*}
        \big(c'(\gamma), c_0(\gamma)\big) \cdot\big(\gamma(w'), \gamma(w_0)\big) &= \left(c'(\gamma) + c_0(\gamma) \cdot \gamma(w') + z\big(c_0(\gamma),\gamma(w_0)\big), c_0(\gamma) \gamma(w_0)\right).
    \end{align*}
    Since $(w_0,c_0)$ satisfies \eqref{PG_commutation_rule}, the condition can be rewritten as
    \begin{align*}
        w' + w_0 \cdot \varphi\big(c'(\gamma)\big) - \big(c'(\gamma) + c_0(\gamma) \cdot \gamma(w')\big) = z\big(c_0(\gamma),\gamma(w_0)\big) - z\left(w_0, \varphi\big(c_0(\gamma)\big)\right).
    \end{align*}
    This is \eqref{PG_additive_commutation_rule} using the definition of the $\star$-actions.
\end{proof}

\begin{lemma}\label{fiber_is_torsor_PG}
    Pointwise multiplication $Z^{1, \star}_{\varphi,\Gamma}(\gU(R)) \times \gP(R) \times \Cont(\Gamma, \gP(R)) \to \gP(R) \times \Cont(\Gamma, \gP(R))$ preserves the fiber of $(w_0,c_0)$ in $Z^1_{\varphi,\Gamma}(\gP(R))$  and turns it into a set-theoretic torsor under $Z^{1, \star}_{\varphi,\Gamma}(\gU(R))$.
\end{lemma}

\begin{proof}
    The right hand side of \eqref{PG_additive_commutation_rule} is a function on $\Gamma$ that does not depend on $(w',c')$.
    The same applies to \eqref{cocyc_cond_Gamma} as in the proof of \Cref{fiber_is_torsor}.
    The left hand side of \eqref{PG_additive_commutation_rule} together with the left hand side of \eqref{PG_additive_commutation_rule} are $\delta^1_{\varphi, \Gamma}(w',c')(\gamma_1, \gamma_2)$.
    We conclude in the same way as in \Cref{fiber_is_torsor}.
\end{proof}

On the right hand side of \eqref{PG_additive_commutation_rule}, the negative of the $2$-cocycle associated to $z$ shows up:

\begin{lemma}\label{is_two_cocycle_PG}
    The pair $$ \tilde z = (F, \overline z) := \left(\big[\gamma \mapsto z\big(w_0, \varphi(c_0(\gamma))\big) - z\big(c_0(\gamma),\gamma(w_0)\big)\big], \big[(\gamma_1, \gamma_2) \mapsto z\big(c_0(\gamma_1), \gamma_1(c_0(\gamma_2))\big)\big]\right)$$ is a $(\varphi,\Gamma)$-$2$-cocycle in the standard complex \eqref{std_cplx} for the $\star$-twisted actions on $\gU(R)$, i.e. $\delta^{2,\star}_{\varphi, \Gamma}(\tilde z) = 0$.
\end{lemma}


\begin{proof} The second entry of the equation $\delta^2_{\varphi, \Gamma}(\tilde z) = 0$ is condition \eqref{zbar_cocyc}.
The first entry of the equation $\delta^2_{\varphi, \Gamma}(\tilde z) = 0$ is the condition:
\begin{align}
\varphi \star  \overline z(\gamma_1, \gamma_2) - \overline z(\gamma_1, \gamma_2) - F(\gamma_1) - \gamma_1 \star F(\gamma_2) + F(\gamma_1\gamma_2) = 0. \label{first_entry}
\end{align}
We expand each term of \eqref{first_entry} using the definition of $F$, $\overline{z}$, and the $\star$-actions:
\begin{align*}
    \varphi \star \overline z(\gamma_1, \gamma_2) &= w_0 \cdot z\left(\varphi\big(c_0(\gamma_1)\big), \varphi\big(\gamma_1(c_0(\gamma_2))\big)\right), \\
    \overline z(\gamma_1, \gamma_2) &= z\left(c_0(\gamma_1), \gamma_1
    \big(c_0(\gamma_2)\big)\right), \\
    F(\gamma_1) &= z\left(w_0, \varphi\big(c_0(\gamma_1)\big)\right) - z\big(c_0(\gamma_1), \gamma_1(w_0)\big), \\
    \gamma_1 \star F(\gamma_2) &= c_0(\gamma_1) \cdot z\left(\gamma_1(w_0), \gamma_1\big(\varphi(c_0(\gamma_2))\big)\right) - c_0(\gamma_1) \cdot z\left(\gamma_1\big(c_0(\gamma_2)\big), \gamma_1\gamma_2(w_0)\right), \\
    F(\gamma_1\gamma_2) &= z\left(w_0, \varphi\big(c_0(\gamma_1\gamma_2)\big)\right) - z\big(c_0(\gamma_1\gamma_2), \gamma_1\gamma_2(w_0)\big).
\end{align*}
The left hand side of \eqref{first_entry} becomes
\begin{equation} \label{expansion_1}
    \begin{split}
    \mathrm{LHS} = & + w_0 \cdot z\left(\varphi\big(c_0(\gamma_1)\big), \varphi\big(\gamma_1(c_0(\gamma_2))\big)\right) - z\big(c_0(\gamma_1), \gamma_1\big(c_0(\gamma_2)\big)\big) \\
    & - z\left(w_0, \varphi(c_0(\gamma_1)\big)\right) + z\big(c_0(\gamma_1), \gamma_1(w_0)\big) \\
    & - c_0(\gamma_1) \cdot z\left(\gamma_1(w_0), \gamma_1\big(\varphi(c_0(\gamma_2))\big)\right) + c_0(\gamma_1) \cdot z\big(\gamma_1(c_0(\gamma_2)), \gamma_1\gamma_2(w_0)\big) \\
    & + z\left(w_0, \varphi\big(c_0(\gamma_1\gamma_2)\big)\right) - z\big(c_0(\gamma_1\gamma_2), \gamma_1\gamma_2(w_0)\big).
    \end{split}
\end{equation}
Taking in \eqref{z_cocyc} $m_1 = w_0$, $m_2 = \varphi(c_0(\gamma_1))$ and $m_3 = \varphi(\gamma_1(c_0(\gamma_2)))$ we obtain
\begin{equation} \label{eq:z_cocycle_applied_first} 
    \begin{split}
    0 = & \ z\left(w_0, \varphi\big(c_0(\gamma_1)\big)\right) + z\left(w_0\varphi(c_0(\gamma_1)), \varphi\big(\gamma_1(c_0(\gamma_2))\big)\right) \\ 
    & - w_0 \cdot z(\varphi(c_0(\gamma_1)), \varphi(\gamma_1(c_0(\gamma_2)))) - z\left(w_0, \varphi\big(c_0(\gamma_1)\gamma_1(c_0(\gamma_2))\big)\right).
    \end{split}
\end{equation}
Adding \eqref{expansion_1} + \eqref{eq:z_cocycle_applied_first}, then using the fact that $c_0 \in Z^1(\Gamma, \gM(R))$ to cancel the 7th term of \eqref{expansion_1} with the 4th term of \eqref{eq:z_cocycle_applied_first}, we obtain
\begin{equation} \label{expansion_2}
    \begin{split}
        \mathrm{LHS} = & - z\left(c_0(\gamma_1), \gamma_1
        \big(c_0(\gamma_2)\big)\right) + z\big(c_0(\gamma_1), \gamma_1(w_0)\big) \\
        & - c_0(\gamma_1) \cdot z\left(\gamma_1(w_0), \gamma_1\big(\varphi(c_0(\gamma_2))\big)\right) + c_0(\gamma_1) \cdot z\left(\gamma_1\big(c_0(\gamma_2)\big), \gamma_1\gamma_2(w_0)\right) \\
        & - z\big(c_0(\gamma_1\gamma_2), \gamma_1\gamma_2(w_0)\big) + z\left(w_0\varphi(c_0(\gamma_1)), \varphi\big(\gamma_1(c_0(\gamma_2))\big)\right).
    \end{split}
\end{equation}
Taking in \eqref{z_cocyc} $m_1 = c_0(\gamma_1)$, $m_2 = \gamma_1(w_0)$ and $m_3 = \gamma_1(\varphi(c_0(\gamma_2)))$, we obtain
\begin{equation} \label{eq:z_cocycle_applied_second} 
    \begin{split}
    0 = & \ z\big(c_0(\gamma_1), \gamma_1(w_0)\big) + z\left(c_0(\gamma_1)\gamma_1(w_0),\gamma_1\big(\varphi(c_0(\gamma_2))\big)\right) \\ 
    & - c_0(\gamma_1) \cdot z\left(\gamma_1(w_0), \gamma_1\big(\varphi(c_0(\gamma_2))\big)\right) - z\left(c_0(\gamma_1), \gamma_1\big(w_0\varphi(c_0(\gamma_2))\big)\right).
    \end{split}
\end{equation}
Subtracting \eqref{expansion_2} $-$ \eqref{eq:z_cocycle_applied_second}, using the commutation rule to cancel the 6th term of \eqref{expansion_2} with the 2nd term of \eqref{eq:z_cocycle_applied_second} we obtain
\begin{equation} \label{expansion_3}
    \begin{split}
        \mathrm{LHS} = & - z\left(c_0(\gamma_1), \gamma_1\big(c_0(\gamma_2)\big)\right) + c_0(\gamma_1) \cdot z\left(\gamma_1\big(c_0(\gamma_2)\big), \gamma_1\gamma_2(w_0)\right) \\
        & - z\big(c_0(\gamma_1\gamma_2), \gamma_1\gamma_2(w_0)\big) + z\left(c_0(\gamma_1), \gamma_1\big(w_0\varphi(c_0(\gamma_2))\big)\right).
    \end{split}
\end{equation}
Taking in \eqref{z_cocyc} $m_1 = c_0(\gamma_1)$, $m_2 = \gamma_1(c_0(\gamma_2))$ and $m_3 = \gamma_1(\gamma_2(w_0))$, using the commutation rule $w_0\varphi(c_0(\gamma_2)) = c_0(\gamma_2)\gamma_2(w_0)$ from \eqref{PG_commutation_rule}, we see that \eqref{expansion_3} $= 0$.
\end{proof}

We conclude with the statement of the obstruction theory for $(\varphi, \Gamma)$-modules over $R$.

\begin{lemma}\label{final_obs_lemma}
    The equations \eqref{cocyc_cond_Gamma} and \eqref{PG_additive_commutation_rule} are equivalent to $\delta^{1,\star}_{\varphi, \Gamma}(w',c') = -\tilde z$.
    Moreover, the torsor in \Cref{fiber_is_torsor_PG} is trivial if and only if $[\tilde z] = 0$ in $H^{2,\star}_{\varphi, \Gamma}(\gU(R))$.
\end{lemma}
\begin{proof}
    By \Cref{is_two_cocycle_PG}, we have that $\tilde z \in Z^{2,\star}_{\varphi, \Gamma}(\gU(R))$, so the class $[\tilde z] \in H^{2,\star}_{\varphi, \Gamma}(\gU(R))$ is well-defined.
    Clearly $[\tilde z] = 0$ if and only if $\tilde z$ is a boundary if and only if the equations \eqref{cocyc_cond_Gamma} and \eqref{PG_additive_commutation_rule} can be solved by some $(w',c')$ if and only if the torsor is trivial.
\end{proof}

\bigskip


\printbibliography

\bigskip

\end{document}